\documentclass[10]{amsart}
\usepackage{caption}
\usepackage{longtable}
\usepackage{latexsym}
\usepackage{amsmath,mathtools}
\usepackage{lscape} 
\usepackage{pdflscape}
\usepackage{float}
\usepackage{color}
\usepackage{amssymb}
\newtheorem{thm}{Theorem}[section]
\newtheorem{con}[thm]{Conjecture}

\newtheorem{lem}[thm]{Lemma}
\newtheorem{cor}[thm]{Corollary}
\newtheorem{pro}[thm]{Proposition}

\theoremstyle{remark}
\newtheorem{rem}[thm]{Remark}

\theoremstyle{definition}
\newtheorem{Def}[thm]{Definition}

\numberwithin{equation}{section}

\newcommand{\e}{\text{e}}               
\newcommand{\R}{\mathbb{R}}           
\newcommand{\C}{\mathbb{C}}           
\newcommand{\N}{\mathbb{N}}           
\newcommand{\Z}{\mathbb{Z}}         
\newcommand{\LL}{\mathbb{L}}
\newcommand{\Q}{{\mathbb{Q}}}           
\newcommand{\HH}{{\mathbb{H}}}           

\newcommand{\Rrank}{\text{rank}_{\mathbb{R}}}

\newcommand{\rank}{\text{rank}}
\newcommand{\End}{\text{End}}
\newcommand{\Tr}{\text{Tr}} 
\newcommand{\Hom}{\mathrm{Hom}}
\newcommand{\id}{\mathrm{id}}
\newcommand{\Ind}{\text{Ind}}
\newcommand{\Ad}{\text{Ad}}
\newcommand{\ad}{\text{ad}}
\newcommand{\supp}{\mathrm{supp}}
\newcommand{\spec}{{\mathrm{spec}}}
\newcommand{\sgn}{{\mathrm{sgn}}}
\newcommand{\vol}{{\mathrm{vol}}}

\newcommand{\ind}{{\chi}}
\newcommand{\Aut}{{\mathrm{Aut}}}

\newcommand{\fa}{{\mathfrak a}}             
\newcommand{\fb}{{\mathfrak b}}
\newcommand{\fc}{{\mathfrak c}}
\newcommand{\fg}{{\mathfrak g}}
\newcommand{\fd}{{\mathfrak d}}
\newcommand{\ff}{{\mathfrak f}}
\newcommand{\fgl}{{\mathfrak{gl}}}
\newcommand{\fsp}{{\mathfrak{sp}}}
\newcommand{\fspin}{{\mathfrak{spin}}}
\newcommand{\fh}{{\mathfrak h}}
\newcommand{\fk}{{\mathfrak k}}
\newcommand{\fl}{{\mathfrak l}}
\newcommand{\fm}{{\mathfrak m}}
\newcommand{\fn}{{\mathfrak n}}

\newcommand{\fo}{{\mathfrak{so}}}
\newcommand{\fq}{{\mathfrak q}}
\newcommand{\ft}{{\mathfrak t}}
\newcommand{\fu}{{\mathfrak u}}
\newcommand{\fr}{{\mathfrak r}}
\newcommand{\fs}{{\mathfrak s}}
\newcommand{\fz}{{\mathfrak z}}

\newcommand{\gt}{\theta}

\newcommand{\vp}{\varphi}
\newcommand{\ve}{\varepsilon}

\newcommand{\Cal}{\mathcal}

\newcommand{\be}{\begin{equation}}
\newcommand{\beu}{\begin{equation*}}

\newcommand{\IP}[2]{\langle#1\,, #2\rangle}     

\newcommand{\defn}{{:=}}
\renewcommand{\Im}{\mathrm{Im}}

\begin{document}

\baselineskip=18pt
\sloppy
\title[Harmonic analysis on standard quotients]
{Harmonic analysis on compact standard quotients of non-Riemannian semisimple symmetric spaces}
\author{S. Mehdi}
\author{M. Olbrich}
\thanks{{\em Mathematics Subject Classification MSC2020}. Primary: 22E46; Secondary: 22E40, 53C30, 58J50, 11F72}
\address{Université de Lorraine, CNRS, IECL, F-57000 Metz, France}
\email{salah.mehdi@univ-lorraine.fr}
\address{Facult\'e des Sciences, de la Technologie et de la Communication , Universit\'e du Luxembourg, Luxembourg}
\email{martin.olbrich@uni.lu}
\maketitle

\begin{abstract}
Let $Y=\Gamma\backslash G/H$ be a compact standard quotient of a non-Riemannian semisimple symmetric space $X=G/H$. We investigate the spectral decomposition of the algebra ${\bf D}(X)$ of $G$-invariant differential operators on $X$ acting on $L^2(Y)$. The absence of elliptic invariant differential operators makes the spectral theory fundamentally different from the Riemannian case. We first show that standard quotients arise from {\it transitive} actions on $X$ of real reductive subgroups $L$ of $G$ containing the discrete subgroup $\Gamma$. Our approach is based on the geometry of properly transitive triples $(G,H,L)$. We derive explicit formulas expressing the Casimir operator of $G$ in terms of Casimir operators of $L$. Triples fall into two classes: Type I and Type II. For triples of Type I, we prove essential self-adjointness of invariant differential operators and discreteness of the corresponding spectral decomposition. This decomposition is illustrated by a detailed analysis of compact standard quotients of anti-de Sitter spaces. In contrast, Type II triples exhibit genuinely continuous spectral phenomena. A central theme of the paper is the interaction between the representation theories of $G$ and $L$. For Type I triples, we prove $L$-admissibility of $H$-spherical $G$-representations of finite length and establish multiplicity formulas. We show that the resulting correspondence defines a map between irreducible spherical $L$-representations and spherical $G$-representations. For triples of both types, we obtain a representation-theoretic description of eigendistributions via distributional matrix coefficients. As an application, we show that every integrable discrete series representation of $G/H$ contributes an infinite-dimensional family of $L^2$-eigenfunctions on every compact standard quotient of Type I.
\end{abstract}

\tableofcontents

\section{Introduction}\label{Intro}

Throughout this introduction, $G$ will be a non-compact connected semisimple real Lie group with finite center. We fix a maximal compact subgroup $K\subset G$. 

\subsection{A general question.}\label{gq} 
 Let $\Gamma\subset G$ be a uniform lattice. The theory of automorphic forms on 
$\Gamma\backslash G$ can be viewed as the study of admissible representations of $G$ occurring in various functions spaces on $\Gamma\backslash G$. In particular, one wishes to decompose the right regular representation on $L^2(\Gamma\backslash G)$ into irreducible unitary representations. Since $\Gamma\backslash G$ is compact, $L^2(\Gamma\backslash G)$ decomposes discretely with finite multiplicities (see e.g. \cite{GGPS})
\begin{equation}\label{GGPS}
L^2(\Gamma\backslash G)\cong \widehat{\bigoplus_{\pi\in\widehat{G}}}N_\Gamma(\pi)\otimes V_\pi,
\end{equation}
where $\widehat{G}$ is the set of equivalence classes of irreducible unitary representations $(\pi,V_\pi)$ of $G$ and $N_\Gamma(\pi)$
is the (finite dimensional) multiplicity space of $V_\pi$ in $L^2(\Gamma\backslash G)$. Although usually proved using convolution operators, the decomposition could be seen as a consequence of the spectral theorem for the elliptic operator given by the Casimir of $G$ acting on sections of vector bundles over the  compact Riemannian locally symmetric space $\Gamma\backslash G/K$. We mention that even for arbitrary discrete subgroups 
$\Gamma\subset G$ there is a direct integral decomposition involving possible continuous contributions
\begin{equation}\label{GGPSter}
L^2(\Gamma\backslash G)\cong \int^{\oplus}_{\pi\in\widehat{G}}M_\Gamma(\pi)\widehat{\otimes} V_\pi\; d\mu_\Gamma(\pi)
\end{equation}
by the abstract Plancherel Theorem (see e.g. \cite{Di}, or  \cite[Ch. 14]{Wa92}). For some classes of such $\Gamma$'s, for instance non-uniform lattices or convex cocompact subgroups, one can make the continuous contributions quite explicit in terms of Eisenstein series (Langlands \cite{La76}, Bunke-Olbrich \cite{BO00}). Although most of what follows can be formulated for more general discrete subgroups, we will keep the assumption that $\Gamma$ is cocompact in~$G$. 

Classically, one is interested in the joint eigenfunctions with respect to the center ${\Cal Z}({\mathfrak g})$ of the complex enveloping algebra ${\Cal U}({\mathfrak g})$ of the Lie algebra ${\mathfrak g}$ of $G$. More precisely, we fix a character $\chi:{\Cal Z}({\mathfrak g})\rightarrow{\mathbb C}$. The corresponding space of automorphic forms
could be defined as follows (where the subscript $K$ always stands for $K$-finite vectors):
%
\begin{eqnarray*}
{\mathcal A}_{\chi,K}(\Gamma\backslash G)&=&\{f\in C^\infty(\Gamma\backslash G)_{K}\mid\; Df=\chi(D)f,\;\;\forall D\in {\Cal Z}({\mathfrak g})\}.
\end{eqnarray*}
For our purposes it is convenient to drop the $K$-finiteness condition and consider the spaces
\begin{eqnarray*}
{\mathcal A}_{\chi,\infty}(\Gamma\backslash G)&=&\{f\in C^\infty(\Gamma\backslash G)\mid\; Df=\chi(D)f,\;\;\forall D\in{\Cal Z}({\mathfrak g})\},\\
{\mathcal A}_{\chi}(\Gamma\backslash G)&=&\{f\in L^2(\Gamma\backslash G)\mid\; Df=\chi(D)f,\;\;\forall D\in {\Cal Z}({\mathfrak g})\}.
\end{eqnarray*}
Since the number of representations $\pi\in\widehat{G}$ with fixed infinitesimal character $\chi_\pi=\chi$ is finite, we get the following decompositions:
\begin{equation*}
{\mathcal A}_{\chi,\star}(\Gamma\backslash G)\cong\bigoplus_{\{\pi\in\widehat{G}\mid\;\chi_\pi=\chi\}}N_\Gamma(\pi)\otimes V_{\pi,\star}\ ,
\end{equation*}
where $\star=K,\infty,\emptyset$ ($V_{\pi,\infty}$ is the subspace of  smooth vectors in $V_\pi$), and the inclusions
\begin{equation*}
{\mathcal A}_{\chi,K}(\Gamma\backslash G)\subset{\mathcal A}_{\chi,\infty}(\Gamma\backslash G)\subset {\mathcal A}_{\chi}(\Gamma\backslash G)
\end{equation*}
with dense image. In particular, the decomposition (\ref{GGPS}) can be coarsened a bit to 
\begin{equation}\label{GGPSbis}
L^2(\Gamma\backslash G)\cong \widehat{\bigoplus}_{\chi\in\widehat{{\Cal Z}({\mathfrak g})}}{\mathcal A}_{\chi}(\Gamma\backslash G),
\end{equation}
which gives in turn the joint spectral decomposition of the (commutative) algebra ${\Cal Z}({\mathfrak g})$ acting as unbounded operators on $L^2(\Gamma\backslash G)$. Here 
$\widehat{{\Cal Z}({\mathfrak g})}$ denotes the set of $*$-characters, i.e $\displaystyle{\chi(A^\star)=\overline{\chi(A)}}$ for all $A\in {\Cal Z}({\mathfrak g})$. 
The set of eigencharacters 
\begin{equation*}
\{\chi\in\widehat{Z({\mathfrak g})}\mid\;{\mathcal A}_{\chi}(\Gamma\backslash G)\neq\{0\}\} 
\end{equation*}
of $\Gamma\backslash G$ is discrete and closed in $\widehat{Z({\mathfrak g})}$ and therefore coincides with the spectrum $\text{Spec}(\Gamma\backslash G)$ of the action of ${\Cal Z}({\mathfrak g})$ on $L^2(\Gamma\backslash G)$. Moreover, if $G$ has no compact factors, the spaces ${\mathcal A}_{\chi}(\Gamma\backslash G)$ are all infinite-dimensional, except possibly when $\chi$ is the trivial character.

In contrast to (\ref{GGPS}) the decomposition (\ref{GGPSbis}) does not explicitly refer to the group structure of $G$. In fact, only its underlying structure as a pseudo-Riemannian symmetric 
space seems to matter: $\Cal Z(\fg)$ is the algebra of (bi-)invariant differential operators of the symmetric space $G\cong G\times G/\Delta G$, where $\Delta G\subset G\times G$ is the diagonal subgroup, and $\Gamma\backslash G$ is a compact locally symmetric space
covered by $G$. 


Thus one might ask whether the decomposition (\ref{GGPSbis}) is true in the more general setting of compact locally symmetric spaces. Namely, suppose $X=G/H$ is a semisimple symmetric space covering a compact locally symmetric space $Y=\Gamma\backslash G/H$, also called a compact Clifford-Klein form or shortly a compact quotient of $X$. Here we are only interested in the case that $H$ is non-compact, i.e. that $X$ is non-Riemannian. The compact case $H=K$ is already covered by (\ref{GGPSbis}) (one just has to take $K$-invariants). Let ${\bf D}(G/H)$ be the (commutative) algebra of $G$-invariant differential operators which acts naturally on $C^\infty(\Gamma\backslash G/H)$. It can be considered as an algebra of unbounded operators acting on $L^2(\Gamma\backslash G/H)$. 
It contains the pseudo-Riemannian Laplacian $\Delta_Y$, which is induced by the Casimir of $G$, but no elliptic operator. For a character $\chi:{\bf D}(G/H)\rightarrow{\mathbb C}$ of ${\bf D}(G/H)$, define the spaces
\begin{equation}\label{q0}
E_{\chi}^{\pm \infty}(\Gamma\backslash G/H):=\{f\in C^{\pm\infty}(\Gamma\backslash G/H)\mid\; Df=\chi(D)f\;\;\forall D\in {\bf D}(G/H)\},
\end{equation}
where $C^{-\infty}(\Gamma\backslash G/H)$ is the space of distributions on $\Gamma\backslash G/H$, and 
$$E^{(2)}_\chi(\Gamma\backslash G/H):=E^{-\infty}_\chi(\Gamma\backslash G/H)\cap L^2(\Gamma\backslash G/H).$$ 
In this setting, one can ask whether one has a discrete spectral decomposition 
\begin{equation}\label{q1}
L^2(\Gamma\backslash G/H)\cong \widehat{\bigoplus}_{\chi\in\widehat{{\bf D}(G/H)}}E_{\chi}^{(2)}(\Gamma\backslash G/H)
\end{equation}
or at least a direct integral decomposition
\begin{equation}\label{q2}
L^2(\Gamma\backslash G/H)\cong \int^\oplus_{\chi\in\widehat{{\bf D}(G/H)}}{\mathcal M}_\chi\; d\mu(\chi)
\end{equation}
where $\chi\mapsto{\mathcal M}_\chi$ is a measurable field of Hilbert spaces, $\mu$ is a Borel measure on $\widehat{{\bf D}(G/H)}$ and $D\in {\bf D}(G/H)$ acts by $\chi(D)$ on ${\mathcal M}_\chi$. Note that in case (\ref{q2}) holds, ${\mathcal M}_\chi$ can be identified, as a vector space, with a subspace of $E_{\chi}^{-\infty}(\Gamma\backslash G/H)$ ($\mu$-almost everywhere). Of course, one is interested not only in the existence of such a decomposition but also in a more precise description of the involved eigenspaces and spectral measures.

These questions were raised for the first time in the 1990's by Kobayashi \cite{Kos96} and have initiated a wide research program. The most far reaching published results in this direction have been obtained by Kassel and Kobayashi \cite{KK2},\cite{KK25}, see also the announcement papers \cite{KK11},\cite{KK20}. Here we present our own results addressing these questions which were obtained during the last two decades. Some of these results were previously announced in \cite{MO21}. 

In order to understand these questions properly, there are several issues of different nature which should be discussed first.

\subsection{Geometric issues.}\label{gi} 
Besides group manifolds there are apparently not so many non-Riemannian semisimple symmetric spaces admitting compact quotients. For instance, the de Sitter space  
$$SO_e(1,n+1)/SO_e(1,n)\ , \quad n\geq 1,$$ 
does not admit compact quotients (Calabi-Markus phenomenon \cite{CM62}). The same is true for the anti-de Sitter space $SO_e(2,n)/SO_e(1,n)$ with $n$ odd \cite{Ku81}.
The situation changes for even $n$. In this case, a basic standard construction applies. A discrete subgroup $\Gamma\subset G$ defines such a compact quotient if and only if it acts properly, cocompactly and freely on $G/H$.  If we have a (connected) reductively embedded group $L\subset G$ acting properly and cocompactly
on the symmetric space $G/H$, then any torsion free cocompact lattice $\Gamma\subset L$ defines a compact quotient $\Gamma\backslash G/H$. We will show, see Prop.~\ref{cocpt}, that in this situation $L$ has to act transitively (except if $G$ has compact factors).
Moreover, it is a classical result of Borel \cite{Bo63} that every reductive group $L$ admits such lattices. For the anti-de Sitter space with even $n$, we can take $L=U(1,\frac{n}{2})$.
We call compact quotients constructed in this way standard quotients.

In particular, standard quotients arise via triples $(G,H,L)$ with $G/H$ symmetric, $H$ non-compact, and  
\begin{itemize}
\item[(i)] $G=LH$,
\item[(ii)] $L\cap H$ is compact.
\end{itemize}
We will call such triples properly transitive triples, for the precise technical conditions see Definition \ref{defck}. Thanks to the important work of  Oni\v s\v cik
on decompositions of semisimple Lie groups $G=LH$ \cite{Oni62}\cite{Oni69}, all properly transitive triples, in particular the symmetric spaces $G/H$ admitting standard quotients, are known, at least in principle, see 
Theorem~\ref{listck} and the remarks following it.  Note that the quotients $\Gamma\backslash G$ of group manifolds appear naturally among the standard quotients.
Some group manifolds, however, admit also standard quotients that are significantly different from these basic examples, see Table 2, Case~12 and Remark \ref{remindecomp}. If we restrict considerations to symmetric spaces with $G$ simple, then there are - up to local isometry - five series of non-Riemannian symmetric spaces and five sporadic spaces $G/H$
admitting standard quotients, see Table 1.   

It is conjectured in \cite{KY05} that a symmetric space admits compact quotients if and only if it admits a standard one. This does not mean that every compact quotient is standard. Some
standard quotients have deformations which are no longer standard quotients (see Goldman \cite{Gol85}, Ghys \cite{Ghy87}, Kassel\cite{Ka12}, and Kannaka-Kobayashi \cite{KaK25}). Moreover, there
is a recent construction \cite{MST23} of quotients of $SO_e(2,2n)/U(1,n)$ (which admits standard quotients) that are completely different from standard quotients and their deformations. However, for some symmetric spaces one might expect that all compact quotients are standard (see \cite{Ze98}). We will restrict our considerations mainly to standard quotients which makes many things (see for instance the discussion of (\ref{matcoef}) below) more accessible due to the presence of the group $L$.

One could even consider the more general task of spectral theory of compact quotients of homogeneous spaces $G/H$ of reductive type (with $H$ non-compact) with respect to a certain commutative algebra of locally invariant operators containing the Laplacian. What has been said above holds more
or less also in this case. In particular, we can speak of standard quotients. Again, Oni\v s\v cik's results provide a complete list of spaces admitting standard quotients, at least for 
simple $G$. It turns out that this list is not much longer than the corresponding list of symmetric spaces, see again the discussion after  Theorem~\ref{listck}. Therefore we decided to stay
in the framework of locally symmetric spaces, where we can take advantage of the far developed theory of harmonic analysis on symmetric spaces $G/H$ (see e.g. \cite{vdBFJS}).

It is not known whether a compact locally symmetric space which is locally isometric to $G/H$ is of the form $\Gamma\backslash G/H$. This would be true if such spaces were always geodesically complete, as it is the case for Lorentzian spaces with constant curvature (see  Klingler \cite{Kli96}, Carri\`ere \cite{Car89}). But even if there were non-complete examples, a spectral theory for them would be really problematic. Indeed, in this case the Laplacian $\Delta_Y$ is expected not to be essentially self-adjoint and one has to distinguish a self-adjoint extension of $\Delta_Y$ as an additional datum, see the next subsection.

\subsection{Analytic issues.}\label{anal} For a general (compact) 
quotient $Y=\Gamma\backslash G/H$, the algebra ${\bf D}(G/H)$ acts on $L^2(\Gamma\backslash G/H)$ via unbounded 
non-elliptic 
operators. It is not clear whether a decomposition as in (\ref{q2}) (or even as in (\ref{q1})) exists at all, and if it exists whether it is unique. When ${\bf D}(G/H)$ is generated by a single formally self-adjoint operator $D$, as it is the case for rank one symmetric spaces $G/H$ with $D=\Delta_Y$, the question is whether $D$, originally defined on $C^\infty(\Gamma\backslash G/H)$, has a self-adjoint extension and whether this extension is unique, i.e., whether $D$ is essentially self-adjoint. Note that the Laplacian $\Delta_Y$ always has a self-adjoint extension while uniqueness
is not a priori clear. In higher rank, even the existence of a spectral decomposition is a quite subtle problem. But there are sufficient criteria again in terms of essential self-adjointness assuring existence and uniqueness. We will review the relevant general theory behind that in Section \ref{G2}.

This issue is related to a conjecture of Colin de Verdi\`ere and Le Bihan, known as the {\it quantum completeness problem} (see \cite{CL22}, Conjecture 1.2):
\vspace*{0.2cm}

{\it Let $M$ be a closed smooth manifold equipped with a smooth density. Suppose $D$ is a formally self-adjoint operator of degree at most $2$ with smooth coefficients on $M$. Then $D$ is essentially self-adjoint if and only if the Hamiltonian flow of its symbol is complete.}
\vspace*{0.2cm}\\
\noindent 
Colin de Verdi\`ere - Le Bihan established this conjecture in various cases, including Laplacians on some Lorentzian surfaces \cite{CL22}. We observe that in the case where $M$ is a compact  pseudo-Riemannian manifold, the above conjecture says that the Laplacian $\Delta_M$ is essentially self-adjoint if and only if $M$ is geodesically complete.
Since  compact quotients are geodesically complete the conjecture suggests that our problem of finding {\em the} spectral decomposition of ${\bf D}(G/H)$ on $Y$ might be well-posed.
Vice versa, some of our results (see (D) and ({G}) below) will confirm the conjecture for Laplacians on certain pseudo-Riemannian locally symmetric spaces.

Concerning the construction of eigenfunctions or eigendistributions on $Y$ the following strategy seems natural. Though the group $G$ does not act on the double coset space $\Gamma\backslash G/H$, one would like to describe the eigenspaces $E_\chi^\star(\Gamma\backslash G/H)$ defined in (\ref{q0}) in terms of the representation theory of $G$. We first note that for every admissible $H$-(co)spherical representation $(\rho,W_\rho)$ of $G$, distributional matrix coefficients provide a map 
\begin{equation}\label{matcoef}
(W_{\widetilde{\rho},-\infty})^\Gamma\otimes (W_{\rho,-\infty})^{H}_\chi\longrightarrow E_\chi^{-\infty}(\Gamma\backslash G/H),
\end{equation}
where $\displaystyle{(W_{\rho,-\infty})^{H}_\chi}$ denotes of $H$-invariant distributions vectors in $W_\rho$ transforming via the character $\chi$ under ${\bf D}(G/H)$ and 
$\displaystyle{(W_{\widetilde{\rho},-\infty})^\Gamma}$ is the space of $\Gamma$-invariant distribution vectors for the dual representation $W_{\tilde{\rho}}$ of $W_\rho$. 
It turns out that even for general compact quotients, the sum of all of these maps exhausts $E_{\chi}^{-\infty}(\Gamma\backslash G/H)$ (see Section \ref{eigen}). However, it is not clear (and even not true in general) that it is sufficient to restrict to irreducible or even unitary representations $\rho$. While the spaces $(W_{\rho,-\infty})^{H}_\chi$ are finite dimensional and well understood, the primary aim is to gain insight into the spaces $(W_{\widetilde{\rho},-\infty})^\Gamma$. 
We are not aware of general principles or even construction methods providing such an insight - except for two special cases, see below.
For instance, the scattering theoretic approach of \cite{BO00}, to the analoguous problem for (convex) cocompact subgroups $\Gamma$ of real rank one groups $G$, seems
to generalize from groups to symmetric spaces (of real rank one) at a first glance, only. The generalization would involve the embedding of $X$ into a higher dimensional symmetric space and the construction of a kind of Eisenstein series associated to this space. As soon as it comes to their meromorphic continuation the analogy does not persist anymore, i.e.
for a proof of a possible meromorphic continuation completely new ideas are required. Similar remarks apply to the approach \cite{DG16} via resonant states of the geodesic flow. Note, however, that this approach has recently lead to very interesting spectral theoretical results for certain non-compact and incomplete Lorentzian locally symmetric spaces, namely quasi-Fuchsian anti de-Sitter $3$-manifolds \cite{DGM25}.   

The two special cases mentioned above are the following: If $W_\rho$ is an integrable discrete series representation for $X=G/H$, then the corresponding integrable   eigenfunctions on $X$ can be averaged over 
$\Gamma$ providing (at least weak) eigenfunctions on $Y$ (Poincar\' e series). Of course, one has to show that this procedure does not only produce the zero eigenfunction. 
This is the approach of \cite{KK2}. Or one assumes that the quotient $Y=\Gamma\backslash X$ is standard  corresponding to a triple $(G,H,L)$. Then one can describe 
$(W_{\widetilde{\rho},-\infty})^\Gamma$ via representations of $L$ having $\Gamma$-invariants (Lemma~\ref{kirchner}). This will be the main approach of the present paper.

\subsection{Strategy and main results.} 

We now proceed to list and explain the paper's main results, all of which concern compact standard quotients of symmetric spaces. As a starting point we establish
\begin{itemize}
\item[({\bf A})] {\em Assume that $G$ has no compact factor. Let  $L,H\subset G$ be reductively embedded closed subgroups with finitely many components. If $L$ acts properly and cocompactly
on $G/H$, then it acts transitively.} (Prop.~\ref{cocpt})
\end{itemize}
The proof 
rests on some previously known consequences of properness and cocompactness for the non-compact parts of the Lie algebras $\fg,\fl,\fh$ of $G,L,H$, respectively, and a basic fact about the fundamental group of the Lagrangian Grassmannian. The result was already mentioned in \cite{MO21}. Different proofs for some special cases can be found also in \cite{BoTr25}.

This result tells us in particular that every standard quotient $Y$ of a symmetric space $X=G/H$  is of the form
$$   Y=\Gamma\backslash G/H \cong \Gamma\backslash L/L\cap H  $$
for some properly transitive triple $(G,H,L)$ (see (i) and (ii) above)  and some cocompact lattice $\Gamma\subset L$. Since $L\cap H$ is compact, we conclude from (\ref{GGPS}) that
\begin{equation}\label{decompo}  
L^2(Y)\cong L^2(\Gamma\backslash L)^{L\cap H}\cong  \widehat{\bigoplus_{\pi\in\widehat{L}}}N_\Gamma(\pi)\otimes V_\pi^{L\cap H}.
\end{equation}
This is a decomposition of the action of ${\bf D}(G/H)$. In fact, the action of the  larger algebra ${\bf D}(L/L\cap H)\supset {\bf D}(G/H)$ of $L$-invariant differential operators on $X$ can be described
via the  regular representation on $C^\infty(\Gamma\backslash L)$ of $L\cap H$-invariant elements in the universal enveloping algebra ${\Cal U}(\fl)$. Hence the decomposition (\ref{decompo}) even respects the action of ${\bf D}(L/L\cap H)$. It would be already close to a spectral decomposition for ${\bf D}(G/H)$ if the spaces  $V_\pi^{L\cap H}$, $\pi\in\widehat L$, were finite dimensional or, at least, would split into finitely many joint eigenspaces for ${\bf D}(G/H)$. 

Next we introduce geometric conditions on a triple $(G,H,L)$ that will ensure these properties.
We call a triple {\em spherical} if not only $L$, but also the minimal parabolic $P_L$ of $L$ acts transitively on $X=G/H\cong L/L\cap H$ (Def.~\ref{sphertriples}). In other words, the homogeneous space $L/L\cap H$
is a real spherical homogeneous space. Real spherical homogeneous spaces $L/Q$ are natural generalizations of symmetric spaces. Harmonic analysis on them is nowadays one of the most active research fields in representation theory of real reductive groups, see e.g. \cite{DKKS}, \cite{SV}, \cite{KKS}. Of course, this research focuses on the difficult situation where $Q$ is non-compact, while in our situation $Q=L\cap H$ is compact. 

It will be natural to distinguish not too much between triples $(G,H,L)$ and $(G,H,L')$ if $L$ and $L'$ differ only  by compact factors, i.e. we will
consider them as equivalent. Unfortunately, the property of being spherical is not preserved by such equivalences, in general. To remedy the situation, we say that a triple $(G,H,L)$
is of Type~I, if for a maximal properly acting subgroup $L_{max}\supset L$ of $G$ the triple $(G,H,L_{max})$ is spherical. Otherwise, we say that the triple is of Type~II (Def.~\ref{types}). 
For instance, the triple 
\be\label{T1}
(SO_e(2,2n), SO_e(1,2n),U(1,n))\ , 
\end{equation}
which is responsible for standard quotients of anti-de Sitter spaces, is of Type~I, while 
\be\label{T9}
(SO_e(4,4), SO(3)\times SO_e(1,4), Spin(3,4))
\end{equation} 
is of Type~II. Now we have
\begin{itemize}
\item[({\bf B})] {\em A properly transitive triple $(G,H,L)$ is spherical if and only if $\dim V_\pi^{L\cap H}<\infty$ for all $\pi\in\widehat L$} (Cor.~\ref{inv3}). {\em Moreover, for any triple 
$(G,H,L)$ of Type~I the algebra ${\bf D}(G/H)$ acts by a single character $\chi_\pi$ on  $V_\pi^{L\cap H}$, $\pi\in \widehat L$} (Prop. \ref{staun}).
\end{itemize}
While the first statement has a direct proof using Casselman's subrepresentation theorem, the proof of the second is based on the classification of Type~I triples and involves multiplicity one results for many triples (Prop.~\ref{mufree}).
We remark that the first statement could also be viewed as a special case of results for general spherical homogeneous spaces $L/Q$ (\cite{KO13}, \cite{KS16}) (which have a much more involved proof).

Let us now discuss the classification of properly transitive triples $(G,H,L)$. We call a triple indecomposable if it is not equivalent to a product of two triples. 
Certainly, the classification of general triples is reduced to the classification of indecomposable triples. If the underlying symmetric space $G/H$ is irreducible, in which case we also call the triple
$(G,H,L)$ irreducible, then $(G,H,L)$ is indecomposable. Note that the converse is not true. However, we have

\begin{itemize}
\item[({\bf C})] {\em Indecomposable triples of Type~I are irreducible} (Prop.~\ref{productck}). {\em  Up to
equivalence, the list of all irreducible properly transitive triples  (of both types) is given in Tables 1 and 2} (Thm.~\ref{listck}).
\end{itemize}
As a consequence, triples of Type~I are completely classified. The second part of this classification result is an immediate consequence of Oni\v s\v cik's results \cite{Oni69} on decompositions
of semisimple Lie groups, as already mentioned in Subsection~\ref{gi}. With a bit more effort, it is in fact possible to deduce from \cite{Oni69} also a classification of indecomposable non-irreducible triples
(which are necessarily of Type~II). They arise from certain diagonal constructions applied to the triples appearing in Tables 1 and 2. We do not formulate the complete classification result in this paper.
We remark that the underlying symmetric spaces of such indecomposable triples can have arbitrarily many irreducible factors, but at most two, if we would insist (what we don't do here) that there is no group manifold among the irreducible factors of $G/H$. The most interesting indecomposable triples with two
factors are written down at the end of Section \ref{triples}.

Now we return to spectral decompositions of $L^2(Y)$. We first observe that  (\ref{decompo}) combined with ({B}) implies 
\begin{itemize}
\item[({\bf D})] {\em Let $(G,H,L)$ be a triple of Type~I, and let $\Gamma\subset L$ be a cocompact lattice. Then there is a unique and discrete spectral decomposition of the form} (\ref{q1})
{\em which induces corresponding topological decompositions of $C^{\pm\infty}(\Gamma\backslash G/H)$. All formally self-adjoint elements of   ${\bf D}(G/H)$ are essentially self-adjoint on $\Gamma\backslash G/H$.} (Thm.~\ref{selim})
\end{itemize}

To get some information on the spectrum also in case of Type~II triples and more information for Type~I triples
we express the Casimir operator $\Omega_G\in {\bf D}(G/H)$, which acts as $\Delta_Y$ on $L^2(Y)=L^2(\Gamma\backslash G/H)$, in terms of elements of ${\Cal U}(\fl)$.
This should lead to some understanding of the action of  $\Delta_Y$ on the summands of (\ref{decompo}) indexed by $\pi\in\widehat L$,
in particular - in the case of a Type~I triple - to the determination of the corresponding eigenvalues $\chi_\pi(\Omega_G)$. We give a general procedure leading to such expressions
in terms of the principal angles between $\fl$ and $\fh$ in $\fg$ (Lemma \ref{sense}). We make the result completely explicit for all properly transitive triples
$(G,H,L)$ with $G$ simple, i.e. for the triples appearing in Table~1. The result is in terms of the Casimir $\Omega_L$ of $L$ and of Casimir operators of certain subgroups of $L$. Let $K_L\subset L$ be a maximal compact subgroup containing $L\cap H$, and let $\sigma\in\Aut(G)$ be the involution defining the symmetric space structure of $G/H$.
\begin{itemize}
\item[({\bf E})] {\em Let $(G,H,L)$ be a properly transitive triple with $G$ simple. For each such triple, there are explicitly known non-zero rational constants $\alpha,\beta$ and 
$\gamma$ such
that the following identities hold in $D(L/L\cap H)$:  
\begin{eqnarray*} \Omega_G &= &\alpha\Omega_L+\beta\Omega_{K_L},\quad\quad\quad\mbox{ if }(G,H,L)\mbox{ is of Type~I},\\
\Omega_G& =& \alpha\Omega_L+\beta\Omega_{K_L} +\gamma\Omega_{L\cap\sigma(L)},\  \mbox{ if }(G,H,L)\mbox{ is of Type~II}.
\end{eqnarray*}}
(Prop.~\ref{casimir})
\end{itemize}
Note that the group $L\cap\sigma(L)$ is non-compact in the relevant cases. For instance, for the triple (\ref{T9}) it is given by the normal real form of the exceptional group
$G_2$ sitting in $L=Spin(3,4)$. The appearance of such non-compact ``correction terms" like $\Omega_{L\cap\sigma(L)}$  which, moreover, do not commute with $\Omega_{L\cap K}$,
led to the - at least for us -  surprising  discovery that standard quotients associated to Type~II triples can hardly admit a discrete spectral decomposition of the form (\ref{q1}).
On the other hand, the above formulas for triples of Type~I, when combined with knowledge about the representations $\pi\in\widehat L$, $V_\pi^{L\cap H}\ne\{0\}$,  and some properties of the corresponding multiplicities $\dim N_\Gamma(\pi)$, would lead to a rather clear picture of the spectrum of $\Delta_Y$ for standard quotients $Y$ associated to Type~I triples.
We implement this in detail for the Triple (\ref{T1}), i.e. for compact standard quotients of the anti-de Sitter space. Note that the anti-de Sitter space is a symmetric space of rank one,
hence  ${\bf D}(G/H)$ is generated by the Laplacian $\Delta_Y$.

\begin{enumerate}
\item[({\bf F})] {\em Let $X$ be the anti-de Sitter space of dimension $2n+1$, $n\ge1$, $\Gamma\subset L=U(1,n)$ discrete, cocompact and torsion free. Let $Y=\Gamma\backslash X$.
Then the set of eigenvalues $\mu$ of $\Delta_Y$ splits into three disjoint subsets:
\begin{itemize}
\item $B_+:=\{\ell^2-n^2\mid\ell\in\N\}$. All these eigenvalues appear with infinite multiplicity.
\item A (possibly empty) finite set $B_0\subset (1-n^2,0)\setminus B_+$. These eigenvalues appear with finite multiplicity. 
\item An unbounded set $B_-\subset (-\infty,1-n^2)$.
\end{itemize}
}
(Thm.~\ref{something})
\end{enumerate}
Via (\ref{decompo}), these three subsets correspond roughly to the subsets of discrete series, complementary series and unitary principal series representations of  $\{\pi\in\widehat L\mid V_\pi^{L\cap H}\ne\{0\}\}$, respectively. 

In order to show that the eigenspaces for the eigenvalues $\mu\in B_+$ are infinite dimensional, we determine all discrete series
representations $\pi$ of $L$ having an $L\cap H$ invariant vector, observe that infinitely many of them contribute to a given $\mu=\ell^2-n^2$, and establish a non-vanishing result
for almost all of the multiplicity spaces $N_\Gamma(\pi)$ (Prop.~\ref{crux} (c),(d) and Prop.~\ref{pavle}). For $\ell\ge 2$, this non-vanishing result is a direct consequence of known multiplicity formulas for discrete series representations. For $\ell=1$, additional arguments are needed. We remark that the corresponding discrete series representation is integrable if and only if $\ell\ge n+1$. 

We know more about the size of $B_-$ than just that it is unbounded, see 
the proof of Thm.~\ref{something}.
But for $n\ge 2$, we still do not know enough in order
to say something about the possible accumulation points of $B_-$ and eigenvalues $\mu\in B_-$ of infinite multiplicity, i.e. about the negative part of the essential spectrum of $\Delta_Y$.
For $n=1$ this part of the essential spectrum is unbounded as well (Cor.~\ref{late}). One might expect the same also for $n\ge 2$. This is certainly a topic for further research. Recall that in the automorphic form case (\ref{GGPSbis}) all eigenvalues occur with infinite multiplicity, and there are no accumulation points.
Similar questions about accumulation points of eigenvalues of compact standard quotients  remain unanswered also by the spectral theoretic results  ({G}) and ({K}) that follow.

For Type~II triples with $G$ simple, it seems to be very difficult to investigate the action of $\Omega_G$ on $V_\pi^{L\cap H}$ just based on the formulas given in ({E}).
There is, however, a class of indecomposable non-irreducible Type~II triples (cf. Remark~\ref{remindecomp}) which is more accessible by such methods:
\begin{equation}\label{T13} 
(G'\times G'\times G'\times G',\Delta_{12} G'\times \Delta_{34} G', G'\times \Delta_{23} G'),
\end{equation}
where $G'$ is an arbitrary simple non-compact Lie group. In other words, $X=G/H$ is the product of group manifolds $G'\times G'$, and $L\cong G'\times G'$ acts on it in the following way:
The first factor of $L$ acts on the first factor of $X$ from the left, while the second factor of $L$ acts diagonally (from the right on the first factor and from the left on the second). 
We have $\Omega_G=\Omega_1+\Omega_2$, corresponding to the product decomposition of $X$. Both summands belong to ${\bf D}(G/H)$. Let $L_i\subset L$, $i=1,2$,  be the factors
of $L$, and $\Delta\subset L$ be the diagonal subgroup. Then the formulas analogous to ({E}) are  (in a suitable normalization of $\Omega_\Delta$, cf. Lemma~\ref{bur}) :
$$ \Omega_1=\Omega_{L_1},\quad \Omega_2 =\Omega_\Delta\ .$$
Now the actions of $\Omega_i$ on irreducible unitary representations of $L$ which are tensor products of representations
of $L_i\cong G'$ are easily understood if one understands the decomposition of tensor products when restricted to the diagonal $\Delta$.
We work this out when $G'=PSL(2,\R)$.
\begin{itemize}
\item[({\bf G})] {\em Let $(G,H,L)$ be the triple (\ref{T13}) with $G'=PSL(2,\R)$, and let $\Gamma\subset L\cong PSL(2,\R)\times PSL(2,\R)$ be a cocompact torsion free lattice. Then there is a unique spectral decomposition of the form} (\ref{q2}),
{\em where $\widehat{{\bf D}(G/H)}$ identifies with $\R^2$ via $\chi\mapsto (\chi(\Omega_1),\chi(\Omega_2))$, the measure $\mu$ is a countably infinite sum of Dirac measures supported in the half plane $\{\chi(\Omega_2)>-\frac{1}{4}\}$ and one-dimensional Lebesgue measures supported on half lines $\{\chi(\Omega_1)=c, \chi(\Omega_2)\le-\frac{1}{4}\}$, and the
multiplicity spaces ${\mathcal M}_\chi$ are all infinite dimensional.
All formally self-adjoint elements of ${\bf D}(G/H)$ are essentially self-adjoint on $\Gamma\backslash G/H$.
The absolutely continuous spectrum of the Laplacian $\Delta_Y$ is $\R$ and with infinite multiplicity with embedded eigenvalues of infinite multiplicity. Moreover, the set of embedded eigenvalues is bounded neither above nor below.} (Thm.~\ref{surprise}, Figure~2)
\end{itemize}

We remark that the results discussed so far do not rely on the well developed $H$-spherical representation theory of $G$. We only used Casimir like operators on $G$
and representation theory of $L$. We also have not yet discussed the relation between the approaches to a spectral theory for ${\bf D}(G/H)$ on $L^2(Y)$ via (\ref{decompo})
(representation theory of $L$) and via (\ref{matcoef}) (representation theory of $G$). However, at least in the case of spherical triples $(G,H,L)$, the results above - in particular the classification results ({C}) and the formulas ({E})
for $\Omega_G$ -  imply an important property of $H$-spherical $G$-representations, namely $L$-admissibility:
\begin{itemize}
\item[({\bf H})] {\em Let $(G,H,L)$ be a triple of Type~I, and let $W_\rho$ be a non-trivial  irreducible unitary $H$-spherical representation of $G$.  As an $L$-representation, $W_\rho$ decomposes discretely into an infinite Hilbert sum of irreducible unitary representations of $L$. The multiplicity $m_\rho(\pi)$ of a given $\pi\in\widehat L$ in $W_\rho$ is either $0$ or equal to $\dim V_\pi^{L\cap H}$. In particular, the multiplicities are finite if $(G,H,L)$ is spherical.} (Thm.~\ref{mainbranching})
\end{itemize}
Recall that an admissible $G$-representation $W_\rho$ of finite length is called $H$-spherical, if $W_{\rho,\infty}$ embeds into $C^\infty(G/H)$. By Frobenius reciprocity, this
is equivalent to: $(W_{\tilde\rho,-\infty})^H$ contains a vector which is cyclic for the $G$-representation $W_{\tilde\rho,-\infty}$. There is a dual notion: $W_\rho$
is called cospherical, if $(W_{\rho,-\infty})^H$ contains a cyclic vector. For an irreducible admissible, in particular an irreducible unitary, representation $W_\rho$ both notions coincide and are equivalent to $(W_{\rho,-\infty})^H\ne\{0\}$. By $\widehat G_H$ we denote the set of equivalence classes of irreducible unitary $H$-(co)spherical representations of $G$.

In this paper, we present a more geometric proof of the result ({H}) which is essentially classification free and does not use the formulas for $\Omega_G$ in ({E}). In fact, we show 
that for a spherical triple $(G,H,L)$ the maximal compact $K_L\subset L$ (and hence also $L$) acts transitively on $G/P_{\sigma\theta}$, where $P_{\sigma\theta}\subset G$
is a minimal $\sigma\theta$-stable parabolic (Lemma \ref{sigmatheta}). The result (except for the precise multiplicity formula which is related to (I) below) then follows in combination with Delorme's embedding theorem stating that any irreducible admissible
$H$-spherical representation can be embedded into a representation induced from a finite dimensional representation of $P_{\sigma\theta}$ (principal series representations for $G/H$).
Among other things, this proof has the advantage that it gives a way to obtain the $L$-decomposition of $W_\rho$ rather explicitly as soon as $W_\rho$ is realized as
a subrepresentation of a principal series for $G/H$, see e.g. (\ref{schal}) and Tables 3--6 in Section~\ref{eigen}.

From the point of view of spectral theory, with (\ref{decompo}) as a starting point, we are more interested in the reciprocal question: Given $\pi\in\widehat L$ with $V_\pi^{L\cap H}\ne\{0\}$ determine
all $\rho\in \widehat G_H$ such that $m_\rho(\pi)\ne \{0\}$. It will turn out that there is at most one such representation $\rho$, see (I) below. But the difficulty is that there might be none. We have to drop the unitarity and the irreducibility conditions. This is the reason that we treat the branching problem from $G$ to $L$ in Section~\ref{branching} in the framework of general admissible $H$-spherical $G$-representations of finite length. We call such a representation $W_\rho$ $\pi$-minimal, if $V_{\pi,\infty}$ embeds into $W_{\rho,\infty}$ but does not embed into a proper subrepresentation of it. It will be called $L$-minimal, if it is $\pi$-minimal for some irreducible admissible $L$-representation (necessarily with $V_\pi^{L\cap H}\ne\{0\}$). Note that every irreducible admissible $H$-spherical $G$-representation is $L$-minimal.
\begin{itemize}
\item[({\bf I})] {\em Let $(G,H,L)$ be a triple of Type~I. Then for every irreducible admissible $L$-representation $V_\pi$ with $V_\pi^{L\cap H}\ne\{0\}$ there exists a unique
$\pi$-minimal $H$-spherical $G$-representation. Moreover, $L$-minimal $G$-representations $W_\rho$  satisfy $\dim (W_{\tilde\rho,-\infty})^H=1$.} (Thm.~\ref{aal})
\end{itemize}
This result can be viewed as a refinement of the second part of (B). It says that $\pi\in \widehat L$ does not only determine a character of ${\bf D}(G/H)$, it determines
an $H$-spherical $G$-representation. As the second part of (B), it rests on the multiplicity one results of Prop.~\ref{mufree}.

From now on let us swap the roles of $\rho$ and $\tilde\rho$, i.e. $\rho$ will be cospherical, and $\tilde \rho$ will be $H$-spherical. Let us also return for a moment to 
general
properly transitive triples $(G,H,L)$. The multiplicity space for the restriction of $W_{\tilde\rho}$ to $L$ is always contained in $\Hom_L(W_{\rho,\infty}, V_{\pi,\infty})$ and contains $\Hom_L(W_{\rho,-\infty}, V_{\pi,-\infty})$ (intertwining operators that extend to distribution vectors),
even if the decomposition is non-discrete, see Lemma~\ref{sardine}. Note that $(W_{\rho,-\infty})^H$ is a ${\bf D}(G/H)$-module. We have a natural ${\bf D}(G/H)$-equivariant map
$$ E_0: \Hom_L(W_{\rho,-\infty}, V_{\pi,-\infty})\otimes (W_{\rho,-\infty})^H\longrightarrow (V_{\pi,-\infty})^{L\cap H}\ .$$
\begin{itemize}
\item[({\bf J})] {\em Let $(G,H,L)$ be a properly transitive triple. Let $W_\rho$ and $V_\pi$ be continuous representations on reflexive Banach spaces of $G$ and $L$, respectively.
Then there is a natural ${\bf D}(G/H)$-equivariant map 
\be\label{thomas1}\nonumber   E: \Hom_L(W_{\rho,\infty},V_{\pi,\infty})\otimes (W_{\rho,-\infty})^H\longrightarrow  (V_{\pi,-\infty})^{L\cap H}
\end{equation}
extending $E_0$ above. If $(G,H,L)$ is of Type~I, $\pi$ is irreducible admissible, and $\tilde\rho$ is $\tilde\pi$-minimal, then $E$ is a bijection.} (Prop.~\ref{paulklee})
\end{itemize}
We will use the map $E$ for representations $W_\rho$ having an infinitesimal character. We will also assume that the natural map ${\Cal Z}(\fg)\rightarrow {\bf D}(G/H)$ is surjective.
This is justified by the classification of properly transitive triples. These assumptions imply that $(W_{\rho,-\infty})^H=(W_{\rho,-\infty})_\chi^H$, where $\chi$ is the character of ${\bf D}(G/H)$ determined by the infinitesimal character of $W_\rho$. In particular, the image of $E$ consists of eigenvectors for the action of ${\bf D}(G/H)$ on  $(V_{\pi,-\infty})^{L\cap H}$.

We want to advertise the approach to use the map $E$ as the key tool to transform the decomposition (\ref{decompo})
into a spectral decomposition of $L^2(Y)$. Moreover, the resulting decomposition will be even even finer than (\ref{q1}) or (\ref{q2}) since it is indexed by  a set of (co)spherical representations $\rho$ of $G$ instead of eigencharacters $\chi$ of 
${\bf D}(G/H)$ (in other words, we want to produce a decomposition that generalizes (\ref{GGPS}) instead of (\ref{GGPSbis})).

Let us now explain the relation of (\ref{decompo}) to
the matrix coefficient map (\ref{matcoef}) using $E$. The injection $N_\Gamma(\pi)\otimes V_\pi^{L\cap H}\hookrightarrow L^2(Y)$ extends to 
$N_\Gamma(\pi)\otimes (V_{\pi,-\infty})^{L\cap H}\hookrightarrow C^{-\infty}(Y)$, which is given by distributional matrix coefficients $c^L$ for $L$ as follows:
$N_\Gamma(\pi)$ identifies with $(V_{\tilde\pi,-\infty})^\Gamma$ and the map in question is
$$  (V_{\tilde\pi,-\infty})^\Gamma\otimes (V_{\pi,-\infty})^{L\cap H}\ni \tilde v\otimes v\mapsto c^L_{\tilde v, v}\in C^{-\infty}(L)^{\Gamma\times (L\cap H)}\cong C^{-\infty}(Y)\ .$$
We denote matrix coefficients for $G$ by the symbol $c^G$. Now let $\Phi\in \Hom_L(W_{\rho,\infty},V_{\pi,\infty})$, $w\in (W_{\rho,-\infty})^H$, $\tilde v\in (V_{\tilde\pi,-\infty})^\Gamma$. Then we have the following
equality of eigendistributions on $Y$: 
\be\label{wilhelm} c^L_{\tilde v,E(\Phi\otimes w)} =c^G_{\Phi^t\tilde v, w}\ .
\end{equation}
This identity also gives a useful characterization of smooth and $L^2$-eigenfunctions: 
$$c^G_{\Phi^t\tilde v, w}\in C^\infty(Y)\  (\mbox{or }L^2(Y))\ \Leftrightarrow \ E(\Phi\otimes w)\in (V_{\pi,\infty})^{L\cap H}\  (\mbox{or }V_\pi^{L\cap H})\ .$$

Back to triples of Type~I. In order to formulate the spectral decomposition that generalizes (\ref{GGPS}) we need some more notation. Let $\widehat G_H^L\supset \widehat G_H$
be the set of infinitesimal equivalence classes of admissible $H$-cospherical $G$-representations $W_\rho$ of finite length such that the conjugate dual $\rho^*$ is $\pi$-minimal
for some $\pi\in\widehat L$. We obtain a partition of the set $\{\pi\in\widehat L\mid V_\pi^{L\cap H}\ne\{0\}\}$ into disjoint subsets
$$ \widehat L_\rho:=\{  \pi\in\widehat L\mid V_\pi^{L\cap H}\ne\{0\}, \rho^*\mbox{ is }\pi\mbox{-minimal}\},\quad \rho\in \widehat G_H^L.$$
Recall from (I) that $\dim (W_{\rho,-\infty})^H=1$ for $\rho\in \widehat G_H^L$. Moreover, ${\bf D}(G/H)$  acts on this space of invariants by a character $\chi\in \widehat{{\bf D}(G/H)}$, which gives rise to a finite-to-one map $p: \widehat G_H^L\rightarrow \widehat{{\bf D}(G/H)}$.
\begin{itemize}
\item[({\bf K})] {\em Let $(G,H,L)$ be a triple of Type~I. Let $\Gamma\subset L$ be a cocompact lattice. There exist vector subspaces $(W_{\widetilde{\rho},-\infty})^\Gamma_0\subset (W_{\widetilde{\rho},-\infty})^\Gamma$ (dense whenever $\rho$ is irreducible) and dense subspaces $\Hom_L(W_{\rho,\infty},V_{\pi,\infty})_0\subset \Hom_L(W_{\rho,\infty},V_{\pi,\infty})$ carrying
Hilbert space structures (depending on the normalization of the scalar product on the one-dimensional Hilbert space $(W_{\rho,-\infty})^H$) such that we have isomorphisms of Hilbert spaces 
$$  L^2(Y)\cong \widehat{\bigoplus_{\rho\in\widehat G_H^L}} (W_{\tilde \rho,-\infty})_0^\Gamma\otimes  (W_{\rho,-\infty})^H\quad ({\bf D}(G/H)\mbox{-equivariant})$$
and 
$$  (W_{\widetilde{\rho},-\infty})^\Gamma_0\cong  \widehat{\bigoplus_{\pi\in\widehat L_\rho}} \Hom_L(W_{\rho,\infty},V_{\pi,\infty})_0\otimes N_\Gamma(\pi)\ .$$
There are corresponding decompositions of $C^{\pm\infty}(Y)$.}
(Thm.~\ref{gogo})
\end{itemize}
In Thm.~\ref{gogo} we use slightly different notations for the various Hilbert spaces involved in the above decompositions. We also give explicit descriptions of them.
For instance, if $\rho\in\widehat G_H$, then $\Hom_L(W_{\rho,\infty},V_{\pi,\infty})_0=\Hom_L(W_{\rho},V_{\pi})$ equipped with a certain multiple of its natural scalar product.

For certain subsets of $\widehat G_H^L$ we can make quite precise statements about their contribution to $L^2(Y)$.
\begin{itemize}
\item[({\bf L})] {\em Let $(G,H,L)$ be a triple of Type~I, and let $\Gamma\subset L$ be a cocompact torsion free lattice. Let $\rho\in\widehat G_H$ be an integrable discrete series representation for $G/H$. Then $\dim(W_{\widetilde{\rho},-\infty})^\Gamma_0=\infty$.}
(Prop.~\ref{horn} and Thm.~\ref{gogo})
\end{itemize}
Indeed, $\widehat L_\rho$ is an infinite set by (H). Moreover, the analogue of (\ref{wilhelm}) for $\Phi\in \Hom_L(W_{\rho,-\infty},V_{\pi,-\infty})$, $\tilde v\in V_{\tilde\pi,\infty}$ implies that every $\pi$ in $\widehat L_\rho$ is an integrable discrete series representation of $L$.
Therefore, $N_\Gamma(\pi)\ne\{0\}$. Now apply (K).

In the case of the triple (\ref{T1}), i.e. for standard quotients of the anti-de Sitter space, we have the same result for {\em all} discrete series representations $\rho$ for $G/H$,
see the discussion of (F) above.

We expect that the contribution of non-unitary representations $\rho$ of $G$ 
to $L^2(Y)$  is small in some sense. If $G/H$ has rank one, we can make this precise.
We also show that non-unitary representations $\rho$ really contribute to the spectral decomposition of some standard quotients.
The latter is based in part on results of Bergeron and Clozel \cite{BeCl13}, \cite{BeCl17}. 
\begin{itemize}
\item[({\bf M})] {\em Let $(G,H,L)$ be a triple of Type~I. Let $Y$ be a corresponding standard quotient. Let $L^2(Y)_{nu}\subset L^2(Y)$ be the subspace given by the sum of the contributions of all $\rho\in \widehat G_H^L\setminus \widehat G_H$ in (K).  If $G$ is linear and $\rank(G/H)=1$, then  
$\dim L^2(Y)_{nu}<\infty$.
If $G/H$ is not a group manifold, then there exists  a standard quotient $Y$ of $G/H$ such that $L^2(Y)_{nu}\ne\{0\}$.} (Prop.~\ref{wanndenn}, Prop.~\ref{nanu}, Prop.~\ref{nonu}, and Cor.~\ref{nanunana})
\end{itemize}
Moreover, for the triples $(G,H,L)$ with $\rank(G/H)=1$ we determine the sets $\widehat G_H^L$ and $\widehat L_\rho$ completely, see Tables 3--6 in Section~\ref{eigen}.
This is based on the detailed description of the composition series of the principal series for $G/H$ given in \cite{HoTa93} and \cite{Sch87} for these cases. As a consequence, one obtains a description of $\spec(\Delta_Y)$ similar to (F) for all standard quotients associated to these triples.

For standard quotients associated to properly transitive triples $(G,H,L)$ of Type~II we expect a spectral decomposition similar to the one in (K), but the first direct sum
should be replaced by a direct integral, and the sets $\widehat L_\rho$ appearing in the second sum are no longer mutually disjoint (one could replace them just by 
the set $\{\pi\in\widehat L\mid V_\pi^{L\cap H}\neq\{0\}\}$).         
Note that we have not yet defined $\widehat G_H^L$ for Type~II triples. But certainly, this set should contain $\widehat G_H$.

As a starting point we make the following observations.
\begin{enumerate}
\item[({\bf N})] 
{\em Let $Y$ be a standard quotient corresponding to a properly transitive triple $(G,H,L)$. Then every joint eigendistribution for ${\bf D}(G/H)$ on $Y$
can be written as a series of matrix coefficients of the form (\ref{wilhelm}) for some admissible $H$-cospherical $G$-representations $W_\rho$ of finite length with infinitesimal character,
$\pi\in\hat L$ with $N_\Gamma(\pi)\ne\{0\}$, $\Phi\in\Hom_L(W_{\rho,\infty},V_{\pi,\infty})$.} (Lemma \ref{otto}, Lemma \ref{kirchner}, cf. also Prop.~\ref{monika}) 
\end{enumerate}
\begin{itemize}
\item[({\bf O})] {\em For some triples $(G,H,L)$ of Type~II including (\ref{T9}) and (\ref{T13}) with $\rank_\R(G')=1$, we find non-trivial continuous families of representations $\lambda\mapsto\rho_\lambda\in \widehat G_H$ and of intertwining operators $\lambda\mapsto\Phi_\lambda\in \Hom_L(W_{\rho_\lambda,\infty},V_{\pi,\infty})$, $\pi\in\widehat L$, that for every standard quotient $Y$ corresponding to $(G,H,L)$ give rise to non-trivial families
of eigendistributions of the form (\ref{wilhelm}) on $Y$.} (Prop.~\ref{knarz} and Cor.~\ref{knorz})
\end{itemize}
These observations stress again the importance of the spaces $\Hom_L(W_{\rho,\infty},V_{\pi,\infty})$ which appear as multiplicity spaces for the
restriction of $G$-representations to $L$. 
Moreover, we gained further evidence that the spectrum of standard quotients corresponding to Type~II triples should not be purely discrete.
The missing link which will show how the (families of) eigendistributions in (N) and (O) fit together to give a spectral decomposition 
which is similar to (K) 
is the following conjecture. It is independent of the discrete subgroup $\Gamma\subset L$ and gives a spectral decomposition of ${\bf D}(G/H)$ acting on $V_\pi^{L\cap H}$, $\pi\in\widehat L$.

\begin{con}\label{C0}
Let $(G,H,L)$ be a properly transitive triple. There exist
\begin{itemize} 
\item a topological space $\widehat G_H^L\supset \widehat G_H$ whose points are certain (infinitesimal) equivalence classes of $H$-cospherical $G$-representations of finite
length with infinitesimal character which gives rise to a continuous finite-to-one map $p: \widehat G_H^L\rightarrow \widehat{{\bf D}(G/H)}$,
\item for each $\pi\in\widehat L$ with $V_\pi^{L\cap H}\ne\{0\}$ a family of subspaces 
$$\widehat G_H^L\ni\rho\mapsto\left [\Hom_L(W_{\rho,\infty},V_{\pi,\infty})\otimes (W_{\rho,-\infty})^H\right ]_0\subset \Hom_L(W_{\rho,\infty},V_{\pi,\infty})\otimes (W_{\rho,-\infty})^H$$ 
carrying the additional structure of a measurable family of Hilbert spaces,
\item for each $\pi\in\widehat L$ with $V_\pi^{L\cap H}\ne\{0\}$ a $\sigma$-finite Borel measure $\mu_\pi$ on $\widehat G_H^L$
\end{itemize}
such that the natural ${\bf D}(G/H)$-equivariant maps 
$$   E: \Hom_L(W_{\rho,\infty},V_{\pi,\infty})\otimes (W_{\rho,-\infty})^H\longrightarrow  (V_{\pi,-\infty})^{L\cap H}$$
from (J) induce a unitary equivalence
\be\label{D0}   V_\pi^{L\cap H} \cong \int^\oplus_{\widehat G_H^L}\left [\Hom_L(W_{\rho,\infty},V_{\pi,\infty})\otimes (W_{\rho,-\infty})^H\right ]_0\:d\mu_\pi(\rho)
\end{equation}
for all $\pi\in\widehat L$ with $V_\pi^{L\cap H}\ne\{0\}$. \rm{(Conjecture~\ref{C1})}
\end{con}
Indeed, inserting the hypothetical decompositions (\ref{D0}) into (\ref{decompo}) and interchanging the sum over $\widehat L$ with the direct integral over $\widehat G_H^L$ we obtain the desired
spectral decomposition of $L^2(Y)$, see Conjecture~\ref{C2}.

The decompositions (\ref{D0}) could be considered as Frobenius reciprocal to the restriction problem from $G$ to $L$ in the category of $H$-(co)spherical representations
in the same way as the decomposition of $\Ind_L^G(V_\pi)$ is reciprocal to the restriction problem in the category of  ``all'' $G$-representations.
But in contrast to restriction and induction, there is no general abstract Plancherel theorem in the background which ensures the existence of these decompositions.

The above conjecture is true for triples of Type~I by the second part of (J). We also have
\begin{itemize}
\item[({\bf P})] {\em Let $(G,H,L)$ be equivalent to one of the triples (\ref{T13}). Then Conjecture~\ref{C0} is true with $\widehat G_H^L=\widehat G_H$.} (Prop.~\ref{last})
\end{itemize}
Here, as in (G), the problem can be reduced to the decomposition of certain tensor product representations. To establish Conjecture \ref{C0} for {\em irreducible}  triples $(G,H,L)$ of Type~II
(see Tables 1 and 2)
and at least some $\pi\in\widehat L$ with $V_\pi^{L\cap H}\ne\{0\}$, as well as make its ingredients as explicit as possible, is work in progress.

As indicated at several places in this introduction, our results can be generalized in at least two directions.

First, one can weaken the condition that $H\subset G$ is symmetric to just being reductively embedded. As already mentioned in Subsection \ref{gi}, the resulting list of triples $(G,H,L)$,
at least for simple $G$,
would be not much longer than ours. If we would, in addition, insist on the Type~I condition that $L_{max}/L_{max}\cap H$ is a {real} spherical homogeneous space, then the only new cases to
consider would be certain circle bundles over symmetric spaces corresponding to Triples 2, 3 in Table \ref{fix} and the triple 
$$(G,H,L)=(Spin(3,4), G_{2(2)},Spin(2)\cdot Spin(1,4)),$$
which arises form Triple 8 in Table \ref{fix} by interchanging the roles of $L$ and $H$. Note that only in the last case we have that $L_{max}/L_{max}\cap H$ is  complex spherical, c.f. the next subsection. Moreover, in this last
case the space $X=G/H$ is isometric to the symmetric space $SO_e(4,4)/SO_e(3,4)$. In addition, every $Spin(3,4)$-invariant differential operator on $X$ is $SO_e(4,4)$-invariant. In fact,
it is a polynomial in the Laplacian. Thus the analysis on compact standard quotients of $X$ is already covered by the discussion of Triple 5 in Table \ref{fix} for $n=1$.  

Second, we can skip the condition that $\Gamma\subset L$ is cocompact. In fact, our results (and conjectures) give a recipe how to translate a decomposition indexed by $\widehat L$
(as (\ref{decompo})) into a decomposition indexed by $\widehat{{\bf D}(G/H)}$ or even $\widehat G_H^L$ (as in (K)). With some care one can just deal with non-compact
standard quotients in the same way. One just has to replace the direct sum decomposition (\ref{decompo}) by a direct integral decomposition, which comes from (\ref{GGPSter}) applied to
$L^2(\Gamma\backslash L)$, as the starting point.

In contrast to these two directions, a possible generalization of our results to non-standard quotients $Y$ would really require new ideas. 

\subsection{Relations to the work of Kassel and Kobayashi.}\label{kako}

There are some overlaps beween our results discussed in the previous subsection and the results of  \cite{KK2} and \cite{KK25}.

The only overlap with the earlier paper \cite{KK2} is related to the result $(L)$ above. Under the assumption of (L), the main result of  \cite{KK2} says that 
$$ (W_{\widetilde{\rho},-\infty})^\Gamma_0\ne \{0\} $$
provided that the parameter of $\rho$ is sufficiently regular. Thus the conclusion is weaker than in (L). On the other hand, the main result of \cite{KK2} applies to much more general situations. In fact, it ensures the existence of $L^2$-eigenfunctions on $Y$ corresponding to ``very  integrable" discrete series representations
$\rho$ for the symmetric space $X=G/H$
not only for compact standard quotients $Y$ coming from triples of Type~I as in (L), but - if combined with the main result of the recent paper \cite{KT24} - for arbitrary quotients $Y$ of semisimple symmetric spaces $X$ admitting discrete series. We note that the methods of \cite{KK2}, namely the use of Poincar\'e series, are completely different from ours.

The paper \cite{KK25} of Kassel and Kobayashi collects results on spectral theory for (not only compact) standard quotients corresponding to triples $(G,H,L)$ mainly satisfying the following assumptions:
Either $(G,H,L)$ is a group case triple listed under Case 11 in Table~{\ref{foxi} or $G$ is simple with $H\subset G$ reductively embedded and $X\cong L/L\cap H$ is a {\em complex} spherical
homogeneous space. The latter means that a Borel subgroup of the complexification $L_\C$ of $L$ has an open orbit on the complexification of $X$. Note that being complex spherical
is a significantly stronger condition than being real spherical. However, it is an outcome of Oni\v s\v cik's classification that for $G$ simple and $L=L_{max}$, $H=H_{max}$ both notions
for $X=L/L\cap H$ coincide. Taking the discussion of non-symmetric examples from the end of the previous subsection into account we can now say in our language that in \cite{KK25} 
exactly the standard quotients of
(irreducible) Type~I triples are investigated.

The main overlap of our results for Type~I triples with the results in \cite{KK25} consists in the two basic discrete decomposability results (D), cf. \cite{KK25}, Theorems ~1.8 and 1.9,
and (H) (without the explicit multiplicity assertions), cf. \cite{KK25}, Thm.~8.3. The first result, which is relatively easy to obtain as an abstract statement, only becomes really powerful  when
combined with some explicit transfer techniques how to go from the ``$L$-spectrum" to the ``$G$-spectrum"  of the locally symmetric space $Y$. The transfer techniques developed
in our paper, at least the $G$-representation theoretic approach of Sections~\ref{branching} and \ref{eigen}, are quite different from the ones of \cite{KK25} . The latter, which were already developed in \cite{KK19}, rest on the determination of a system of generators of ${\bf D}(L/L\cap H)$ containing the generators of ${\bf D}(G/H)$. Also the proof of (H) we give in the present paper is completely different from the one given in \cite{KK25}. We remark that Theorem 1.9 in \cite{KK25} is formulated for arbitrary not necessarily compact standard quotients.
The reader should be aware that its statement is not entirely correct for noncompact quotients. 

For two reasons we have included both results in the present paper. First, we have obtained the results (D) and (H) independently of the work of Kassel and Kobayashi. In fact, they belong to a group among our results that have been obtained
in the early phase of our work on the subject, i.e. much more than 10 years ago. We have first publicly reported on some of these early results, including the results (D), (E) and (F), 
at a conference in Bonn in 2011 - as mentioned already in the introduction of \cite{KK19}. Second, it is our wish to present here a complete and systematic account of our approach to the spectrum of standard quotients which is based on the clear distinction
between triples of Type~I and Type~II.

In the early phase of our work we were not aware of the importance of complex spherical homogeneous spaces (via the commutativity of ${\bf D}(L/L\cap H)$) for our problem.
In this respect, we benefited and learned a lot from \cite{KK19}. Although complex sphericity still does not play a very decisive role in our paper, the understanding of its meaning helped us
to clarify some multiplicity one phenomena and to simplify some of our proofs. In this respect, we mention Remark \ref{complexsph}, the proof of Prop.~\ref{mufree}, and our comments following the proof of Lemma~\ref{schlips}. We also remark that Cor.~1.7  in  \cite{KK19} is quite close to the result (E) above (for triples of Type~I).

In \cite{KK25} Kassel and Kobayashi also discuss the rather subtle relationship between $L^2$-eigenfunctions on $Y$ and {\em unitary} $G$-representations, i.e. the question whether the $G$-representation generated by such an eigenfunction has to be unitary. They conjecture that this $G$-representation always has at least an irreducible unitary {\em sub}representation (Conj.~11.2).
Our Proposition \ref{nanu}, which also goes into the result (M) above, disproves this and the related Conjecture~11.3 in  \cite{KK25}. 

\subsection{Structure of the paper.}\label{sp} 

Besides the introduction, the paper consists of ten sections of different length and nature, as well as of two appendices. 

Section \ref{triples} provides the geometric basis for our considerations. After a short discussion of reductively embedded subgroups $L\subset G$ we prove Result (A), introduce properly
transitive triples $(G,H,L)$ and discuss their types, relations between real ranks (Prop.~\ref{ranksymm}), and their classification (Result (C)).

The entirely algebraic Section \ref{pbw} is devoted to the relation between the Casimir operators of $G$ and $L$, respectively, for properly transitive triples $(G,H,L)$. In particular,
we prove Result (E). The case by case calculations  necessary for its proof are given in Appendix \ref{A}.

The two intermediate Sections \ref{dist} and \ref{G2} collect essential well-known material which is rather scattered throughout the literature. Thus we found it useful to present it in compact form for later reference. Section~\ref{dist} deals with distributional matrix coefficients and decompositions of spaces of smooth and distribution vectors of unitary representations. The only less known result of independent interest in this section is probably
Prop.~\ref{kunst}, which gives an unusual form of Frobenius reciprocity in terms of distributions on homogeneous spaces. This result will be the key for the construction of the map $E$
appearing in Result (J). In Section~\ref{G2} we present some spectral theory of algebras of commuting unbounded operators on Hilbert spaces. In particular, we fix our definition of a spectral
decomposition (in direct integral form) of such an algebra (Def. \ref{kobold1}) and discuss existence and uniqueness of spectral decompositions.

The investigation of properly transitive triples, in particular those of Type~I, continues in Section \ref{invariants}, which is devoted to the proof of Result (B) and an essential multiplicity one result (Prop.~\ref{mufree}). The cococompact discrete subgroup $\Gamma\subset L$ enters the stage in the short Section~\ref{spectrum}, where the decomposition (\ref{decompo}) and its consequences for Type~I triples, in particular Result (D), are discussed.

In Sections \ref{detail} and \ref{infinitemulti} we study the spectral decomposition of two examples, one of Type I (standard quotients of anti-de Sitter spaces) and one of Type II (Triple (\ref{T13}) with
$G'=PSL(2,\R)$). We establish the results (F) and (G). We already mentioned that (F) involves a non-vanishing result for the multiplicity spaces $N_\Gamma(\pi)$ for certain discrete series represantions $\pi$ of $U(1,n)$. It rests on Prop.~\ref{nichts}, whose proof is given in Appendix~\ref{B}.

It is only in the last two sections, where representation theory of $G$ plays a role. Section \ref{branching} is entirely devoted to the restriction problem of (weakly) $H$-spherical admissible
$G$-representations to $L$, where $(G,H,L)$ is a properly transitive triple of Type I. We treat this problem in a way suitable for applications to the spectral theory for standard quotients.
In particular, we introduce and discuss the notion of $\pi$-minimal spherical $G$-representations and prove the results (I) and (H). We also prove Prop.~\ref{horn}, which essentially gives Result~(L).

The paper culminates in Section \ref{eigen} in which we describe eigenspaces and spectral decompositions in terms of the matrix coefficient map (\ref{matcoef}). The section is divided into three subsections. In the first one, we deal with eigendistributions on compact quotients (not necessarily standard ones) in a representation theoretic way and prove the results (N) and (J).
The second subsection, which contains the major results of the paper, among them (K) and (M),  is devoted to the spectral theory of standard quotients corresponding to Type I triples. As a byproduct, we obtain a complete explicit solution of the restriction problem of Section \ref{branching} and a description of $\spec(\Delta_Y)$ similar to (F) for all properly transitive triples
$(G,H,L)$ with $\rank(G/H)=1$. These explicit results make essential use of the reducibility and unitary results obtained by Howe and Tan \cite{HoTa93} and Schlichtkrull \cite{Sch87}. In the last subsection, we turn to triples of Type II. We prove Result~(O) on the existence of continuous families of eigendistributions with varying eigenvalues. Eventually, we present Conjecture~\ref{C0} and discuss its consequences, e.g. Conjecture~\ref{C2}, for the spectral theory of compact standard quotients. Our reasoning in this last subsection is influenced by the work of Frahm \cite{Mol17} on intertwining operators associated to strongly spherical pairs $(G,L)$.

The paper provides a - hopefully - structured collection of several results,
which are also of independent interest. For a reader, who is mainly interested in the major results of Section~\ref{eigen}, it suffices to study Sections~\ref{triples},
\ref{invariants}, \ref{branching} and the beginning Section~\ref{spectrum} before. A reader, who is interested in the branching problem from $G$ to $L$, may read Sections~\ref{triples},
\ref{invariants}, \ref{branching} (and Tables \ref{fax}--\ref{tauto} in Section~\ref{eigen}) only. Alternatively, just reading Sections~\ref{triples}--\ref{infinitemulti} and leaving out the last two sections completely, the reader would  already
get a rather broad overview over spectral phenomena for compact standard quotients (in particular results (A)--(F)), which, in addition, avoids completely the use of representation theory of $G$.

\bigskip

\noindent{\bf Acknowledgements.} {S. M. was partially supported by the project OpART (ANR-23-CE40-0016) of the Agence Nationale de la Recherche.}


\section*{Notational conventions}

\begin{longtable}{p{3cm}p{10cm}}

$\N$ & the set of positive integers.\\[0.3cm]

$\N_0$ & the set of non-negative integers.\\[0.3cm]

$\Delta _{i_1\dots i_k}G$ & image of the diagonal embedding of $G\hookrightarrow \underbrace{G\times\dots\times G}_{n}$ where non-trivial entries occur at places $\{i_1,\dots,i_k\}\subset\{1,\dots,n\}$.\\[0.3cm]

$\Delta G$ & the diagonal subgroup $\Delta_{12} G$ of $G\times G$
\end{longtable}

\section{Properly transitive triples}\label{triples}

Let $G$ be a connected non-compact real semisimple Lie group with finite center, and let $L\subset G$ be a closed subgroup with finitely many connected components. The terminology ``$L$ is reductive in $G$'' has slightly varying meanings across the literature. First we want to clarify some of these differences for ourselves. The first variant is 

\begin{Def}\label{redem}
Let $\fg$ be a real semisimple Lie algebra, and let $\fl$ be a subalgebra. We say that $\fl\subset \fg$ is reductively embedded, if the adjoint action
of $\fl$ on $\fg$ is semisimple, i.e. completely reducible. Let $G$, $L$ be as above with Lie algebras $\fg$, $\fl$. Then $L\subset G$ is called reductively embedded if $\fl\subset \fg$ is reductively embedded.
\end{Def}
Note that $\fl\subset \fg$ is reductively embedded if and only if $\fl$ is reductive as an abstract Lie algebra and the operators $\ad(Z):\fg\rightarrow\fg$, for $Z$ in the center $\fz(\fl)$ of $\fl$, are diagonalizable (over $\C$). Often we would like to have that $\fl$ is stable under a Cartan involution $\theta$. This would give rise to decompositions
$\fz(\fl)\ni Z=Z_++Z_-$, $Z_\pm\in \fz(\fl)$, such that $\ad(Z_+)$ has purely imaginary and $\ad(Z_-)$ has real eigenvalues. This motivates the construction in the following
paragraph.

Let $Z\in \fz(\fl)$. We consider the eigenspace decomposition   
$$ \fg_\C = \bigoplus_{\lambda\in \C} \fg_\lambda $$
of  $\ad(Z)$. We define an operator $A_+:\fg_\C\rightarrow \fg_\C$ by
$$  A_+\vert_{\fg_\lambda}\defn{i}\Im(\lambda)\id\ .$$
One checks that $A_+$ is a derivation of $\fg_\C$ that commutes with conjugation with respect to $\fg$ and annihilates $\fl$.
We conclude that there is a (unique) $Z_+\in\fz_\fg(\fl)\subset \fg$ such that $\ad(Z_+)=A_+$.  We define the abelian Lie algebra $\tilde\fz_+:=\{Z_+\mid Z\in\fz(\fl)\}$.
The question is whether  $\tilde\fz_+\subset \fl$. On the level of groups we make the following definition.

\begin{Def}\label{qa}
Let $L\subset G$ be reductively embedded.
Let $Q$ be the closure of the analytic subgroup of  $\tilde\fz_+$. We define the strong closure $\bar L$ of $L$ by $\bar L:=LQ$ and say that $L$ is strongly closed,
if $\bar L =L$, i.e. $Q\subset L$. A subgroup  $L\subset G$ is called strongly reductive in $G$ if it is reductively embedded and strongly closed.
\end{Def}
 
Note that the strong closure $\bar L$ of $L$ is intimately related to the Zariski closure of $\Ad(L)$ in $\Aut(\fg)$.
The restriction of the Killing form of $\fg$
to $\tilde\fz_+$ is negative definite, hence $Q$ is compact. We see that $L\subset \bar L$ is cocompact. 
We also see that every closed subgroup $L\subset G$ with finitely many connected components, such that $\fl$ is reductive as an abstract Lie algebra and the center $Z(L_0)$ is compact, is strongly reductive in $\fg$. The following classical lemma makes clear the significance of being strongly reductive in $G$. It implies in particular that for strongly reductive $L\subset G$ the restriction to $\fl$ of the Killing form of $\fg$ is non-degenerate.

\begin{lem} [see e.g.~\cite{Mos55}]\label{knuff}
Let $G$ be a connected non-compact real semisimple Lie group with finite center, and let $L\subset G$ be a closed subgroup with finitely many connected components. 
There exists a Cartan involution $\theta$ of $G$ such that $L$ is $\theta$-stable if and only if $L$ is strongly reductive in $G$.
\end{lem}

Let us indicate the argument for the less obvious direction of the lemma. Being strongly reductive in $G$ implies that there is a
Cartan decomposition $\fl=\fk_\fl\oplus\fs_\fl$ such that $\fu_\fl:=\fk_\fl \oplus i\fs_\fl$ is a compact  subalgebra of $\fg_\C$. The Lie algebra $\fu_\fl$ is therefore contained in a compact real form $\fu$ of $\fg_\C$.
Let $\tau$ be the corresponding conjugate linear involution on $\fg_\C$, and let $\sigma$ be the involution corresponding to our original real form $\fg$.
Then a standard argument shows that there is a unique $X\in i\fu\subset \fg_\C$ such that $(\sigma\tau)^2=\exp(\ad(X))$. Moreover, $\tau_1:= \exp(\frac{1}{4}\ad(X))\tau \exp(-\frac{1}{4}\ad(X))$
commutes with $\sigma$. Then put $\theta\defn\tau_1|_\fg$. Our construction of $\tau$ implies that $\fl$ is $\tau$-stable and that $X$ centralizes $\fl$. Therefore $\fl$ is also $\theta$-stable.

Recall that an action of a locally compact group $R$ on a locally compact Hausdorff space $X$ is called proper if for all compact $C\subset X$ the set
$$   \{r\in R\mid  rC\cap C\ne\emptyset\}\subset R$$
is compact. In particular, for a proper action all stabilizers are compact and all orbits are closed.

We are interested in proper and cocompact actions of reductively embedded groups $L\subset G$ on certain homogeneous spaces
of $G$, in particular on symmetric spaces. The 
following proposition tells us - under natural assumptions - that these actions are always transitive.

\begin{pro}\label{cocpt}
Let $G$ be a non-compact connected semisimple real Lie group with finite center and without compact factors. Let $L,H\subset G$ be  reductively embedded. 
Assume that $L$ acts properly and cocompactly on $G/H$. Then $L$ acts transitively on $G/H$. Moreover, the centers $Z(L_0)$ and $Z(H_0)$ are compact.
\end{pro} 
The last assertion is essentially due to Oni\v s\v cik \cite{Oni69}, cf. the end of the proof of the proposition.

For a reductive Lie algebra $\fl$ we define $\fl_{min}$ as the sum of its simple ideals of non-compact type, i.e. we remove the center and all compact simple
ideals from $\fl$. If $L\subset G$ is a closed subgroup having Lie algebra  $\fl$, then $L_{min}\subset G$ denotes the analytic subgroup corresponding to $\fl_{min}$. 
Now the proposition has the following immediate consequence:
\begin{cor}\label{unfug}
Under the assumptions of Prop.~\ref{cocpt}, we have
\begin{enumerate}
\item[(i)] $L$ and $H$ are strongly reductive in $G$.
\item[(ii)] $L_{min}$ acts transitively on $G/H$ and $G/H_{min}$.
\end{enumerate}
\end{cor}

\noindent
{\it Proof of Proposition \ref{cocpt}.}
We first observe that it suffices to prove the proposition for connected $L$ and $H$ (here it is essential that we assume $G$ to be connected). Let us next assume in addition that
$L$ and $H$ are strongly reductive in $G$. Thus by Lemma \ref{knuff} and conjugating $L$ if necessary we may assume that both $\fl$ and $\fh$ are stable with respect to a Cartan
involution $\theta$ of $\fg$. Let $\fg=\fk\oplus\fs$ be the corresponding Cartan decomposition.
Put $\fl_\fs\defn\fl\cap\fs$ and $\fh_\fs\defn\fh\cap\fs$. The compactness of stabilizers of the $L$-action on $G/H$ implies that for all $k\in K$
\begin{equation}\label{knax}\fl_\fs\cap \Ad(k)\fh_\fs = \{0\}\ .
\end{equation}
Moreover, by a result of Kobayashi (\cite{Kos89}, Thm.~4.7) properness and cocompactness of the $L$-action imply that
$$ \dim \fl_\fs +\dim \fh_\fs =\dim \fs\ .$$
We obtain
\begin{equation} \label{asterix}
\fs=\fl_\fs+ \fh_\fs\ .
\end{equation}
We want to show that
\begin{equation} \label{obelix}
\fr\defn[\fl_\fs,\fl_\fs]^{\perp_\fk}\cap[\fh_\fs,\fh_\fs]^{\perp_\fk}=0,
\end{equation}
where the orthogonal complements in $\fk$ are taken with respect to the Killing form $\IP{}{}$ of $\fg$.
Then (\ref{asterix}) and (\ref{obelix}) would yield
\beu 
\fg=\fl+\fh,
\end{equation*}
which implies that $L$ has an open orbit on the connected manifold $G/H$. By properness this orbit is also closed. Hence $L$ acts transitively.
(We remark that for reductively embedded $\fl,\fh\subset\fg$ the implication $\fg=\fl+\fh \Rightarrow G=LH$ is also valid without the properness assumption (\cite{Oni69}, Thm.~3.1).) 
To prove (\ref{obelix}) it is enough to show that 
\begin{equation}\label{keystep}
\ad_{\fs}(X)=0\text{ for all }X\in\fr.
\end{equation}
Indeed, since $\fg$ is without compact factors, the adjoint representation of $\fk$ on $\fs$ is faithful. 

Now pick $X\in\fr$ and consider the compact group $K_X\defn Z_K(X)\subset K$ with Lie algebra $\fk_X$. The $K_X$-invariant alternating bilinear form 
\beu
\omega(Y,Z)\defn\IP{ [X,Y]}{Z}
\end{equation*}
on $\fs$ induces a $K_X$-invariant symplectic form $\bar\omega$ on 
\beu
\bar\fs\defn\fs/\ker(\ad_{\fs}(X)).
\end{equation*}
The definition (\ref{obelix}) of $\fr$ implies that $[X,\fl_\fs]\subset \fl_\fs^\perp$ and $[X,\fh_\fs]\subset \fh_\fs^\perp$, i.e. the subspaces $\fl_{\fs}$ and $\fh_{\fs}$ are 
isotropic with respect to $\omega$, and, in view of (\ref{asterix}), they induce transversal Lagrangian subspaces $\bar\fl_{\fs}$, $\bar\fh_{\fs}$ of $\bar\fs$. 

If $\ad_{\bar\fs}(X)^{\star}$ denotes the adjoint of $\ad_{\bar\fs}(X)$ with respect to $\IP{}{}$, then 
\begin{equation}\label{knurr} 
\ad_{\bar\fs}(X)^{\star}=-\ad_{\bar\fs}(X)\ .
\end{equation}
In particular, $\ad_{\bar\fs}(X)$ commutes with its adjoint. Moreover, $\ad_{\bar\fs}(X)$ is bijective.  We obtain a polar decomposition
\beu
\ad_{\bar\fs}(X)=|X|J=J|X|\ ,
\end{equation*}
where $|X|=\sqrt{\ad_{\bar\fs}(X)^{\star}\ad_{\bar\fs}(X)}$ is self-adjoint and positive definite, and $J$ is an orthogonal transformation. Moreover, (\ref{knurr}) implies that $J^2=-\id$, i.e.  $J$ induces an orthogonal complex structure on $\bar\fs$.
One checks that 
\beu
(\bar Y,\bar Z)\defn\IP{|X|\bar Y}{\bar Z}-{i}\bar\omega(\bar Y,\bar Z)\;\;\text{ for all }\bar Y,\bar Z\in\bar\fs
\end{equation*}
defines a Hermitian form on $\bar\fs$.
Therefore, letting $n\defn\frac{1}{2}\dim(\bar\fs)$ and choosing an orthonormal basis 
of $\bar\fl_{\fs}$, the above form identifies $\bar\fs$ with $\mathbb{C}^n$ such that $-\bar\omega$ is the imaginary part of the standard Hermitian form. 

Next we consider the $K_X$-action on the manifold 
\beu
\mathcal L\defn\{ V\subset \fs\mbox{ $\omega$-isotropic subspace} \mid \dim V=\dim\fl_{\fs}\}
\end{equation*}
together with a $K_X$-equivariant projection $\mathcal L\rightarrow \bar{\mathcal L}$, where $\bar{\mathcal L}$ is the Lagrangian Grassmannian of $\bar\fs$, 
i.e the $\frac{n(n+1)}{2}$-dimensional smooth manifold of Lagrangian subspaces of $\bar\fs$. It is a standard fact that (see e.g. \cite{GS77}), 
under the above identification of $\bar\fs$ with $\mathbb{C}^{n}$, $\bar{\mathcal L}$ identifies with the homogeneous space $U(n)/O(n)$. In particular, $\Ad_{\bar\fs}(K_X)$ may 
be identified with a subgroup of $U(n)$. Moreover, the determinant function $\det:U(n)\rightarrow S^{1}$ induces a well defined map 
\begin{equation}\label{paula}
U(n)/O(n)\ni [A] \mapsto \det(A)^2\in S^1\subset\mathbb{C},
\end{equation}
the square of the (complex) determinant. In turn this map induces a map of fundamental groups 
\begin{equation}\label{ernst}
\pi_1(\bar{\mathcal L})\rightarrow\pi_1(S^1)=\mathbb{Z}\ ,
\end{equation}
which is an isomorphism for $n>0$. 

On the other hand, by (\ref{knax}) the $K_X$-orbit of $\fl_\fs$ in $\mathcal L$ is contained in the open subset
\beu 
\mathcal L_0\defn\{ V\in\mathcal L\mid V\cap\fh_\fs=\{0\}\}\subset \mathcal L\ 
\end{equation*}
which is diffeomorphic to the vector space 
\beu 
\{\Phi\in\Hom(\fl_\fs,\fh_\fs)\mid \ad_\fs (X)\circ\Phi\text{ self-adjoint w.r.t. }\IP{}{}_{|\fl_\fs} \}
\end{equation*}
via the map
\beu
\Phi\mapsto \{(Y,\Phi(Y))\in\fl_\fs\oplus\fh_\fs=\fs\mid Y\in\fl_\fs\}\in \mathcal L_0.
\end{equation*}
This implies that $\mathcal L_0$ is contractible. 

Given an element $Y\in\fk_{X}$, write $<\exp(tY)>$ for the corresponding one-parameter subgroup in $K_{X}$. Then, since $K_{X}$ is compact, the set 
\beu
\fk_{X}^{\prime}\defn\{Y\in\fk_{X}\mid <\exp(tY)>\text{ is closed in }K_{X}\}
\end{equation*}
is dense in $\fk_{X}$. For a fixed $Y\in\fk_{X}^{\prime}$, consider a real number $t_{1}>0$ such that $\exp(t_{1}Y)=e$ and the closed loop 
\beu
\gamma:\lbrack0;t_{1}\rbrack\ni t\mapsto\Ad_{\fs}(\exp tY)(\fl_{\fs})\subset\mathcal L_0.
\end{equation*}
Since $\mathcal L_0$ is contractible, the projection $\mathcal L\rightarrow \bar{\mathcal L}$ maps $\gamma$  to a contractible loop in $\bar{\mathcal L}$. 
In view of (\ref{paula}) and (\ref{ernst})
the path
\beu
[0,t_1]\ni t\mapsto \big(\det(\Ad_{\bar\fs}(\exp(tY)\big)^2=e^{2t\Tr(\ad_{\bar\fs} (Y))}
\subset S^1
\end{equation*}
has winding number $0$, i.e. $\Tr(\ad_{\bar\fs}(Y))=0$. The trace and determinant are taken over the ground field $\mathbb{C}$. By density of $\fk_{X}^{\prime}$, we deduce that $\Tr(\ad_{\bar\fs} (X))=0$. 
Using the real trace $\Tr_\mathbb{R}$ on $\bar\fs$ we can rewrite this as
\beu 
0=\frac{1}{2}\big(\Tr_\mathbb{R} (\ad_{\bar\fs} X) -{i}\Tr_\mathbb{R} (J\ad_{\bar\fs} (X))\big)=0-\frac{{i}}{2}\Tr_\mathbb{R} (J\ad_{\bar\fs} X)=\frac{{i}}{2}\Tr_\mathbb{R} (|X|)\ .
\end{equation*}
We conclude that $|X|=0$ which implies that $\ad_\fs (X)=0$.  This shows (\ref{keystep}) and hence transitivity of $L$ under the assumption that $L$ and $H$ are strongly reductive in $G$.

For general connected reductively embedded $L,H\subset G$ we apply the above to the strong closures (observe that $\bar L$ still acts properly and cocompactly on $G/\bar H$) and obtain that $G=\bar L\bar H$ as well as $\fg=\bar\fl+\bar\fh$ on the level of Lie algebras. Since $\fg$ is semisimple, a result of Oni\v s\v cik (\cite{Oni69}, Thm.~3.2) tells us that only the
semisimple parts of  $\bar\fl$, $\bar\fh$ are needed for the latter decomposition. But these semisimple parts are contained in $\fl$ and $\fh$, respectively.
We obtain that $\fg=\fl+\fh$, and thus the transitivity of $L$ and $L'$ on $G/H$, where $L'$ is the analytic subgroup corresponding to the semisimple part $\fl'$ of $\fl$.
Again using the properness of the $L$-action we conclude the compactness of the quotients
$$L/L'\cong Z(L)/L'\cap Z(L)=Z(L)/Z(L')\ .$$
Now $Z(L')$ is finite, and therefore $Z(L)$ is compact. Interchanging the roles of $L$ and $H$ we also conclude the compactness of $Z(H)$. 
$\hfill\Box$

From now on we will restrict our considerations to symmetric spaces $G/H$. 
\begin{Def}\label{defck}
Let $G$ be a non-compact connected real semisimple Lie group with finite center, and let $L,H\subset G$ be closed subgroups.
We call $(G,H,L)$ a properly transitive triple if the following conditions hold:
\begin{itemize}
\item[(i)] There exists an involutive automorphism $\sigma: G\rightarrow G$ such that for any $\sigma$-stable normal subgroup $G_1\subset G$  of positive dimension the
fixed point group $G_1^\sigma$ is a proper non-compact subgroup of $G_1$
 and $(G^\sigma)_0\subset H\subset G^\sigma$.
\item[(ii)] $L\subset G$ is connected and reductively embedded.
\item[(iii)] $G=LH$.
\item[(iv)] $L\cap H$ is compact.
\end{itemize}
\end{Def}
One could rephrase the definition roughly as follows: $X\defn G/H$ is a connected semisimple symmetric space without Riemannian factors (where $G$ acts almost effectively by the full
identity component of the isometry group) and the reductively embedded connected group $L\subset G$ acts properly and transitively on $X$. Note that the absence of Riemannian
factors of $X$ implies the absence of compact factors of $G$. Thus in view of Proposition~\ref{cocpt}
the definition would not change if we only would require that $L$ acts cocompactly (instead of transitively) on $X$.

An {\em isomorphism} between two triples $(G,H,L)$ and $(G',H',L')$ is a Lie group isomorphism $\Phi: G\rightarrow G'$ which sends $H$ onto $H'$ and $L$ onto $L'$.  It is natural to do the classification of properly transitive triples $(G,H,L)$  only up to coverings of the underlying symmetric spaces $X=G/H$, i.e. to classify
the underlying triples of Lie algebras $(\fg,\fh,\fl)$. Two triples having isomorphic underlying  triples of Lie algebras are called {\em locally isomorphic}. Moreover, we don't want to distinguish too much between triples that only differ by adding or deleting compact factors of $\fl$.
Before Cor.~\ref{unfug} we introduced $\fl_{min}$ and $L_{min}$. 
Two properly transitive triples $(G,H,L)$ and $(G^{\prime},H^{\prime},L^{\prime})$ are said to be {\em equivalent} if there exists a Lie 
algebra isomorphism $\Phi:\fg\longrightarrow\fg^{\prime}$ such that 
\begin{equation}\label{equiv}\nonumber
\Phi(\fh)=\fh^{\prime}\;\text{ and }\;\Phi(\fl_{min})=\fl^{\prime}_{min}.
\end{equation}
Similarly, we define $\fl_{max}$ to be the sum of $\fl_{min}$ with a maximal compact subalgebra in the centralizer $\fz_{\fg}(\fl_{min})$ that contains $\fl\cap\fz_{\fg}(\fl_{min})$ (in general it is only determined up to certain conjugations).  Let $L_{max}$ be the corresponding analytic subgroup. Then $L_{min}\subset L\subset L_{max}$, $L_{max}/L_{min}$ is compact, $L_{max}$
is maximal subject to the conditions $L\subset L_{max}$, $L_{max}$ acts properly on $G/H$, while $L_{min}$
is minimal subject to the conditions $L_{min}\subset L$, $L_{min}$ acts transitively on $G/H$. Moreover, $(G,H,L)$ and $(G^{\prime},H^{\prime},L^{\prime})$ are equivalent if and only if there exists a Lie 
algebra isomorphism $\Phi:\fg\longrightarrow\fg^{\prime}$ such that 
\begin{equation}\nonumber
\Phi(\fh)=\fh^{\prime}\;\text{ and }\;\Phi(\fl_{max})=\fl^{\prime}_{max}.
\end{equation}

Note that $H\subset G$ is automatically strongly reductive in $G$, since we can always find a Cartan involution $\theta$ that commutes with $\sigma$. $L\subset G$ is then also strongly
reductive by Corollary~\ref{unfug}.
Using that $G=LH$ we can find a Cartan involution that stabilizes $L$ {\it and} commutes with $\sigma$. 

Therefore we can and will tacitly assume - unless stated otherwise - that $L$ (and also $H$) is stable under a given Cartan involution of $G$ that commutes with $\sigma$. Set $\fk_L\defn \fk\cap\fl$ and $\fk_H\defn\fk\cap\fh$.  Then we get the following 
decompositions on the Lie algebra level:
\begin{itemize}
\item[] $\fg=\fk\oplus\fs=\fh\oplus\fq=\fh+\fl$\ ,\quad where $\fq\defn \fg^{-\sigma}$
\item[] $\fk=\fk_H\oplus(\fk\cap\fq)=\fk_H+\fk_L$\hspace*{0.5cm}$\fs=(\fs\cap\fh)\oplus(\fs\cap\fq)=(\fs\cap\fh)\oplus(\fs\cap\fl)$
\item[]$\fh=\fk_H\oplus(\fh\cap\fs)$\hspace*{2.0cm}$\fq=(\fq\cap\fk)\oplus(\fq\cap\fs)$\hspace*{2.0cm}$\fl=\fk_L\oplus(\fl\cap\fs)$
\item[]$\fk=[\fs,\fs]$\hspace*{3.3cm}$\fh=[\fq,\fq]$
\end{itemize}
Requiring $\theta$-stability fixes the choice of $\fl_{max}$: $\fl_{max}=\fl_{min}\oplus (\fz_{\fg}(\fl_{min})\cap\fk)$, provided that already $\fl$ and hence also $\fl_{min}$ were 
$\theta$-stable. We will also use the notation $K_L\defn K\cap L$, $K_H\defn K\cap H$.

Before we come to the classification of properly transitive triples we want  to establish two general structural results.
The first has to do with real ranks.
Let us for a moment return to the more general assumptions of Prop.~\ref{cocpt}. Properness of the action implies (\cite{Kos89}, Cor. 4.2) that 
$$\Rrank(G)\ge\Rrank(H)+\Rrank(L)\ .$$
Using in addition (\ref{knax}) and (\ref{asterix}) it is easy to see that in the cocompact case we must have equality.
We want to reprove and strengthen this result for properly transitive triples $(G,H,L)$. 
Recall that the real rank $\Rrank(G/H)$ of the symmetric space $G/H$ is defined as the dimension of a maximal abelian subspace of $\fq\cap\fs$. 
\begin{pro}\label{ranksymm}
If $(G,H,L)$ is a properly transitive triple, then $$\rank_{\mathbb {R}} (G/H)=\rank_{\mathbb {R}} (L)=\Rrank(G)-\Rrank(H)\ .$$
\end{pro}
\begin{proof}
Let $\fa_L$ be a maximal abelian subspace of $\fl\cap\fs$, and let
$\fa$ be a $\sigma$-stable maximal abelian subspace of $\fs$. We want to show that
\be\label{murphy}
\dim \fa_L=\dim \fa\cap\fq\ .
\end{equation}
Taking $\fa$ such that $\fa\cap\fq$ is maximal abelian in $\fq\cap\fs$ the first equation in Prop.~\ref{ranksymm} will follow, while for $\fa$ such that $\fa\cap\fh$ is maximal abelian in $\fh\cap\fs$ we will obtain the second. We proceed in several steps.
\begin{itemize}
\item[(i)]  We may assume that $\fa_L\subset\fa$. Indeed, there exists $k\in K$ such that $\Ad(k){\fa_L}\subset\fa$. Write $k=k_Hk_L$ with $k_H\in K_H$ and $k_L\in K_L$. Now we replace $\fa_L$ by $\Ad(k_L)\fa_L$ and $\fa$ by $\Ad(k_H^{-1})\fa$.
\item[(ii)] We consider the orthogonal projection $p_\fq$ to $\fq$ (with kernel $\fh$), 
\be\label{hermann}
p_\fq(X)=\frac{1}{2}(X-\sigma(X))\ .
\end{equation} 
In view of $\sigma$-stability it maps $\fa$ to itself and thus to $\fa\cap \fq$.
Since $\fs=(\fh\cap\fs)\oplus(\fl\cap\fs)$ the map $p_\fq$ maps $\fl\cap\fs$ bijectively to $\fq\cap\fs$. We obtain
$$ \dim \fa_L = \dim p_\fq(\fa_L)\le \dim \fa\cap \fq\ .$$
\item[(iii)] Similarly as in (ii), we obtain that the orthogonal projection $p_\fh$
to $\fh$,  $p_\fh (X)=\frac{1}{2}(X+\sigma(X))$, maps $\fa$ to $\fa\cap \fh$ and
induces an isomorphism of  $(\fl\cap\fs)^{\perp_\fs}$ onto $\fh\cap\fs$. Here $(\fl\cap\fs)^{\perp_\fs}$ denotes the orthogonal complement of $\fl\cap\fs$ in $\fs$ with respect to the Killing form. We obtain
$$ \dim  (\fa\cap(\fl\cap\fs)^{\perp_\fs}) = \dim p_\fh(\fa\cap(\fl\cap\fs)^{\perp_\fs})\le \dim \fa\cap \fh\ .$$
\item[(iv)] 
We have $\fa=\fa_L\oplus (\fa\cap(\fl\cap\fs)^{\perp_\fs})$. Indeed, for $X\in\fa$, write $X=X_\fl+X^\perp$ with $X_\fl\in\fl\cap\fs$ and $X^\perp\in(\fl\cap\fs)^{\perp_\fs}$. Then one has:
\begin{eqnarray*}
\{0\}&=&\lbrack\fa_L,X\rbrack\\
&=&\lbrack\fa_L,X_\fl\rbrack+\lbrack\fa_L,X^\perp\rbrack
\end{eqnarray*}
where, since $\lbrack\fl,\fl^\perp\rbrack\subset\fl^\perp$, the first term is contained in $\fl$ and the second term in $\fl^\perp$. We deduce that both 
$\lbrack\fa_L,X_\fl\rbrack$ and $\lbrack\fa_L,X^\perp\rbrack$ are trivial. In particular, $X_\fl\in\fa_L$ by maximality, and hence also $X^\perp\in\fa$. 
\end{itemize}
Since $\fa=\fa\cap\fh\oplus\fa\cap\fq$ the steps (ii), (iii) and (iv) now imply (\ref{murphy}).
\end{proof}
Note that the equality $\rank_{\mathbb {R}} (G/H)=\Rrank(G)-\Rrank(H)$ already tells us that many symmetric spaces $X=G/H$ cannot appear as the underlying
symmetric space of a properly transitive triple.

\begin{Def}\label{sphertriples}
A properly transitive triple $(G,H,L)$ is said to be {\it spherical} if a minimal parabolic subgroup $P_L$ of $L$ acts transitively on $G/H$.
\end{Def}

Recall that a homogeneous space  of a reductive group $L$ is called (real) spherical if a minimal parabolic subgroup of $P_L$ of $L$ has an open orbit on it (see e.g. \cite{KKPS}).
Under our assumption of properness this is equivalent to the transitivity of $P_L$. In particular, $(G,H,L)$  is a spherical triple if and only if the homogeneous space $L/L\cap H$ is
(real) spherical. We also have the following equivalent characterizations of being spherical:
\be\label{char}  L\cap H\mbox{ acts transitively on } L/P_L \Leftrightarrow  L\cap H\mbox{ acts transitively on } K_L/M_L\ ,
\end{equation}
where $M_L\defn Z_{K_L}(\fa_L)$. It is the last characterization that we will employ most often.

The property of being a spherical triple is invariant under local isomorphisms but not invariant under equivalences of properly transitive triples, in general. This motivates the following definition.
\begin{Def}\label{types}
A properly transitive triple $(G,H,L)$ is said to be:
\begin{itemize}
\item[(i)] of {\it Type~I} if $(G,H,L_{max})$ is spherical.
\item[(ii)]  of {\it Type Ia} if $(G,H,L_{min})$ is spherical.
\item[(iii)] of {\it Type Ib} if $(G,H,L_{max})$ is spherical but $(G,H,L_{min})$ is not spherical.
\item[(iv)] of {\it Type~II} if $(G,H,L_{max})$ is not spherical.
\end{itemize}
\end{Def}
\noindent
Now we come to the second structural result.
The product of two properly transitive triples $(G_1,H_1,L_1)$ and $(G_2,H_2,L_2)$ is the properly transitive triple $(G_1\times G_2, H_1\times H_2, L_1\times L_2)$.
A properly transitive triple $(G,H,L)$ is said to be {\it decomposable} if it is equivalent to a non-trivial product, and {\it indecomposable} otherwise. A properly transitive triple $(G,H,L)$ is said to be {\it irreducible} if the  underlying symmetric 
space $X=G/H$ is irreducible.

\begin{rem}\label{remindecomp}
An irreducible properly transitive triple is indecomposable, but the converse is not true. E.g., the triple $(G^\prime\times G^\prime\times G^\prime\times G^\prime,\Delta_{12}G^\prime\times\Delta_{34}G^\prime,G^\prime\times\Delta_{23}G^\prime)$ 
is an indecomposable properly transitive triple but $G/H\simeq G^\prime\times G^\prime$ is not irreducible, for any real simple Lie group $G^\prime$. However, the following proposition
tells us that a triple of Type~I  is irreducible if and only if it is indecomposable.
\end{rem}
\begin{pro}\label{productck}
Let $(G,H,L)$ be a properly transitive triple of Type $I$. Then it is equivalent to a (finite) product of irreducible triples $(G_i,H_i, L_i)$ of Type~I.
\end{pro}
\begin{proof}
Note that a product of triples is of Type~I if and only if each factor is of Type~I. Thus it suffices to prove the following: Suppose $(G,H,L)$ is a triple of Type~I such that $\fg$ is the direct sum $\fg_1\oplus\fg_2$ of $\sigma$-stable ideals. Then there are subalgebras $\fl_j\subset \fg_j$ such that $\fl_{min}=\fl_1\oplus \fl_2$. 

Then $\fq\cap\fs$ splits, as an $L\cap H$-module, into a direct 
sum $\fq_1\oplus\fq_2$, with $\fq_j\subset\fg_j$, $j=1,2$. Write $E_j$ for the inverse image of $\fq_j$ via the isomorphism which sends $X\in\fl\cap\fs$ to the element $\frac{X-\sigma(X)}{2}\in\fq\cap\fs$. 
In particular, $\fl\cap\fs$ is the direct sum $E_1\oplus E_2$ of $L\cap H$-modules. Put ${\fl}_j\defn E_j+\lbrack E_j,E_j\rbrack$ for $j=1,2$. We have $\fl_j\subset \fl$.

{\it Claim 1: ${\fl}_j$ is an ideal in $\fl$ for $j=1,2$.}\\
{\it Proof of Claim 1.} By replacing $L$ by $L_{max}$ if necessary, we may assume that $(G,H,L)$ is spherical, i.e. $L\cap H$ acts transitively on $K_L/M_L$.
The stabilizer in $K_L$ of any $X\in \fl\cap\fs$, in particular of any $X\in E_j$, contains a subgroup conjugate to $M_L$. We conclude that $\Ad(K_L)X=\Ad(L\cap H)X$.
Hence $E_j$ is $K_L$-invariant. 
In particular, one has $\lbrack \fk_L,{\fl}_j\rbrack\subset{\fl}_j$. Now it follows that also
$\lbrack E_j,{\fl}_j\rbrack\subset{\fl}_j$. Moreover, by invariance of the Killing form of $\fg$ restricted to $\fl$, 
one gets $\lbrack E_j^\perp,E_j\rbrack=\{0\}$, where $E_j^\perp$ denotes the orthogonal of $E_j$ in $\fl\cap\fs$ with respect to the Killing form of $\fg$ restricted to $\fl$. The claim now follows. 

{\it Claim 2:  ${\fl}_1\cap{\fl}_2=\{0\}$.}\\
{\it Proof of Claim 2.} By definition of ${\fl}_j$, one has 
${\fl}_1\cap{\fl}_2=\lbrack E_1,E_1\rbrack\cap\lbrack E_2,E_2\rbrack\subset [\fl\cap\fs,\fl\cap\fs]\subset\fk$. Thus $[\fl\cap\fs,\fl_1\cap\fl_2]\subset\fs$.
On the other hand, by Claim 1, $[\fl\cap\fs,\fl_1\cap\fl_2]\subset \fl_1\cap\fl_2\subset\fk$. We obtain $[\fl\cap\fs,\fl_1\cap\fl_2]=0$ and therefore
$\IP{{\fl}_1\cap{\fl}_2}{\lbrack\fl\cap\fs,\fl\cap\fs\rbrack}=0$. Claim 2 follows.

{\it Claim 3:  $E_j\subset\fg_j$ for $j=1,2$.}\\
{\it Proof of Claim 3.} Since $\fg_1$ and $\fg_2$ are both stable under the involution $\sigma$, given $X\in E_1$ and $Y\in E_2$, we may write: $X=A_1+B_1+B_2$ 
and $Y=A_2^\prime+B_1^\prime+B_2^\prime$, with $A_j,A_j^\prime\in\fg_j\cap\fq\cap\fs$ and $B_j, B_j^\prime\in\fg_j\cap\fh\cap\fs$. 
By Claim~2, $\lbrack X,Y\rbrack=0$ which implies that $\lbrack A_1,B_1^\prime\rbrack=\lbrack B_2,A_2^\prime\rbrack=0$. In particular, we have 
$\lbrack B_2,\fq\cap\fs\rbrack=\{0\}$ and therefore $0=\IP{\lbrack B_2,\fq\cap\fs\rbrack}{\fq\cap\fk}=\IP{B_2}{\lbrack\fq\cap\fs,\fq\cap\fk\rbrack}=\IP{B_2}{\fh\cap\fs}$, i.e $B_2=0$. We conclude $E_1\subset\fg_1$. With the same arguments, one proves that $E_2\subset\fg_2$ which implies the claim.

Claim 3 implies that $\fl_j\subset\fg_j$. Since $\fl\cap\fs=E_1\oplus E_2$ and thus, by construction, $\fl_1\oplus\fl_2$ is generated by $\fl\cap\fs$,
we conclude that $\fl_1\oplus\fl_2=\fl_{min}$.
\end{proof}

The following classification of {\em irreducible} properly transitive triples is a rather direct consequence of the general classification results of Oni\v s\v cik \cite{Oni69} on decompositions of semisimple Lie groups
$G$ into products $G=LH$, where $L,H\subset G$ are reductively embedded. Note that Oni\v s\v cik neither requires the compactness condition (iv) nor the symmetry
condition (i) in Definition \ref{defck}. However, he requires that $G$ acts almost effectively not only on $G/H$ but also on $G/L$. It is only this last requirement
which prevents us to read of a full classification of all {\em indecomposable} triples directly from Oni\v s\v cik's results. 
\begin{thm}\label{listck}
Let $(G,H,L)$ be an irreducible properly transitive triple. Then either $G$ is simple and $(G,H,L)$ is equivalent to exactly one of the triples described in Table 1, or $G/H=
G_1\times G_1/\Delta G_1$ is a group manifold and $(G,H,L)$ is equivalent to exactly one of the triples described in Table \ref{foxi}. Strictly speaking, in Case~12 one has to run only over
the Cases 1,3,5,6,7,8 in Table 1 in order to avoid doublings arising by interchanging $L_1$ and $H_1$.  Vice versa, each row of Tables 1 and 2 defines an irreducible
properly transitive triple. In the column $\bf L$ we always list $L_{max}$; $L_{min}$ just arises by removing all compact factors. For $\bf H$ we always list the connected choice $H=(G^\sigma)_0$.
\setcounter{table}{0}
\begin{table}[h]
\caption{$G$ simple}\label{fix}
\begin{center}
\scalebox{1}{
\begin{tabular}{|r||l||c||c||c||c||r|}
  \hline
  &${\bf G}$ & ${\bf H}$ & ${\bf L}$ & ${\bf L\cap H}$ & ${\bf \text{\bf rank}(G/H)}$ & {\bf type}\\
  \hline
    \hline
  \bf 1& $SO_e(2,2n)$, $n\ge 2$ & $SO_e(1,2n)$ & $U(1,n)$ & $U(n)$ & $1$ & Ia \\
   \hline
    \hline
    \bf 2&$SO_e(2,2n)$, $n\ge 3$ & $U(1,n)$ & $SO_e(1,2n)$ & $U(n)$ & $\lbrack\frac{n+1}{2}\rbrack$ & Ia \\
   \hline
    \hline
 \bf 3 &$SU(2,2n)$, $n\ge 1$  & $S(U(1)\times U(1,2n))$ & $Sp(1,n)$ & $U(1)\times Sp(n)$ & $1$ & Ia \\
   \hline
    \hline
    \bf 4& $SU(2,2n)$, $n\ge 2$& $Sp(1,n)$ & $S(U(1)\times U(1,2n))$ & $U(1)\times Sp(n)$ & $n$ & Ia \\
  \hline
    \hline
    \bf 5& $SO_e(4,4n)$, $n\ge 1$ & $SO_e(3,4n)$ & $Sp(1,n)\cdot Sp(1)$ & $\Delta Sp(1)\cdot Sp(n)$ & $1$ & Ia \\ 
\hline
    \hline
  \bf 6    & $SO_e(8,8)$& $SO_e(7,8)$ & $Spin(1,8)$ & $Spin(7)$ & $1$ & Ia \\
   \hline
    \hline
    \bf 7    & $SO(8,{\mathbb C})$& $SO(7,{\mathbb C})$ & $Spin(1,7)$ & $G_2$ & $2$ & Ia \\
     \hline
    \hline
    \bf 8      & $SO_e(3,4)$& $SO(2)\times SO_e(1,4)$ & $G_{2(2)}$ & $U(1)\times SU(2)$& $2$ & II \\   
 \hline
    \hline
     \bf 9   & $SO_e(4,4)$& $SO(3)\times SO_e(1,4)$ & $Spin(3,4)$ &$SU(2)\times SU(2)$ & $3$ & II \\
   
   \hline
    \hline
    \bf 10    & $SO(8,{\mathbb C})$& $SO_e(1,7)$ & $Spin(7,{\mathbb C})$ & $G_2$ & $4$ & II \\
   \hline
\end{tabular}
}
\end{center}
\end{table}
\begin{table}[h]
\caption{Group cases}\label{foxi}
\begin{center}\scalebox{1}{
\begin{tabular}{|r||l||c||c||c||r|}
  \hline
  &${\bf G_1}$  & ${\bf L}$ & ${\bf L\cap H}$ & ${\bf \text{\bf rank}(G/H)}$ & {\bf type}\\
  \hline
    \hline
   \bf 11&   non-compact simple &  $G_1\times K_1$, $K_1\subset G_1$ max. compact & $\Delta K_1$ & $\text{rank}_{{\mathbb C}}(G_1)$ & Ib \\
     \hline
    \hline
  \bf 12& as $G$ in Table 1 & $L_1\times H_1$ with $(L_1,H_1)$ as $(L,H)$  in Table 1& $\Delta(L_1\cap H_1)$ & $\text{rank}_{{\mathbb C}}(G_1)$ & II \\
   \hline
\end{tabular}}
\end{center}
\end{table}
\end{thm}
\begin{proof}
We have to distinguish between 3 cases: $\fg$ is simple and does not admit the structure of a complex Lie algebra, $\fg$ is a complex simple Lie algebra, or $G/H$ is a group manifold
(we call it the group case). In the first case we apply Thm.~4.1 (Table 2) in  \cite{Oni69}. We have just to select those cases from Oni\v s\v cik's list, where $L\cap H$ is compact and one of the two subgroups
is symmetric. This is easily done because the description of $L\cap H$ is also given in this list. We get precisely (up to equivalence) Cases 1--6, 8, 9 of our Table 1.
We made a choice of excluding certain spaces with small parameter $n$ in order to avoid doublings in the table caused by special isomorphisms in small dimensions.
Note that the possible abelian factors of $H$ and $L$ are not written down in Oni\v s\v cik's list. However, they can be found in his earlier list of decompositions of compact groups
(\cite{Oni62}, Thm.~4.1, Table 7). For the second case we use \cite{Oni69}, Thm.~4.2, instead and observe that in this case not both $L$ and $H$ can be complex. Indeed, if they
were, then the intersection $L\cap H$ would be a non-trivial (which follows from the one-dimensionality of the third cohomology of simple Lie algebras - one may also consult the classification \cite{Oni62}, Thm.~4.1)  complex reductive linear group, i.e. $L\cap H$ would be non-compact.
We obtain the cases 7 and 10 of our Table~1.

We now discuss the group case: Either $L$ contains $G_1$, then we obtain Case 11, or $G$ acts almost effectively on $G/L$, which means in Oni\v s\v cik's language: the
triple $(\fg,\fh,\fl)$ is an effective decomposition. Now Thm.~ 4.3 in \cite{Oni69} tells us that this triple is what Oni\v s\v cik calls primitive, i.e. it can be constructed in an explicit
way (described on pp.~564/565 in \cite{Oni69}) from decompositions of simple Lie algebras (i.e. in our situation from decompostions $\fg_1=\fl_1+\fh_1$ with $\fl_1\cap\fh_1$ compact). Observing that all these decompositions
already appear in our Table 1 we obtain the description given as Case~12.

It remains to determine the types for all the triples $(G,H,L)$ appearing in Table 1 and 2. We just have to check whether $L_{*}\cap H$ acts transitively on $K_{L_*}/M_{L_*}$,
where $*=min$ or $max$. For Cases 1--7 we just employ the transitivity of the groups $SU(n)$, $Sp(n)$, $G_2$,  and $Spin(7)$ on the spheres $S^{2n-1}$, $S^{4n-1}$, $S^6$ and
$S^7$, respectively. In Cases 8--10 we have $\dim (L_{max}\cap H)<\dim (L_{max}/P_{L_{max}})$. Note that in Case 11 we have $L_{min}=G_1$ and $L_{min}\cap H =\{e\}$. Thus $L_{min}\cap H$
does not act transitively on $K_{L_{min}}/M_{L_{min}}=K_1/M_1$, where $M_1$ is the $M$-component of a minimal parabolic of $G_1$. Note that $K_{L_{max}}/M_{L_{max}}=K_1/M_1\times K_1/K_1$. Thus $L_{max}\cap H=\Delta K_1$ acts transitively on this manifold. 

There are several ways to show that all the triples in Case~12 are of Type~II.
One possibility is the following: $K_{L_{max}}/M_{L_{max}}=K_{L_1}/M_{L_1}\times K_{H_1}/M_{H_1}$, where $(G_1, H_1,L_1)$ runs through the triples in Cases 1,3,5,6,8,10.
In the last 2 cases $(G_1, H_1,L_1)$ is of Type~II, i.e. $L_1\cap H_1$ does not act transitively on $K_{L_1}/M_{L_1}$. Thus $\Delta(L_1\cap H_1)$ does not act transitively on 
$K_{L_1}/M_{L_1}\times K_{H_1}/M_{H_1}$. In the remaining cases one can find an $L_1\cap H_1$-equivariant map $\Phi: K_{H_1}/M_{H_1}\rightarrow K_{L_1}/M_{L_1}$.
Then the graph of $\Phi$ is a proper $\Delta(L_1\cap H_1)$-invariant subset of $K_{L_{max}}/M_{L_{max}}$, so $\Delta(L_1\cap H_1)$ does not act transitively.
\end{proof}
We conclude this section by a list of comments on Theorem~\ref{listck}.
\begin{itemize}
\item Combining the theorem with Proposition \ref{productck} gives a full classification of properly transitive triples of Type~I.
\item For a full classification of all properly transitive triples one would still need a classification of indecomposable non-irreducible triples $(G,H,L)$ (which are necessarily of Type~II).
It is possible to give such a classification based on a generalization of Thm.~ 4.3 in \cite{Oni69}
to certain non-effective triples. We will not formulate the result here.
All such triples arise by certain diagonal constructions applied to irreducible triples. The simplest example of such a diagonal construction is given in Remark \ref{remindecomp}. It can be generalized to $k$-fold
products of group manifolds for all $k\ge 2$. 
\item A similar class of examples is given by triples of the form $(L_1\times L_1\times G_1, \Delta_{12}L_1\times H_1, L_1\times \Delta_{23} L_1)$,
where $(G_1, H_1, L_1)$ runs through the triples in Table~1 (with $L_1=L_{min}$). If we require in addition that $G_1/L_1$ is symmetric (but now $L_1=L_{max}$), we can also form the indecomposable triple
$(G_1\times G_1, L_1\times H_1, \Delta G_1)$, i.e. the triple formed by interchanging the roles of $L$ and $H$ in Case~12, Table 2.
\item Another interesting indecomposable triple arises when connecting Cases 4 and 5 in Table 1 by a diagonal construction: $(SU(2,2n)\times SO_e(4,4n), Sp(1,n)\times SO_e(3,4n),\Delta SU(2,2n))$.
\item Inspecting Oni\v s\v cik's  lists (Thm.~4.1 (Table 2) and  Thm.~4.2 in  \cite{Oni69}) of decompositions $G=LH$, $L\cap H$ compact, with $G$ simple, one observes that either $L$ or $H$ is
symmetric (up to compact factors). It follows that Table~1, if we allow to interchange the roles of $H$ and $L$, also gives a classification of all triples $(G,H,L)$ satisfying the assumptions of Prop.~\ref{cocpt}
with $G$ simple up to a natural equivalence relation.
\item It is worth noting that for a properly transitive triple $(G,H,L)$ with $G$ simple (see Table 1) the following properties are equivalent: (i) the triple is of Type~I; (ii) it is of Type Ia; 
(iii) $\rank_{\mathbb {R}} (L)=1$.
\item Note that the triple in Table~1, Case~1 also makes sense for $n=1$. But in this case $G$ is not simple. In fact, the triple is equivalent to Case 11 with $G_1=SL(2,\R)$.
It is the properly transitive triple of smallest dimension $d$, namely  $d=3$. By dimension we mean the dimension of the underlying symmetric space $X=G/H$.
The triple of smallest possible dimension with $G$ simple is given by Case 1, $n=2$. Here $d=5$. The smallest irreducible triple of Type~II is given by Case 8, here $d=10$.
Note that there is a smaller non-irreducible triple of Type~II with $d=6$: take $G'=SL(2,\R)$ in Remark \ref{remindecomp}.
\item All the triples listed in the above tables also appear in several papers of Kobayashi and his coauthors, e.g. in  \cite{KY05}. In their earlier papers, classification is not discussed.
The classification given in the later paper \cite{KK25} also rests on Oni\v s\v cik's results, but is given under slightly different assumptions. See also the paper by Tojo (\cite{To19}), where a classification of irreducible symmetric spaces admitting compact standard quotients is announced. Tojo works without the knowledge of Prop.~\ref{cocpt} and therefore did not build on the the work of  Oni\v s\v cik.
\end{itemize}


\section{Embedding of Casimir operators}\label{pbw}

Let $G$ be a connected non-compact semisimple real Lie group with finite center, $K$ a maximal compact subgroup of $G$ associated with a Cartan involution $\theta$ 
and $H\subset G$ a reductively embedded closed subgroup such that $G/H$ is a symmetric space with respect to an involution $\sigma$ commuting with $\theta$. 
We first recall the description of the algebra ${\bf D}(G/H)$ of left $G$-invariant differential operators on $G/H$ in terms of the universal enveloping algebra ${\Cal U}(\fg)$ of $\fg^{\mathbb C}$ of $G$. The right translation $r$ of $G$ on 
$G$
%
%
induces an action, still denoted $r$, of the enveloping algebra ${\Cal U}(\fg)$ on $C^\infty(G)$
%
%
by left-invariant differential operators. On the other hand, $H$ acts on ${\Cal U}(\fg)$ via the adjoint 
action, we write ${\Cal U}(\fg)^{H}$ for the subalgebra of $H$-invariant elements. This induces an action of ${\Cal U}(\fg)^{H}$ on the vector space $C^{\infty}(G/H)$ of smooth complex functions on $G/H$, viewed as right $H$-invariant smooth functions on $G$. One obtains a homomorphism of algebras
\begin{equation}\label{homo1}
r: {\Cal U}(\fg)^{H}\rightarrow {\bf D}(G/H),\;A\mapsto R(A).
\end{equation}
%
Define $Z(G/H)$ to be the subalgebra $r({\Cal Z}(\fg))$ of ${\bf D}(G/H)$, where ${\Cal Z}(\fg)$ is the center of ${\Cal U}(\fg)$. Then it is known that (see Proposition 4.1, Theorem 4.3 and Lemma 4.4 of \cite{HS94}):
\begin{itemize}
\item[(i)] $r$ induces an isomorphism of algebras: 
$${\bf D}(G/H)\cong{\mathcal U}(\fg)^{H}/({\Cal U}(\fg)^{H}\cap {\Cal U}(\fg)\fh)\cong\left [{\mathcal U}(\fg)/{\Cal U}(\fg)\fh\right ]^H. $$
These isomorphisms hold for
all homogeneous spaces $G/H$ with $H\subset G$ reductively embedded.
\item[(ii)] ${\bf D}(G/H)$ is isomorphic to the algebra $S(\mathfrak{g})^H$, of $H$-invariant polynomials on $\mathfrak{g}_{\mathbb C}$, with $\rank(G/H)$ independent generators. In particular, ${\bf D}(G/H)$ is a commutative algebra. 
\item[(iii)]  If $G$ is a classical Lie group, or if the rank of $G/H$ is one, then $Z(G/H)={\bf D}(G/H)$.
\end{itemize}

Assume now that $(G,H,L)$ is a properly transitive triple. Then, as above, write ${\bf D}(L/L\cap H)$ for the algebra of left $L$-invariant differential operators on $L/L\cap H$ which we will
identify with 
$${\mathcal U}(\fl)^{L\cap H}/({\Cal U}(\fl)^{L\cap H}\cap {\Cal U}(\fl)(\fl\cap \fh)). $$
By the transitivity of the $L$-action on $G/H$, the algebra ${\bf D}(L/L\cap H)$ can be viewed as the algebra of $L$-invariant  differential operators on $G/H$ which contains ${\bf D}(G/H)$. Thus we have an embedding
\begin{equation}\label{embed1}
{\bf D}(G/H)\stackrel{\imath}\hookrightarrow {\bf D}(L/L\cap H).
\end{equation}

As always, we consider a Cartan involution stabilizing $L$ and commuting with $\sigma$. We fix a $G$-, $\theta$- and $\sigma$-invariant nondegenerate bilinear form $\IP{}{}$ on $\fg$ (not necessarily the Killing form to allow for more flexibility). In the remainder of this section we will compute the image of the corresponding Casimir operator of $G$ via the embedding (\ref{embed1}), i.e. we want to represent its image by an element of  ${\mathcal U}(\fl)^{L\cap H}$. Note that in the group case 11 (see Table 2) the corresponding embedding problem for the full algebra ${\bf D}(G/H)$  is already completely solved by the isomorphism 
\be\label{center}
{\bf D}(G/H)\cong {\Cal Z}(\fg_1),
\end{equation}
which sends the Casimir of $G$ to twice the Casimir of $G_1=L_{min}$.

The restrictions of the bilinear form to ${\mathfrak h}$ and ${\mathfrak l}$, still denoted $\IP{}{}$, are non-degenerate. Since ${\mathfrak q}$ is orthogonal to ${\mathfrak h}$, with respect to $\IP{}{}$, and ${\mathfrak g}={\mathfrak l}+{\mathfrak h}$, observe that 
\begin{equation}\label{qperpl}
{\mathfrak q}\cap{\mathfrak l}^\perp=\{0\}.
\end{equation}
Let $F$ be the orthogonal complement of $\fl\cap\fh$ in $\fl\cap\fk$ with respect to $\IP{}{}$, so that 
$$E\defn F\oplus(\fl\cap\fs)$$ 
is the 
orthogonal complement of $\fl\cap\fh$ in $\fl$ and
\begin{equation}\label{defE}\nonumber
\fl=(\fl\cap\fh)\oplus E.
\end{equation} 
Note that both summands $F$ and $\fl\cap\fs$ of $E$  are non-zero vector spaces invariant under the (adjoint) action of $L\cap H$. 
We want to construct convenient bases for ${\mathfrak q}$ and $E$. For this, we consider the additional symmetric bilinear form $(\;,\;)$ on $\fl$ defined by 
\beu
(X,Y)\defn\IP{ \sigma (X)}{Y} .
\end{equation*}
It is not necessarily non-degenerate.
For $\beta\in\mathbb{R}$, we consider the vector subspace $\fl(\beta)$ 
of $\fl$ defined by
\beu
\fl(\beta)=\{X\in\fl\mid (X,Y)=\beta\IP{X}{Y}\;\;\forall Y\in\fl\}.
\end{equation*}
The next lemma collects useful properties of the $\fl(\beta)$'s.
\begin{lem}\label{basislem}
\begin{itemize}
\item[(1)] $\fl(\beta)$ is  $L\cap H$-invariant and $\theta$-stable. In particular,
$$\fl(\beta)=(\fl(\beta)\cap\fk)\oplus (\fl(\beta)\cap\fs). $$
\item[(2)] $\fl(1)=\fl\cap\fh$ and $\fl(-1)=\fl\cap\fq$.
\item[(3)] If $\fl(\beta)\ne \{0\}$, then $-1\leq \beta\leq 1$.
\item[(4)] There is a decomposition 
\beu
\fl=\bigoplus_{\beta}\fl(\beta)
\end{equation*}
which is orthogonal with respect to $\IP{}{}$, i.e $\fl(\beta)\perp\fl(\beta^\prime)$ if $\beta\neq\beta^\prime$. \\
\item[(5)]  One has:
\begin{eqnarray*}
E&=&\bigoplus_{\beta\ne 1}\fl(\beta)\\
F&=&\bigoplus_{\beta\ne 1} (\fl(\beta)\cap\fk)\\
\fl\cap\fs&=&\bigoplus_{\beta\ne 1} (\fl(\beta)\cap\fs)
\end{eqnarray*}
Define $\text{Spec}(F):=\{\beta\in{\mathbb R\setminus\{1\}}\mid \fl(\beta)\cap \fk\neq\{0\}\}$, $\text{Spec}(\fl\cap\fs):=\{\beta\in{\mathbb R}\mid \fl(\beta)\cap\fs\neq\{0\}\}\subset [-1,1)$ and 
$\text{Spec}(G,H,L):=\text{Spec}(F)\cup \text{Spec}(\fl\cap\fs)$.
\item[(6)] We have $[\fl(-1),\fl(\beta)]\subset \fl(-\beta)$. If $L\cap K$ is $\sigma$-stable, then $\text{Spec}(F)=\{-1\}$. If moreover $\fl\cap\fs$ is irreducible, under the adjoint action of $L\cap H$, then $\text{Spec}(\fl\cap\fs)=\{0\}$.
\item[(7)] Fix $\beta\in\mathbb{R}$. If $X\in \fl(\beta)$, then $\lbrack X,\sigma(X)\rbrack=0$.
\item[(8)] If $G/L$ is a symmetric and $\IP{}{}$ is invariant under the corresponding involution, then 
$$\text{Spec}(G,L,H)\setminus\{-1\}=\text{Spec}(G,H,L)\setminus\{-1\}. $$
\end{itemize}
\end{lem}
\begin{proof}
Assertion (1) follows from the fact that the operators $\Ad(l)$, $l\in L\cap H$, and $\theta$ commute with $\sigma$ and preserve $\fl$ and $\IP{}{}$.
For (2) we first observe the obvious inclusions $\fl\cap\fh\subset\fl(1)$ and $\fl\cap\fq\subset\fl(-1)$. To show the opposite inclusion and Assertion (3) we take $0\ne X\in\fl(\beta)\cap\fk$.
Then $\sigma(X)=\beta X+Y$ for some $Y\in\fl^\perp\cap\fk$. Since $\IP{}{}$ is $\sigma$-invariant we obtain
$$ (1-\beta^2)\IP{X}{X}=\IP{Y}{Y}\ .$$
Using the definiteness of $\IP{}{}$ on $\fk$ we conclude $1-\beta^2\ge 0$ and, for $\beta=\pm 1$, that $Y=0$, i.e. $\sigma(X)=\pm X$. The same argument
works for $0\ne X\in\fl(\beta)\cap\fs$. We conclude by $\theta$-stability of $\fl(\beta)$, see (1).
Next, for (4), the decomposition $\fl=(\fl\cap \fk)\oplus(\fl\cap\fs)$ is orthogonal with respect to both forms $\IP{}{}$ and $(\;,\;)$. Moreover the restriction of $\IP{}{}$ to either summand is definite. Then we diagonalize $(\;,\;)$ with respect to $\IP{}{}$ 
on both summands separately. Now (5) follows directly from (1), (2) and (4). 
A simple computation using $\fl(-1)=\fl\cap\fq$, see (2), shows the first part of (6).
To prepare the proof of the last statement we claim that in general 
$[F,\fl\cap \fs]\ne\{0\}$. To see this, we look at the
orthogonal decomposition into ideals $\fl=\fl_{min}\oplus \fc$, with $\fl_{min}$ semisimple of non-compact type and $\fc\subset\fk$. Then
$$ \ker\Big (\ad_{\fl\cap\fs} :\fl\cap \fk \rightarrow \fgl(\fl\cap\fs)\Big )=\fc\ .$$
Since $L_{min}\cap K$ still acts transitively on $K/H\cap K\cong L\cap K/L\cap H$ the natural inclusion
$$   \fl_{min}\cap\fk/\fl_{min}\cap \fh \rightarrow ((\fl_{min}\cap\fk)\oplus\fc )/\fl\cap\fh $$
is surjective, which implies: For every $X\in\fc$ there is an $Y$ in $\fl_{min}$ such that $X-Y\in \fl\cap\fh$. If $X\in\fc\cap F$, then
$\langle X,X\rangle=\langle X,X-Y\rangle=0$. The claim follows. Now we assume that $L\cap K$ is $\sigma$-stable. Then $\fl\cap\fk=(\fl\cap\fh)\oplus(\fl\cap\fk\cap\fq)$ (orthogonal decomposition). Hence, $F=\fl\cap\fk\cap\fq$ and $\text{Spec}(F)=\{-1\}$. 
Now suppose in addition that $\fl\cap\fs$ is irreducible. Then $\fl\cap\fs\subset\fl(\mu)$ for some $\mu$ and we obtain using the first assertion 
$\{0\}\neq [F,\fl\cap\fs]\subset\fl\cap\fs\cap [F,\fl(\mu)]\subset\fl(\mu)\cap\fl(-\mu)$. Thus $\mu=0$. For (7), let $X\in\fl(\beta)$ and $Y\in \fl$. Then
\begin{eqnarray*} 
\IP{\lbrack X,\sigma(X)\rbrack}{Y}&=&-\IP{ \sigma(X)}{\lbrack X,Y\rbrack}
=-\beta\IP{X}{\lbrack X,Y\rbrack}\\&=&\beta\IP{\lbrack X,X\rbrack}{Y}
=0.
\end{eqnarray*}
Since $\lbrack X,\sigma(X)\rbrack\in\fq$ and $\fq\cap\fl^{\perp}=\{0\}$, we deduce that $\lbrack X,\sigma(X)\rbrack=0$. 
Here $\fl^{\perp}$ denotes, as in (\ref{qperpl}), the orthogonal complement of $\fl$ in $\fg$ with respect to $\IP{}{}$. Finally, for (8), suppose that $G/L$ is symmetric with respect to an involution $\tau$. Consider the orthogonal projections 
$$p_\fh:\fl\rightarrow\fh,\; X\mapsto\frac{1}{2}(X+\sigma(X))\;\;\text{ and }\;\;p_\fl:\fh\rightarrow\fl,\; X\mapsto\frac{1}{2}(X+\tau(X)).$$ 
In particular, the eigenspaces $\fl(\beta)$ and $\fh(\beta)$ can be written as follows:
\begin{eqnarray*}
\fl(\beta)&=&\Big\{X\in\fl\mid\langle p_\fh(X),Y\rangle=\frac{1+\beta}{2}\langle X,Y\rangle\;\forall Y\in\fl\Big\}\\
\fh(\beta)&=&\Big\{X\in\fh\mid\langle p_\fl(X),Y\rangle=\frac{1+\beta}{2}\langle X,Y\rangle\;\forall Y\in\fh\Big\}.
\end{eqnarray*}
For $\beta\in\text{Spec}(G,H,L)$, $X\in\fl(\beta)$ and $Z\in\fh$, we have successively
\begin{eqnarray*}
\langle p_\fl(p_\fh(X)),Z\rangle&=&\langle p_\fh(X),p_\fl(Z)\rangle\;\;\text{ by self-adjointness}\\
&=&\frac{1+\beta}{2}\langle X,p_\fl(Z)\rangle\;\;\text{ since }X\in\fl(\beta)\\
&=&\frac{1+\beta}{2}\langle X,p_\fh(Z)\rangle\;\;\text{ since }X\in\fl\text{ and }Z\in\fh\\
&=&\frac{1+\beta}{2}\langle p_\fh(X),Z\rangle\;\;\text{ by self-adjointness}
\end{eqnarray*}
which implies that $p_\fh(X)\in\fh(\beta)$ and proves (8).
\end{proof}

The decomposition (4) in Lemma \ref{basislem} behaves well with respect to Cartan subalgebras. Namely, let ${\mathfrak t}$ denote either a $\sigma$-stable Cartan subalgebra of ${\mathfrak k}$ (resp. ${\mathfrak g}$) that contains a Cartan subalgebra of ${\mathfrak l}\cap{\mathfrak k}$ (resp. ${\mathfrak l}$), or a $\sigma$-stable maximal abelian subspace of ${\mathfrak s}$ containing ${\mathfrak l}\cap{\mathfrak s}$. If, as above, one {\it diagonalizes} $(\ ,\,)$ on ${\mathfrak l}\cap{\mathfrak t}$, then one can see that
\begin{equation}\label{cartanprinciple}
{\mathfrak l}\cap{\mathfrak t}=\bigoplus_\beta({\mathfrak l}\cap{\mathfrak t})(\beta)
\end{equation}
with $({\mathfrak l}\cap{\mathfrak t})(\beta)\subset{\mathfrak l}(\beta)$. We will use this property in Appendix A for the explicit computation of $\text{Spec}(F)$ and $\text{Spec}({\mathfrak l}\cap{\mathfrak s})$ for Triples 8--10 in Table~1.

We return to invariant differential operators. Let us recall the definition of the Casimir operator $\Omega_G$. There is a sequence of $G$-equivariant maps
\begin{equation}\label{casimir}
\End(\fg)\stackrel{\text{canonical}}\longrightarrow\fg\otimes\fg^{\star}\stackrel{\IP{\;}{\;}}\longrightarrow\fg\otimes\fg
\stackrel{\text{injection}}\hookrightarrow T(\fg)\stackrel{\text{quotient}}\hookrightarrow {\Cal U}(\fg),
\end{equation}
where $\fg^\star$ denotes the vector space dual of $\fg$ and $T(\fg)$ the tensor algebra of $\fg_{\mathbb C}$. The Casimir operator $\Omega_{G}\in{\Cal Z}(\fg)={\Cal U}(\fg)^G\subset {\Cal U}(\fg)^H$ is the image of the identity through (\ref{casimir}). We will denote its image in ${\bf D}(G/H)$ via (\ref{homo1}) by the same symbol. 
This invariant differential operator
coincides with the pseudo-Riemannian Laplacian on $G/H$ corresponding to the $G$-invariant metric induced by the restriction of $\IP{}{}$ to $\fq$. 
In a similar way, one defines the Casimir 
operator $\Omega_{L}\in {\Cal Z}(\fl)$ 
with respect to the restriction of $\IP{}{}$ to $\fl$. 
For vector subspaces $W\subset\fg$, $V\subset \fl$ that are non-degenerate with respect to $\IP{}{}$ we define $\Omega_W\in{\Cal U}(\fg)$, $\Omega_V\in{\Cal U}(\fl)$
as the images of the orthogonal projections to $V$, $W$ via (\ref{casimir}). In particular, if $V$ is $L\cap H$-invariant, then $\Omega_V$ defines an element of ${\bf D}(L/L\cap H)$.
In this way we obtain $L$-invariant differential operators $\Omega_{\fl(\beta)}$ for all $\beta\in \mathrm{Spec}(G,H,L)$ and $\Omega_{L_1}:=\Omega_{\fl_1}$ for each $\theta$-stable
subgroup $L_1\subset L$ with Lie algebra $\fl_1$ normalized by $L\cap H$.
In any case, $\Omega_V$ can be expressed in terms of an orthogonal basis $\{X_j\}$ of $V$ as $\sum_{j}\frac{1}{\IP{X_{j}}{X_{j}}}X_{j}^{2}$.

The relation of the decomposition (4) in Lemma \ref{basislem} to $i(\Omega_G)\in {\mathcal U}(\fl)^{L\cap H}/({\Cal U}(\fl)^{L\cap H}\cap {\Cal U}(\fl)(\fl\cap \fh))$ is 
given by the following lemma.
\begin{lem}\label{sense}
$$ i(\Omega_G)=\sum_{\beta\in \mathrm{Spec}(G,H,L)} \frac{2}{1-\beta}\: \Omega_{\fl(\beta)}\ .$$
\end{lem}
\begin{proof} 
We consider the orthogonal projection
\beu
p_{\fq}:\fl\rightarrow\fq,\;X\mapsto\frac{1}{2}(X-\sigma(X))=X-p_\fh(X).
\end{equation*}
Its restriction to $E\subset \fl$ is an isomorphism of vector spaces. For $\beta,\beta'\in \mathrm{Spec}(G,H,L)$, $\beta\ne\beta'$, we have $p_\fq(\fl(\beta))\perp p_\fq(\fl(\beta'))$.
Moreover,
$$ \langle p_\fq(X),p_\fq(Y)\rangle=\frac{1-\beta}{2}\langle X,Y\rangle\quad\mbox{for }X,Y\in\fl(\beta)\ .$$
Now let $\{X_{\beta,j}\}$ be an orthogonal basis of $E$ adapted to the orthogonal decomposition of $E$ in Lemma~\ref{basislem}(5).
The above considerations show that $\{p_\fq(X_{\beta,j})\}$ is an orthogonal basis of $\fq$ and that we have the following equalities in ${\Cal U}(\fg)$:
\begin{eqnarray*}
 \Omega_\fq&=&\sum_{\beta,j} \frac{1}{\langle p_\fq(X_{\beta,j}), p_\fq(X_{\beta,j})\rangle} p_\fq(X_{\beta,j})^2\\
&=& \sum_\beta \frac{2}{1-\beta}\:\sum_j \frac{1}{\langle X_{\beta,j},X_{\beta,j}\rangle} p_\fq(X_{\beta,j})^2\ .
\end{eqnarray*}
By Lemma~\ref{basislem}(7) we have $[X_{\beta,j},\sigma (X_{\beta,j})]=0$ which implies $[X_{\beta,j},p_\fh(X_{\beta,j})]=0$.
It follows that
$$ p_\fq(X_{\beta,j})^2=(X_{\beta,j}-p_\fh(X_{\beta,j}))^2=X_{\beta,j}^2-2 X_{\beta,j}\,p_\fh(X_{\beta,j})+p_\fh(X_{\beta,j})^2
\equiv X_{\beta,j}^2\quad \mod {\Cal U}(\fg)\fh\ .$$
We continue the computation in the quotient space  ${\Cal U}(\fg)/{\Cal U}(\fg)\fh$.
\begin{eqnarray*}
\Omega_G=\Omega_\fq
&=& \sum_\beta \frac{2}{1-\beta}\:\sum_j \frac{1}{\langle X_{\beta,j},X_{\beta,j}\rangle}X_{\beta,j}^2\\
&=&  \sum_\beta \frac{2}{1-\beta} \: \Omega_{\fl(\beta)}\ .
\end{eqnarray*}
The lemma follows.
\end{proof}
We want to make the formula for $i(\Omega_G)$ explicit for triples $(G,H,L)$ with $G$ simple. For that we would have to compute $\mathrm{Spec}(G,H,L)$
(and the corresponding eigenspaces $\fl(\beta)$) for all triples equivalent to one of the triples listed in Table 1. If $\fl_{min}\ne\fl_{max}$, i.e. in Cases 1,4,5 in Table 1, then
$\mathrm{Spec}(F)$ (in contrast to $ \mathrm{Spec}(\fl\cap\fs)$) is not invariant under equivalences of triples. Nevertheless, we restrict our computations to one particular convenient triple in each equivalence class. This will be sufficient for applications.

The following observations will simplify these computations significantly.
\begin{itemize}
\item The cardinality of $\mathrm{Spec}(\fl\cap\fs)$ (or $\mathrm{Spec}(F)$) is bounded by the number of irreducible components of the $L\cap H$-representation $\fl\cap\fs$ (or $F$).
For $L=L_{min}$ (or $L_{max}$) the same is true if we consider $\fl\cap\fs$ (or $F$) as a representation of $L_{max}\cap H$.
\item Assume that $(G,H,L)$ is spherical and that $L$ has real rank one. Then $\fl\cap\fs$ is an irreducible $L\cap H$-representation. Indeed, under the rank one assumption the homogeneous space
$K_L/M_L$ identifies with the unit sphere in $\fl\cap\fs$.  $(G,H,L)$ being spherical means that $L\cap H$ acts transitively on it.
\item In case of irreducibility of $\fl\cap\fs$ (or $F$) the corresponding eigenvalue is just given by $\frac{\IP{\sigma(X)}{X}}{\IP{X}{X}}$ for $0\ne X\in\fl\cap\fs$ ( or $F$) arbitrary.
\end{itemize}
Note that the second observation, combined Lemma~\ref{basislem}(6), already gives the result for Cases 1--4 in Table~1. For the remaining triples the necessary computations
are done in Appendix A. It turns out that $F$ is always an irreducible $L_{max}\cap H$-representation while $\fl\cap\fs$ has at most two irreducible components.
Here we list the results. Elements of $\mathrm{Spec}(F)$ are denoted by $\lambda$, and elements of $\mathrm{Spec}(\fl\cap\fs)$ by $\mu$.
\vspace*{0.2cm}\\
\begin{equation*}
 \left.
 \begin{array}{ll}
     (SO_e(2,2n),SO_e(1,2n),U(1,n))\\
     \\
     (SO_e(2,2n),U(1,2n),SO_e(1,2n))\\
     \\
     (SU(2,2n),S(U(1)\times U(1,2n)),Sp(1,n))\\
     \\
     (SU(2,2n),Sp(1,n),S(U(1)\times U(1,2n)))
    \end{array}
   \right \}\hspace*{0.2cm}L\cap K\text{ $\sigma$-stable, $\fl\cap\fs$ irreducible, $\lambda=-1$, $\mu=0$}
  \end{equation*}
\vspace*{0.1cm}\\
 \begin{equation*}
 \left.
 \begin{array}{ll}
   (SO_e(4,4n),SO_e(3,4n),Sp(1,n)) & \text{ with }\; \lambda=0,\;\mu=\frac{1}{2}\\
     \\
     (SO_e(8,8),SO_e(7,8),Spin(1,8)) & \text{ with }\; \lambda=0,\;\mu=\frac{3}{4}\\
     \\
     (SO(8,{\mathbb C}),SO(7,{\mathbb C}),Spin(1,7))& \text{ with }\; \lambda=-\frac{1}{2},\;\mu=\frac{1}{2}
         \end{array}
   \right \}\hspace*{0.2cm}L\cap K\text{ not $\sigma$-stable, }\fl\cap\fs\text{ irreducible}
  \end{equation*}
\vspace*{0.1cm}\\
 \begin{equation*}
 \left.
 \begin{array}{ll}
  (SO(8,{\mathbb C}),SO_e(1,7),Spin(7,{\mathbb C})) & \text{ with }\; \lambda=-\frac{1}{2},\;\mu=\frac{1}{2},\;-1\\
  \\
   (SO_e(4,4),SO(3)\times SO_e(1,4),Spin(3,4)) & \text{ with }\; \lambda=-\frac{1}{2},\;\mu=\frac{1}{2},\;-1\\
    \\
     (SO_e(3,4),SO(2)\times SO_e(1,4),G_{2(2)}) & \text{ with }\; \lambda=-\frac{1}{3},\;\mu=\frac{1}{3},\;-1\\
       \end{array}
   \right \}\hspace*{0.2cm}\fl\cap\fs\text{ splits into two irreducible components}
  \end{equation*}
\vspace*{0.5cm}

The next proposition gives the image under the embedding (\ref{embed1}) of $\Omega_{G}$ in terms of Casimir operators of subgroups of $L$ for one representative
of each equivalence class of  properly transitive triples $(G,H,L)$ with $G$ simple.
\begin{pro}\label{Casimir}${}$
%
\begin{itemize}
\item[(1)] $\displaystyle{\imath(\Omega_{G})=2\Omega_{L}-\Omega_{L\cap K}}$, when
\beu (G,H,L)=\left\{\begin{array}{l} 
(SO_{e}(2,2n),SO_{e}(1,2n),U(1,n))\\
(SO_{e}(2,2n),U(1,n),SO_{e}(1,2n))\\
(SU(2,2n),S(U(1)\times U(1,2n)),Sp(1,n))\\
(SU(2,2n),Sp(1,n),S(U(1)\times U(1,2n)))\\
\end{array}\right.
\end{equation*}
\item[]
\item[(2)] $\displaystyle{\imath(\Omega_{G})=4\Omega_{L}-2\Omega_{L\cap K}}$ for $(G,H,L)=(SO_{e}(4,4n),SO_{e}(3,4n),Sp(1,n))$.
\item[(3)] $\displaystyle{\imath(\Omega_{G})=8\Omega_{L}-6\Omega_{L\cap K}}$ for $(G,H,L)=(SO_{e}(8,8),SO_{e}(7,8),Spin(1,8))$.
\item[(4)] $\displaystyle{\imath(\Omega_{G})=4\Omega_{L}-\frac{8}{3}\Omega_{L\cap K}}$ for $(G,H,L)=(SO(8,\mathbb{C}),SO(7,\mathbb{C}),Spin(1,7))$.
\item[(5)] $\displaystyle{\imath(\Omega_{G})=3\Omega_{L}-2\Omega_{L\cap\sigma(L)}-\frac{3}{2}\Omega_{L\cap K}}$ for $(G,H,L)=(SO_{e}(3,4),SO(2)\times SO_{e}(1,4),G_{2(2)})$.
\vspace*{0.1cm}
\item[(6)] $\displaystyle{\imath(\Omega_{G})=4\Omega_{L}-3\Omega_{L\cap\sigma(L)}-\frac{8}{3}\Omega_{L\cap K}}$, when
\beu (G,H,L)=\left\{\begin{array}{l} 
(SO(8,\mathbb{C}),SO_{e}(1,7),Spin(7,\mathbb{C}))\\
\\
(SO_{e}(4,4),SO(3)\times SO_{e}(1,4),Spin(3,4)).
\end{array}\right.\\
\end{equation*}
\end{itemize}
\end{pro}
\begin{proof}
The proposition is an immediate consequence of the above description of $\text{Spec}(F)$ and $\text{Spec}(\fl\cap\fs)$ combined with Lemma \ref{sense}. Let us give some
more detail. For the triples of Type~I in the above list, i.e. in Cases (1)--(4), we obtain
$$ \imath(\Omega_G)=\frac{2}{1-\mu}\: \Omega_{\fl(\mu)}+\frac{2}{1-\lambda}\: \Omega_{\fl(\lambda)}\ .$$
Moreover, in ${\bf D}(L/L\cap H)$ we have
$$   \Omega_{\fl(\mu)}=\Omega_L-\Omega_{L\cap K}\quad\mbox{and}\quad \Omega_{\fl(\lambda)}=\Omega_{L\cap K}\ .$$
For the remaining three triples let $\mu$ be the element of $\text{Spec}(\fl\cap\fs)\setminus \{-1\}$. Then 
$$ \imath(\Omega_G)=\frac{2}{1-\mu}\: \Omega_{\fl(\mu)}+\Omega_{\fl(-1)}+\frac{2}{1-\lambda}\: \Omega_{\fl(\lambda)}$$
and
$$ \Omega_{\fl(-1)}=\Omega_{L\cap\sigma(L)},\  \Omega_{\fl(\mu)}=\Omega_L-\Omega_{L\cap\sigma(L)} -\Omega_{L\cap K},\ \Omega_{\fl(\lambda)}=\Omega_{L\cap K}.$$
\end{proof}

It is also possible to obtain an explicit formula for $\imath(\Omega_G)$ for remaining irreducible triples, i.e. for the triples of the form $(G_1\times G_1,\Delta G_1,G_1\times H_1)$ appearing as Case~12 in Table 2, see the end of the section. We  mention that for the Cases 1,5,6  of Table 1, a similar formula for $\imath(\Omega_G)$ is given in \cite{STV18}. A thorough discussion of the embedding $\imath:{\bf D}(G/H)\longrightarrow {\bf D}(L/L\cap H)$ for most
spherical triples in a different setting can be found in \cite{KK19}, cf. Remark~\ref{complexsph} below. 

We also mention that analogous embedding formulas were obtained in the case of (cubic) twisted Dirac operators on $G/H$ and $L/L\cap H$ in \cite{MP21}.

The following result is a direct consequence of the above computations. It has decisive consequences for the spectral analysis of ${\bf D}(G/H)$ on compact standard quotients of  $G/H$
corresponding to a triple of Type~I, see Thm.~\ref{selim} and the remark following it.
\begin{cor}\label{commutation}
Suppose $(G,H,L)$ is a triple of Type~I. Then, as  $L$-invariant differential operators acting on $C^\infty(L/L\cap H)$, 
$\imath(D)$ and  $\Omega_{L\cap K}$ commute whenever $D$ belongs to ${\bf D}(G/H)$. In particular, $\imath(D)$ commutes with the $L$-invariant elliptic operator $\Omega_L-2\Omega_{L\cap K}$ acting on $C^\infty(L/L\cap H)$.
\end{cor}
\begin{proof}
Let us check first that the property to prove is invariant under equivalences of triples. We consider the decomposition into ideals $\fl=\fl_{min}\oplus \fc$ for some   $\fc\subset 
\fz_{\fk}(\fl_{min})$. The Casimir $\Omega_\fc$ belongs to the center of ${\Cal U}(\fl)$, hence 
$$ [\imath(D),\Omega_{L\cap K}]=[\imath(D),\Omega_{L_{min}\cap K}+\Omega_\fc]=[\imath(D),\Omega_{L_{min}\cap K}]\ .$$
This shows the invariance.
Now using Prop.~\ref{productck}, the assertion is easily reduced to irreducible triples. For the group case 11 of Table 2 the assertion is obvious, see (\ref{center}). By the classification result Thm.~\ref{listck} it remains to consider the Cases (1)-(4) in Prop.~\ref{Casimir}.
The latter proposition tells us that $\imath(\Omega_G)=a\Omega_L+b\Omega_{L\cap K}$ for non-zero rational numbers $a$ and $b$. In particular, one has:
\begin{eqnarray*}
[\imath(D),\Omega_{L\cap K}]&=&[\imath(D),\frac{a}{b}\Omega_L+\Omega_{L\cap K}]\;\;\;\text{since $\Omega_L$ belongs to the center of ${\Cal U}(\fl)$}\\
&=&\frac{1}{b}\imath([D,\Omega_G])\\
&=&0.
\end{eqnarray*}
\end{proof}
\begin{rem}\label{complexsph}
We want to put Cor.~\ref{commutation} into a different perspective that we have learned from \cite{KK19}. A reductive homogeneous space $L/Q$ is called {\em complex spherical} if its
complexification $L_\C/Q_\C$ is spherical, i.e. the Borel subgroup of $L_\C$ has an open orbit on $L_\C/Q_\C$. Complex spherical spaces $L/Q$ have a commutative algebra ${\bf D}(L/Q)$
of invariant differential operators (see e.g. \cite{Vi01}). It is now a nice accident coming out of Oni\v s\v cik's classification (see Table 1 in Thm.~\ref{listck} and compare with Table 1.1  in \cite{KK19}), that $L/L\cap H$ is complex spherical for all triples $(G, H, L)$ of Type~I with $G$ simple
and $L=L_{max}$. Thus in these cases the corollary  just follows from the commutativity of ${\bf D}(L/L\cap H)$. Note that a complex spherical space is necessarily (real) spherical (see e.g.
\cite{KKS}, Lemma 2.1), but not vice versa.
\end{rem}

\begin{rem}\label{noncomm}
Cor.~\ref{commutation} is not true for triples of Type~II. 
For instance, for the triples of Type~II with $G$ simple, it is the additional non-compact Casimir
$\Omega_{L\cap\sigma(L)}$ appearing in the formulas for $\imath(\Omega_G)$ which causes $[\imath(\Omega_G),\Omega_{L\cap K}]\ne 0$.
\end{rem}
%

Given a triple $(G^\prime,H^\prime,L^\prime)$, where $G^\prime/H^\prime$ is symmetric with respect to an involution $\sigma^\prime$, we consider the new triple 
$(G,H,L):=(G^\prime\times G^\prime,\Delta G', L^\prime\times H^\prime)$, where the new involution $\sigma$ is given by switching the factors. It is possible to express $\text{Spec}(G,H,L)$ in terms of $\text{Spec}(G^\prime,H^\prime,L^\prime)$. Indeed, from Lemma \ref{basislem} (4), one has: ${\mathfrak l}^\prime=\bigoplus_{\beta^\prime}{\mathfrak l}^\prime(\beta^\prime),\;\;\text{ with }\;{\mathfrak l}^\prime(1)={\mathfrak l}^\prime\cap{\mathfrak h}^\prime,\;\;\;{\mathfrak l}^\prime(-1)={\mathfrak l}^\prime\cap{\mathfrak q}^\prime\;\text{ and }\;-1\leq\beta^\prime\leq 1$. We consider the orthogonal projection 
$$p_{\fh^\prime}:{\mathfrak l}^\prime\rightarrow{\mathfrak h}^\prime,\;X\mapsto\frac{1}{2}(X+\sigma^\prime(X)).$$ 
Then one can check that
\begin{equation*}
p_{\fh^\prime}({\mathfrak l}^\prime)=p_{\fh^\prime}(\bigoplus_{\beta^\prime}{\mathfrak l}^\prime(\beta^\prime))=\bigoplus_{\beta^\prime\neq -1}p_{\fh^\prime}({\mathfrak l}^\prime(\beta^\prime))
\end{equation*}
since $\text{ker}(p_{\fh^\prime})=\fl'(-1)$ and the restriction of $p_{\fh^\prime}$ to $\bigoplus_{\beta^\prime\neq -1}{\mathfrak l}^\prime(\beta^\prime)$ is injective, and that
\begin{equation*}
{\mathfrak h}^\prime=({\mathfrak h}^\prime\cap{\mathfrak l}^{\prime\perp})\oplus\bigoplus_{\beta^\prime\neq -1}p_{\fh^\prime}({\mathfrak l}^\prime(\beta^\prime))
\end{equation*}
where ${\mathfrak l}^{\prime\perp}$ denotes the orthogonal complement of ${\mathfrak l}^\prime$ in ${\mathfrak g}^\prime$. In particular, we have an orthogonal decomposition
$\fl=\bigoplus_{\beta'}\fl_{\beta'}$ with
$$ \fl_{\beta'}:= \fl'(\beta')\oplus p_{\fh'}( \fl'(\beta'))\subset \fl'\oplus \fh'\mbox{ for }\beta'\ne -1,\quad \fl_{-1}:=\fl'(-1)\oplus ({\mathfrak h}^\prime\cap{\mathfrak l}^{\prime\perp}).$$
%
%
Since we also have $\sigma(\fl_{\beta_1'})\perp \fl_{\beta_2'}$ for $\beta_1'\ne\beta_2'$ both decompositions of $\fl$ are compatible: $\fl=\bigoplus_{\beta,\beta'}\fl(\beta)\cap \fl_{\beta'}$.
Observe that $\fl_{-1}\subset \fl(0)$. Now let $\beta'\ne-1$,
$X\in{\mathfrak l}^\prime(\beta^\prime)$, and $\lambda\in\R$. Then the element $(X,\lambda p_{\fh^\prime}(X))\in \fl_{\beta'}$ belongs to $\fl(\beta)$ if and only if for
all $Y\in {\mathfrak l}^\prime(\beta^\prime)$
\begin{eqnarray*}
\beta\langle X,Y\rangle&=&\langle \sigma(X,\lambda p_{\fh^\prime}(X)),(Y,0)\rangle=\frac{\lambda(1+\beta^\prime)}{2}\langle X,Y\rangle\\
\beta\lambda\frac{(1+\beta^\prime)}{2}\langle X,Y\rangle&=&\langle \sigma(X,\lambda p_{\fh^\prime}(X)),(0,p_{\fh^\prime}(Y))\rangle=\frac{(1+\beta^\prime)}{2}\langle X,Y\rangle,
\end{eqnarray*}
i.e.  $\beta=\frac{1}{\lambda}=\frac{\lambda(1+\beta^\prime)}{2}$, which implies $\beta=\pm\sqrt{\frac{1+\beta^\prime}{2}}$. 
We set
$$ {\Cal B}:=\Big\{\beta=\pm\sqrt{\frac{1+\beta^\prime}{2}}\mid\beta^\prime\in\text{Spec}(G^\prime,H^\prime, L^\prime)\Big\} .$$
Recall (Lemma \ref{basislem}, (2) and (5)) that the eigenvalue $1$ is never an element of the spectrum of a properly transitive triple. We deduce that: 
\begin{equation*}
{\Cal B}\subset\text{Spec}(G,H,L)\subset{\Cal B}\cup\{-1,0\}.
\end{equation*}
%
The corresponding eigenspaces $\fl(\beta)$ are:
\begin{eqnarray*}
\Big\{(X,\pm\sqrt{\frac{2}{1+\beta^\prime}}\; p_{\fh^\prime}(X))\mid X\in{\mathfrak l}^\prime(\beta^\prime)\Big\} &\text{ if }& \beta^\prime\neq -1,1\\
\Big\{(X,Y)\mid X\in{\mathfrak l}^\prime(-1),\; Y\in{\mathfrak h}^\prime\cap{\mathfrak l}^{\prime\perp}\Big\} &\text{ if }& \beta=0\\
\Big\{(X,-X)\mid X\in{\mathfrak l}^\prime\cap \fh'\Big\} &\text{ if }& \beta= -1\: .
\end{eqnarray*}
Now an explicit formula for the embedding of the Casimir operator for the triples appearing as Case 12 in Table~2 can be easily obtained from  the above  computation of $\mathrm{Spec}(G',H',L')$
for $G'$ simple and Lemma~\ref{sense}. However, it is no longer possible to express the result as a linear combination of Casimir operators of subgroups of $L$.

Eventually, we remark that the theory developed in this section also works for non-symmetric $G/H$, i.e. for triples $(G,H,L)$ only satisfying the assumptions of Prop.~\ref{cocpt}.
We have not essentially used that the reflection $\sigma=2p_\fh-\id$, which is defined in general, is an automorphism of $\fg$. Only Assertion (6) of Lemma~\ref{basislem} is affected.
The crucial property $[X,p_\fh (X)]=0$ for $X\in\fl(\beta)$ holds in general.


\section{Generalities I: On distribution vectors, distributional matrix coefficients and decompositions}\label{dist}

The reader may skip the following two sections in first reading and use them only as a source of references for definitions and more or less well-known
results of more general nature. The real story will continue in Section \ref{invariants}.

In this section we collect some facts about distribution vectors, distributional matrix coefficients, and discrete decompositions of spaces of smooth and distribution vectors of representations. Only Proposition~\ref{kunst} below might probably not have
been observed before. The reason for inserting this section is the lack of appropriate references, for instance the existing treatment \cite{He14} of distributional matrix coefficients does not really fit our needs. 

In this section we allow $G$ to be a general Lie group which, for simplicity, we assume to be unimodular with fixed Haar measure $dg$.
Let $\rho$ be a representation of $G$ on a reflexive Banach space $W_\rho$. We form its space of smooth vectors $W_{\rho,\infty}$ and equip it
with its natural Fr\'echet topology. The latter is generated by the seminorms $w\mapsto\| \rho(X)w\|$, $X\in {\Cal U}(\fg)$. Let $W_{\tilde \rho}=(W_\rho)'$ be the dual Banach space
representation. Reflexivity ensures that also $\tilde\rho$ is a continuous representation, i.e. that the action $G\times W_{\tilde \rho}\ni (g,\tilde w)\mapsto \tilde\rho(g)\tilde w=\rho(g^{-1})^t\tilde w\in W_{\tilde \rho}$
is continuous. 

By definition, the space of distribution vectors $W_{\rho,-\infty}$ of $W_\rho$ is the topological
dual of $W_{\tilde \rho,\infty}$: $W_{\rho,-\infty}:=(W_{\tilde \rho,\infty})'$ equipped with the strong topology, i.e. the topology of uniform convergence on bounded sets in $W_{\tilde \rho,\infty}$. It also carries actions of $G$ and ${\Cal U}(\fg)$.
We get continuous inclusions with dense images 
$$W_{\rho,\infty}\subset W_{\rho}\subset W_{\rho,-\infty}\ .$$

Note that  $W_{\rho,\infty}$ also carries an action of the convolution algebra $C_c^{-\infty}(G)$ of compactly supported distributions on $G$. It extends the natural actions of
$G$, $ {\Cal U}(\fg)$, and $C_c(G)$ on $W_{\rho,\infty}$ and is uniquely characterized by

$$ \rho(\vp)\circ\rho(T)=\rho(\vp*T)\,\quad \vp\in C_c^{\infty}(G),\ T\in    C_c^{-\infty}(G)\ .$$
Note that $\vp*T \in C_c^{\infty}(G)$.

In order to construct this action one may use that
\be
 C_c^{-\infty}(G) =l({\Cal U}(\fg)) C_c^{k}(G)=r({\Cal U}(\fg))  C_c^{k}(G) \label{hut}
\end{equation}
for $k=0$ (or any $k\in$ $\mathbb N_0$). Here $l$ and $r$ denote the left and right regular representation of  ${\Cal U}(\fg)$ on  the space of distributions $C^{-\infty}(G)$, respectively. This action can be extended to $ W_{\rho,-\infty}$ by duality: $\rho(T):=\tilde\rho(\check{T})^t$, where $T\mapsto\check T$ is the involution induced by inversion
in $G$.

Note that the notion of reflexivity extends to locally convex Hausdorff topological vector spaces $W$. Such a space is called reflexive if the natural inclusion $W\hookrightarrow W''$
(duals are equipped with the strong topology) is a topological isomorphism. It is called semi-reflexive if this inclusion is just bijective. Semi-reflexive Fr\'echet spaces are automatically
reflexive. For a continuous $G$-representation $\rho$ on a semi-reflexive space $W$, the natural action of $G$ on $W'$ is again a continuous representation.
In particular, Lemma \ref{stock} (iv) below shows that $W_{\rho,-\infty}$ is a continuous $G$-representation.

We have the following basic properties of the space of distribution vectors $W_{\rho,-\infty}$:

\begin{lem}\label{stock}
Let $W_\rho$ be a continuous $G$-representation on a reflexive Banach space. Then
\begin{enumerate}
\item[(i)] For every $w\in W_{\rho,-\infty}$ there exists $k\in\N_0$ such that  $\rho(f) w\in W_\rho$ for all $f\in C_c^k(G)$.
\item[(ii)] $W_{\rho,\infty}=\{w\in W_{\rho,-\infty}\mid \rho(X)w\in W_\rho \mbox{ for all } X\in {\Cal U}(\fg)\}.$
\item[(iii)] For all  $w\in W_{\rho,-\infty}$ and $\vp\in C_c^\infty(G)$ we have $\rho(\vp)w\in W_{\rho,\infty}$. Moreover, 
the map
$$  C^\infty_c(G)\ni\vp \mapsto \rho(\vp) w\in  W_{\rho,\infty} $$
is continuous.
\item[(iv)] $W_{\rho,\infty}$ is reflexive.
\end{enumerate}
\end{lem}
\begin{proof} Let $w\in W_{\rho,-\infty}$. By definition, $w$ is a continuous linear functional on $W_{\tilde \rho,\infty}$. Hence there exist $X_1,\dots,X_r\in {\Cal U}(\fg)$
such that for all $\tilde w\in W_{\tilde \rho,\infty}$
$$  |\langle w,\tilde w\rangle |\le \sum_{j=1}^{r}\|\tilde\rho(X_j)\tilde w\|\ .$$
Let $k$ be the maximal filtration order of the elements $X_1,\dots,X_r\in  {\Cal U}(\fg)$. In the following we will write $l_{X_j}$ and $r_{X_j}$ (instead of $l(X_j)$ and $r(X_j)$) for left and right differentiation by $X_j$. Then for $f\in C^k_c(G)$ and $\tilde w\in W_{\tilde \rho,\infty}$
\begin{eqnarray}
|\langle \rho(f)w,\tilde w\rangle|&=&|\langle w,\tilde\rho(\check f)\tilde w\rangle| \nonumber\\
&\le &\sum_{j=1}^{r}\|\tilde\rho(X_j)\tilde\rho(\check f)\tilde w\|
= \sum_{j=1}^{r}\|\tilde\rho(l_{X_j}\check f)\tilde w\|\nonumber\\
&\le& C \left ( \sum_{j=1}^{r}\| r_{X_j}f\|_\infty\right )\|\tilde w\| \label{korbinian}
\end{eqnarray}
for some constant $C$ depending only on $\supp\, f$. The symbol $\|.\|_\infty$ denotes the supremum norm. 
We see that $\rho(f)w$ extends to a continuous linear functional on $W_{\tilde \rho}$
which is by reflexivity given by some element of $W_\rho$. This shows that $\rho(f)w\in W_\rho$ and proves (i).

Concerning (ii) we consider the delta distribution $\delta_e\in C_c^{-\infty}(G)$ supported at the unit $e\in G$. Using (\ref{hut}) we find for any given $k\in\N_0$ finitely
many elements  $X_1,\dots,X_r\in {\Cal U}(\fg)$ and $f_1,\dots f_r\in C_c^k(G)$ such that
\be \label{rums}
\delta_e = \sum_{j=1}^{r} r_{X_j}f_j
\end{equation}
Let now $w\in W_{\rho,-\infty}$ be such that $\rho(X)w\in W_\rho$ for all $X\in {\Cal U}(\fg)$. Then
$$ w=\rho(\delta_e)w= \sum_{j=1}^{r}\rho(f_j)\rho(X_j^\top) w\ ,$$
where $X\mapsto X^\top$ is the antihomomorphism of  ${\Cal U}(\fg)$ induced by $X\mapsto -X$ on $\fg$. Since $\rho(X_j^\top) w\in W_\rho$ the
elements  $\rho(f_j)\rho(X_j^\top) w$, and therefore also $w$, are $k$-times continuously differentiable vectors in $W_\rho$. This holds for all $k\in\N_0$, hence $w\in W_{\rho,\infty}$.
Thus the right hand side of the equation in (ii) is contained in the left. The opposite inclusion is trivial.

We now show (iii). Let  $w\in W_{\rho,-\infty}$. By (i) we have $\rho(\vp)w\in W_{\rho}$ for all $\vp\in C_c^\infty(G)$.
Then $\rho(X)\rho(\vp)w=\rho(l_X\vp)w\in W_\rho$
for all $X\in  {\Cal U}(\fg)$. Assertion (ii) implies that $\rho(\vp)w\in W_{\rho,\infty}$.
Moreover, the inequality (\ref{korbinian}) implies that 
$$ \|\rho(X)\rho(\vp)w\|\le C \sum_{j=1}^{r}\| r_{X_j}l_X\vp\|_\infty\ .$$
This gives the desired continuity.

Since $W_{\rho,\infty}$ is a Fr\'echet space, for (iv) it suffices to show that the inclusion $W_{\rho,\infty}\hookrightarrow (W_{\tilde\rho,-\infty})'$ is surjective.
Let $F$ be a continuous linear functional on $W_{\tilde\rho,-\infty}$. Its restriction to $W_{\tilde\rho}$ is  continuous. Reflexivity of $W_\rho$ implies that there
exists $w\in W_\rho$ such that $F(\tilde w)=\langle w,\tilde w\rangle$ for all $\tilde w\in W_{\tilde\rho}$. The same argument applies to the continuous linear 
functional $\tilde\rho(X^\top)^t F$, $X\in  {\Cal U}(\fg)$:  there
exists $w_X\in W_\rho$ such that 
$$\tilde\rho(X^\top)^tF(\tilde w)=\langle w_X,\tilde w\rangle \mbox{ for all }\tilde w\in W_{\tilde\rho}\ .$$
Let now $\tilde w\in W_{\tilde \rho,\infty}$. Then
\begin{eqnarray*}
\langle \rho(X)w,\tilde w\rangle&=&\langle w,\tilde\rho(X^\top)\tilde w\rangle= F(\tilde\rho(X^\top)\tilde w)\\
&=& \tilde\rho(X^\top)^tF(\tilde w)=\langle w_X,\tilde w\rangle\ .
\end{eqnarray*}
Since $W_{\tilde \rho,\infty}\subset W_{\tilde \rho}$ is dense we conclude that $\rho(X)w=w_X\in W_\rho$. Now (ii) implies that $w\in W_{\rho,\infty}$.
Eventually we conclude that $F(\tilde w)=\langle w,\tilde w\rangle$ for all $\tilde w\in W_{\tilde\rho,-\infty}$.
%
\end{proof}

For $\tilde w\in W_{\tilde\rho,-\infty}$ and $w\in W_{\rho,\infty}$ we have the matrix coefficient $c_{\tilde w,w}\in C^\infty(G)$,
$$ c_{\tilde w,w}(g):=\langle \tilde w, \rho(g)w\rangle\ .$$
Because of Lemma \ref{stock} (iii) we can extend the definition to two distribution vectors $\tilde w\in W_{\tilde\rho,-\infty}$ and $w\in W_{\rho,-\infty}$
giving a continuous linear functional $c_{\tilde w,w}:C_c^\infty(G)\rightarrow \C$ by 
$$ c_{\tilde w,w}(\vp):=\langle \tilde w, \rho(\vp)w\rangle\ ,$$
 i.e. $c_{\tilde w,w}\in C^{-\infty}(G)$.
\begin{lem}\label{stein}
\begin{enumerate}
\item[(i)]  For $g_1,g_2\in G$ we have $c_{\tilde\rho(g_1)\tilde w,\rho(g_2)w}=l(g_1)r(g_2) c_{\tilde w,w}$.
\item[(ii)] We equip $ C^{-\infty}(G)$ with the strong dual topology. For fixed $w\in W_{\rho,-\infty}$ the map
$$  W_{\tilde\rho,-\infty}\ni\tilde w\mapsto c_{\tilde w,w}\in C^{-\infty}(G)$$
is continuous.
\end{enumerate}
\end{lem}
\begin{proof} 
Assertion (i) can be easily verified by a direct computation. For (ii) we consider a bounded subset $B\subset C_c^\infty(G)$. By Lemma~\ref{stock} (iii)
its image $B_1=\{\rho(\vp)w\mid \vp\in B\}$ is bounded in $W_{\rho,\infty}$. The desired continuity now follows from
$$ \sup_{\vp\in B} |c_{\tilde w,w}(\vp)|= \sup_{w_1\in B_1}|\langle \tilde w,w_1\rangle|\ .$$
\end{proof}

We denote the space of continuous intertwining operators between two $G$-representations by $\Hom_G(.,.)$.  Let $H\subset G$ be a closed subgroup. Usual Frobenius reciprocity gives an isomorphism 
\be\label{plum} \Hom_G(W_{\tilde\rho,\infty},C^\infty(G/H))\cong (W_{\rho,-\infty})^H
\end{equation}
given by evaluation at $1\in G$: $\Phi\mapsto w_\Phi$, where $\langle w_\Phi,\tilde w\rangle=\Phi(\tilde w)(1)$.
Its inverse is given by the matrix coefficient map $w\mapsto \Phi_w$, $\Phi_w(\tilde w)=c_{\tilde w,w}$.
We want to obtain an alternative version of (\ref{plum}) for distributions, namely
$$  \Hom_G(W_{\tilde\rho,-\infty},C^{-\infty}(G/H))\cong (W_{\rho,-\infty})^H\ .$$
Here again, we consider the strong dual topology on $C^{-\infty}(G/H)$. (The result would not change if we consider the weak$^*$ topology instead.)
We also note that in case of non-unimodular $H$ (which will be not of importance for our purposes) one has to define $C^{-\infty}(G/H)$ as the dual of the space
of compactly supported smooth densities on $G/H$ instead of $C_c^\infty(G/H)$ so that $C^{-\infty}(G/H)$ is really the space of generalized functions
on $G/H$.

\begin{pro}\label{kunst}
The matrix coefficient map
$  (W_{\rho,-\infty})^H\ni w\mapsto c_{.,w} $
induces a bijection
$$  (W_{\rho,-\infty})^H\stackrel{\cong}{\longrightarrow} \Hom_G(W_{\tilde\rho,-\infty},C^{-\infty}(G/H)) .$$
\end{pro}
\begin{proof}
We first discuss the case $H=\{1\}$. We observe that by Lemma \ref{stein} (ii), (iii)
$$ \Phi(\tilde w):=c_{\tilde w,w}\in C^{-\infty}(G), \quad \tilde w\in W_{\tilde\rho, -\infty},$$
is well-defined, continuous, and $G$-equivariant (with respect to the left regular representation on  $C^{-\infty}(G)$), i.e. $\Phi\in \Hom_G(W_{\tilde\rho,-\infty},C^{-\infty}(G))$.
Therefore the matrix coefficient map induces a map
$$ W_{\rho,-\infty}\ni w\mapsto \Phi\in \Hom_G(W_{\tilde\rho,-\infty},C^{-\infty}(G))\ .$$
We want to define an inverse map. The idea is to take $\Phi\in \Hom_G(W_{\tilde\rho,-\infty},C^{-\infty}(G))$, restrict it to $W_{\tilde\rho,\infty}$,  and to show
that 
$$W_{\tilde\rho,\infty}\ni \tilde w\mapsto \Phi(\tilde w)(e)\in\C$$
is well-defined and continuous.

Formally we proceed as follows. Let $B\subset W_{\tilde\rho}$ be the closed unit ball. Then $B$ is a bounded subset of $W_{\tilde\rho,-\infty}$, hence also $\Phi(B)\subset C^{-\infty}(G)$
is bounded. Let $\chi\in C^{-\infty}_c(G)$ with $\chi\equiv 1$ in a neighborhood $U$ of $1\in G$. Then $\chi\Phi(B)\subset  C^{-\infty}_c(G)$ is bounded, too.
Bounded subsets of $C^{-\infty}_c(G)$ are equicontinuous (see e.g. \cite{Sch71}, p. 141), i.e. there exists a continuous seminorm $\nu$ on $C^\infty(G)$ with
$$|\langle \chi\Phi(\tilde w),f\rangle|\le \nu(f)\ \mbox{for all }\tilde w\in B,\ f\in C^\infty(G).$$
We may assume that 
$$\nu(f)=\sum_{i=1}^{l} \ \sup_{g\in\Omega}|l_{X_i} f(g)|$$
for some $\Omega\subset G$ compact, $X_i\in {\Cal U}(\fg)$. Let now $\vp\in C_c^\infty(G)$ with $\supp\, \vp\subset U$. We obtain for $\tilde w\in W_{\tilde \rho}$
\begin{equation}
\label{art}
|\langle \Phi(\tilde w),\vp\rangle| = |\langle \chi\Phi(\tilde w),\vp\rangle| \le \|\tilde w\| \sum_{i=1}^{l} \|l_{X_i} \vp\|_\infty\ .
\end{equation}
In particular, $\Phi(\tilde w)$ extends continuously to $C^k_c (U)\ni\vp$ for $k$ large enough such that (\ref{art}) still holds. We find a representation similar to (\ref{rums}) of the delta distribution
even if we require that the $f_i$ have support in a prescribed neighborhood of $1\in G$ (making $r$ possibly larger):
$$  \delta_e=\sum_{i=1}^m l_{X_i}f_i,\  f_i\in C^k_c(U)\subset C^k_c(G). $$
We define a linear functional $F: W_{\tilde \rho, \infty}\rightarrow \C$ by
\be \label{gnu} F(\tilde w)= \sum_{i=1}^m \langle \Phi(\tilde\rho(X_i^\top)\tilde w), f_i\rangle\ .
\end{equation}
Inequality (\ref{art}) shows that $F$ is continuous and therefore given by some $w\in W_{\rho,-\infty}$.

It is now a straightforward computation to show that $\Phi\mapsto w$ inverts the matrix coefficient map $w\mapsto \Phi=c_{.,w}$ considered at the beginning of the proof.
For the convenience of the reader we also present this computation:

For $\Phi=c_{.,w}$ we obtain
\begin{eqnarray*}
 F(\tilde w)&=& \sum_{j=1}^m \langle c_{\tilde\rho(X_i^\top)\tilde w,w}, f_i\rangle\\
&=& \sum_{i=1}^m \langle l_{X_i^\top}c_{\tilde w,w}, f_i\rangle\\
&=& \langle c_{\tilde w,w}, \sum_{i=1}^m l_{X_i}f_i\rangle = c_{\tilde w,w}(e)\\
&=& \langle \tilde w, w\rangle\ .
\end{eqnarray*}
Vice versa, let $\Phi\in \Hom_G(W_{\tilde\rho,-\infty},C^{-\infty}(G))$. Let $F$ be given as in (\ref{gnu}), and let $w\in W_{\rho,-\infty}$ be such that 
$\langle \tilde w, w\rangle = F(\tilde w)$
for all $\tilde w \in W_{\tilde\rho,\infty}$. Then for $\vp\in C_c^\infty(G)$
\begin{eqnarray*}
c_{\tilde w,w}(\vp)&=&\langle \tilde w,\rho(\vp)w\rangle =\langle \tilde\rho(\check\vp)\tilde w,w\rangle\\
&=& F( \tilde\rho(\check\vp)\tilde w)\\
&=&\sum_{i=1}^m \langle \Phi(\tilde\rho(l_{X_i^\top}\check\vp)\tilde w, f_i\rangle\\
&=& \langle \Phi(\tilde w),\vp*\delta_e\rangle\\
&=& \langle \Phi(\tilde w),\vp\rangle\ .
\end{eqnarray*}
We now have established that the matrix coefficient map gives an isomorphism 
\be \label{rims} W_{\rho,-\infty}\stackrel{\cong}{\longrightarrow} \Hom_G(W_{\tilde\rho,-\infty},C^{-\infty}(G))\ .
\end{equation}
By Lemma \ref{stein} (ii) it is $G$-equivariant if we equip the right hand side by the $G$-action induced with the right regular representation on $C^{-\infty}(G)$.
There is a canonical isomorphism $C^{-\infty}(G)^H\cong C^{-\infty}(G/H)$. The proposition now follows by taking $H$-invariants in (\ref{rims}).
\end{proof}

Let $I$ be an index set, and let $H$ and $H_\alpha$, $\alpha\in I$, be a collection of Hibert spaces. Then the Hilbert space direct sum $\widehat{\bigoplus}_{\alpha\in I} H_\alpha$ is
a well-defined new Hilbert space. An isomorphism
$$    H\cong \widehat{\bigoplus_{\alpha\in I}} H_\alpha $$
of Hilbert spaces is seen as a decomposition of $H$ into the Hilbert spaces $H_\alpha$. Of particular importance is the situation that all the Hilbert spaces carry unitary representations
of $G$ and the isomorphism is a unitary equivalence of representations.

Now let $V$, $V_\alpha$, $\alpha\in I$, be locally convex Hausdorff topologically vector spaces. We want to make sense of decompositions written as
$$    V\cong \overline{\bigoplus_{\alpha\in I}} V_\alpha\ . $$
A natural choice of requirements for this would be: There exist continuous linear maps $i_\alpha: V_\alpha \rightarrow V$ and $p_\alpha: V\rightarrow V_\alpha$ such that
\begin{enumerate}
\item $p_\alpha i_\alpha=\id_{V_\alpha}$ for all $\alpha\in I$ and $p_\alpha i_\beta=0$ for all $\alpha\ne\beta\in I$.
\item $\sum_{\alpha} i_\alpha(V_\alpha)\subset V$ is dense.
\item Let $v\in V$ such that $p_\alpha (v)=0$ for all $\alpha\in I$. Then $v=0$.
\end{enumerate}

In the examples we have in mind stronger conditions are satisfied which involve the convergence of (partial) Fourier series. We restrict our considerations to countable decompositions.
\begin{Def}\label{kobold}
Let $I$ be an at most countable index set, and let $V$, $V_\alpha$, $\alpha\in I$, be locally convex Hausdorff topologically vector spaces. We write 
$$    V\cong \overline{\bigoplus_{\alpha\in I}} V_\alpha $$
if there exist continuous linear maps $i_\alpha: V_\alpha \rightarrow V$, $p_\alpha: V\rightarrow V_\alpha$, $\alpha\in I$, satisfying (a) above, and for all subsets $J\subset I$
continuous linear maps $P_J: V\rightarrow V$ such that:
\begin{enumerate}
\item[(i)] $P_{\{\alpha\}}=i_\alpha p_\alpha$, $P_I =\id$.
\item[(ii)] For every sequence $(J_n)$ of mutually disjoint subsets of $I$ the series $\sum_{n=1}^{\infty}P_{J_n}(v) $
converges in $V$ to $P_J(v)$, where $J=\bigcup_{n=1}^\infty J_n$.
\end{enumerate}
\end{Def}
As a special case of the series appearing in (ii) we get the convergence of the full Fourier series $$v=\sum_{\alpha\in I} i_\alpha p_\alpha (v)\mbox{  for all }v\in V$$
(independent of the ordering of $I$).
This implies Properties (b) and (c) above. Note that given the collection of embeddings $i_\alpha$ the other structural maps $p_\alpha$, $P_J$ are uniquely determined by
the conditions in the definition (if they exist). This also holds with the roles of $i_\alpha$ and $p_\alpha$ interchanged.
We have included the existence of the partial projections $P_J$ into the definition in order to be able to make easily conclusions of the following kind: 
Assume we have a countable index set $K$, subspaces $V_\chi\subset V$, $\chi\in K$, and mutually disjoint subsets $J_\chi\subset I$ with $\bigcup_{\chi\in K} J_\chi =I$. Then
\be\label{matisse}
V\cong \overline{\bigoplus_{\alpha\in I}} V_\alpha, \ V_\chi\cong \overline{\bigoplus_{\alpha\in J_\chi}} V_\alpha\ \Longrightarrow\ V\cong \overline{\bigoplus_{\chi\in K}} V_\chi\ .
\end{equation}
Here it is understood that the structural maps for the decomposition of $V_\chi$ are inherited from the decomposition of $V$.

We also note that  
$$  H\cong \widehat{\bigoplus_{\alpha\in I}} H_\alpha\quad \mbox{implies}\quad  H\cong \overline{\bigoplus_{\alpha\in I}} H_\alpha\ . $$
Observe however, that in contrast to  $\widehat{\bigoplus}_{\alpha\in I} H_\alpha$ the symbol $\overline{\bigoplus}_{\alpha\in I} V_\alpha$ has no intrinsic meaning
as a topological vector space (even not as a set).

The decomposition also makes sense for an uncountable index set $I$, if  $V_\alpha\ne\{0\}$ for at most countably many
$\alpha\in I$. In case that we work with representations of groups or algebras we require the maps $i_\alpha$, $p_\alpha$ to
be equivariant.

Recall that, by definition, a Fr\'echet space is a Montel space if each of its closed and bounded subsets is compact. Montel spaces have the property that weak* convergent sequences in its dual
are strongly convergent. Spaces of smooth sections of  vector bundles over manifolds equipped with their standard Fr\'echet topology are Montel.
The prototype of the decompositions we are interested in is given by the following

\begin{lem}\label{deko}
Let $W_\rho$ be a unitary representation of $G$ on a separable Hilbert space. We assume that there is a decomposition
$$  W_\rho\cong \widehat{\bigoplus_{\alpha \in I}} V_\alpha$$
with some unitary $G$-representations  $V_\alpha$. Then
$$  W_{\rho,\infty}\cong \overline{\bigoplus_{\alpha \in I}} V_{\alpha,\infty}\ ,$$
and, if $W_{\rho,\infty}$ is Montel,
$$  W_{\rho,-\infty}\cong \overline{\bigoplus_{\alpha \in I}} V_{\alpha,-\infty}\ .$$
In particular, if   $W_\rho$
decomposes discretely with finite multiplicities into irreducibles, i.e.
$$  W_\rho\cong \widehat{\bigoplus_{\pi\in\hat G}}M(\pi)\otimes V_\pi\ ,$$
where $M(\pi)$ is a finite dimensional Hilbert space (called multiplicity space), and $W_{\rho,\infty}$ is Montel, then
$$  W_{\rho,\pm\infty}\cong \overline{\bigoplus_{\pi\in\hat G}}M(\pi)\otimes V_{\pi,\pm\infty}\ .$$
\end{lem}
\begin{proof} We choose a unitary equivalence $U: W_\rho\cong \widehat{\bigoplus}_{\alpha \in I} V_\alpha$. It is given
by a collection of $G$-equivariant maps $p_\alpha: W_\rho\rightarrow V_\alpha$, while its inverse is given by a collection $i_\alpha:  V_\alpha\rightarrow W_\rho$.
For $J\subset I$ let $P_J: W_\rho\rightarrow W_\rho$ be the orthogonal projection to the subspace $\{w\in W_\rho\mid p_\alpha(w)=0\mbox{ for all }\alpha\not\in J\}$. It is also
$G$-equivariant.
These maps satisfy the conditions of Def.~\ref{kobold}. In particular, by Hilbert space theory, the partial Fourier series
\be\label{muck} \sum_n P_{J_n}(w) 
\end{equation}
converge in $W_\rho$ to  $P_J(w)$. Restricting to smooth vectors we obtain $G$- and $ {\Cal U}(\fg)$-equivariant maps $p_\alpha: W_{\rho,\infty}\rightarrow  V_{\alpha,\infty}$ and
$i_\alpha: V_{\alpha,\infty} \rightarrow W_{\rho,\infty}$, $P_J: W_{\rho,\infty}\rightarrow W_{\rho,\infty}$. It remains to check that for $w\in W_{\rho,\infty}$  the series (\ref{muck}) converges in $W_{\rho,\infty}$-topology. But this
follows easily from the convergence in  $W_\rho$ and the  ${\Cal U}(\fg)$-equivariance of $ P_{J_n}$, $P_J$.

By duality we can extend the maps $i_\alpha$, $p_\alpha$, $P_J$ to the corresponding spaces of distribution vectors and obtain immediately that for $w\in W_{\rho,-\infty}$ the Fourier
series (\ref{muck}) converges to $P_J(w)$ in the weak* topology. If $W_{\rho,\infty}$ is Montel, then so is $W_{\tilde\rho,\infty}$ (the conjugate representation by unitarity). By the above mentioned
property of Montel spaces we deduce the convergence of the Fourier series in (the strong dual topology of)  $W_{\rho,-\infty}$.
\end{proof}

\section{Generalities II: Spectral decompositions}\label{G2}

In this section we collect some facts on the spectral theory of unbounded operators that will be useful later.

Let $\Cal A$ be a finitely generated commutative $*$-algebra over $\C$ with unit acting on a dense subspace  $H^\infty$ of a separable Hilbert space $H$. We view it as an algebra
of unbounded operators on $H$ with common invariant domain $H^\infty$.

We will assume in addition that $H^\infty$ is equipped with a locally convex topology such that the linear maps $D: H^\infty \rightarrow H^\infty$, $D\in\Cal A$, and
$H^\infty\hookrightarrow H$ are continuous. Let $H^{-\infty}$ be the conjugate topological dual of $H^\infty$. We get dense embeddings
$$ H^\infty\subset H \subset H^{-\infty}$$
as well as an extension of the $\Cal A$-action to $H^{-\infty}$: For $D\in\Cal A$, $\tilde v\in H^{-\infty}$, $v\in H^\infty$ we put
$$\langle D\tilde v, v\rangle\defn \langle \tilde v, D^* v\rangle\ .$$
We are interested in spectral decompositions of this action. The main examples we have in mind are of the following form:
\begin{enumerate}
\item $H=L^2(Y)$, where $Y$ is a smooth manifold equipped with a smooth density, and $\Cal A$ consists of differential operators acting on
$H^\infty=C^\infty_c(Y)$, $H^{-\infty}=C^{-\infty}(Y)$. The star operation corresponds to taking formal adjoints. In our examples $Y$ will be compact while
$\Cal A$ will not contain any elliptic operator which makes the situation sufficiently complicated.
\item $H=V_\pi^Q$, where $V_\pi$ is a unitary representation of a reductive Lie group $L$, and $Q\subset L$ is a compact subgroup. We consider 
$H^{\pm\infty}=( V_{\pi,\pm\infty})^Q$. The algebra $\Cal A$ will be an appropriate commutative subquotient of ${\Cal U}(\fl)^Q$.
\end{enumerate}

For any character $\chi\in\Hom(\Cal A, \C)$ we consider the joint eigenspaces for $\Cal A$ 
$$ H_\chi^{\pm\infty}\defn\{v\in H^{\pm\infty}\mid Dv=\chi(D)v\mbox{ for all }D\in\Cal A\}\ .$$
We set $H_\chi\defn H_\chi^{-\infty}\cap H$.

Similarly, for $\lambda\in \C$ we have the eigenspaces $H_{D,\lambda}^\infty\subset H_{D,\lambda}\subset H_{D,\lambda}^{-\infty}$ of a single operator $D\in\Cal A$.
The spaces $H_\chi$ and $H_{D,\lambda}$ are closed subspaces of $H$, in particular Hilbert spaces.

Let $\Xi\subset \Hom(\Cal A, \C)$ be the subset of $*$-homomorphisms,  i.e. the set of characters $\chi$  satisfying $\chi(D^*)=\overline{\chi(D)}$ for all $D\in\Cal{A}$.
If $H_\chi^{\infty}\ne \{0\}$, then $\chi\in\Xi$. Moreover, for $\chi\ne\chi'$ the spaces $H_\chi^{\infty}$ and $H_{\chi'}$ are mutually orthogonal.
Note that these assertions do not hold in general if we remove the superscript $\infty$.

We want to decompose the action of $\Cal A$  into eigenspaces. The simplest case would be a discrete decomposition: There are mutually orthogonal closed subspaces
$H_\chi^0\subset H_\chi$, $\chi\in\Xi$, such that
\begin{equation}\label{disc}
H\cong \widehat{\bigoplus_{\chi\in\Xi}} {H}_\chi^0\ .
\end{equation}
In general, we will need direct integrals and distributional eigenspaces. For the basic facts about direct integrals we refer to Dixmier \cite{Di} or Wallach \cite{Wa92}, Section 14.8.

Let $D_1,D_2,\dots, D_r$ be a set of (formally) self-adjoint generators of $\Cal A$. They give rise to an embedding $\Xi\hookrightarrow \R^r$,
$$ \Xi\ni\chi\mapsto (\chi(D_1),\dots,\chi(D_r))\in\R^r\ ,$$
as a closed subset,
and therefore induce a topology on $\Xi$. This topology is independent of the choice of generators.

\begin{Def}\label{kobold1} A spectral decomposition of $(\Cal A, H^\infty, H)$ (or shortly of $\Cal A$ if $H^\infty$ and $H$ are understood) is given by a unitary
bijection
$$\Cal F: H\stackrel{\sim}{\longrightarrow} \int_\Xi^\oplus W_\chi\:d\mu(\chi)$$
such that for all $D\in\Cal A$, $v\in H^\infty$
\be\label{IQ} \Cal F(Dv)(\chi)=\chi(D) \Cal F(v)(\chi)  \quad(\mu\mbox{-almost everywhere}).
\end{equation}
Here $\{W_\chi\mid\chi\in\Xi\}$ is a measurable family of Hilbert spaces, and $\mu$ is a $\sigma$-finite Borel measure on $\Xi$.
We also require that $\mu(\{\chi\in\Xi\mid W_\chi=\{0\}\})=0$.
Two such spectral decompositions $\Cal F_1$, $\Cal F_2$ are equivalent if there exists a decomposable unitary bijection
$$ B:  \int_\Xi^\oplus W^1_\chi\:d\mu_1(\chi)\longrightarrow  \int_\Xi^\oplus W^2_\chi\:d\mu_2(\chi) $$
such that $B\circ \Cal F_1=\Cal F_2$. Decomposable means that $B$ is of the form
$$  Bf(\chi)=B(\chi)f(\chi) \quad(\mu_1\mbox{-almost everywhere}) $$
for some measurable field of bounded operators $B(\chi): W_\chi^1\rightarrow W_\chi^2$.
We say that a spectral decomposition $\Cal F$ is unique if all other spectral decompositions of $\Cal A$ are equivalent to $\Cal F$.
\end{Def}
\begin{rem} The measure class of $\mu$ (with respect to absolute continuity) as well as the multiplicity function (modulo changes on sets of measure zero)
$$ \Xi \in \chi\mapsto \dim W_\chi\in \N_0\cup \infty $$
are invariant under equivalences of spectral decompostions.
\end{rem}

\begin{rem}\label{bernstein} 
Often the embedding $H^\infty\subset H$ is Hilbert-Schmidt, i.e. it factors through an intermediate Hilbert space $H^1$ such that the embedding $H^1\hookrightarrow H$
is Hilbert-Schmidt. For instance in our main examples (a) (for $Y$ compact) and (b) (for $V_\pi$ irreducible) one can take appropriate Sobolev spaces for $H^1$. By a theorem of
Gelfand-Kostyuchenko (see Bernstein \cite{Ber88}) any spectral decomposition is then pointwise defined on $H^\infty$, i.e. for $\mu$-almost all $\chi\in\Xi$ there exists a continuous
linear operator $\Cal F_\chi: H^\infty\rightarrow W_\chi$ with dense image such that for $v\in H^\infty$ and $D\in\Cal A$ we have $\Cal F_\chi(Dv)=\chi(D)\Cal F_\chi(v)$  and
$\Cal F(v)(\chi)=\Cal F_\chi(v)$. The adjoint $\Cal F_\chi^*: W_\chi\hookrightarrow H^{-\infty}$ identifies $W_\chi$ with a subspace $H_\chi^0\subset H_\chi^{-\infty}$
and equips $H_\chi^0$ with a (non-obvious) Hilbert space structure. This point of view supports the philosophy that a spectral decomposition is a (continuous) decomposition into (distributional)
joint eigenvectors. Note that a discrete spectral decomposition, i.e. $\mu$ is equivalent to the counting measure of an at most countable subset of $\Xi$ and therefore
the direct integral is really a Hilbert space direct sum, is always pointwise defined on $H$ (not only on $H^\infty$) which implies that $H_\chi^0\subset H_\chi$ with the usual Hilbert space structure. Thus we are in the situation of (\ref{disc}).
\end{rem} 

\begin{rem}\label{plums} There is a one-to-one correspondence between equivalence classes of spectral decompositions of $(\Cal A, H^\infty, H)$ and $P(H)$-valued
Borel measures $E$ on $\Xi$ such that for any $v\in H^\infty$ and any $D\in\Cal A$ we  have
$$ Dv = \int_\Xi \chi(D)\: dE_\chi(v) \ .$$
Here $P(H)\subset \End(H)$ denotes the set of orthogonal projections. Indeed, each spectral decomposition defines such a measure by setting $E(A)\defn \Cal F^{-1}m_A\Cal F$ for a Borel set $A\subset \Xi$. Here $m_A$ denotes the multiplication by the characteristic function $1_A$ of $A$. The correspondence in the other direction is most easily described
when $E$ admits a cyclic vector $v\in H$, i.e. the set $\{E(A)v\mid A\subset \Xi\mbox{ Borel}\}$ is linearly dense in $H$. In this case the map
$$ E(A)v\mapsto 1_A\in L^2(\Xi, d\mu)\ ,$$
where $\mu$ is the finite measure given by $\mu(A)\defn \langle E(A)v,v\rangle$, extends uniquely to a unitary bijection
$$ H\stackrel{\sim}{\longrightarrow} L^2(\Xi, d\mu)$$
(note that $L^2(\Xi, d\mu)$ is a direct integral with $W_\chi=\C$ for all $\chi\in\Xi$). In general, one constructs the corresponding direct integral by decomposing first $H$ into a Hilbert sum of at
most countably many $E$-invariant subspaces admitting a cyclic vector.
\end{rem} 

In the generality discussed so far, it may happen that a given $(\Cal A, H^\infty, H)$ has no spectral decomposition or several inequivalent ones. We want to give
some criteria for existence and uniqueness. 

Let us start with the case $r=1$, i.e. that $\Cal A$ is generated by a single formally self-adjoint operator $D\ne 0,\id$.
Note that then $\Hom(\Cal A, \C)=\C$, $\Xi=\R$. Here we are dealing just with the spectral theory of the symmetric operator $D$ with initial domain $H^\infty\subset H$.
For basic notions and facts about unbounded operators on Hilbert spaces we refer e.g. to Reed-Simon I \cite{RSI}, Ch. VIII, and II \cite{RSII}, Ch. X.
Note that in spectral theory unbounded operators always come together with their domains. In order to associate a spectrum to $D$ we need a closed extension
$\tilde D$ of $D$ with domain $H^{\tilde D}\supset H^\infty$. Then
$$ \spec(\tilde D)\defn\{\lambda\in\C\mid D-\lambda : H^{\tilde D}\rightarrow H\mbox{ is not bijective}\}\ .$$
If, in addition $\tilde D$ is self-adjoint, i.e. $  \langle \tilde Dv,w\rangle= \langle  v,\tilde Dw\rangle$ for all $v,w\in  H^{\tilde D}$ and
$$ H^{\tilde D^*}\defn \{w\in H\mid H^{\tilde D}\ni v\mapsto \langle \tilde Dv,w\rangle\in \C\mbox{ is continuous}\}=H^{\tilde D}\ ,$$
then $\spec(\tilde D)\subset\R$, and $\tilde D$ has a spectral resolution
\be\label{sr}\Cal F: H\stackrel{\sim}{\longrightarrow} \int_\R^\oplus W_\lambda\:d\mu(\lambda)
\end{equation}
such that
$$\Cal F(\tilde Dv)(\lambda)=\lambda \Cal F(v)(\lambda)  \quad(\mu\mbox{-almost everywhere}) $$
not only for $v\in H^\infty$, but for $v\in H^{\tilde D}$. This is a consequence the spectral theorem for unbounded self-adjoint operators in  projection valued measure form combined with Remark \ref{plums}.
It also follows that $\spec(\tilde D)=\supp\, \mu$ and that
$$ H^{\tilde D}=H_{\Cal F}^D\defn \{v\in H\mid \int_\R \lambda^2 \|\Cal F(v)(\lambda)\|_\lambda^2\: d\mu(\lambda)<\infty\}\ .$$
In particular, (\ref{sr}) gives a spectral decomposition for $(\C[D], H^\infty, H)$. Vice versa, any spectral decomposition of  $(\C[D], H^\infty, H)$ defines
a self-adjoint extension $D_\Cal F$ with domain  $H_{\Cal F}^D$. We just have seen

\begin{pro}\label{stark} There is a one-to-one correspondence between equivalence classes of spectral decompositions of $\Cal A=\C[D]$ and self-adjoint extensions $\tilde D$ of $D$.
\end{pro}

The operator $D$ has two natural closed extensions: its closure $\bar D$ and its adjoint $D^*$. The domain $H^{\bar D}$ of $\bar D$ is (by definition) the completion of $H^\infty$ with respect
to the graph norm, while it is easy to see that $D^*$ has domain
\be\label{lab} H^{D^*}=\{v\in H\subset H^{-\infty}\mid Dv\in H\}\ .\end{equation}
Note that $H^{\bar D}\subset  H^{D^*}$.
One says that $D$ is essentially self-adjoint if $\bar D$ is self-adjoint. Equivalent characterizations of essential self-adjointness are
\begin{enumerate}
\item $\bar D=D^*$.
\item $ D^*$ is symmetric.
\item \label{quark} $D$ has a unique self-adjoint extension.
\item For some (equivalently for all) $\lambda\in \C\setminus\R$ we have $H_{D,\lambda}=H_{D,\bar\lambda}=\{0\}$. 
\end{enumerate}
Combining (\ref{quark}) with Proposition~\ref{stark} we obtain

\begin{cor}\label{uwe} $\Cal A=\C[D]$ has a unique spectral decomposition if and only if $D$ is essentially self-adjoint.
\end{cor}

We also remark that $D$ has a self-adjoint extension $\tilde D$ if and only if
\begin{enumerate}
\item[(e)] for some (equivalently for all) $\lambda\in \C\setminus\R$ we have $\dim H_{D,\lambda}=\dim H_{D,\bar\lambda}$. 
\end{enumerate}
Moreover, the self-adjoint extensions of $\tilde D$ of $D$ are in one-to-one correspondence to unitary bijections $I: H_{D,\lambda}\rightarrow H_{D,\bar\lambda}$, and they
satisfy $\bar D\subset\tilde D\subset D^*$  (see  Reed-Simon II \cite{RSII}, Ch. X).

Now we are going to discuss existence and uniqueness of spectral decompositions for general $\Cal A$. For that we fix a set of formally self-adjoint generators $D_1,\dots,D_r$
of $\Cal A$. Note that existence and uniqueness of a spectral decomposition for $\Cal A$ is equivalent to existence and uniqueness of a unitary representation $\pi$ of the
abelian group $(\R^r,+)$ on $H$ such that $H^\infty$ is contained in the smooth vectors of $\pi$ and for all $v\in H^\infty$, $j=1,\dots, r$,
$$ \frac{\partial}{\partial x_j} \pi(x) v|_{x=0}={i} D_j v\ ;$$
i.e. we want to integrate the representation of the Lie algebra ${i}\cdot\mbox{span}_\R(D_1,\dots,D_r)$ on $H^\infty\subset H$ to a unitary representation on $H$ of the corresponding
simply connected Lie group. Such questions (even for non-abelian Lie algebras) have been investigated for a long time (see e.g. Nelson \cite{Nel59}, Flato \cite{FSSS72}, Powers \cite{Pow74}).
Nevertheless, there is no such simple necessary and sufficient condition as in Corollary \ref{uwe}. However, we have a complete analogue to Proposition \ref{stark},
see Proposition~\ref{schwach} below.

We consider two unbounded self-adjoint operators $C_1$, $C_2$ on $H$ with corresponding projector-valued measures $E_1$, $E_2$ provided
by the spectral theorem. One says that $C_1$ and $C_2$ strongly commute if all the projectors $E_1(A)$, $E_2(B)$, $A,B\subset \R$ Borel, commute.

\begin{pro}\label{schwach} Equivalence classes of spectral decompositions of $\Cal A$ are in one-to-one correspondence to choices of mutually strongly commuting  self-adjoint extensions $\tilde D_1,\dots,\tilde D_r$ of $D_1,\dots,D_r$.
\end{pro}
\begin{proof} Let $\Cal F$ be a spectral decomposition, and let $D\in\Cal A$ be a formally self-adjoint element. As in the discussion preceding Prop.~\ref{stark}, we obtain
a self-adjoint extension $\tilde D=D_\Cal F$ with domain  
$$H_{\Cal F}^D\defn \{v\in H\mid \int_\Xi\chi(D)^2 \|\Cal F(v)(\chi)\|_\chi^2\: d\mu(\chi)<\infty\}\ .$$
In particular, we obtain self-adjoint extensions  $\tilde D_1,\dots,\tilde D_r$ of $D_1,\dots,D_r$. The corresponding projector-valued measures are given
by $E_j(A)=\Cal F^{-1} m_{A_j}\Cal F$, $A\subset \R$ Borel, where $A_j\defn\{\chi\in\Xi\mid\chi(D_j) \in A\}$. It follows that $\tilde D_1,\dots,\tilde D_r$  mutually strongly commute.

Vice versa, given self-adjoint extensions $\tilde D_1,\dots,\tilde D_r$ of $D_1,\dots,D_r$ with commuting spectral resolutions $E_1,\dots,E_r$, there is
a corresponding $P(H)$-valued measure on $\Xi$, the ``product measure'' of $E_1,\dots, E_r$. More precisely, given Borel sets $A_1,\dots, A_r\subset \R$
we define
$$ A_1\cdot A_2\cdot\dots\cdot A_r\defn\{\chi\in\Xi\mid \chi(D_j)\in A_j, j=1,\dots,r\}\subset \Xi$$
and set
$$ E( A_1\cdot A_2\cdot \dots \cdot A_r)\defn E_1(A_1)E_2(A_2)\dots E_r(A_r)\ .$$
Since the sets of the form $A_1\cdot\dots\cdot A_r$ generate the Borel $\sigma$-algebra of $\Xi$ and are stable under intersections, the last formula determines uniquely
a $P(H)$-valued Borel measure on $\Xi$. $E$ provides us with a spectral decomposition $\Cal F$, see Remark~\ref{plums}.
\end{proof} 
As consequences, we obtain immediately the following criteria.
\begin{cor}\label{karl} Assume that $\Cal A$ has a spectral decomposition and that $D_1,\dots,D_r$ are essentially self-adjoint. Then the spectral decomposition is unique.
\end{cor}

\begin{cor}\label{klaus} Assume that $D_1,\dots,D_r$ are essentially self-adjoint and that their closures mutually strongly commute. Then $\Cal A$ has a unique spectral decomposition.
\end{cor}

\begin{cor}\label{heinz}
 Assume that $D_1^2+D_2^2+\dots+D_r^2$ is essentially self-adjoint. Then $\Cal A$ has a unique spectral decomposition.
\end{cor}
\begin{proof} It is a result of Nelson (Theorem 5 in \cite{Nel59}) that the assumption of Cor.~\ref{heinz} implies the one of Cor.~\ref{klaus}.
\end{proof}
In contrast to Cor.~\ref{uwe}, the criteria in Cor.~\ref{klaus} and Cor.~\ref{heinz} are not necessary. Moreover, all these criteria are in terms of essential self-adjointness.
However, in practice it is often difficult to decide a priori whether a given operator is essentially self-adjoint. Even the Criterion (d) above might be difficult
to check. Sometimes it is easier to construct one spectral decomposition of $\Cal A$ directly and discuss essential self-adjointness (and therefore uniqueness) ``a  posteriori" using this
spectral decomposition. The simplest instance for this is the following

\begin{pro}\label{schwoch} Assume that $H$ has a complete orthonormal system consisting of joint eigenvectors for $\Cal A$ belonging to $H^\infty$.
In other words, we assume that
$$ \bigoplus_{\chi\in\Xi} H_\chi^\infty $$
is dense in $H$. Then every formally self-adjoint operator $D\in \Cal A$ is essentially self-adjoint, we have  $\overline{H_\chi^\infty}= H_\chi$  for all $\chi\in\Xi$,
and 
$$  H=\widehat{\bigoplus_{\chi\in\Xi}} H_\chi$$
is the unique spectral decomposition of $\Cal{A}$.
\end{pro}
\begin{proof} Let $D\in\Cal A$ be formally self-adjoint, and let $0\ne v\in H_{D,\lambda}$ for some $\lambda\in \C$. By our assumptions there is a $\chi\in\Xi$ and some
$e\in  H_\chi^\infty$ satisfying $\langle v,e\rangle\ne 0$.
We get $\lambda \langle v,e\rangle=\langle Dv,e\rangle=\langle v,De\rangle=\langle v,\chi(D)e\rangle$. Since $\chi\in\Xi$ we obtain $\chi(D)\in\R$ and thus
$\lambda=\chi(D)\in\R$. Using the Criterion (d) above we see that $D$ is essentially self-adjoint.

The density of $\bigoplus_{\chi\in\Xi} H_\chi^\infty$ in $H$ and the mutual orthogonality of the spaces  $H_\chi^\infty$ easily imply that  $\overline{H_\chi^\infty}= H_\chi$ 
for all $\chi\in\Xi$ and that $H=\widehat{\bigoplus}_{\chi\in\Xi} H_\chi$. The latter is a spectral decomposition of $\Cal A$. Uniqueness follows from Cor.~\ref{karl}.
\end{proof}

However, in the context of Example (b), there is a class of operators, where essential self-adjointness can be established a priori. 

\begin{pro}\label{spezi} In the situation of Example (b) above, let $N$ be a closed and connected subgroup of $L$ containing $Q$ with Lie algebra by $\fn$.
Let $Z\in{\Cal Z}(\fn)\subset {\Cal U}(\fl)^Q$ be formally self-adjoint. Then the operator $D:=\pi(Z)$ with domain $(V_{\pi,\infty})^Q$ is essentially self-adjoint on $H=V_\pi^Q$.
\end{pro}

\begin{proof} We want to show that $\bar D=D^*$. Note that $D^*$ has domain $\{v\in V_\pi^Q\mid \pi(Z)v\in V_\pi^Q\subset (V_{\pi,-\infty})^Q\}$, see (\ref{lab}). For $v$ in that space
we have to find $v_i \in (V_{\pi,\infty})^Q$ such that $v_i\to v$ {and} $\pi(Z)v_i\to \pi(Z)v$ in $V^Q_{\pi}$. As a first step we consider a $Q$-invariant delta sequence $(\vp_i)$
in $C_c^\infty (N/Q)$. We set $w_i:=\pi(\vp_i)v$. This vector is smooth with respect to $N$: $w_i\in (V_{\pi,\infty_N})^Q$. Moreover, $w_i\to v$ and, since $Z\in Z(\fn)$,
$$ \pi(Z)w_i=\pi(Z)\pi(\vp_i)w=\pi(\vp_i)\pi(Z)w\to \pi(Z)w\ .$$
In other words, if $D'$ denotes the operator $\pi(Z)$ with domain $(V_{\pi,\infty_N})^Q$, then $\overline{D'}=D^*$.

An observation due to Poulsen (\cite{Pou72}, Cor.~1.2) is the following: \\
Let $W\subset V_{\pi,\infty_N}$ be an $N$-invariant subspace which is dense in $V_\pi$, and let $X\in {\Cal U}(\fn)$. Then the two operators given by $\pi(X)$ with domains
$W$ and $V_{\pi,\infty_N}$, respectively, have the same closure.

Using this observation for $X=Z$ and $W=V_{\pi,\infty}$ and going over to $Q$-invariants we obtain $\bar D= \overline{D'}$ and thus $\bar D=D^*$.
\end{proof} 

We conclude this section by a few remarks on the spectral theory of Laplacians $\Delta_Y$  on pseudo-Riemannian manifolds $Y$. Since $\Delta_Y$ has real coefficients
we see that Property (e) above holds. Thus $\Delta_Y$ always has a self-adjoint extension. If the metric tensor of $Y$ is not definite, then the spectrum
of any self-adjoint extension of  $\Delta_Y$ is unbounded from above and below. This is easily seen using Rayleigh quotients.

Concerning essential self-adjointness we can say the following. If $Y$ is Riemannian and complete (in particular, if it is Riemannian and compact), then $\Delta_Y$ is essentially self-adjoint. Compact pseudo-Riemannian manifolds of indefinite signature are not necessarily geodesically complete. It has been conjectured recently by Colin de Verdi\`ere and Le Bihan \cite{CL22}, that for compact pseudo-Riemannian manifolds, $\Delta_Y$ is essentially self-adjoint if and only if $Y$ is geodesically complete, cf. Subsection~\ref{anal} of the introduction. 
\section{The space of $L\cap H$-invariants}\label{invariants}

The first goal of the present section section is to give a characterization of {\em spherical} properly transitive triples $(G,H,L)$  in terms of the finite-dimensionality of the spaces of invariants
$V_\pi^{L\cap H}$ in irreducible unitary representations $V_\pi$ of $L$, see Corollary~\ref{inv3} below. This will be a consequence of Casselman's subrepresentation theorem
combined with the characterization (\ref{char}). We remark that this result is a special case of the multiplicity bounds for general spherical homogeneous spaces $L/Q$
due to \cite{KO13} and \cite{KS16}, which are more complicated to obtain. Since for compact $Q$ (in our situation $Q=L\cap H$) a direct proof is quite simple and short, we feel that it is
adequate to present it here. 
We also show that for many triples the spaces $V_\pi^{L\cap H}$ are in fact at most one-dimensional, see Proposition~\ref{mufree}. This allows us to prove the second main result (Prop.~\ref{staun}) of this chapter asserting that ${\bf D}(G/H)$ acts on $V_\pi^{L\cap H}$ by a {\em single} character - this even holds for all triples of Type~I.

Let $L$ be a connected reductive Lie group with compact center with fixed maximal compact subgroup $K_L$. Let $P_{L}=M_{L}A_{L}N_{L}$ be the Langlands decomposition of a minimal parabolic subgroup of $L$, 
$\sigma$ a (necessarily finite dimensional) irreducible representation of $M_{L}$ and $\nu$ a complex linear form on the complexified Lie algebra of $A_{L}$. 
Then we form the unitarily induced representation 
$$ H^{\sigma,\lambda}:=\Ind_{P_{L}}^{L}\sigma\otimes\e^{\nu}\otimes 1$$
of $L$, also called principal series representation. Let $Q\subset L$ be a compact subgroup. We may and will assume that $Q\subset K_L$.
For an admissible representation $V_\pi$ of $L$ on a Hilbert space such that $K_L$ acts unitarily we consider the spaces of invariants
$$  ( V_{\pi,K_L})^Q \subset ( V_{\pi,\infty})^Q \subset V_{\pi}^Q\subset (V_{\pi,-\infty})^Q\ .$$
Here, as usual, $V_{\pi,K_L}$ denotes the $(\fl,K_L)$-module of $K_L$-finite vectors in $V_\pi$. Because of the compactness of $Q$ all inclusions have dense image, and thus all these four spaces have the same dimension (viewed as an element of $\N_0\cup\{\infty\}$).
We are particularly interested in the cases that $V_\pi$ is unitary or that it is a principal series representation $ H^{\sigma,\lambda}$. Note that in the latter case
the space of invariants is independent of $\lambda$. Indeed,  $H^{\sigma,\lambda}\cong\Ind_{M_{L}}^{K_L}\sigma$ as $K_L$- and hence as $Q$-representation.

For later reference we first express these spaces in terms of the $K_L$-type decomposition of  $V_\pi$. We use the notation for decompositions introduced at the end of Section \ref{dist}, in particular in Lemma \ref{deko} (here the role of $G$ is played by $K_L$).
\begin{lem}\label{frob}
Let $V_\pi$ an admissible representation of $L$ as above. For $*=\infty,\emptyset,-\infty$ we have
\beu
(V_{\pi,*})^{Q}\cong\overline{\bigoplus_{\gamma\in\widehat K_L}}M_\pi(\gamma)\otimes V_\gamma^Q\ ,
\end{equation*}
in particular
\beu
(H^{\sigma,\lambda}_{*})^{Q}\cong\overline{\bigoplus_{\gamma\in\widehat K_L}}\Hom_{M_L}(V_\gamma,V_\sigma)\otimes V_\gamma^Q.
\end{equation*}
\end{lem}
\begin{proof}

The first formula just arises by taking $Q$-invariants in the three decompositions in Lemma~\ref{deko}. Here we use that $L$-smooth and $K_L$-smooth vectors coincide for admissible representations.  In the case $V_\pi=H^{\sigma,\lambda}$, we obtain using Frobenius reciprocity
$$M_\pi(\gamma)\cong \Hom_{K_L}(V_\gamma, \Ind_{M_{L}}^{K_L}\sigma)\cong \Hom_{M_L}(V_\gamma,V_\sigma)\ .$$
\end{proof}

We will need the following well-known consequence of the slice theorem and the principal orbit theorem for compact Lie group actions (see e.g. \cite{Ja} or Theorem 2.1 of \cite{Mos57}):
\begin{lem}\label{prince}
Let $C$ be a compact Lie group acting on a connected manifold $S$. There exists a non-empty open $C$-invariant subset $U$ of $S$, a closed subgroup 
$C^{\prime}$ of $C$, and a non negative integer $m$ such that, as $C$-manifolds, $U\simeq (C/C^{\prime})\times{\R^m}$, 
where the $C$-action on ${\R^m}$ is the trivial action.
\end{lem}
\noindent
Now the main observation is the following
\begin{pro}\label{translem}
In the situation described above the following assertions are equivalent:
\begin{itemize}
\item[(i)] The space $(H^{\sigma,\lambda})^{Q}$
is finite dimensional for all $\sigma$.
\item[(ii)] The space $(H^{1,\lambda})^{Q}$
is finite dimensional. Here $1$ denotes the trivial one dimendional representation of $M_L$.
\item[(iii)] The space $V_\pi^{Q}$
is finite dimensional for all irreducible admissible representations of $L$.
\item[(iv)] The space $V_\pi^{Q}$
is finite dimensional for all $\pi\in \widehat L$.
\item[(v)] $Q$ acts transitively on $L/P_{L}\cong K_L/M_L$.
\end{itemize}
\end{pro}
\begin{proof}
We prove the following chain of implications: (v) $\Rightarrow$ (i) $\Rightarrow$ (iii ) $\Rightarrow$ (iv)$\Rightarrow$ (ii) $ \Rightarrow$ (v).
 
Assume (v). Then we obtain using Frobenius reciprocity
\be\label{karl-hermann}
(H^{\sigma,\lambda})^{Q}\cong
\Big(\Ind_{Q\cap M_L}^{Q}(\sigma|_{Q\cap M_L})\Big)^Q\cong
V_\sigma^{Q\cap M_L}\ .
\end{equation}
The space on the right hand side is finite dimensional. We obtain (i). The implication (i) $\Rightarrow$ (iii) is a consequence of Casselman's subrepresentation theorem
(strictly speaking we first conclude the finite dimensionality of  $(V_{\pi,K_L})^Q$). Assertion (iii) implies (iv) (irreducible unitary representations are admissible). For purely imaginary $\lambda$ the representation $H^{1,\lambda}$ is irreducible and unitary.
This shows the implication  (iv) $\Rightarrow$ (ii). Now we assume (ii). We apply Lemma \ref{prince} to $S=K_L/M_L$ und $C=Q$. Then the $Q$-equivariant diffeomorphism 
$U\cong Q/C'\times \R^m$ induces a linear embedding
$$ C_c^\infty(\R^m)\hookrightarrow C^\infty(K_L/M_L)^Q\cong (H^{1,\lambda}_\infty)^Q\ .$$
Thus $C_c^\infty(\R^m)$ is finite dimensional which implies $m=0$. We conclude $U=S$ and the transitivity of $Q$.
\end{proof}

\begin{rem}
We could have added $Q$-admissibility as a sixth equivalent condition: All $Q$-isotypic components of  irreducible admissible representations $V_\pi$ of $L$ are finite dimensional.
\end{rem}
\begin{cor}\label{inv3} Let $(G,H,L)$ be a properly transitive triple. Then $(G,H,L)$ is spherical if and only if $V_\pi^{L\cap H}$
is finite dimensional for all $\pi\in \widehat L$.
\end{cor}
\begin{proof} Use the characterization (\ref{char}) for being spherical and apply Prop.~\ref{translem} to the compact group $Q=L\cap H$.
\end{proof}

\noindent
%
%

\begin{rem} 
Assume that $(G,H,L)$ is not spherical. By a refinement of the argument of the proof of Proposition \ref{translem}, one can show  that then
$$ \dim (H^{\sigma,\lambda})^{L\cap H}\in \{0,\infty\}\ .$$
Nevertheless, there might be infinite dimensional irreducible unitary representations $V_\pi$ of $L$ such that $$0<\dim V_\pi^{L\cap H}<\infty$$ 
(take for example a triple
of Case~12 in Thm.~\ref{listck} with $(G_1,H_1,L_1)$ spherical and $V_\pi=V_{\pi_1}\otimes \C$ for certain $\pi_1\in \widehat L_1$).
\end{rem}

Back to spherical triples: it is natural to ask in which cases $V_\pi^{L\cap H}$ is not only finite dimensional but is of dimension at most one.
This happens to be the case for many spherical triples.

\begin{pro}\label{mufree}
Assume that $(G,H,L)$ is a spherical triple satisfying one of the following conditions:
\begin{itemize}
\item[(i)] $G$ is simple and $(G,H,L)$ is not locally isomorphic to $(SO_e(4,4n), SO_e(3,4n), Sp(1,n))$, if $n\geq 1$.
\item[(ii)] $L=L_{max}$ and all group manifolds appearing in the decomposition of $X=G/H$ into irreducible symmetric spaces are locally isomorphic to $SO_e(1,n)$ or $SU(1,n)$, $n\geq 2$.
\end{itemize} 
Then for all irreducible admissible representations $V_\pi$ of $L$ we have $\dim V_\pi^{L\cap H}\le 1$.

Now let $(G,H,L)=(SO_e(4,4n), SO_e(3,4n), Sp(1,n))$, and let $V_\pi$ be an irreducible admissible representation of $L$. Since $L\cap H=Sp(n)$, the $Sp(1)$-factor
of $K_L=Sp(1)\times Sp(n)$ acts on $V_\pi^{L\cap H}$. This action is irreducible.
\end{pro}

\begin{proof} We first assume that $G$ is simple. By Casselman's embedding theorem it suffices to check the assertion for principal series representations $H^{\sigma,\lambda}$.
This can be done by explicit computations where $(G,H,L_{max})$ runs through the Triples 1 -- 7 in Table 1 and $L_{min}\subset L\subset L_{max}$ (strictly speaking through all
isomorphism classes of triples which are locally isomorphic to them, but this leads essentially to the same computations).  One can use
either Lemma \ref{frob} or Formula (\ref{karl-hermann}). For Triple 1 cf. Lemma \ref{bummi} and its proof. For the higher rank triples 2 and 4 related computations can be found in the proof of Lemma \ref{schlips}. We obtain the desired multiplicity one result except for Triple 5 with $L=L_{min}$. In the latter case $L\cap H=Sp(n)$ and $(H^{\sigma,\lambda})^{Sp(n)}$
is an irreducible $Sp(1)$-module. The argument is again as in the proof of Lemma~\ref{bummi}. A slightly more conceptual approach to the proof under Assumption (i) would be the following: One checks that for all cases under consideration the homogeneous
space $L/L\cap H$ is {\em complex} spherical which implies that ${\bf D}(L/L\cap H)$ is commutative, see Remark \ref{complexsph} and  Table 1.1 in \cite{KK19}.
For a compact $Q\subset L$ and an irreducible admissible $L$-representation $V_\pi$ such that  $V_\pi^Q$ is finite dimensional the representation of ${\Cal U}(\fl)^Q$ on $V_\pi^Q$ is irreducible (or zero).
This representation descends to ${\bf D}(L/Q)$. In our cases, for $Q=L\cap H$, the algebra ${\bf D}(L/Q)$ is commutative and hence has only one-dimensional irreducible representations. 

Now assume (ii). By Prop.~\ref{productck}, one can reduce the assertion to irreducible triples with $L=L_{max}$. If $G$ is simple, then (i) holds and we are done.
If $G/H\cong G_1$ is a group manifold and $L=L_{max}$, then the assertion is equivalent to: the restriction of irreducible $G_1$-representations to the maximal
compact subgroup is multiplicity free. The two mentioned groups have this property (and are the only simple non-compact Lie groups with this property - up to local isomorphisms).
\end{proof}

As discussed in the above proof, the algebra ${\bf D}(L/L\cap H)$ acts on $V_\pi^{L\cap H}$ (in case that the latter space is infinite dimensional as an algebra of  a  priori unbounded operators
with initial domain of definition $(V_{\pi,\infty})^{L\cap H}$). Via the embedding $\imath: {\bf D}(G/H)\rightarrow {\bf D}(  L/L\cap H)$ also ${\bf D}(G/H)$ does so.

\begin{pro}\label{staun}
Let $(G,H,L)$ be a triple of Type~I, and let $V_\pi$ be an irreducible admissible representation  of $L$ with $V_\pi^{L\cap H}\ne\{0\}$. Then the algebra ${\bf D}(G/H)$ acts by a single character on  $V_\pi^{L\cap H}$:
there exists $\chi_\pi\in\Hom({\bf D}(G/H),\C)$ such that the action of $D\in {\bf D}(G/H)$ on $V_\pi^{L\cap H}$ is given by $\chi_\pi(D)\cdot\id$.
\end{pro}

\begin{proof} For the cases with $\dim V_\pi^{L\cap H}=1$ covered by Prop.~\ref{mufree} the assertion is obvious. It is also easy to see that the assertion is true if $X=G/H\cong G_1$
is a group manifold. Let us check this for the case $L=L_{min}=G_1$ (the other cases are similar, see also the discussion of non-minimal $L$ below). Now $L\cap H=\{e\}$ and ${\bf D}(G/H)$
acts on an irreducible admissible representation $V_\pi=V_\pi^{L\cap H}$ of $L$ via ${\Cal Z}(\fl)$, hence by the infinitesimal character of $\pi$.

In order to cover general triples $(G,H,L)$ we take a closer look on the relation of $V_\pi^{L\cap H}$ with $V_{\pi_1}^{L_1\cap H}$, where $L_1:=L_{min}$ and
$\pi$, $\pi_1$ are irreducible admissible representations of $L$ and $L_1$, respectively. By construction of $L_{min}=L_1$ there is a (compact) connected subgroup $C\subset Z_{K_L}(L_1)$
such that $L=L_1\cdot C$. Consequently, irreducible admissible representations of $L$ are always of the form $V_{\pi_1}\otimes F_\tau$, where $\pi_1$ is as above and $\tau\in\hat C$ (and
the characters of $L_1\cap C$ defined by $\pi_1$ and $\tau$ coincide). Let $X=G/H\cong L/L\cap H\cong L_1/L_1\cap H$. By Frobenius reciprocity
$$ (V_{\pi_1,-\infty})^{L_1\cap H}\cong \Hom_{L_1}(V_{\widetilde{\pi_1},\infty},C^\infty (X))\ .$$
The natural action of ${\bf D}(G/H)$ comes from its action on $C^\infty(X)$ on the right hand side. The action of $C\subset L$ on $C^\infty(X)$ commutes with the ${\bf D}(G/H)$-action as well as
with the $L_1$-action, and thus defines a $C$-action on $(V_{\pi_1,-\infty})^{L_1\cap H}$ commuting with the ${\bf D}(G/H)$-action. 
We claim that as ${\bf D}(G/H)$-modules 
\be\label{fidel}
(V_{\pi,-\infty})^{L\cap H}=\left [ V_{\pi_1,-\infty}\otimes F_\tau\right ] ^{L\cap H}\cong \left[  (V_{\pi_1,-\infty})^{L_1\cap H}\otimes F_\tau\right ]^C\ .
\end{equation}
Here the action of ${\bf D}(G/H)$ on the right hand side is induced by its action on the first factor. Indeed, again by Frobenius reciprocity we have
\begin{eqnarray*}
\left[ V_{\pi_1,-\infty}\otimes F_\tau\right ] ^{L\cap H}&\cong& \Hom_{L}(V_{\widetilde{\pi_1},\infty}\otimes F_{\tilde\tau},C^\infty (X))\\
&\cong& \Hom_C\left(F_{\tilde\tau}, \Hom_{L_1}(V_{\widetilde{\pi_1},\infty},C^\infty (X))\right)\\
&\cong&  \left[  (V_{\pi_1,-\infty})^{L_1\cap H}\otimes F_\tau\right ]^C\ .
\end{eqnarray*}
The isomorphism (\ref{fidel}) tells us that the assertion of the corollary  is true for $(G,H,L)$ whenever it is true for $(G,H,L_{min})$. Using Prop.~\ref{productck} we can
further reduce the assertion to the case that either $G$ is simple or that $G/H$ is a group manifold (and still $L=L_{min})$. The group case is discussed above, while the case of simple $G$
(with one exception) follows from Prop.~\ref{mufree}. It remains to prove the assertion for the triple $(SO_e(4,4n), SO_e(3,4n), Sp(1,n))$. Indeed, since $X$ is simply connected in this case
the assertion will also follow for all other triples locally isomorphic to it. We have $L_{max}=Sp(1,n)\cdot Sp(1)$. Thus the corresponding group $C$ considered above is given
by the $Sp(1)$-factor. It is different from the $Sp(1)$-factor considered in the last assertion of Prop.~\ref{mufree}, which is contained in $Sp(1,n)$. Let us call it $C_1$. The group 
$\{(p,p)\in C_1\times C\mid p\in Sp(1)\}$ is contained in $L_{max}\cap H$, see Table 1, Triple 5. In fact, it coincides with $(C_1\times C)\cap H$.
Thus for every $c\in C$ there is a unique $p(c)\in C_1$ such that $p(c)c=cp(c)\in H$, and $p: C\rightarrow C_1$ is an isomorphism of groups. Now we compute the action of $c\in C$ on
$V_\pi^{Sp(n)}$
defined before (\ref{fidel}) in terms of $p(c)\in C_1$. For that let $v\in V_\pi^{Sp(n)}=V_{\pi,-\infty}^{Sp(n)}$. Let  $\Phi_v\in \Hom_{Sp(1,n)}(V_{\tilde\pi,\infty}, C^\infty(X))$  be the element
corresponding to $v$ by Frobenius reciprocity, and let $\tilde v \in   V_{\tilde\pi,\infty}$. We obtain
\begin{eqnarray*}
\langle c\cdot v,\tilde v\rangle&=&\Phi_v(\tilde v)(c^{-1})=\Phi_v(\tilde\pi(p(c))^{-1}\tilde v)(p(c)^{-1}c^{-1})\\
&=& \Phi_v(\tilde\pi(p(c))^{-1}\tilde v)(e)=\langle v,\tilde\pi(p(c))^{-1}\tilde v\rangle\\
&=& \langle\pi(p(c))v,\tilde v\rangle\ .
\end{eqnarray*} 
Thus the two $Sp(1)$-actions on $V_\pi^{Sp(n)}$ coincide. In particular, the irreducible $Sp(1)$-action in Prop.~\ref{mufree} commutes with ${\bf D}(G/H)$. We
conclude that ${\bf D}(G/H)$ acts by a character.
\end{proof}

\section{The spectrum of compact Clifford-Klein forms: the Type~I situation}\label{spectrum}

Let $(G,H,L)$ be a properly transitive triple. Now we fix a cocompact lattice $\Gamma\subset L$. Then $\Gamma$ acts properly and cocompactly
on the symmetric space $X=G/H$. 
We assume in addition that $\Gamma$ is without non-trivial elements conjugate to $L\cap H$. Then $\Gamma$ also acts freely and the quotient 
$Y=\Gamma\backslash X$ is a compact locally symmetric {\em manifold} (and not only an orbifold). At least if we choose $G$ (and hence $L$) to be linear we can achieve the latter condition always by going
over to a finite index subgroup of the original lattice (in fact, we even have a torsion free finite index subgroup by Selberg's lemma, see \cite{Rat06}).
The results of the present section hold, suitably interpreted in the setting of orbifolds, also for general cocompact lattices $\Gamma\subset L$.

The commutative algebra  ${\bf D}(G/H)$ of $G$-invariant differential operators on $X$ acts on $C^{\pm\infty}(Y)$. We also view it as an algebra of unbounded operators (with common domain $C^{\infty}(Y)$) acting on the Hilbert space $L^2(Y)$. We now turn to the spectral properties of these actions.
Of particular interest are the spectral properties of the pseudo-Riemannian Laplacian $\Delta_Y$ which is induced by the Casimir operator $\Omega_G$.

Following the discussion in Section \ref{G2},
for any character $\chi\in\Hom( {\bf D}(G/H),\C)$
we consider the corresponding spaces of joint eigenfunctions and eigendistributions
$$ {E}_\chi^{\pm\infty}(Y):=\{f\in C^{\pm\infty}(Y)\mid Df=\chi(D)\cdot f\mbox{ for all } D\in {\bf D}(G/H)\}$$
as well as the space ${E}_\chi^{(2)}(Y):={E}_\chi^{-\infty}(Y)\cap L^2(Y)$.
${E}_\chi^{(2)}(Y)
\subset L^2(Y)$
 is closed (and hence a Hilbert space).

The description of the algebra ${\bf D}(L/L\cap H)$ of $L$-invariant differential operators on $X$ in terms of the universal enveloping algebra ${\Cal U}(\fl)$
shows that  ${\bf D}(L/L\cap H)$ acts on $(V_{\pi,\pm\infty})^{L\cap H}$.
Via the embedding $\imath: {\bf D}(G/H)\rightarrow {\bf D}(L/L\cap H)$ (see (\ref{embed1})) the algebra ${\bf D}(G/H)$ also does so.
We denote the corresponding joint eigenspaces by $(V_{\pi,\pm\infty})_\chi^{L\cap H}$, and set 
$(V_{\pi})^{L\cap H}_\chi:=(V_{\pi,-\infty})^{L\cap H}_\chi\cap V_\pi^{L\cap H}$.

Note that ${\bf D}(L/L\cap H)$ also acts on $C^{\pm\infty}(Y)$. We equip $C^\infty(Y)$ with its standard Fr\' echet topology and $C^{-\infty}(Y)$ with the strong dual topology.
Recall the meaning of the topological decompositions indicated by the symbol $\overline{\bigoplus}$ from Def.~\ref{kobold}.

\begin{lem}\label{bassa}
For $\pi\in\hat L$, let $N_\Gamma(\pi)= \Hom_L(V_\pi, L^2(\Gamma\backslash L))$ be the (finite dimensional) multiplicity space of $V_\pi$ in $L^2(\Gamma\backslash L)$.
Then we obtain the following decompositions:
\begin{eqnarray}
L^2(Y)&\cong &\widehat{\bigoplus_{\pi\in\widehat L}} N_\Gamma(\pi)\otimes V_{\pi}^{L\cap H}\ ,\nonumber\\
C^{\pm\infty}(Y)&\cong& \overline{\bigoplus_{\pi\in\widehat L}} N_\Gamma(\pi)\otimes (V_{\pi,\pm\infty})^{L\cap H}\label{abc}\ ,
\end{eqnarray}
and for $\chi\in\Hom( {\bf D}(G/H),\C)$
\begin{eqnarray*}
{E}_\chi^{\pm\infty}(Y)&\cong& \overline{\bigoplus_{\pi\in\widehat L}} N_\Gamma(\pi)\otimes (V_{\pi,\pm\infty})^{L\cap H}_\chi\ ,\nonumber\\
{E}_\chi^{(2)}(Y)&\cong &\widehat{\bigoplus_{\pi\in\widehat L}} N_\Gamma(\pi)\otimes (V_{\pi})^{L\cap H}_\chi\ .\nonumber
\end{eqnarray*}
 The second decomposition (\ref{abc}) is a decomposition of ${\bf D}(G/H)$-modules, and even of ${\bf D}(L/L\cap H)$-modules.
\end{lem} 

\begin{proof} Since $\Gamma\subset L$ is cocompact, the lemma is an immediate consequence of the well-known discrete decomposition with finite multiplicities of the 
unitary $L$-representation 
\be\label{monet} 
L^2(\Gamma\backslash L)\cong\widehat{\bigoplus_{\pi\in\widehat L}} N_\Gamma(\pi)\otimes V_{\pi}\ .
\end{equation}
Going to smooth or distribution vectors (see Lemma \ref{deko}) we obtain
\be\label{vangogh}
C^{\pm\infty}(\Gamma\backslash L)\cong \overline{\bigoplus_{\pi\in\widehat L}} N_\Gamma(\pi)\otimes V_{\pi,\pm\infty}\ .
\end{equation}
%
Taking $L\cap H$-invariants we obtain the first two decompositions. The last two follow by restricting the ${\bf D}(G/H)$-equivariant decomposition (\ref{abc}) to eigenspaces.
\end{proof}
\noindent
The lemma reduces the spectral theory of the action ${\bf D}(G/H)$ on 
$$ C^{\infty}(Y)\subset L^2(Y)\subset C^{-\infty}(Y)$$
to the understanding of its action on 
$$(V_{\pi,\infty})^{L\cap H}\subset V_{\pi}^{L\cap H}\subset(V_{\pi,-\infty})^{L\cap H}$$
for $\pi\in\widehat L$. 
By Prop.~\ref{staun} this action is the simplest possible, namely given by a single character of ${\bf D}(G/H)$, for triples of Type~I, while for Type~II triples the action remains complicated in general.

\begin{cor}\label{hussa} Let $(G,H,L)$ be a triple of Type~I. Let $\pi\in \widehat L$ be such that $V_\pi^{L\cap H}\ne\{0\}$. \\ 
Let $\chi_\pi\in \Hom({\bf D}(G/H),\C)$ be as in Prop.~\ref{staun}.
For $\chi\in \Hom({\bf D}(G/H),\C)$ we set 
$$\widehat L_\chi:=\{\pi\in \widehat L\mid V_\pi^{L\cap H}\ne \{0\}, \chi_\pi=\chi\}\ .$$ 
Then we have
\begin{eqnarray*}
{E}_\chi^{\pm\infty}(Y)&\cong& \overline{\bigoplus_{\pi\in\widehat L_\chi}} N_\Gamma(\pi)\otimes (V_{\pi,\pm\infty})^{L\cap H}\ ,\nonumber\\
{E}_\chi^{(2)}(Y)&\cong &\widehat{\bigoplus_{\pi\in\widehat L_\chi}} N_\Gamma(\pi)\otimes V_{\pi}^{L\cap H}.\nonumber
\end{eqnarray*}
Note that for $\chi\notin \Xi$ we have $\widehat L_\chi=\emptyset$.
\end{cor}

Now we easily conclude the discrete spectral decomposition for Type~I triples.

\begin{thm}\label{selim}
Let $(G,H,L)$ be of Type~I, and let $\Gamma\subset L$ be a cocompact lattice as described at the beginning of the section. Then we have a unique spectral decompositon
\be\label{else}
L^2(Y)\cong \widehat{\bigoplus_{\chi\in\Xi}} {E}_\chi^{(2)}(Y)\ ,
\end{equation}
which in addition induces decompositions of smooth functions and distributions on $Y$:
\be\label{ilse}
C^{\pm\infty}(Y)\cong\overline{\bigoplus_{\chi\in\Xi}} {E}_\chi^{\pm\infty}(Y)\ .
\end{equation}
In particular, we have
\begin{enumerate}
\item For all $\chi\in\Hom({\bf D}(G/H),\C)$, the inclusions ${E}_\chi^{\infty}(Y)\subset {E}_\chi^{(2)}(Y)\subset {E}_\chi^{-\infty}(Y)$ have dense images.
\item $L^2(Y)$ has a complete orthonormal system consisting of smooth joint eigenfunctions.
\item There are only countably many $\chi\in\Hom({\bf D}(G/H),\C)$ with ${E}_\chi^{-\infty}(Y)\ne\{0\}$. If the latter condition holds, then $\chi\in\Xi$.
\item All formally self-adjoint operators $D\in {\bf D}(G/H)$, in particular the Laplacian $\Delta_Y$, are essentially self-adjoint.
\end{enumerate}
\end{thm}

\begin{proof} Assertions (a)--(c), (\ref{else}) and (\ref{ilse}) are  direct consequences  of Cor.~\ref{hussa} and Lemma \ref{bassa}. Note that for the proof of  (\ref{ilse})
we need a conclusion of the form (\ref{matisse}). For (d) and the uniqueness of the spectral decomposition we  employ Prop.~\ref{schwoch}. 
\end{proof}

We remark that Thm.~\ref{selim} can be proved without representation theory (in particular without the use of Prop.~\ref{staun}) and only based on the existence of an elliptic
differential operator on $Y$ commuting with all elements of ${\bf D}(G/H)$.  This existence is ensured by Cor. \ref{commutation}.

Observe that $Y$, as a quotient of the symmetric space $G/H$, is geodesically complete. Thus, Thm.~\ref{selim}~(d) confirms the conjecture of Colin de Verdi\`ere and Le Bihan on the essential self-adjointness of $\Delta_Y$ \cite{CL22} mentioned at the end of Section~\ref{G2} for compact Type~I quotients $Y$. 

Theorem \ref{selim} shows that every $D\in {\bf D}(G/H)$, in particular the Laplacian $\Delta_Y$,  has a discrete spectral decomposition. Note that in case $\rank(G/H)=1$ the spectral decomposition of $\Delta_Y$ is equivalent to the one of ${\bf D}(G/H)$. We also mention that discrete here does not mean that the eigenvalues/eigencharacters do not have accumulation points nor that the corresponding eigenspaces are finite dimensional (as it would be the case for elliptic
differential operators). Combining Thm.~\ref{selim} with Cor.~\ref{hussa}, properties of the multicities $N_\Gamma(\pi)$, $\pi\in \widehat L$, and the formulas for $i(\Omega_G)$ in Prop.~\ref{Casimir}, one can get a rather clear picture of $\spec(\Delta_Y)$ in case $(G,H,L)$ is of Type~I. In the next section we will do this in detail for $(G,H,L)=(SO_{e}(2,2n),SO_{e}(1,2n),U(1,n))$. In Section \ref{infinitemulti}, we demonstrate by an example that a discrete spectral decomposition as in Theorem \ref{selim} cannot be expected for triples of Type~II. All these computations rest on Lemma~\ref{bassa} and Prop.~\ref{Casimir} and make no use of the representation theory of $G$ which will be brought into play in the last two sections. 

\section{The spectrum of compact Lorentzian manifolds of constant curvature}\label{detail}

In this section, unless otherwise stated, we consider the irreducible spherical triple 
\beu
(G,H,L)\defn(SO_{e}(2,2n),SO_{e}(1,2n),U(1,n))\text{ with }n\geq 1. 
\end{equation*}
The symmetric space $G/H$ is Lorentzian of constant negative sectional curvature. Let $\Gamma\subset L=U(1,n)$ be a discrete 
cocompact subgroup without non-trivial elements conjugate to $L\cap H\simeq U(n)$. Then 
\beu
Y=\Gamma\backslash G/H \cong \Gamma\backslash L/L\cap H
\end{equation*}
is a compact Lorentzian manifold of dimension $2n+1$ and constant negative sectional curvature. In order to get more precise formulas for multiplicities of discrete series representations
we make the slightly stronger assumption that $\Gamma$ is torsion free. By Selberg's lemma
\cite{Rat06} this can always be achieved by replacing $\Gamma$ by a finite index subgroup. We normalize the metric on $Y$ such that the curvature 
equals $-1$. Let $\Delta_Y$ be the corresponding Laplace operator. It is given by the Casimir operator $\Omega_G$ with respect to the bilinear form 
$\langle X,Y\rangle = \frac{1}{2}\Tr XY$ on $\fg$. The induced bilinear form on $\fl=\fu(1,n)$ is $\langle X,Y\rangle = \mathrm{Re} (\Tr XY)$.
Since the rank of the symmetric space $G/H$ equals $1$, the algebra ${\bf D}(G/H)$ of invariant differential operators on $G/H$ is generated by the 
Laplacian. By Thm.~\ref{selim}, we know that $\Delta_Y$ is essentially self-adjoint. 
We denote its self-adjoint extension 
also by $\Delta_Y$. Let $\spec(\Delta_Y)\subset\mathbb{R}$ be its spectrum. In our situation, it will be just the closure of the set of eigenvalues. For $\mu\in\mathbb{R}$, we consider 
the eigenspaces 
\beu
E_\mu:=E_\mu^{(2)}(Y)=\{f\in L^2(Y)\mid \Delta_Yf=\mu f\}. 
\end{equation*}
The aim of this section is to prove the following theorem.
\begin{thm}\label{something}\mbox{}
\begin{enumerate}
\item $\displaystyle\spec(\Delta_Y)\subset (-\infty,0]\cup\bigcup_{\ell=n+1}^\infty \{\ell^2-n^2\}\ .$
\item $\spec(\Delta_Y)\cap [1-n^2,\infty)$ is a discrete subset of $\mathbb{R}$.
\item $\dim E_{\ell^2-n^2} =\infty$ for all integers $\ell \geq 
1$.
\item $\displaystyle \sum_{\mu\in (1-n^2,0)\setminus \bigcup_{\ell=2}^{n-1} \{\ell^2-n^2\}} \dim E_\mu <\infty\ .$
\\
\item $-\infty$ is an accumulation point of $\spec(\Delta_Y)$.
\end{enumerate}
\end{thm}

Let us first explain the relationship between the structure of $\spec(\Delta_Y)$ given by the theorem and the unitary dual of $L=U(1,n)$ appearing in 
Lemma \ref{bassa} (see Proposition \ref{crux} below).
\begin{itemize}
\item $\spec(\Delta_Y)\cap (-\infty,-n^2]$ comes from unitary principal series and, at $-n^2$, also from limits of discrete series of $L$. 
\item $\spec(\Delta_Y)\cap (-n^2,0]$ consists of contributions from complementary series, ends of complementary series and non-integrable discrete series.
The discrete series contributions are responsible for the infinite dimension of the eigenspaces $E_{\ell^2-n^2}$, $\ell=1,\dots, n$.
\item $\spec(\Delta_Y)\cap (0,\infty)=\displaystyle\bigcup_{\ell=n+1}^\infty \{\ell^2-n^2\}$ is the contribution of integrable discrete series.
\end{itemize}
\vspace*{-1.5cm}
\begin{figure}[H]
\includegraphics[scale=0.56]{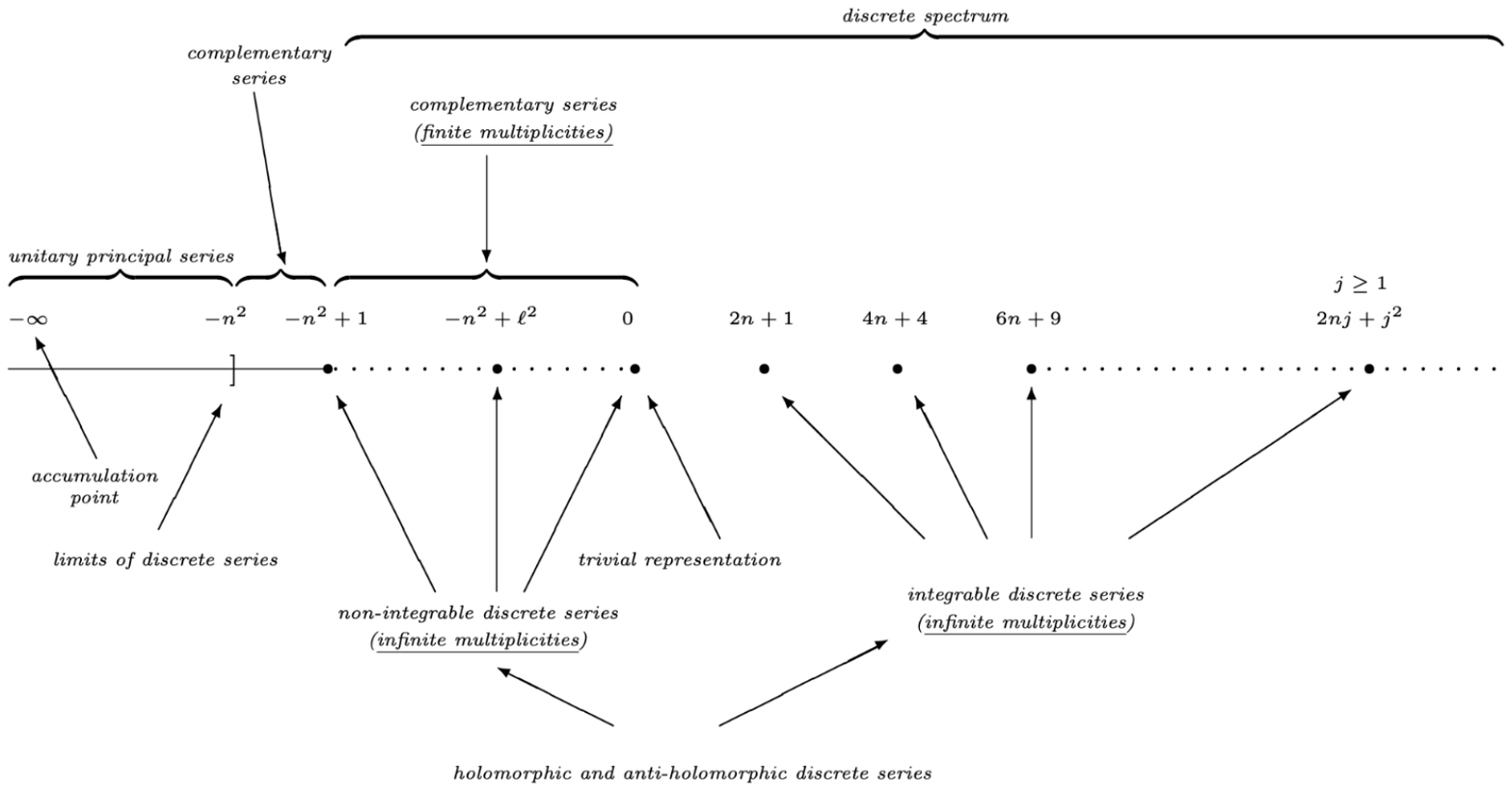}
\vspace*{-2cm}
\caption{Spectrum of $\Delta_{Y}$.}
\end{figure}

\noindent
Note that the space $E_0$ of solutions of the relativistic wave equation is infinite dimensional.

Of course, it would be highly desirable to uncover some structural properties at least of the {\it essential} spectrum (i.e. eigenvalues of infinite multiplicity and 
accumulation points of eigenvalues) in the half-line $(-\infty,1-n^2)$, in particular in $(-\infty,-n^2]$ (contributions of unitary principal series). However, at the moment we cannot exclude the extreme cases that the latter part of the essential spectrum 
is either empty or the full half-line, except for $n=1$, where it turns out that this part of the spectrum is always infinite, see Cor.~\ref{late}. 
Note that for $n=1$ we are in the group manifold situation $(G,H,L)\cong (G^{\prime}\times G^{\prime},\Delta G^{\prime},G^{\prime}\times K')$ with $G'=SU(1,1)$.

In order to prove Theorem \ref{something}, we use Lemma \ref{bassa} combined with Proposition \ref{Casimir}(2). For $k\in\mathbb{Z}$, we consider the one-dimensional representation of $L\cap K$ given by
\beu
\gamma_k:L\cap K =U(1)\times U(n)\rightarrow \mathbb{C}^*,\;(z,A)\mapsto z^k
\end{equation*}
and set
\beu
\widehat L_k\defn\{\pi\in\widehat L\mid \gamma_k\mbox{ occurs in }\pi_{|L\cap K}\}.
\end{equation*}
%
Let $\Omega_L$, $\Omega_{L\cap K}$ be the  Casimir operators corresponding to the bilinear form 
$$\langle X,Y\rangle = \mathrm{Re} (\Tr XY)$$ 
on $\fu(1,n)$. By Proposition \ref{Casimir}, the Laplacian $\Delta_Y$ is given by the action of $2\Omega_L-\Omega_{L\cap K}$. Next, we let $\fa$ be the image 
of the natural embedding 
\beu
j:\fo(1,1)\rightarrow\fu(1,n)
\end{equation*} 
into the left upper corner. For $\lambda\in\mathbb{C}$, we define $\nu_\lambda\in\fa^*$ by 
\beu
\nu_\lambda (j(H_t))\defn \lambda t
\end{equation*}
where $H_t\in\fo(1,1)$ is given by
$\begin{pmatrix}
0&t\cr
t&0
\end{pmatrix}$. 
Let $P=MAN$ be the associated (minimal) parabolic subgroup of $L$ and let $\sigma_k$ be the restriction of $\gamma_k$ to 
\beu
M\simeq\Delta U(1)\times U(n-1)\subset L\cap K.
\end{equation*}
%
We then form the principal series representations of $L$
$$ H^{k,\lambda}\defn\Ind_P^L\sigma_k\otimes e^{\nu_\lambda}\otimes 1\  ,\qquad k\in\mathbb{Z},\lambda\in\mathbb{C}\ .$$

\begin{lem}\label{bummi}
Let $\pi\in\widehat L$ be such that $V_\pi^{L\cap H}\ne\{0\}$. Then there is a unique $k\in\mathbb{Z}$ such that $\pi\in\widehat L_k$. 
For $\pi\in\widehat L_k$ the space $V_{\pi}^{L\cap H}$ is one-dimensional and coincides with the isotypic component 
$V_{\pi}(\gamma_k)$. There is a unique parametrization of $\widehat L_k$ by $\lambda\in Q_k\cup i[0,\infty)\subset \mathbb{C}$ 
for a certain subset $Q_k\subset (0,\infty)$ which is explicitly described in Proposition \ref{crux} below. Moreover, this parametrization is 
characterized by the existence of an $L$-equivariant embedding of $V_{\pi,\infty}$ into $H^{k,\lambda}$ or $H^{k,-\lambda}$. 
In particular, the operator $2\Omega_L-\Omega_{L\cap K}$ acts on $V_{\pi}^{L\cap H}$ 
by the scalar $\lambda^2-n^2$.
\end{lem}

\begin{proof} Combining Casselman's embedding theorem with the Frobenius reciprocity formula in Lemma \ref{frob} (with $Q=L\cap H=U(n)$), 
one gets a unique integer $k$ such that $V_{\pi,\infty}$ embeds into $H^{k,\lambda}$ for some $\lambda\in\mathbb{C}$ (in fact, $\lambda\in\mathbb{R}\cup i\mathbb{R}$ by unitarity) 
and the asserted formula for the space of invariants follows. Since $\Omega_L$ acts on $H^{k,\lambda}$ by 
$$\frac{\lambda^2-n^2}{2}+\sigma_k(\Omega_M),$$ 
we see that actually $\lambda$ is uniquely determined (up to sign). The last assertion now follows from the computation
$$\sigma_k(\Omega_M)=\frac{k^{2}}{4}+\frac{k^{2}}{4}=\frac{k^2}{2}=\frac{1}{2}\gamma_k(\Omega_{L\cap K})\ .$$
\end{proof}
\noindent
We mention that the part of $\widehat L_k$ parametrized by $\lambda\in i(0,\infty)$ corresponds to the irreducible unitary principal series representations $H^{k,\lambda}$. 
The principal series representation $H^{k,0}$ is sometimes irreducible, sometimes it decomposes into two {\it limit of discrete series} representations, and exactly one of 
these belongs to $\widehat L_k$. The latter situation occurs precisely when $|k|\geq n$, $k\equiv n \mod 2$. For these facts we refer to, e.g., \cite{Kn86}, or Theorem 5 and Proposition 51 
in \cite{KnS71}. 
We now determine the sets $Q_k$, $k\in\mathbb{Z}$, introduced in Lemma~\ref{bor}.
By general theory (\cite{Kn86}, \cite{KnS71}) they correspond to {\it discrete series} representations and non-tempered unitary representations, which are exhausted for $L=U(1,n)$ by the {\it complementary series} 
(including their ends). Recall that a discrete series representation of a group $G$ is {\it integrable} if its $K\times K$-finite matrix coefficients belong to $L^1(G)$.

\begin{pro}\label{crux}\mbox{}
\begin{enumerate}
\item If $|k|<n$, then $Q_k=(0,n-|k|]$ (complementary series including the end).
\item If $|k|=n$, then $Q_k=\emptyset$.
\item If $|k|> n$, $k\equiv n \mod 2$, then $Q_k=\{2,4,\dots,|k|-n\}$ (holomorphic discrete series if $k>0$ and anti-holomorphic discrete series if $k<0$).
\item If $|k|> n$, $k\not\equiv n \mod 2$, then $Q_k=(0,1)\cup\{1,3,\dots,|k|-n\}$ (complementary series without end and discrete series).
\end{enumerate}
Moreover a discrete series representation corresponding to $\lambda\in Q_k$ is integrable if and only if
$\lambda>n$ (this occurs if and only if $|k|>2n$).
\end{pro}
\begin{proof}
The form of the complementary series part of $Q_k$ follows from \cite{KnS71}, Section II. 
By the theory of leading exponents of matrix coefficients of admissible representations (see \cite{Wa88}, Chapters 4 and 5 or \cite{Kn86}, Theorem 8.48) a 
discrete representation with a unique embedding to principal series representation with $\fa^*$-parameter $\lambda$ is integrable if and only if 
$\lambda>\rho_\fa=n$. It remains to determine the parameters of the discrete series representations in $\widehat L_k$. For this we could simply refer to \cite{Sh94}, 
where the discrete spectrum (even the full spectral decomposition) of Laplacians acting on line bundles over Hermitian symmetric spaces is determined. Instead, we prefer 
to give a more representation theoretic argument which is, in principle, applicable to much more general situations.

First recall the description of discrete series representations by their {\it Harish-Chandra parameters} together with information coming from the 
{\it Blattner formula} concerning their $K$-types (see e.g. \cite{Wa88}, Chapters 6 and 8, in particular Theorem. 6.7.6). Let $T$ be the torus of diagonal 
matrices in $L=U(1,n)$. It is a Cartan subgroup of $L$ as well as of $L\cap K$. Irreducible $L\cap K$-representations are parametrized by integral 
$L\cap K$-dominant weights of the form (with respect to the standard basis $e_0,e_1,\dots, e_n$ of ${i}\ft^*$) 
$$ (m_0,m_1,\dots,m_n)\ ,\quad m_i\in\mathbb{Z},\  m_1\geq m_2\geq\dots\geq m_n\ .$$
Note that $\gamma_k$ has highest weight $(k,0,\dots,0)$. Discrete series representations of $L$ are parametrized via their Harish-Chandra parameter
by a similar set of regular dominant weights 
$$ (m_0,m_1,\dots,m_n)\ ,\quad m_i+\frac{n}{2}\in\mathbb{Z},\  m_0\ne m_i, i=1,\dots n,\, m_1> m_2>\dots> m_n\ .$$
In particular, Harish-Chandra parameters fall into $n+1$ disjoint open Weyl chambers $C_0,C_1,\dots,C_n$, where $C_r$ is characterized by 
$m_{r}\geq m_0>m_{r+1}$. The Weyl chamber $C_r$ corresponds to the system of positive noncompact roots 
$$\Sigma_r=\{e_0-e_{r+1},\dots,e_0-e_{n}, e_1-e_0,\dots, e_r-e_0\}$$
with half sum
$$\rho_r =\big (\frac{n}{2}-r,\underbrace{\frac{1}{2},\dots,\frac{1}{2}}_r,\underbrace{-\frac{1}{2},\dots,-\frac{1}{2}}_{n-r}\big )\ .$$
Compact positive roots are independent of $r$ with half sum
$$\rho_c =\big (0,\frac{n-1}{2},\frac{n-3}{2},\dots,\frac{1-n}{2}\big )\ .$$
A discrete series representation with Harish-Chandra parameter $\Lambda\in C_r$ has infinitesimal character $\Lambda$ and {\it minimal $L\cap K$-type} 
\beu
\nu_\Lambda\defn\Lambda+\rho_r-\rho_c. 
\end{equation*}
For any other $L\cap K$-type $\nu$, one has that $\nu-\nu_\Lambda$ is a non-negative integer linear combination 
of elements of $\Sigma_r$. 
If we choose the complex structure of the Hermitian symmetric space $L/(L\cap K)$ such that its holomomorphic tangent space corresponds to
$$ \bigoplus_{\alpha\in \Sigma_0} \fl_{-\alpha}\subset (\fl\cap\fs)_\C\ ,$$
then discrete series representations with Harish-Chandra parameter $\Lambda\in C_0$ are holomorphic, 
while discrete series representations with Harish-Chandra parameter $\Lambda\in C_n$ are anti-holomorphic (Theorem 6.6 in \cite{Kn86}).

Now let $\pi\in\widehat L_k$ be a discrete series representation corresponding to $\lambda\in Q_k\subset (0,\infty)$ with Harish-Chandra parameter 
$\Lambda\in C_r$ for some $r\in\{0,1,\dots,n\}$ and minimal $(L\cap K)$-type $\nu_\Lambda$. Since 
$$\Lambda=\nu_\Lambda+\rho_c-\rho_r=\gamma_k+\rho_c-\rho_r-R_r,$$ 
where $R_r$ lies in the cone spanned by $\Sigma_r$, it follows that
\begin{eqnarray*}
\|\Lambda\|^2&=&\langle \Lambda,\gamma_k+\rho_c-\rho_r-R_r\rangle\\
&<&\langle \Lambda,\gamma_k+\rho_c\rangle\\
&\le& \|\Lambda\| \|\gamma_k+\rho_c\|
\end{eqnarray*}
and hence $\|\Lambda\|< \|\gamma_k+\rho_c\|$. The infinitesimal character of $\pi$ coincides with the one of the principal series representation $H^{k,\lambda}$ which is
\beu
\big (\frac{k+\lambda}{2},\frac{k-\lambda}{2},\frac{n}{2}-1,\frac{n}{2}-2,\dots,1-\frac{n}{2}\big )
\end{equation*}
and $\Lambda$ is a permutation (Weyl conjugate) of this vector. We restrict our attention to the case $k\geq 0$, since the arguments are analogous for $k<0$. 
We obtain
\begin{eqnarray*}
0&<&  \|\gamma_k+\rho_c\|^2-\|\Lambda\|^2 
\le k^2 +2\left (\frac{n-1}{2}\right )^2 -\left (\frac{k+\lambda}{2}\right )^2 -\left (\frac{k-\lambda}{2}\right )^2\\
&=&\frac{1}{2}\left (k^2-\lambda^2+(n-1)^2 \right ),
\end{eqnarray*}
and therefore $\lambda< k+n-1$. The regularity and integrality conditions on $\Lambda$ imply that either $\frac{k-\lambda}{2}\geq\frac{n}{2}$, i.e.,
\be\label{aupoint}
\lambda\in\Big\{k-n-2\ell\mid \ell=0,1,\dots,\left [ \frac{k-n-1}{2}\right ]\Big\}\ ,
\end{equation}
or $\frac{k+\lambda}{2}\geq\frac{n}{2}$ and $\frac{k-\lambda}{2}\le-\frac{n}{2}$, i.e.
$
\lambda\geq k+n
$, which is impossible. Thus any $\lambda\in Q_k$ satisfies (\ref{aupoint}). Conversely, assume that $\lambda$ satisfies (\ref{aupoint}).
We define a Harish-Chandra parameter belonging to $C_0$ by 
$$\Lambda\defn\big (k-\ell-\frac{n}{2},\frac{n}{2}+\ell,\frac{n}{2}-1,\frac{n}{2}-2,\dots,1-\frac{n}{2}\big ).$$
The corresponding discrete series representation has minimal $L\cap K$-type
$$ \nu_{\Lambda}=\Lambda+\rho_0-\rho_c=(k-\ell,\ell,0,\dots,0).$$
Note that $\gamma_k-\nu_{\Lambda}=(\ell,-\ell,0,\dots,0)$ belongs to the cone spanned by $\Sigma_0$.
Therefore $\gamma_k$ may occur in  $\pi_{\Lambda}$. To check that it does occur we use the full Blattner formula (\cite{Wa88}, Theorem~6.7.6)
which gives the precise multiplicity of a $(L\cap K)$-type as an alternating sum over the compact Weyl group $W_{\fl\cap\fk}$
\begin{eqnarray*} [\pi_{\Lambda}:\nu]&=&\sum_{w\in W_{\fl\cap\fk}}\det(w) p(w(\nu+\rho_c)-(\Lambda+\rho_0))\\
&=&\sum_{s\in S_n}\sgn(s) p(s(\nu+\rho_c)-(\Lambda+\rho_0))\ ,
\end{eqnarray*}
where 
\beu
p:{i}\ft^\star\rightarrow\mathbb{N}_0
\end{equation*}
is the partition function associated with $\Sigma_0$, defined as the number of ways in that $\tau$ can be written 
as a non-negative integer linear combination of elements of $\Sigma_0$. Now $\gamma_k+\rho_c=\big (k,\frac{n-1}{2},\frac{n-3}{2},\dots,\frac{1-n}{2}\big )$ 
and $\Lambda+\rho_0=\big (k-\ell,\frac{n-1}{2}+\ell,\frac{n-3}{2},\dots,\frac{1-n}{2}\big )$. Hence for $s\in S_n\setminus\{1\}$ we have:
$$ s(\gamma_k+\rho_c)-(\Lambda+\rho_0)=(\ell,m_1,\dots,m_n)$$
for some $m_i\in\mathbb{Z}$ with $m_i>0$ for at least one $i$. Therefore, $s(\gamma_k+\rho_c)-(\Lambda+\rho_0)$ is not in the cone spanned by $\Sigma_0$, i.e., 
\beu
p(s(\gamma_k+\rho_c)-(\Lambda+\rho_0))=0.
\end{equation*}
We obtain that
$$ [\pi_\Lambda:\gamma_k]=p((\gamma_k+\rho_c)-(\Lambda+\rho_0))=p(\ell,-\ell,0,\dots,0)=1\ .$$
Thus $\pi_{\Lambda}\in L_k$. It follows that $\lambda\in Q_k$.
\end{proof}
For 
$k\in \Z$ 
and $\lambda\in Q_{k}\cup i\lbrack 0,\infty)$, let $\pi_\lambda^k$ be the corresponding element of $\widehat L_k$. The following proposition tells
us that almost all discrete series representations among them really occur in $L^2(\Gamma\backslash L)$. The only possible exceptions are the two representations $\pi^{\pm(n+1)}_1$, $n\ge 2$.
Recall our standing assumption that $\Gamma$ is torsion free.

\begin{pro}\label{pavle}Let $|k|\ge n+2$ or $n=1$.
If $\pi_\lambda^k$ is a discrete series representation (see Prop.~\ref{crux}), then we have
\be\label{kusmin} 
\dim N_\Gamma(\pi_{\lambda}^{k})\ge d_{\pi_{\lambda}^{k}}\vol(\Gamma\backslash L)>0\ ,
\end{equation}
where $d_\pi$ is the formal dimension of $\pi$. For $\lambda\ge 2$ we have equality in (\ref{kusmin}).
\end{pro}

Let us mention that for a discrete series representation of a connected reductive group with compact center with sufficiently regular Harish-Chandra parameter $\Lambda$
the Formula (\ref{kusmin}) with equality is already known. In particular, for integrable discrete series (in our situation $\lambda> n$) the formula was first observed by Langlands \cite{La66},
see also \cite{Wa76}, and is a quite direct consequence of the Selberg trace formula. Under the weaker regularity condition (in the notation of the proof of Prop.~\ref{crux})
$$ \langle \Lambda-\rho_r-\rho_c,\alpha\rangle >0 \:,\quad \alpha\in \Sigma_r, $$
the formula was established in \cite{HuPa06}, Ch.~8. For the group $U(1,n)$ (and $Spin(1,n)$) the result was already shown in Proposition 9.2 of \cite{HoPa74}. In our situation, this regularity
condition amounts to $\lambda\ge 2$. Both proofs use index theoretic arguments. Our proof of (\ref{kusmin}), which also covers $\lambda=1$, will be similar in spirit. 

{\em Proof of Proposition \ref{pavle}}. 
Let $\ell$ be given by $2\ell=|k|-n-\lambda$. We consider the $K_L$-representation with highest weight $\nu=(|k|-\ell,\ell,0,\dots,0)$. (It is the minimal $K$-type of $\pi_\lambda^{|k|}$, see the proof of Prop.~\ref{crux}.)
It defines a complex vector bundle $E$ over the complex hyperbolic manifold $Z=\Gamma\backslash L/K_L$. We consider the Euler characteristic $\chi(Z,E)$ for its
Dolbeaut cohomology groups. It can be computed explicitly by the index theorem of Atiyah-Singer (which is in this case Hirzebruch-Riemann-Roch) or rather by Hirzebruch's proportionality principle (see e.g. \cite{HuPa06}, Ch. 8) coupled with the Theorem of Borel-Weil-Bott. The result is
\be \label{brecht}  \chi(Z,E)=  d_{\pi_{\lambda}^{k}}\vol(\Gamma\backslash L)\:.\end{equation}
On the other hand, $\chi(Z,F_\nu)$ can be expressed by the multiplicities $N_\Gamma(\pi)$, where $\pi\in \widehat L$ has the same infinitesimal character
as $\pi_\lambda^{k}$ (or $\pi_\lambda^{-k}$). This leads to Proposition \ref{nichts} below, which will be shown in Appendix B. Proposition~\ref{nichts} combined with (\ref{brecht}) implies Prop.~\ref{pavle}.   
\hfill$\Box$

\begin{pro}\label{nichts} Under the assumptions of Prop.~\ref{pavle} and in the notation of its proof we have 
$$  \chi(Z,E)=\left\{\begin{array}{ll} \dim N_\Gamma(\pi_{\lambda}^{k})&\lambda\ge 2\\ \dim N_\Gamma(\pi_{\lambda}^{k})-\dim N_\Gamma(\pi')&\lambda=1\end{array}\right.  \ ,$$
where $\pi'$ is the end of complementary series representation that is excluded from $Q_k$ in Prop.~\ref{crux} (d).
\end{pro} 

{\em Proof of Theorem \ref{something}}.
The theorem follows from Lemma \ref{bassa}, Lemma \ref{bummi}, and Proposition \ref{crux} combined with the following facts:
\begin{itemize}
\item One has
\beu
L^2(Y)\cong\widehat{\bigoplus_{k}}\;\widehat{\bigoplus_{\lambda\in Q_{k}\cup i\lbrack 0,\infty)}}N_\Gamma(\pi_{\lambda}^{k})\otimes V_{\pi_{\lambda}^{k}}(\gamma_k)
\end{equation*}
so that the multiplicities of the eigenvalues are given by:
\be\label{klumpfuss}
\dim E_{\lambda^{2}-n^{2}}=\sum_{k}\dim N_\Gamma(\pi_{\lambda}^{k}).
\end{equation}
Thus in view of Prop.~\ref{pavle} the discrete series contributions make the eigenspaces infinite-dimensional for $\lambda\in\N$.
\item For fixed $k$, the set $\big\{\lambda^2-n^2\mid \lambda\in Q_{k}\cup i\lbrack 0,\infty), N_\Gamma(\pi_{\lambda}^{k})\ne\{0\}\big\}\subset\mathbb{R}$ is closed, discrete and infinite as the spectrum 
of an elliptic self-adjoint operator on the compact manifold $Z=\Gamma\backslash L/L\cap K$ (actually, as a set with multiplicities it obeys Weyl's asymptotics for $\lambda\to i\infty$).
\hfill$\Box$
\end{itemize}

It should be noted that the above mentioned Weyl type asymptotics and related classical results do not seem to have serious implications for the structure 
of the negative part of $\spec(\Delta_Y)$. For this, one would like to understand other types of ``semi-classical limits'':
for a fixed bounded interval $I\subset i\mathbb{R}$ study the behavior of the numbers
$N_\Gamma(I,k)\defn\sum_{\{\pi\in \widehat L_k\mid \lambda_\pi\in I\} }\dim N_\Gamma(\pi)$, when $|k|$ goes to $\infty$. For instance, if for all sufficiently large $I$ these numbers
tend to infinity for $|k|\to\infty$, then the negative part of the {\em essential} spectrum of $\Delta_Y$ will be infinite.

For $n=1$ we can clarify the situation rather completely. We have $L=U(1,1)=SU(1,1)\cdot M$ with $SU(1,1)\cap M=\{\pm\id\}$. From this, one deduces (using that the projection of $\Gamma$ to $PU(1,1)$ can be lifted to $SU(1,1)$) that any torsion-free cocompact lattice $\Gamma \subset L=U(1,1)$ is of the form
$$\Gamma_{0,\chi}:=\{(\gamma\chi(\gamma)\mid \gamma\in\Gamma_0\}\ ,$$
where $\Gamma_0\subset SU(1,1)$ is a cocompact torsion-free lattice and $\chi:\Gamma_0\rightarrow M
\cong U(1)$ is a character. Note that the group ${\Cal X}$ of all such characters
is a torus of dimension $2g$, where $g\ge 2$ is the genus of the Riemann surface $Z\cong\Gamma_0\backslash H^2$, $H^2$ being the real hyperbolic plane.
Let ${\Cal X}_\Gamma\subset {\Cal X}$ be the closure of the subgroup $\{\chi^{2k}\mid k\in\Z\}$. For any $\vp\in {\Cal X}$ we consider the flat line bundle $E_\vp\rightarrow Z$ defined by $\vp$. We also consider the (non-flat) locally homogeneous line bundle $F_1\rightarrow Z$ defined
by the restriction of the representation $\gamma_1$ to $S(U(1)\times U(1))$. We set $E_{1,\vp}:=F_1\otimes E_\vp$. We look at the discrete spectral decompositions  of
the action of the operator $2\Omega_L-\Omega_{L\cap K}=2\Omega'$, where $\Omega'$ is the Casimir operator of $SU(1,1)$, on the spaces of sections $L^2(Z,E_\vp)$ and  $L^2(Z,E_{1,\vp})$.
Ordering the eigenvalues by size and repeating them according to their finite multiplicity we obtain functions
$\mu_i,\nu_i: {\Cal X}\rightarrow (-\infty, 0]$, $i\in\N$, such that for every $\vp\in{\Cal X}$ 
$$ \dots\le \mu_3(\vp)\le\mu_2(\vp)\le \mu_1(\vp)\le 0\ ,\qquad \dots\le \nu_3(\vp)\le\nu_2(\vp)\le \nu_1(\vp)\le -1 .$$
It is well-known that these functions are continuous. Moreover,  they converge pointwise to $-\infty$ for $i\to\infty$, and $\mu_k(\vp)=0$ if and only if $k=1$ and $\vp$ is trivial.
This follows from the fact that $-\Omega'$ acts on sections of $E_\vp$ as the Bochner-Laplace operator associated to the canonical flat connection. The upper bound for $\nu_1$ comes from unitary representation theory of $SU(1,1)\cong SL(2,\R)$, see the proof of the following proposition.

\begin{pro}\label{nottoolate}
Let $n=1$, and let $\Gamma=\Gamma_{0,\chi}$ be as above. We consider the subsets ${\Cal X}_\Gamma, \chi {\Cal X}_\Gamma\subset {\Cal X}$. Then the set $\spec (\Delta_Y)\cap (-\infty,0]$ coincides with the union of the images $\mu_i({\Cal X}_\Gamma)$ and $\nu_i(\chi {\Cal X}_\Gamma)$, $i\in \N$.
\end{pro}
\begin{proof} 
We first check that the above union of images is closed. Since ${\Cal X}_\Gamma$ has only finitely may connected components it suffices to
discuss the closedness of sets of the form
$$ \tilde A:=\bigcup_{i=1}^\infty f_i(A)\ ,$$
where $A\subset {\Cal X}_\Gamma$ (or $A\subset \chi{\Cal X}_\Gamma$) is a connected component and $f_i=\mu_i$ (or $f_i=\nu_i$) for all $i$. Then  $f_i(A)=[a_i,b_i]$
for some $a_{i+1}\le a_i\le b_i\le b_{i-1}$. We have to allow degenerate `intervals' $[a,a]:=\{a\}$. Set $a_0:=0$. Now let $c$ be in the complement of $\tilde A$.
Let $i$ be maximal subject to the condition $c<a_i$. Then $a_{i+1}<c$, and hence also $b_{i+1}<c$. By the above properties the intersection of $\tilde A$ with the open interval 
$(b_{i+1}, a_i)\ni c$ is empty.
Hence the complement of $\tilde A$ is open.

Every representation of $L=U(1,1)=SU(1,1)\cdot M$  defines a representation of $SU(1,1)\times M$. As such, we have
\be\label{furunkel} 
H^{k,\lambda}\cong H^{\pm,\lambda}\otimes \C_k\ ,
\end{equation}
where $\pm$ stands for the parity of $k$, $H^{\pm,\lambda}$ is the even/odd principal series for $SU(1,1)\cong SL(2,\R)$,  and $\C_k$ is the one-dimensional representation space of the representation $\sigma_k$ of $M$.
In particular, as representations of $L':=SU(1,1)$, the representations $\pi_{\lambda}^{k_1}$ and $\pi_{\lambda}^{k_2}$ are equivalent, whenever $k_1-k_2$ is even and $H^{k_1,\lambda}$ is irreducible. 

By (\ref{klumpfuss}) and Prop.~\ref{crux} the eigenspaces of $\Delta_Y$ for negative eigenvalues are described by the spaces $N_\Gamma(\pi^k_\lambda)$ for $\lambda$ imaginary or,
if $k$ is even, $\lambda\in (0,1)$. With the exception of $\lambda=0$, $k$ odd, these representations are irreducible principal series representations. In order to get the eigenvalue
$\mu=0$ into the picture we also consider the trivial representation $\pi^0_1$.

Let us assume for a monent that $H^{k,\lambda}$ is irreducible. Then $N_\Gamma(\pi^k_\lambda)=\Hom_L(H^{k,\lambda}_\infty,C^\infty(\Gamma\backslash L))$.
Let first $k$ be odd. We fix a non-zero vector $v\in H^{k,\lambda}$ transforming with respect to $\gamma_{-1}$
under $L'\cap K=S(U(1)\times U(1))$. Now we consider the following map
$$ \Hom_L(H^{k,\lambda}_\infty,C^\infty(\Gamma\backslash L))\ni\Phi\mapsto F_\Phi\in C^\infty(L'),\quad F_\Phi(l):=\Phi(v)(l),\ l\in L'.$$ 
Then we have for $l\in L'$, $k'\in L'\cap K$
$$ F_\Phi(lk')=r_{k'}\circ \Phi(v)(l)=\Phi(k'v)(l)=\gamma_{-1}(k') \Phi(v)(l)=\gamma_1(k')^{-1}F_\Phi(l)$$
and for $\gamma_0\in\Gamma_0$, $\gamma=\gamma_0\chi(\gamma_0)\in \Gamma$
$$ F_\Phi (\gamma_0l)=F_\Phi(\gamma l\chi(\gamma_0)^{-1})=\Phi(\chi(\gamma_0)^{-1}v)(l)=\Phi(\sigma_k(\chi(\gamma_0)^{-1})v)(l)=\chi^{-k}(\gamma_0)F_\Phi(l).$$
These transformation rules show that $F_\Phi\in C^\infty(Z, E_{1,\chi^{-k}})$. It is clear that $F_\Phi$ is an eigensection with corresponding eigenvalue $\lambda^2-1$.
On the other hand, these eigenspaces have a representation theoretic description similarly as in Lemma \ref{hussa} coming from the decomposition of $L^2(\C_\vp\times_{\Gamma_0} L')$ as a representation of $L'$. This together with (\ref{furunkel}) shows that for $\lambda\ne 0$ the map $\Phi\mapsto F_\Phi$ induces an isomorphism
between $N_\Gamma(\pi^k_\lambda)$ and the $\lambda^2-1$-eigenspace in $C^\infty(Z, E_{1,\chi^{-k}})$.
For $\lambda=0$ this map still induces an isomorphism between $N_\Gamma(\pi^{-|k|}_0)$ and the $-1$-eigenspace in $C^\infty(Z, E_{1,\chi^{-k}})$.
Since $\pi^k_0$ and $\pi^{-k}_0$ are conjugate representations we have $\dim N_\Gamma(\pi^{k}_0)=\dim N_\Gamma(\pi^{-k}_0)$.
Moreover, we have captured all eigenspaces, since we have already used all unitary representations of $L'$  containing the $L'\cap K$-type $\gamma_{-1}$.

For even $k$ we proceed in a completely analogous way, but now we choose $v$ to be $L'\cap K$-invariant. We obtain isomorphisms between $N_\Gamma(\pi^k_\lambda)$ and the $\lambda^2-1$-eigenspaces in $C^\infty(Z, E_{\chi^{-k}})$. This isomorphism persists for $(k,\lambda)=(0,1)$.

We obtain: A number $\mu\in (-\infty,0]$ is an eigenvalue of $\Delta_Y$ if and only if there exist $k\in\Z$ and $i\in\N$ such that
$\mu=\mu_i(\chi^{2k})$ or $\mu=\nu_i(\chi^{2k+1})$. By continuity of $\mu_i$, $\lambda_i$ and the definition of ${\Cal X}_\Gamma$ the set of eigenvalues of $\Delta_Y$ is dense in the union described in the proposition. Since the latter set is closed, it coincides with $\spec(\Delta_Y)$.
\end{proof}

The description of the negative part of $\spec(\Delta_Y)$ in Prop.~\ref{nottoolate} together with Thm.~\ref{something} has the following consequences.

\begin{cor}\label{late}
Let $n=1$. Then the essential spectrum of $\Delta_Y$ coincides with $\spec(\Delta_Y)$. In particular, the essential spectrum is also not bounded from below. Moreover, the following assertions are
equivalent:
\begin{itemize}
\item[(i)] The intersection $\Gamma\cap SU(1,1)$ has infinite index in $\Gamma$.
\item[(ii)] The set of eigenvalues of $\Delta_Y$ has an accumulation point in $\R$.
\item[(iii)] There are real numbers $a<b$ such that $[a,b]\subset \spec(\Delta_Y)$.
\end{itemize}
\end{cor}
\begin{proof} 
In view of  Thm.~\ref{something} we have to verify the first assertion only for the negative part of the spectrum. Let $\chi$ be such that $\Gamma=\Gamma_{0,\chi}$, and let $\mu\in\spec (\Delta_Y)\cap (-\infty,0]$.
 By Prop.~\ref{nottoolate} 
there exist $\vp\in{\Cal X}_\Gamma$ and $i\in\N$ such that
$\mu=\mu_i(\vp)$ or $\mu=\nu_i(\chi\vp)$. To save notation, we discuss only the first case. We can choose a sequence of integers $(k_j)$ without repetitions such that
$\chi^{-2k_j}\to\vp$ in ${\Cal X}_\Gamma$. We have seen in the proof of  Prop.~\ref{nottoolate} that the $\mu_i(\chi^{-2k_j})$-eigenspace in $L^2(Z,E_{\chi^{-2k_j}})$, which is non-zero, is isomorphic to $N_\Gamma(\pi_{\lambda_j}^{2k_j})$ with $\lambda_j^2-1=\mu_i(\chi^{-2k_j})$. Now we choose a unit vector $w_j$ in the $\pi_{\lambda_j}^{2k_j}$-component of $L^2(Y)$, see Lemma \ref{bassa}.  It is an eigenvector of $\Delta_Y$. Now $(w_j)$ is a Weyl sequence for $\mu=\mu_i(\vp)$. Hence $\mu$ belongs to the
essential spectrum.

In order to establish the equivalence between (i), (ii) and (iii) we first observe the following equivalences:
$$
(i) \mbox{ holds for } \Gamma=\Gamma_{0,\chi} \Leftrightarrow  \mathrm{im} (\chi)\mbox{ is infinite}  \Leftrightarrow {\Cal X}_\Gamma \mbox{ is infinite}\ .$$
Assume that Assertion (i) holds. Since ${\Cal X}_\Gamma$  is infinite, it  has a non-trivial identity component. Let $\vp$ be a non-trivial character belonging to this component. Then 
$\mu_1(\vp)<0$, see the remark  before Prop.~\ref{nottoolate}.
Thus $[\mu_1(\vp),0]\subset\spec(\Delta_Y)$ by  Prop.~\ref{nottoolate}, in particular (iii) holds. The implication (iii) $\Rightarrow$ (ii) is trivial. Now we assume that (i) does not hold. 
In view of the above equivalence and Prop.~\ref{nottoolate}, this implies that the negative part of $\spec(\Delta_Y)$ is a finite union of spectra of elliptic operators on compact manifolds.
The absence of accumulation points follows for the negative part, the positive part of the spectrum was already covered by Thm.~\ref{something}. 
\end{proof}
What makes the case $n=1$ special is that $G/H$ is isomorphic to a group manifold $G_1\times G_1/\Delta G_1$ with $G_1=SU(1,1)$.
We remark that similar arguments as above can be used to exhibit accumulation points of eigencharacters for certain standard quotients of more general group manifolds corresponding to Case 11 in Table 2, namely
for $Y=  \Gamma\backslash G_1\times G_1/\Delta G_1$ with 
$|\Gamma/\Gamma\cap(G_1\times\{e\})|=\infty$ and $\Gamma\subset G_1\times T_1$, where $T_1\subset K_1$ is a maximal torus. Note that such $\Gamma$ exist if and only if $G_1$ is locally isomorphic to $SO(1,n)$ or $SU(1,n)$.

\section{The spectrum of a Type~II example}\label{infinitemulti}

In this section we study the spectral decomposition of ${\bf D}(G/H)$ on $L^2(\Gamma\backslash G/H)$, where $\Gamma\subset L$ is a cocompact lattice and 
$(G,H,L)$ is as in Remark \ref{remindecomp} with $G^\prime=PSL(2,\mathbb{R})$:

$G=G^\prime\times G^\prime\times G^\prime\times G^\prime$, $H=\Delta_{12} G^\prime \times\Delta_{34} G^\prime$, $L=G^\prime\times\Delta_{23} G^\prime$.\\
Observe that $L=L_{min}$ and $L\cap H=\{e\}$. As already remarked at the end of Section 2, $(G,H,L)$ is the triple of Type~II of smallest dimension (although it is not irreducible). We will see that ${\bf D}(G/H)$ has a unique spectral decomposition which in contrast to the Type~I situation (see Thm.~\ref{selim}) is not purely discrete. We expect that the same holds for all
triples of Type~II.

Note that $X=G/H$ is a group manifold. In particular,
$$ {\bf D}(G/H)\cong{\Cal Z}(\fg'\oplus\fg')\cong \mathbb{C} [j_1(\Omega'),j_3(\Omega')]\ ,$$
where $\Omega'$ is the Casimir operator of $G'=PSL(2,\mathbb{R})$ and $j_k$ denotes the embedding of  ${\Cal U}(\fg ')\hookrightarrow {\Cal U}(\fg)$ induced by the embeeding
of $G'$ into $G$ at the $k$-th position. We have $\Omega_G= j_1(\Omega')+j_3(\Omega')$ as elements of ${\bf D}(G/H)$.

As a first step we determine the embedding $\imath: {\bf D}(G/H)\rightarrow {\bf D}(L/L\cap H)$. Since $j_1(\Omega')\in{\Cal U}(\fl)\subset{\Cal U}(\fg)$ we have $\imath(j_1(\Omega'))=j_1(\Omega')$.
Concerning the second generator $j_3(\Omega')$ we have

\begin{lem}\label{bur}$ \imath(j_3(\Omega'))=\Delta_{123}(\Omega')$.\\
 Here we view $\Delta_{123}$ as the embedding ${\Cal U}(\fg ')\hookrightarrow {\Cal U}(\fl)$ induced by the corresponding embedding $G'\hookrightarrow G$ whose image is contained in $L$ .
\end{lem} 

\begin{proof} Let $X\in\fg'$. We consider $Y:=\Delta_{12}(X)\in \fh\subset {\Cal U}(\fg)$ and $Z=i_3(X)\in \fg\subset {\Cal U}(\fg)$. Then $\Delta_{123}(X)=Y+Z$ and
\begin{eqnarray*}\Delta_{123}(X)^2&=&(Y+Z)^2=Y^2+Z^2+YZ+ZY\\
&=&Y^2+Z^2+2ZY\\
&\equiv& Z^2 \mod {\Cal U}(\fg)\fh\ .
\end{eqnarray*}
Now we let $X$ run over an orthonormal basis of $\fg'$ and sum up. The result follows.
\end{proof} 
Now we can use Lemma \ref{bassa}:
\be\label{bor} L^2(\Gamma\backslash G/H)\stackrel{{\bf D}(G/H)}{\cong} \widehat{\bigoplus_{\pi\in\widehat L}} N_\Gamma(\pi)\otimes V_\pi\ ,
\end{equation}
which reduces our task to the study of the action of  $\imath ({\bf D}(G/H))$ on $V_\pi$, $\pi\in \widehat L$. Recall that in this case $L\cap H=\{e\}$. We have
$$ \widehat L\cong \{\eta_1\hat\otimes \eta_2\mid \eta_i\in\widehat{PSL(2,\mathbb{R})}\}\ ,$$
and
$$ \widehat{PSL(2,\mathbb{R})}=\{\pi_\lambda, \pi_k, 1\mid \lambda\in i[0,\infty)\cup(0,\frac{1}{2}), k\in\mathbb{Z}  \mbox{ odd}\}$$
consists of unitary principal series, complementary series, discrete series, and the trivial representation. We denote the corresponding representation spaces by $H^\lambda$, $D_k$, and 
$\C$.
The Casimir operator $\Omega'$ (in suitable normalization) has corresponding eigenvalues $\lambda^2-\frac{1}{4}$, $\frac{k^2-1}{4}$, $0$.

The operator $j_1(\Omega')$ acts on $V_\pi=V_1\hat\otimes V_2$ with the eigenvalue corresponding to $V_1$. In order to understand the action of $j_3(\Omega')\equiv \Delta_{123}(\Omega')$,
which is essentially the Casimir of the diagonal subgroup of $L\cong G'\times G'$,
we can use the well-known branching rules for tensor product representations of  $PSL(2,\mathbb{R})$, see \cite{Rep1}, \cite{Rep2}.

Since $N_\Gamma(\pi)=N_\Gamma(\bar\pi)$, where $\bar\pi$ is the conjugate representation, we list the braching rules for self-conjugate tensor products and for sums $V_\pi\oplus V_{\bar\pi}$, if $\pi\not\cong\bar\pi$.
Note that $\overline{ \pi_\lambda}\cong \pi_\lambda$ and $\overline{\pi_k}\cong \pi_{-k}$:
\begin{itemize}
\item[(a)] $\displaystyle H^{\lambda_1}\hat\otimes H^{\lambda_2}\cong (D_l\oplus D_{-l})\hat\otimes H^\mu\cong \int_{i[0,\infty)}^\oplus \C^2\otimes H^\lambda d\lambda \oplus \widehat{\bigoplus_{k}}\: D_k\ \left(\oplus H^{\lambda_1+\lambda_2-\frac{1}{2}}\right)$ .\\
The term in parentheses  only appears in the decomposition of a tensor product of two complementary series representations $H^{\lambda_1}$ and $H^{\lambda_2}$  
with $\lambda_1+\lambda_2>\frac{1}{2}$. 
\item[(b)] If $l_1$, $l_2$ have opposite signs, then\\
$\displaystyle  (D_{l_1}\hat\otimes D_{l_2})\oplus  (D_{-l_1}\hat\otimes D_{-l_2})\cong \int_{i[0,\infty)}^\oplus \C^2\otimes  H^\lambda d\lambda \oplus{\bigoplus_{|k|\le|l_1+l_2|-1}} D_k$ .
\item[(c)] If $l_1$, $l_2$ have the same sign, then\\
$\displaystyle  (D_{l_1}\hat\otimes D_{l_2})\oplus  (D_{-l_1}\hat\otimes D_{-l_2})\cong \widehat{\bigoplus_{|k|\ge |l_1+l_2|+1}} D_k$.
\item[(d)] $V_{\pi_1}\otimes \C\cong V_{\pi_1}$, $\pi_1\in  \widehat{PSL(2,\mathbb{R})}$.
\end{itemize}

Now we can derive the desired spectral decomposition. We use the notation of Section 7 and identify the character space $\Xi$ with $\R^2$ via $\chi\mapsto (\chi(j_1(\Omega')),\chi(j_3(\Omega')))$. In particular, we can consider the eigenspaces $E^{(2)}_{(\nu_1,\nu_2)}(Y)$ for $(\nu_1,\nu_2)\in\R^2$.  We recall that a lattice $\Gamma\subset L\cong G'\times G'$ is called reducible, if it contains a finite index subgroup of the form 
\be\label{gummi}\Gamma_1\times \Gamma_2\ ,\,\Gamma_i\subset G',
\end{equation}
and irreducible otherwise. In order to be able to make quite precise statements on the appearance and dimension of certain eigenspaces we assume that $\Gamma$ is torsion
free and, if reducible, of the form (\ref{gummi}). Both properties can be achieved by going over to a finite index subgroup of a given general cocompact lattice of~$L$.

\begin{thm}\label{surprise}
Let $(G,H,L)$ and $\Gamma\subset L$ be as above. We set $Y=\Gamma\backslash G/H$. Then we have a spectral decomposition of ${\bf D}(G/H)$ of the following form:
\begin{eqnarray}\label{wums}
L^2(Y)&\cong& \widehat{\bigoplus_{\nu_1\in S_+\cup S_-}}\left (\int_{\nu_2\in (-\infty,-\frac{1}{4}]}^\oplus  W_{(\nu_1,\nu_2)} d\nu_2  \oplus \widehat{\bigoplus_{\nu_2\in\mathrm{sm}(\nu_1)\cup S_+}}
E^{(2)}_{(\nu_1,\nu_2)}(Y)    \right )\\
\nonumber&& \oplus\, \widehat{\bigoplus_{\nu\in S^1_-}} E^{(2)}_{(\nu,\nu)}(Y)  \oplus \widehat{\bigoplus_{\nu\in S^2_-}} E^{(2)}_{(0,\nu)}(Y)  \ , 
\end{eqnarray}
where
\begin{itemize}
\item $S_+:=\{\frac{k^2-1}{4}\mid k\in\N\mbox{ odd}\}=\{0,2,6,12,20,\dots\}$,
\item $S_-:=\{\lambda^2-\frac{1}{4}\mid \lambda\in i[0,\infty)\cup(0,\frac{1}{2}), \exists \pi\in \widehat{PSL(2,\mathbb{R})}\mbox{ s.th. } N_\Gamma(\pi_\lambda\hat\otimes\pi)\ne\{0\}\}$
,
\item  $\mathrm{sm}(\nu_1):=\emptyset$, if $\nu_1\not\in (-\frac{1}{4}, 0)$,  and \\
$\mathrm{sm}(\nu_1):=\{(\lambda_1+\lambda_2)(\lambda_1+\lambda_2-1)\mid \lambda_i\in (0,\frac{1}{2}), \lambda_1^2-
\frac{1}{4}=\nu_1, \lambda_1+\lambda_2>
\frac{1}{2} , 
N_\Gamma(\pi_{\lambda_1}\hat\otimes\pi_{\lambda_2})\ne\{0\}\} $, if $\nu_1\in (-\frac{1}{4}, 0)$,
\item For $i=1,2$, $S^i_-:=\emptyset$, if $\Gamma$ is irreducible, and \\
$S^i_-:=\{\lambda^2-\frac{1}{4}\mid \lambda\in i[0,\infty)\cup(0,\frac{1}{2}), N_{\Gamma_i}(\pi_\lambda)\not=\{0\}\}$, if $\Gamma=\Gamma_1\times \Gamma_2$.
\end{itemize}
In particular, the spectral decomposition is not purely discrete.

For reducible $\Gamma$, we have that $S_-=S^1_-$ and that $S^i_-$ are countably infinite subsets of $(-\infty,0)$ without accumulation points in $(-\infty,0]$. For irreducible $\Gamma$
the subset $S_-\subset (-\infty,0)$ is also countably infinite and unbounded, but it may have accumulation points in $(-\infty,-\frac{5}{36}]$. The sets $\mathrm{sm}(\nu_1)\subset (-\frac{1}{4}, 0)$ are finite, and they are non-empty for at most finitely many $\nu_1$.

Moreover, all eigenspaces $E^{(2)}_{(\nu_1,\nu_2)}(Y)$ and multiplicity spaces $W_{(\nu_1,\nu_2)}$ appearing in (\ref{wums}) are infinite dimensional. 
The spectral decomposition is unique. All formally self-adjoint operators in ${\bf D}(G/H)$, in particular $\Omega_G$, $j_1(\Omega')$, and $j_3(\Omega')$, are essentially self-adjoint.

\begin{figure}[H]
\includegraphics[scale=0.5]{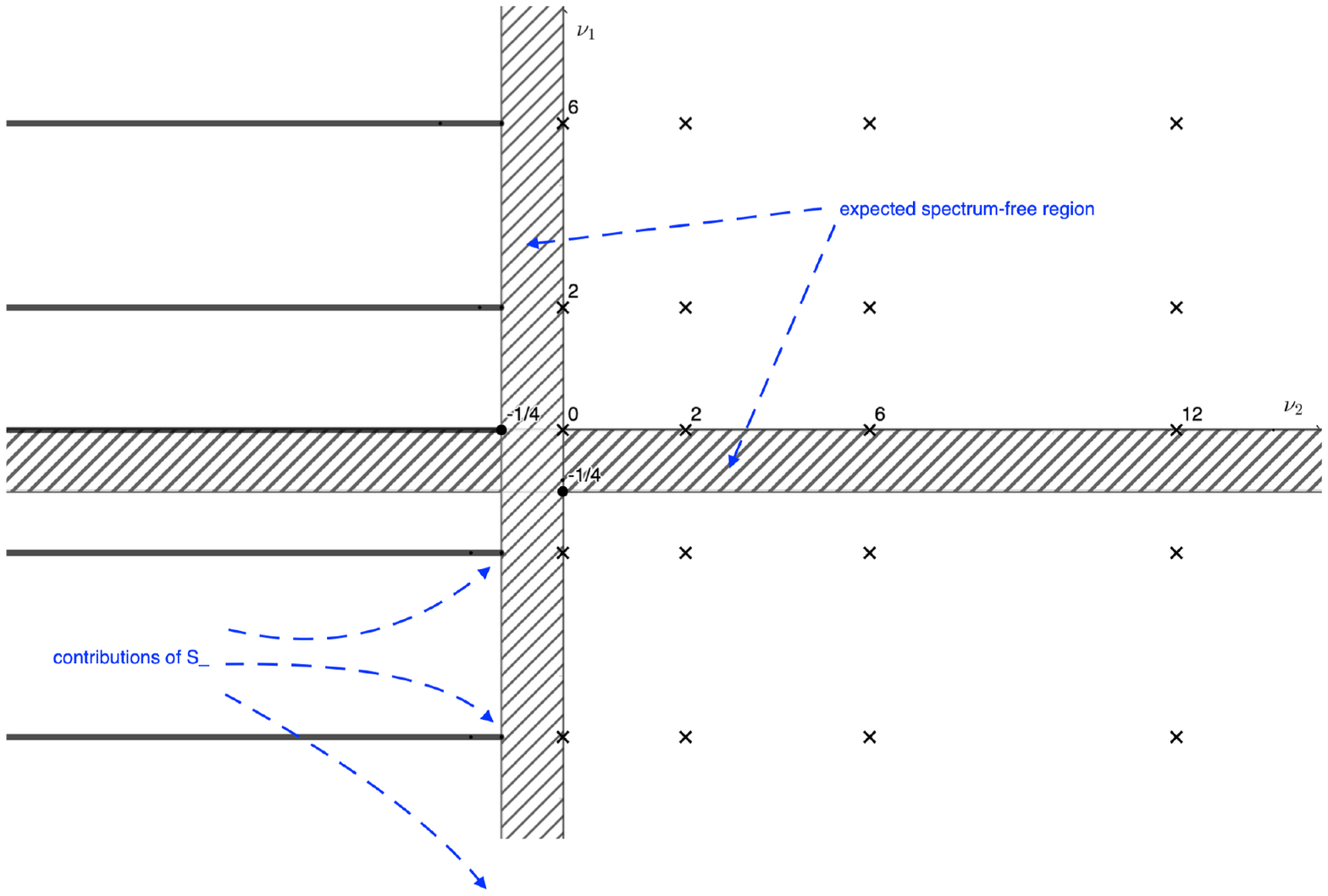}
\caption{Expected spectrum for irreducible $\Gamma$.}
\label{I1}
\end{figure}
\end{thm}

\begin{proof}
That $L^2(Y)$ has a spectral decomposition of the form (\ref{wums}) is a direct consequence of  Lemma \ref{bur}, the decomposition (\ref{bor}), and the branching rules (a)--(d). It remains to verify the more specific claims about the terms involved in the decomposition. 

For irreducible $\Gamma$, we have $N_\Gamma(\pi)=\{0\}$
for all non-trivial $\pi$ of the form (d) (where we allow to switch the factors). Indeed, in the irreducible case the projection of $\Gamma$ to either factor of $L$ is dense in $G'$ (see e.g.
\cite{Rag72}, Ch.~V). Thus a continuous function on $\Gamma\backslash L$ invariant under one factor is constant. We conclude that $S^i_-=\emptyset$ for irreducible $\Gamma$.

We remark that in the reducible case $N_\Gamma(\pi_1\hat\otimes\pi_2)\cong N_{\Gamma_1}(\pi_1)\otimes N_{\Gamma_2}(\pi_2)$ which implies that $S_-=S^1_-$.

That all $\nu_1\in S_+$ really contribute non-trivially to (\ref{wums}) in the way described follows from the fact that for all odd integers $l_1$, $l_2$ with $|l_1l_2|>1$ we have $N_\Gamma(\pi_{l_1}\hat\otimes\pi_{l_2})\ne\{0\}$. For $l_i\ne \pm 1$, $i=1,2$, this is a consequence of the integrability of the discrete series representation $\pi_{l_1}\hat\otimes\pi_{l_2}$, see the remark after Prop.~\ref{pavle}. As in the proof of that proposition, we will employ an index theoretic argument in the remaining cases. 
We consider the complex two-dimensional compact manifold $Z=\Gamma\backslash L/K_L$. It  is locally biholomorphic to the product of two upper half planes. This induces
a splitting of the holomorphic cotangent bundle into two line bundles $K_1$ and $K_2$. Let $k\in\N$ be odd. We consider the Dolbeaut cohomology groups $H^i(Z,K_1\otimes K_2^{\otimes\frac{k+1}{2}})$, $H^i(Z, K_2^{\otimes\frac{k+1}{2}})$ and the corresponding Euler characteristics
$$ \chi(Z,.)=\dim H^0(Z,.)-\dim H^1(Z,.)+\dim H^2(Z,.)\ .$$
One checks, e.g. using Hirzebruch's proportionality principle, that $\chi(Z,K_1\otimes K_2^{\otimes\frac{k+1}{2}})>0$ and $\chi(Z,K_2^{\otimes\frac{k+1}{2}})<0$. On the other hand, the cohomology groups can be expressed in terms of the
multiplicities $N_\Gamma(\pi)$ of those $\pi\in \widehat L$ having the same infinitesimal character as the discrete series representation $\pi_1\hat\otimes\pi_k$. For $k>1$ (i.e. non-trivial infinitesimal character) we obtain 
$$H^0(Z,K_1\otimes K_2^{\otimes\frac{k+1}{2}})\cong N_\Gamma(\pi_{-1}\hat\otimes\pi_{-k}),\  H^2(Z,K_1\otimes K_2^{\otimes\frac{k+1}{2}})=\{0\}, \   H^1(Z, K_2^{\otimes\frac{k+1}{2}})\cong N_\Gamma(\pi_{1}\hat\otimes\pi_{-k}).$$
We conclude $N_\Gamma(\pi_{\pm 1}\hat\otimes\pi_{\pm k})\ne\{0\}$.

The unboundedness of $S_-$ also in the irreducible case follows from the spectral asymptotics result \cite{DKV79}, Thm.~8.8. The restriction on the possible accumulation points of $S_-$
comes from a result of Kelmer and Sarnak (\cite{KSa09}, Thm.~2) on a strong spectral gap for $L^2(\Gamma\backslash L)$.

Let $\Delta\fg'$ be the diagonal subalgebra of $\fl\cong \fg'\oplus\fg'$. Then $j_1(\Omega')\in {\Cal Z}(\fl)$, and, by Lemma \ref{bur}, $j_3(\Omega')\equiv \Delta_{123}(\Omega')\in {\Cal Z}(\Delta\fg')$.
The essential self-adjointness of these two operators follows directly from Proposition \ref{spezi} with $V_\pi=L^2(\Gamma\backslash L)$ and $Q=\{e\}$. Now Cor.~\ref{karl}
implies that the spectral decomposition is unique.

Let $D\in {\bf D}(G/H)\cong \C[j_1(\Omega'),j_3(\Omega')]$ be an arbitrary formally self-adjoint element. Since $j_1(\Omega')\in {\Cal Z}(\fl)$ the operator $D$ acts on each
summand of the decomposition (\ref{bor}) by a polynomial in $j_3(\Omega')$, i.e. by an element of ${\Cal Z}(\Delta\fg')$. Thus by Proposition \ref{spezi} the operator $D$
is essentially self-adjoint on each summand. It follows (e.g. from Criterion (d) in Section \ref{G2}) that $D$ is essentially self-adjoint on the full space $L^2(\Gamma\backslash L)$.
\end{proof}

We conclude this section with a few remarks.

An important problem related to the spectral decomposition is to understand the joint eigenspaces
$$ E^{\infty}_\chi(Y)\subset E^{(2)}_\chi(Y)\subset E^{-\infty}_\chi(Y)\: ,\quad\chi\in\Hom({\bf D}(G/H),\C)$$
(for an abstract description in terms of $\widehat L$ see Lemma \ref{bassa}).
As a consequence of the essential self-adjointness of $j_1(\Omega')$ and $j_3(\Omega')$ all non-zero eigenspaces $E^{(2)}_\chi(Y)$ satisfy $\chi\in\Xi$ and appear in
the spectral decomposition (\ref{wums}) as a discrete summand. This gives a satisfying understanding of the $L^2$-eigenspaces. However, we have not made any claim on the size of their subspaces $E^{\infty}_\chi(Y)$ of smooth eigenfunctions.
A better understanding of them would require additional information on the behaviour of the branching rules (a), (b) with respect to $L$-smooth vectors. The determination of   $E^{-\infty}_\chi(Y)$ is easier.
Using the meromorphic continuation (see e.g. \cite{BSKZ14}, \cite{Mol17}) of invariant trilinear forms on principal series (which are responsible for the continuous contributions in (a),(b)) it is not difficult to see that for any $\nu\in \C$ the space of distribution vectors in the representations appearing in (a), (b) that are $\nu$-eigenvectors for the Casimir of the diagonal is infinite-dimensional.
We obtain for $(\nu_1,\nu_2)\in\C^2$:
$$ E^{-\infty}_{(\nu_1,\nu_2)}(Y)\ne \{0\} \Leftrightarrow \dim E^{-\infty}_{(\nu_1,\nu_2)}(Y)=\infty \Leftrightarrow \nu_1\in S_+\cup S_-\ .$$ 

By Margulis' arithmeticity theorem irreducible lattices $\Gamma\subset L$ are arithmetic. If some standard conjectures in the theory of automorphic forms were true, this would imply
that complementary series representations of $G'$ would not contribute to (\ref{bor}), see \cite{KSa09}, pp.~285--286 for more information. In other words, one expects that 
$S_-\subset (-\infty,-\frac{1}{4}]$,
and therefore also $sm(\nu_1)=\emptyset$ for all $\nu_1$ in the irreducible case, cf. Figure \ref{I1}. That there are at least no accumulation points of $S_-$ in the interval 
$(-\frac{5}{36},0]\subset (-\frac{1}{4},0]$, which comes from
the main result of \cite{KSa09}, could be considered as a provable part of this expectation. Note that there are known examples of reducible $\Gamma$ with  
$S_-\not\subset (-\infty,-\frac{1}{4}]$,  $sm(\nu_1)\not =\emptyset$ for some $\nu_1$. They come from Riemann surfaces with small eigenvalues. Concerning the latter the reader may
consult \cite{Buser}, Ch.~8, and \cite{BMM18} for more recent developments.

Theorem \ref{surprise} also implies a spectral decomposition of the pseudo-Riemannian Laplacian $\Delta_Y\equiv \Omega_G$ (which is essentially self-adjoint). We obtain that
$\spec(\Delta_Y)=\R$ and consists entirely of continuous spectrum with embedded eigenvalues.

If we replace $L=L_{min}$ by $L_{max}=G'\times\Delta_{23}(G')\times PSO(2)$, i.e. we allow for slightly more general $\Gamma\subset G$, then the same methods produce
a very similar spectral decomposition with the difference that some of the appearing eigenspaces and multiplicity spaces might be finite dimensional.


\section{A branching rule for spherical triples}\label{branching}

A classical problem in representation theory is to find the decomposition of  an irreducible unitary representation of $G$ into irreducibles when restricted to a subgroup $L$.
Of particular interest are those situations where this decomposition is discrete with finite multiplicities. In this case the $G$-representation is called $L$-admissible. The main result
of this section (Thm.~\ref{mainbranching}) says that $H$-spherical irreducible unitary representations of $G$ are $L$-admissible provided that $(G,H,L)$ is a {\em spherical} properly transitive triple.

It will turn out (see Prop.~\ref{horn} and Section~\ref{eigen}) that the study of restrictions to $L$ is helpful for the understanding of the spectral decomposition of
the locally symmetric space $Y=\Gamma\backslash G/H\cong \Gamma\backslash L/(L\cap H)$ in terms of representations of $G$. It is this application that forces us to consider also
non-unitary representations of $G$.   

Let $G/H$ be a semisimple symmetric space. As before, we assume $G$ to be connected with finite center. Let $W_\rho$ be an admissible representation of $G$ of finite length on a reflexive Banach space. The representation $W_\rho$ is called
$H$-spherical if there is a $G$-equivariant embedding
$$ W_{\rho,\infty}\hookrightarrow C^{\infty}(G/H)\ .$$
By Frobenius reciprocity (\ref{plum}) this is equivalent to the existence of an element
$$ \tilde w \in ( W_{\tilde\rho,-\infty})^H\mbox{ such that  }\tilde w_{|W_0}\ne 0 $$
for every non-zero closed $G$-invariant subspace $W_0\subset W_\infty$. Equivalently, the $G$-representation $W_{\tilde\rho,-\infty}$ has an $H$-invariant cyclic vector. For $W_\rho$ irreducible the condition is simply $$ (W_{\tilde\rho,-\infty})^H\ne\{0\}\ .$$
Note that by Casselman-Wallach theory \cite{Cas89b}, \cite{Wa92}, Ch.~11, see also \cite{BK14}, the $G$-representations $W_{\rho,\pm\infty}$ only depend on the underlying Harish-Chandra module $W_{\rho,K}$ and not
on the particular globalization $W_\rho$. Thus the same is also true for the property of being $H$-spherical. Moreover, it is known that $\dim (W_{\tilde\rho,-\infty})^H<\infty$ \cite{Ban87}.

Situations where such $H$-invariant distribution vectors are smooth or even $K$-finite are very special. For instance we have

\begin{lem}\label{almosttrivial}
Let $\ff\subset \fg$ be a 
subalgebra such that its projection to every non-compact simple ideal $\fd$ of $\fg$ is not contained in $\fd\cap \fk$. Let $W$ be a $(\fg,K)$-module, and let $w\in W$ be an
${\Cal U}(\ff)$-finite element. Then $w$ is ${\Cal U}(\fg)$-finite.
\end{lem}
\begin{proof}
Note that $K$ is connected. The irreducible components of the $K$-representation on $\fg/\fk\cong\fs$
are in one-to-one correspondence to the non-compact simple ideals of $\fg$. Our assumption on $\ff$ now implies that the image of $\ff$ in $\fg/\fk$ is not contained
in any proper $K$-submodule of $\fg/\fk$. It follows that there exist $k_1, k_2,\dots, k_r\in K$ such that for $\ff_i:= \Ad(k_i)\ff$ we have 
$\fg=\ff+\ff_1+\dots+\ff_r+\fk$ and hence, by Poincar\' e-Birkhoff-Witt,
\be\label{ralf} {\Cal U}(\fg)={\Cal U}(\ff){\Cal U}(\ff_1)\dots{\Cal U}(\ff_r){\Cal U}(\fk)\ .\end{equation}
We consider the subspace $W_\ff\subset W$ of all ${\Cal U}(\ff)$-finite elements of $W$.  It is an ${\Cal U}(\fg)$-submodule. Since $K$ is connected  $W_\ff\subset W$  is  a $(\fg,K)$-submodule. It follows
that all elements of $W_\ff$  are ${\Cal U}(\ff_i)$-finite for all $i$. Equation (\ref{ralf}) now implies that all elements of $W_\ff$ are ${\Cal U}(\fg)$-finite.
\end{proof}

\begin{cor}\label{frank} Assume that $G/H$ has no non-compact Riemannian factors. Let $W_\rho$ be an $H$-spherical admissible representation of finite length on a reflexive Banach space. Then the following assertions are equivalent:
\begin{itemize}
\item[(i)] $W_\rho$ is finite dimensional.
\item[(ii)]$(W_{\rho,K})^Q$ is finite dimensional for some subgroup $Q$ of $K\cap H$.
\item[(iii)] $(W_{\tilde\rho,-\infty})^H\subset W_{\tilde\rho,K}$.
\end{itemize}
\end{cor}

\begin{proof} The implication $(i)\Rightarrow (ii)$ is trivial. We show $(ii)\Rightarrow (iii)\Rightarrow (i)$. 

 Let $Q\subset L\cap H$. We assume that $(W_{\rho,K})^Q$ is finite dimensional. By duality and compactness of $Q$ we conclude that also $(W_{\tilde\rho,K})^{Q}$ is finite dimensional. By density it follows
that $(W_{\tilde\rho,K})^{Q}=(W_{\tilde\rho,-\infty})^{Q}\supset (W_{\tilde\rho,-\infty})^{H}$. We conclude $(W_{\tilde\rho,-\infty})^{H}\subset W_{\tilde\rho,K}$.

Now we assume $(W_{\tilde\rho,-\infty})^H\subset W_{\tilde\rho,K}$. 
Let $w\in (W_{\tilde\rho,-\infty})^H$ be $G$-cyclic for $W_{\tilde\rho,-\infty}$, see the discussion above. Since $K$ is connected this implies ${\Cal U}(\fg)w=W_{\tilde\rho,K}$. The absence of non-compact Riemannian factors of $G/H$ ensures that $\ff:=\fh$ satisfies the assumption of Lemma~\ref{almosttrivial}.  The lemma  
now tells us that $W_{\tilde\rho,K}$  is finite dimensional. Hence $W_\rho$ is finite dimensional as well.
\end{proof}

Subrepresentations of $H$-spherical representations are again $H$-spherical. This is not true for quotients of $H$-spherical representations. We call such quotients weakly
$H$-spherical.

Now we look at spherical properly transitive triples $(G,H,L)$.
For a subgroup $N\subset G$ we denote the space of $N$-smooth and $N$-distribution vectors of $W_\rho$, respectively, by $W_{\rho,\pm\infty_N}$.

\begin{pro}\label{hering}
Let $(G,H,L)$ be a spherical triple, and let $W_\rho$ be a weakly $H$-spherical admissible representation of finite length on a reflexive Banach space. Then
\begin{enumerate}

\item[(i)] ($K_L$-admissibility) For every irreducible $K_L$-representation $V_\gamma$ we have $\dim \Hom_{K_L}(V_\gamma,W_\rho)<\infty$. In particular, for any globalization $W$
of  the Harish-Chandra module $W_{\rho,K}$, i.e. for any continuous $G$-representation  on a Hausdorff locally convex vector space $W$ with $W_K=W_{\rho,K}$, we have
$W_{K}=W_{K_L}$.
\item[(ii)] There is an increasing filtration $(V_i)_{i\in\N_0}$  of $W_{\rho,K}$ by Harish-Chandra modules for $L$
such that $V_{i+1}/V_i$ is irreducible or zero. The multiplicity of any irreducible $(\fl,K_L)$-module as a subquotient of $W_{\rho,K}$ is finite.
\item[(iii)] The inclusions $W_{\rho,\infty}\hookrightarrow W_{\rho,\infty_L}\hookrightarrow W_{\rho,\infty_{K_L}}$ and $W_{\rho,-\infty_{K_L}}\hookrightarrow W_{\rho,-\infty_L}\hookrightarrow W_{\rho,-\infty}$ are topological
isomorphisms.
\item[(iv)] Assume that $W_\rho$ is $H$-spherical. Then a filtration as in (ii) stablizes, i.e. $W_{\rho,K}$ has finite length as $(\fl,K_L)$-module, if and only if $W_\rho$ is finite dimensional.
\end{enumerate}
\end{pro}

For the proof of the proposition, the following lemma which has its own interest is crucial. 

\begin{lem}\label{sigmatheta}
Let $(G,H,L)$ be a spherical triple. Let $P_{\sigma\gt}$ be a minimal $\sigma\gt$-stable parabolic subgroup in $G$. 
Then $K_L$ acts transitively on $G/P_{\sigma\gt}$.
\end{lem}
\begin{proof}
The group $K_L$  acts transitively on the set 
\beu
{\mathcal F}_{{\mathfrak l}\cap{\mathfrak s}}:=\{{\mathfrak b}\subset{\mathfrak l}\cap{\mathfrak s}\;\text{ maximal abelian subspace}\}.
\end{equation*}
Similarly, if one considers the Lie algebra $({\mathfrak h}\cap{\mathfrak k})\oplus({\mathfrak q}\cap{\mathfrak s})$, one sees that the group $K_H$ acts transitively on the set
\beu
{\mathcal F}_{{\mathfrak q}\cap{\mathfrak s}}:=\{{\mathfrak b}\subset{\mathfrak q}\cap{\mathfrak s}\;\text{ maximal abelian subspace}\}.
\end{equation*}
Since $(G,H,L)$ is a spherical triple and ${\mathcal F}_{{\mathfrak l}\cap{\mathfrak s}}$ is $K_L$-equivariantly finitely-covered by $L/P_L\simeq K_L/M_L$, already the smaller group $L\cap H$ acts transitively on ${\mathcal F}_{{\mathfrak l}\cap{\mathfrak s}}$. We claim that $L\cap H$ acts transitively on ${\mathcal F}_{{\mathfrak q}\cap{\mathfrak s}}$ as well. Indeed, consider the $L\cap H$-equivariant isomorphism (\ref{hermann}): 
\beu
p_{{\mathfrak q}}:{\mathfrak l}\cap{\mathfrak s}\rightarrow{\mathfrak q}\cap{\mathfrak s},\;X\mapsto\frac{1}{2}(X-\sigma(X)).
\end{equation*}
In the proof of Proposition \ref{ranksymm}, we have seen that there exists a maximal abelian subspace ${\mathfrak a}_{{\mathfrak l}\cap{\mathfrak s}}$ of ${\mathfrak l}\cap{\mathfrak s}$ such that $p_{{\mathfrak q}}({\mathfrak a}_{{\mathfrak l}\cap{\mathfrak s}})$ is a maximal abelian subspace of ${\mathfrak q}\cap{\mathfrak s}$. By $L\cap H$-equivariance and transitivity of the $L\cap H$-action, this assertion is actually true for every maximal abelian subspace ${\mathfrak b}$ of ${\mathfrak l}\cap{\mathfrak s}$. In particular, we obtain a well-defined injective $L\cap H$-equivariant map ${\mathcal F}_{{\mathfrak l}\cap{\mathfrak s}}\rightarrow{\mathcal F}_{{\mathfrak q}\cap{\mathfrak s}}$. But we have:
\begin{align*}
\text{dim }{\mathcal F}_{{\mathfrak l}\cap{\mathfrak s}}&=\text{dim }{\mathfrak l}\cap{\mathfrak s}-\text{rk}_{\mathbb R}L\;\text{ by Bruhat and Iwasawa decompositions}\\
&=\text{dim }{\mathfrak q}\cap{\mathfrak s}-\text{rk}_{\mathbb R}L\;\text{ by (\ref{hermann})}\\
&=\text{dim }{\mathfrak q}\cap{\mathfrak s}-\text{rk}_{\mathbb R}(G/H)\;\text{ by Proposition \ref{ranksymm}}\\
&=\text{dim }{\mathcal F}_{{\mathfrak q}\cap{\mathfrak s}}\;\text{ by Bruhat and Iwasawa decompositions,}
\end{align*}
i.e., the map ${\mathcal F}_{{\mathfrak l}\cap{\mathfrak s}}\rightarrow{\mathcal F}_{{\mathfrak q}\cap{\mathfrak s}}$ is also surjective. This implies that ${\mathcal F}_{{\mathfrak l}\cap{\mathfrak s}}$ and ${\mathcal F}_{{\mathfrak q}\cap{\mathfrak s}}$ are $L\cap H$-equivariantly isomorphic, and hence $L\cap H$ acts transitively on ${\mathcal F}_{{\mathfrak q}\cap{\mathfrak s}}$. Using
\beu
{\mathcal F}_{{\mathfrak q}\cap{\mathfrak s}}\simeq K_H/N_{K_H}({\mathfrak a}_{{\mathfrak q}\cap{\mathfrak s}})
\end{equation*}
we obtain that
\beu
K_H=(L\cap H)N_{K_H}({\mathfrak a}_{{\mathfrak q}\cap{\mathfrak s}}).
\end{equation*}
Finally, we deduce the following:
\begin{align*}
K=K_L K_H&=K_L N_{K_H}({\mathfrak a}_{{\mathfrak q}\cap{\mathfrak s}})\\
&=K_L Z_{K_H}({\mathfrak a}_{{\mathfrak q}\cap{\mathfrak s}})\;\text{ since $Z_{K_H}({\mathfrak a}_{{\mathfrak q}\cap{\mathfrak s}})$ has finite index in }N_{K_H}({\mathfrak a}_{{\mathfrak q}\cap{\mathfrak s}})\\
&=K_L Z_{K}({\mathfrak a}_{{\mathfrak q}\cap{\mathfrak s}}).
\end{align*}
Since $K\cap P_{\sigma\theta}=Z_{K}({\mathfrak a}_{{\mathfrak q}\cap{\mathfrak s}})$, we conclude that $K_L$ acts transitively on $K/K\cap P_{\sigma\theta}\simeq G/P_{\sigma\theta}$.
\end{proof}
For later reference we also observe
\begin{cor}\label{renoir}
Under the assumptions of Lemma~\ref{sigmatheta} there is a minimal parabolic $P_L\subset L$ such that $P_{\sigma\gt}\cap L\subset P_L$.  In particular, there is an $L$-equivariant fibration
$G/P_{\sigma\gt}\rightarrow L/P_L$. We can even be more precise: There is a minimal parabolic $P_L=M_LA_LN_L$ such that $P_{\sigma\gt}\cap L= Z_{K_L}({\mathfrak a}_{{\mathfrak q}\cap{\mathfrak s}})A_LN_L$ and $Z_{K_L}({\mathfrak a}_{{\mathfrak q}\cap{\mathfrak s}})\subset M_L$.
\end{cor}
\begin{proof} The corollary is a direct consequence of Lemma~\ref{sigmatheta} and the following

{\it Claim: Let $(G,H,L)$ be a properly transitive triple with $L$ stable under a Cartan involution $\theta$ commuting with $\sigma$. Then there exist a minimal $\sigma\theta$-stable parabolic $P_{\sigma\gt}$ of $G$ and a minimal parabolic $P_L=M_LA_LN_L$ of $L$ such that $P_{\sigma\gt}\cap L= Z_{K_L}({\mathfrak a}_{{\mathfrak q}\cap{\mathfrak s}})A_LN_L$ and $Z_{K_L}({\mathfrak a}_{{\mathfrak q}\cap{\mathfrak s}})\subset M_L$.}\\
We will prove the claim by a refinement of the arguments of the proof of Prop.~\ref{ranksymm}. Let $\fa\subset\fs$ be an arbitrary maximal abelian subspace. We say that a system of positive roots $\Phi^+\subset\Phi(\fg,\fa)$ and a subspace $\fa_0\subset\fa$ are compatible if the following implication holds:
$$ \alpha,\beta\in\Phi^+, \alpha_{|\fa_0}=-\beta_{|\fa_0}\ \Rightarrow\  \alpha_{|\fa_0}=\beta_{|\fa_0}=0\ .$$
Note that for every $\fa_0\subset\fa$ there exists a compatible $\Phi^+$ (consider a lexicographic order taking $\fa_0$ first). Now let $\fb_\fq\subset \fq\cap\fs$ be maximal abelian.
We extend $\fb_\fq$ to a $\sigma$-stable maximal abelian subspace $\fa_1\subset\fs$ and choose a system of positive roots $\Phi_1^+\subset\Phi(\fg,\fa_1)$ compatible with $\fb_\fq$.
Similarly, we extend a maximal abelian subspace $\fb_L\subset\fl\cap\fs$ to a maximal abelian $\fa_2\subset \fs$ and choose a system of positive roots $\Phi_2^+\subset\Phi(\fg,\fa_2)$
compatible with $\fb_L$. There exists a $k\in K$ conjugating the pair $(\fa_1,\Phi_1^+)$ to $(\fa_2,\Phi_2^+)$. We have $k=lh$ for some $l\in K_L$, $h\in K_H$.
Now we consider the pair $(\fa,\Phi^+):=\Ad(h)(\fa_1,\Phi_1^+)=\Ad(l^{-1})(\fa_2,\Phi_2^+)$ and the subspaces $\fa_{\fq\cap\fs}:=\fa\cap\fq=\Ad(h)\fb_\fq$, 
$\fa_L:=\fa\cap\fl=\Ad(l^{-1})\fb_L$. Then $\fa$ is a $\sigma$-stable maximal abelian subspace of $\fs$, and $\Phi^+$ is compatible with $\fa_{\fq\cap\fs}$ {\it and} $\fa_L$.

The pair $(\fa,\Phi^+)$ defines a minimal parabolic $P\subset G$ and, together with $\theta$, an Iwasawa decomposition $G=KAN$. Since $\Phi^+$ is compatible with $\fa_L$
the set 
$$ \{\alpha_{|\fa_L}\mid \alpha\in \Phi^+\}\cap \Phi(\fl, \fa_L)$$
is a positive root system for $\fl$. It defines a minimal parabolic $P_L=M_LA_LN_L$ of $L$ satisfying $A_LN_L\subset P$. The resulting Iwasawa decomposition for $L$ is compatible
with the one for $G$: Let $g=kan\in G$. Then
$$  g\in L\ \Leftrightarrow k\in K_L, a\in A_L, n\in N_L\ .$$
Now let $\displaystyle N_{\sigma\gt}:=\exp\left ( \bigoplus_{\{\alpha\in\Phi^+\mid \alpha_{|\fa_{\fq\cap\fs}\ne 0}\}} \fg_\alpha\right )$. Since $\Phi^+$ is compatible with $\fa_{\fq\cap\fs}$ this is really a subgroup of $G$, and
$$ P_{\sigma\gt}:=Z_G(\fa_{\fq\cap\fs})\cdot N_{\sigma\gt}$$
is a minimal $\sigma\theta$-stable parabolic of $G$ containing $P$.

Now let $g=kan\in P_{\sigma\gt}\cap L$. Then $an\in A_LN_L$. Since $A_LN_L\subset P\subset P_{\sigma\gt}$ we conclude that $k\in P_{\sigma\gt}\cap L\cap K= Z_{K_L}(\fa_{\fq\cap\fs})$. In order to finish the proof of the claim it remains to see that $Z_{K_L}({\mathfrak a}_{{\mathfrak q}\cap{\mathfrak s}})\subset M_L(=Z_{K_L}({\mathfrak a}_L))$.
But this follows from the fact that the ($K_L$-equivariant) orthogonal projection $\fs\rightarrow\fl\cap\fs$ maps $\fa_{\fq\cap\fs}$ onto $\fa_L$, see the proof of Prop.~\ref{ranksymm}.
\end{proof}
{\it Proof of Proposition~\ref{hering}.} Let us first assume that $W_\rho$ is $H$-spherical. Delorme's embedding theorem (see Section 4 in \cite{Del85}) provides a non-zero 
$(\fg,K)$-intertwining operator
\beu
\Phi_1: W_{\rho,K}\longrightarrow\big(\text{Ind}_{P_{\sigma\gt}}^{G}\delta_1\big)_{K}
\end{equation*}
to a principal series representation induced from a finite dimensional representation $(\delta_1,E_{\delta_1})$ of $P_{\sigma\gt}$ (of a very special form). By Casselman-Wallach
theory (see e.g. \cite{Wa92}, Thm.~11.6.7)  it extends to a $G$-intertwining operator  between smooth and distribution vectors, respectively,
\beu
\Phi_1: W_{\rho,\pm\infty}\longrightarrow\big(\text{Ind}_{P_{\sigma\gt}}^{G}\delta_1\big)_{\pm\infty}
\end{equation*}
with closed range that  is a topological direct summand of $\big(\text{Ind}_{P_{\sigma\gt}}^{G}\delta_1\big)_{\pm\infty}$.
This implies the existence of a finite dimensional $P_{\sigma\gt}$-representation  $(\delta,E_{\delta})$ and a topological  embedding (i.e. a topological isomorphism onto its image)
\beu
\Phi: W_{\rho,\pm\infty}\hookrightarrow\big(\text{Ind}_{P_{\sigma\gt}}^{G}\delta\big)_{\pm\infty}\cong\big(\text{Ind}_{K\cap P_{\sigma\gt}}^{K}\delta\big)_{\pm\infty}\ ,
\end{equation*}
which is at least $K$-equivariant. Indeed, if $W_\rho$ is irreducible, then we can take $\Phi=\Phi_1$. In general, we proceed by induction on the length of $W_\rho$. As a subrepresentation
of $W_{\rho,\pm\infty}$ the kernel $W_0$ of $\Phi_1$ is also $H$-spherical. By induction hypothesis $W_0$ embeds topologically and $K$-equivariantly into an induced representation $\big(\text{Ind}_{P_{\sigma\gt}}^{G}\delta_2\big)_{\pm\infty}$. There exists a $K$-invariant closed subspace $W_1$ of $W_{\rho,\pm\infty}$ such that $W_{\rho,\pm\infty}=W_0\oplus W_1$ (topologically).
It follows that $W_{\rho,\pm\infty}$ embeds into $\big(\text{Ind}_{P_{\sigma\gt}}^{G}(\delta_2\oplus\delta_1)\big)_{\pm\infty}$. The existence of $W_1$, say for smooth vectors,  is established as follows: Since $W_{\rho,\infty}$
only depends on its underlying Harish-Chandra module we can assume that $W_\rho$ is a Hilbert space on which $K$ acts unitarily. Moreover, again by Casselman-Wallach theory,
the $G$-smooth and $K$-smooth vectors of $W_\rho$ coincide. Thus we can take $W_1$ as the $K$-smooth vectors in the orthogonal complement of $W_0$.

By the above discussion and the coincidence of $K$- and $G$-smooth vectors we see that a vector $w\in W_\rho\subset W_{\rho,-\infty}$ is smooth if and only if
$\Phi(w)\in \big(\text{Ind}_{K\cap P_{\sigma\gt}}^{K}\delta\big)_{\infty}$. By the Sobolev embedding theorem and the compactness of $K/K\cap P_{\sigma\gt}$ the latter space coincides with the space smooth sections
of the corresponding $K$-homogeneous vector bundle over $K/K\cap P_{\sigma\gt}\cong G/ P_{\sigma\gt}$ and therefore also with the space of smooth vectors for any subgroup $C\subset K$ acting transitively on $K/K\cap P_{\sigma\gt}$, see also \cite{BW80} (Chapter III, Theorem 7.5). Now Lemma~\ref{sigmatheta} implies the part of Assertion (iii) concerning smooth vectors.
Analogously, we find that $w\in W_{\rho,-\infty}$ is a distribution vector for the $K_L$-action if $\Phi(w)$ is a distribution section of the above vector bundle which holds for all
$w\in W_{\rho,-\infty}$. This finishes the proof of Assertion (iii).

Again by Lemma \ref{sigmatheta}, 
we get the following isomorphims of $K_L$-modules:
$$
\text{Ind}_{K\cap P_{\sigma\gt}}^{K}\delta
\simeq\text{Ind}_{K_L\cap P_{\sigma\gt}}^{K_L}\delta\ .
$$
In particular, if $V_\gamma$ is an irreducible representation of $K_L$, we obtain by the above discussion and Frobenius reciprocity that:
\beu
\text{Hom}_{K_L}(V_\gamma,W_\rho)\hookrightarrow\text{Hom}_{K_L}(V_\gamma,\text{Ind}_{K_L\cap P_{\sigma\gt}}^{K_L}\delta)\simeq\text{Hom}_{K_L\cap P_{\sigma\gt}}(V_\gamma,E_\delta)\ .
\end{equation*}
The latter space is finite dimensional. The first part of Assertion (i) follows. It says that $W_{\rho,K}$ and all its $(\fl, K_L)$-submodules are admissible $(\fl, K_L)$-modules. In particular, all
finitely generated $(\fl, K_L)$-submodules of $W_{\rho,K}$ are Harish-Chandra modules for $L$. Since the dimension of $W_{\rho,K}$ is at most countable we conclude that also Assertion~(ii) holds. Concerning the second part of  Assertion (i) we first remark that the $K_L$-isotypic components of $W_{\rho, K}=W_K$ are dense in the ones of $W$.
By the already established finite dimensionality they coincide. Thus the  $K_L$-isotypic components of $W$ are contained in $W_K$. Taking the algebraic direct sum of these components
we find $W_{K_L}\subset W_K$. Equality follows.

We now show the non-trivial direction of (iv). Finite length of $W_{\rho,K}$ together with Cor.~\ref{inv3} and the equivalence (iii)$\Leftrightarrow$(iv) in Prop.~\ref{translem} tells us that
$(W_{\rho,K})^{L\cap H}$ is finite dimensional. We now apply Cor.~\ref{frank}.
This finishes the proof of the proposition for $H$-spherical representations $W_\rho$.

Now let $W_\rho$ be only weakly spherical. We choose an admissible $H$-spherical representation $W_{\rho_1}$ that surjects onto $W_\rho$. It is now easy to see
that $W_\rho$ inherits Properties (i) and (ii) from $W_{\rho_1}$. To establish Assertion (iii) for $W_\rho$  we use the coincidence of $G$- and $K$-smooth vectors and that $W_{\rho,\pm\infty}$ can be realized as a closed $K$-submodule of $W_{\rho_1,\pm\infty}$, see the discussion above.
$\hfill\Box$


We also need some information on non-spherical triples $(G,H,L)$ of Type~I. It is easy to  see that Assertions (i) and (iii) of Proposition~\ref{hering} are no longer true, in general.

\begin{cor}\label{makrele}
Let $(G,H,L)$ be a triple of Type~I, and let $W_\rho$ be a weakly $H$-spherical admissible representation of finite length on a reflexive Banach space. Then
there is an increasing filtration $(V_i)_{i\in\N_0}$  of $W_{\rho,K}$ by Harish-Chandra modules for $L$
such that $V_{i+1}/V_i$ is irreducible or zero. If $W_\rho$ is spherical and $W_{\rho,K}$ has finite length as $(\fl,K_L)$-module, then $W_\rho$ is finite dimensional.
\end{cor}

\begin{proof} We apply Prop.~\ref{hering} to the spherical triple $(G,H,L_{max})$ and obtain an increasing filtration $(V^+_j)_{j\in\N_0}$  of $W_{\rho,K}$ by Harish-Chandra modules for $L_{max}$
such that $V^+_{j+1}/V^+_j$ is irreducible or zero. Since $L_{max}=L_{min}\cdot C_{max}$ and $L=L_{min}\cdot C$ for some closed subgroups $C\subset C_{max}\subset Z_K(L_{min})$ every
irreducible Harish-Chandra module for $L_{max}$ is a finite direct sum of irreducible Harish-Chandra modules for $L$. Thus we can refine the filtration $(V_j^+)$ in order to obtain the
desired filtration $(V_i)$.
\end{proof}

Let us discuss {\em multiplicities} for branching problems in a quite general context. Let $L\subset G$ be strongly reductive.
Let $V_\pi$ be an irreducible admissible $L$-representation, and let $W_\rho$ be an admissible $G$-representation of finite length (both on reflexive Banach spaces). We want to describe the multiplicity of ${V_\pi}$ as a subrepresentation of $W_\rho$ (which might be smaller than the multiplicity as a subquotient) by the dimension of certain spaces of intertwining operators.
In the situation of Prop.~\ref{hering} and Cor.~\ref{makrele} the natural space is 
$$ \Hom_{(\fl,K_L)}(V_{\pi,K_L},W_{\rho,K})\ .$$
We would like to relate this space to spaces of intertwining operators between certain {\em group} representations {\em canonically} attached to the Harish-Chandra modules $V_{\pi,K_L}$ and $W_{\rho,K}$, namely
$$  \Hom_{L}(V_{\pi,\infty},W_{\rho,\infty}),\  \Hom_{L}(V_{\pi,-\infty},W_{\rho,-\infty}), $$
and, if $V_\pi$ and $W_\rho$ are unitary, also
$$ \Hom_{L}(V_{\pi},W_{\rho})\ .$$
These spaces carry the topology of uniform convergence on bounded subsets. For the latter space this topology comes from a Hilbert space structure given by 
\be\label{cezanne} \langle A, B\rangle\cdot\id_{V_\pi}\defn  B^*A,\quad A,B\in  \Hom_{L}(V_{\pi},W_{\rho})\ .
\end{equation}
Moreover, in the unitary case we have a direct integral decomposition (see e.g. \cite{Di} or \cite{Wa92}, Ch.~14)
\begin{equation}\label{lachs}  W_\rho\cong \int^\oplus_{\widehat L} M_\rho(\pi)\hat\otimes V_\pi\:  d\mu_\rho(\pi)\ ,\end{equation}
where the multiplicity space $M_\rho(\pi)$ is a Hilbert space uniquely determined up to isomorphism for almost all  $\pi\in \widehat L$ (w.r.t. $\mu_\rho$).

\begin{lem}\label{sardine}
Under the general assumptions of the previous paragraph we have canonical inclusions
$$
 \Hom_{(\fl,K_L)}(V_{\pi,K_L},W_{\rho,K})\subset  \Hom_{L}(V_{\pi,\infty},W_{\rho,\infty})\subset\  \Hom_{L}(V_{\pi,-\infty},W_{\rho,-\infty}).
$$
Moreover, if $V_\pi$ and $W_\rho$ are unitary, we have:
$$ \Hom_{L}(V_{\pi,\infty},W_{\rho,\infty})\subset\Hom_{L}(V_{\pi},W_{\rho})\stackrel{(*)}{\subset}M_\rho(\pi)\subset  \Hom_{L}(V_{\pi,-\infty},W_{\rho,-\infty})\ .$$
If $\Hom_{L}(V_{\pi},W_{\rho})\ne\{0\}$, then $(*)$ is an equality. Under the assumptions of Cor.~\ref{makrele} all inclusions have dense images, while they are equalities
of finite dimensional vector spaces under the assumptions of Prop.~\ref{hering}.
\end{lem}
\begin{proof}
Let $\Phi\in   \Hom_{(\fl,K_L)}(V_{\pi,K_L},W_{\rho,K})$. Let $W_0$ be the closure of the image of $\Phi$ in $W_{\rho,\infty}$. Using $L$-analyticity we see that $W_0$ is $L$-invariant.
Since  $W_{\rho,\infty}$ is a smooth $G$-representation of moderate growth in the sense of Casselman-Wallach (\cite{Wa92}, Ch. 11) we find that $W_0$ is a smooth $L$-representation of moderate growth and finite length. Casselman-Wallach globalization theory implies that $\Phi$ extends to an element $\tilde\Phi\in   \Hom_{L}(V_{\pi,\infty},W_0)\subset
 \Hom_{L}(V_{\pi,\infty},W_{\rho,\infty})$. This provides the first inclusion. Concerning the second we consider the space of $V_{\pi,N}$ of $C^N$-vectors in $V_\pi$ with its natural Banach space topology. It is again a continuous $L$-representation on a reflexive Banach space. Let $\Phi\in\Hom_{L}(V_{\pi,\infty},W_{\rho,\infty})$. Continuity
of $\Phi$ implies that it extends to an element $\tilde \Phi\in  \Hom_{L}(V_{\pi,N},W_{\rho})$ for $N$ sufficiently large. Taking $L$-distribution vectors we obtain a further extension
$\overline \Phi\in \Hom_{L}((V_{\pi,N})_{-\infty},W_{\rho,-\infty_L})$. We have $(V_{\pi,N})_{-\infty}=V_{\pi,-\infty}$ (which follows
from Casselman-Wallach but can be shown also in an elementary way)  and $W_{\rho,-\infty_L}\subset W_{\rho,-\infty}$. Thus we can view $\overline\Phi$ as an element of $\Hom_{L}(V_{\pi,-\infty},W_{\rho,-\infty})$.

Moreover, going over to $K_L$-finite elements, we get a further inclusion $\Hom_{L}(V_{\pi,-\infty},W_{\rho,-\infty})\subset  \Hom_{(\fl,K_L)}(V_{\pi,K_L},(W_{\rho,-\infty})_{K_L})$.
Under the assumptions of Prop.~\ref{hering} the latter space is equal to $\Hom_{(\fl,K_L)}(V_{\pi,K_L},W_{\rho,K})$, see Prop.~\ref{hering} (i). It follows that all inclusions
are equalities (of finite dimensional vector spaces) in this case. We assume now that the assumptions of Cor.~\ref{makrele} hold. We consider the inclusions
$$L_1:=L_{min}\subset L\subset L_{max}=:L_2\ .$$
We have $L=L_1 \cdot C$, $L_2=L_1\cdot C_2$ for some $C\subset C_2\subset Z_{K}(L_1)$. Irreducible admissible representations of $L$ are of the form
$V_\pi=V_{\pi_1}\otimes F_\tau$, cf. the proof of Prop.~\ref{staun}. We obtain
$$ \Hom_{L}(V_{\pi,\infty},W_{\rho,\infty})\cong \Hom_C(F_\tau, \Hom_{L_1}(V_{\pi_1,\infty},W_{\rho,\infty}))\ .$$
The analogous isomorphisms hold for the other multiplicity spaces occuring in the lemma. This shows that it suffices to establish density of the images of the inclusions for $L=L_1$.
Since $(G,H,L_2)$ is spherical we already know that all multiplicity spaces for $L_2$ coincide.
We consider the multiplicity spaces for $L_1$ as $C_2$-modules. The above isomorphisms for $L=L_2$ and $\tau$ running over all of $\hat C_2$ now imply that the $C_2$-finite elements of all multiplicity spaces for $L_1$ coincide.
Density follows.

Now let  $V_\pi$ and $W_\rho$ be unitary. By uniqueness of an invariant Hermitian form on $V_{\pi,\infty}$ any element in $\Hom_{L}(V_{\pi,\infty},W_{\rho,\infty})$
extends continuously to an element in $\Hom_{L}(V_{\pi},W_{\rho})$.  The space $\Hom_{L}(V_{\pi},W_{\rho})$ describes the multiplicity of $V_\pi$ as a subrepresentation
of $W_\rho$. In particular, if it is non-zero, then $\mu_\rho(\{\pi\})\ne 0$, and $M_\rho(\pi)$ has to be equal to this space. It remains to discuss the last
inclusion $M_\rho(\pi)\subset\Hom_L(V_{\pi,-\infty},W_{\rho,-\infty})$. By the theorem of Gelfand-Kostyuchenko already discussed in Remark~\ref{bernstein} the isomorphism
(\ref{lachs}) can be pointwise defined on $W_{\rho,\infty}$. We obtain an element ${\Cal F}_\pi\in \Hom_L(W_{\rho,\infty},  M_\rho(\pi)\hat\otimes V_\pi)$ with dense range. For $m\in M_\rho(\pi)$
we define an element $i(m)\in \Hom_L(W_{\rho,\infty},V_\pi)= \Hom_L(W_{\rho,\infty},V_{\pi,\infty})$ by 
$$ \langle i(m)w,v\rangle_{V_\pi}:= \langle {\Cal F}_\pi (w), m\otimes v\rangle_{M_\rho(\pi)\hat\otimes V_\pi}\ , w\in W_{\rho,\infty}, v\in V_\pi .$$
Let $j(m)\in  \Hom_{L}(V_{\pi,-\infty},W_{\rho,-\infty})$ be the adjoint of $i(m)\in \Hom_L(W_{\rho,\infty},V_{\pi,\infty})$. Then 
$j: M_\rho(\pi)\rightarrow \Hom_{L}(V_{\pi,-\infty},W_{\rho,-\infty})$ is the desired inclusion.
\end{proof}

We return to the discussion of Type~I triples.

\begin{Def}\label{spass}
Let $(G,H,L)$ be a triple of Type~I. Let $V_\pi$ be an irreducible admissible representation of $L$, and let $W_\rho$ be a weakly $H$-spherical admissible representation of $G$ of finite length (both on reflexive Banach spaces). By $m_\rho(\pi)\in \{0,\infty\}\cup \N$ we denote the common dimension of the multiplicity spaces appearing in 
Lemma~\ref{sardine}. 
We say that $\rho$ is $\pi$-minimal if
\begin{enumerate}
\item[(i)] it is $H$-spherical,
\item[(ii)] $m_\rho(\pi)\ne 0$, and
\item[(iii)] for every subrepresentation on a proper closed subspace $W_{\rho'}\subset W_\rho$ we have $m_{\rho'}(\pi)=0$.
\end{enumerate}
We say that $W_\rho$ is $L$-minimal, if it is $\pi$-minimal for some $V_\pi$ as above.
\end{Def}

Note that every irreducible admissible $H$-spherical representation is $\pi$-minimal for all $\pi$ with $m_\rho(\pi)\ne 0$. In particular, it is $L$-minimal. We will see later (Prop.~\ref{kopp}) that, for some triples
$(G,H,L)$, not every $L$-minimal representation is irreducible. However, $L$-minimal representations are always indecomposable, have an infinitesimal character as well as a unique irreducible quotient. 

\begin{thm}\label{aal}
Let $(G,H,L)$ be a triple of Type~I.
\begin{enumerate}
\item[(i)] Let $W_\rho$ be an $L$-minimal $H$-spherical representation of $G$. Then $\dim (W_{\tilde\rho,-\infty})^H=1$.
\item[(ii)] Let $W_\rho$ be an admissible $H$-spherical $G$-representation of finite length, and let $V_\pi$ be an irreducible admissible $L$-representation with $m_\rho(\pi)\ne 0$.
Then $m_\rho(\pi)=\dim V_\pi^{L\cap H}$.
\item[(iii)]For every irreducible admissible $L$-representation $V_\pi$ with $V_\pi^{L\cap H}\ne\{0\}$ there is a unique $\pi$-minimal $H$-spherical $G$-representation $W_\rho$.
Here uniqueness is up to equivalence of the $G$-representations on smooth vectors (or equivalently of the underlying $(\fg,K)$-modules).
\end{enumerate}
\end{thm}

\begin{proof} Let $V_\pi$ be an irreducible admissible representations of $L$, and set $L_1:=L_{min}$, $L_2:=L_{max}$. As in the proof of Lemma \ref{sardine} we consider
the corresponding connected subgroups $C\subset C_2$ with $L=L_1 \cdot C$, $L_2=L_1 \cdot C_2$, and the corresponding decomposition  $V_\pi=V_{\pi_1}\otimes F_\tau$.
We also consider the identity component $C_3$ of $Z_G(L_1)$.
Its maximal compact subgroup is $C_2$.
  
 For any $G$-representation $Z$ the space $\Hom_{L_1}(V_{\pi_1,\infty},Z)$ carries a 
$C_3$-representation such that $\Hom_{L}(V_{\pi,\infty},Z)\cong \Hom_C(F_\tau, \Hom_{L_1}(V_{\pi_1,\infty},Z))$. Let $X=G/H\cong L/L\cap H\cong L_1/L_1\cap H$, and let $W_\rho$ be an admissible 
$H$-spherical $G$-representation of finite length. We consider the $C_3$-representations 
$$ A:= \Hom_{L_1}(V_{\pi_1,\infty},W_{\rho,\infty})\mbox{ and } B:=  \Hom_{L_1}(V_{\pi_1,\infty},C^\infty(X))\ .$$
As vector spaces we have $B\cong (V_{\tilde\pi_1,-\infty})^{L_1\cap H}$ by Frobenius reciprocity. 

The key observation is that $B$ is an irreducible admissible $C_3$-representation (or zero). Indeed, by Prop.~\ref{productck} it suffices to check this for irreducible Type~I triples with $L=L_{min}$.
If $G$ is simple the observation follows from Prop.~\ref{mufree} combined with the argument of the proof of Prop.~\ref{staun}. In the group case with $G=G_1\times G_1\cong L_1\times C_3$
we have $L_1\cap H=\{e\}$ and $B\cong V_{\tilde\pi_1,-\infty}$ as $C_3\cong L_1$-representation. Here again one uses a Frobenius reciprocity argument similar as in the proof of  Prop.~\ref{staun}.

Any $\Phi\in \Hom_G(W_{\rho,\infty}, C^\infty(X))$ defines an element $i(\Phi)\in \Hom_{C_3}(A,B)$ by $i(\Phi)(a):=\Phi\circ a$, $a\in A$.
We fix an injective $\Phi_0\in \Hom_G(W_{\rho,\infty}, C^\infty(X))$. It exists since $W_\rho$ is $H$-spherical. Then also $i(\Phi_0)$ is injective. We obtain an embedding of
$C_3$-representations 
$A\hookrightarrow B$. 

We now assume that $m_\rho(\pi)=\dim \Hom_C(F_\tau,A)\ne 0$. In particular, $A\ne\{0\}$. Then $B$ is non-zero, hence irreducible. It follows that $A$ is irreducible,
$A\hookrightarrow B$ has dense image, and that $\Hom_{C_3}(A,B)$ is one-dimensional.
Strictly speaking, these assertions follow by considering the underlying Harish-Chandra modules $A_{C_2}$ and $B_{C_2}$, and we obtain that:
\be\label{weird}A_{C_2}\cong B_{C_2}\  .\end{equation}
The density of $A$ in $B$ now implies Assertion (ii):
$$ m_\rho(\pi)=\dim \Hom_C(F_\tau,A)=\dim\Hom_C(F_\tau, B)=\dim\, (V_{\tilde\pi,-\infty})^{L\cap H}=\dim V_\pi^{L\cap H}\ .$$

Let $W'\subset W_\rho$ be a closed $G$-invariant subspace. Then $A':=\Hom_{L_1}(V_{\pi_1,\infty},(W')_{\infty})\subset A$ is a closed $C_3$-invariant subspace. Since $A$ is irreducible, 
$A'=\{0\}$ or $A'=A$ 
which corresponds to the alternative 
$$\Hom_{L}(V_{\pi,\infty},(W')_{\infty})=\{0\}\quad\mbox{ or }\quad\Hom_{L}(V_{\pi,\infty},(W')_{\infty})=\Hom_{L}(V_{\pi,\infty},W_{\rho,\infty})\ .$$
If $W_\rho$ is $\pi$-minimal and $W'\ne W_\rho$, this implies $A'=\{0\}$. 
We reconsider the linear map $i: \Hom_G(W_{\rho,\infty}, C^\infty(X))\cong (W_{\tilde\rho,-\infty})^H \rightarrow \Hom_{C_3}(A,B)$ constructed above.
It follows that $i$  is injective for $\pi$-minimal $W_\rho$.
We conclude that in this case $0<\dim\, (W_{\tilde\rho,-\infty})^H\le \dim   \Hom_{C_3}(A,B)=1$. Assertion (i) follows.

Now we assume $V_\pi^{L\cap H}\cong  \Hom_{L}(V_{\pi,\infty},C^\infty(X))\ne\{0\}$ and discuss existence and uniqueness of a $\pi$-minimal  $W_\rho$. The assumption can
be expressed as $\Hom_C(F_\tau, B)\ne\{0\}$. Therefore also the dense subspace $\Hom_C(F_\tau, B_{C_2})$ is non-zero. On the other hand,
by Prop.~\ref{staun} there exists a character $\chi$ such that $ \Hom_{L}(V_{\pi,\infty},C^\infty(X))=\Hom_{L}(V_{\pi,\infty}, E_{\chi}^\infty(X))$. 
We conclude that 
$$\Hom_C(F_\tau, B_{C_2})\cong \{\Phi\in\Hom_{L}(V_{\pi,\infty}, E_{\chi}^\infty(X))\mid \Phi(V_{\pi,K_L})\subset E_{\chi}^\infty(X)_{K_{L_2}}\}\ .$$
It is well-known (\cite{Ban87}, Cor.~3.10) that  $E_{\chi}^\infty(X)_K$
is a Harish-Chandra module for $G$. Moreover, it has a reflexive Banach globalization $W_{\rho_1}$ such that $W_{\rho_1,\infty}\subset  E_{\chi}^\infty(X)$. These facts are recalled as Prop.~\ref{englisch}
at the beginning of the next session. Since $(G,H,L_2)$ is a spherical triple, Prop.~\ref{hering} (i) now implies that $E_{\chi}^\infty(X)_{K_{L_2}}=E_{\chi}^\infty(X)_K$.
Thus 
$$\{\Phi\in\Hom_{L}(V_{\pi,\infty}, E_{\chi}^\infty(X))\mid \Phi(V_{\pi,K_L})\subset E_{\chi}^\infty(X)_{K}\}$$
is non-zero. We pick a non-zero element $\Phi$ of the latter space and a non-zero $v\in V_{\pi,K_L}$. We consider $0\ne \Phi(v)=:f\in E_{\chi}^\infty(X)_K=W_{1,K}$.
Let $W_{\rho_0}\subset W_{\rho_1}$ be the (closed) $G$-subrepresentation generated by $f$. By construction, we have $m_{\rho_0}(\pi)\ne 0$.

Now let $W_\rho$ be an arbitrary $H$-spherical admissible $G$-representation of finite length on a reflexive Banach space such that $m_\rho(\pi)\ne 0$. We may assume that
$W_{\rho,\infty}\subset C^\infty(X)$. By (\ref{weird}) we then have $A_{C_2}= B_{C_2}$ which implies $\Phi\in\Hom_{L}(V_{\pi,\infty},W_{\rho,\infty})$, $f\in W_{\rho,K}$.
It follows that $W_{\rho_0,K}\subset  W_{\rho,K}$. This shows that $W_{\rho_0}$ is the unique $\pi$-minimal $G$-representation.
\end{proof}

Let us mention two properties of $L$-minimal representations which are also shown implicitly in the above proof. The first is that the notion of $L$-minimality is independent of the
choice of $L$, $L_{min}\subset L\subset L_{max}$. The second is that for a $\pi$-minimal representation $W_\rho$ the natural map
\be\label{guru}\Hom_{L}(V_{\pi,\infty},W_{\rho,\infty})\otimes  (W_{\tilde\rho,-\infty})^H\rightarrow (V_{\tilde\pi,-\infty})^{L\cap H}
\end{equation}
is injective and has dense image.

We collect the specializations of the results obtained so far to the particularly important case  of restriction of {\em unitary} representations 
in the following theorem. 

\begin{thm}\label{mainbranching}
Let $(G,H,L)$ be a triple of Type~I, and let $W_\rho$ be a weakly $H$-spherical irreducible unitary representation of $G$. Then
its restriction to $L$ decomposes discretely:
$$ W_\rho\stackrel{L}{\cong}\widehat{\bigoplus_{\pi\in\widehat L}} M_\rho(\pi)\hat\otimes V_\pi\ .$$
Moreover, we have:
\begin{itemize}
\item[(a)] If $(G,H, L)$ is spherical, then the multiplicities $m_\rho(\pi)=\dim M_\rho(\pi)$ are finite.
\item[(b)] Assume that $W_\rho$ is $H$-spherical. Then 
$$m_\rho(\pi)\ne 0\Rightarrow \left\{\begin{array}{rcl}
m_\rho(\pi)&=&\dim V_\pi^{L\cap H}\\
m_{\rho'}(\pi)&=&\quad 0\qquad \mbox{ for all irreducible $H$-spherical unitary }\rho'\not\cong \rho
\end{array}\right. \ .$$
In particular, if $(G,H,L)$ is one of the triples described in Prop.~\ref{mufree}
(i) or (ii), then the restriction of $W_\rho$ to $L$ is multiplicity free.
\item[(c)]Assume that $W_\rho$ is $H$-spherical and not trivial. Then 
$$ \sum_{\pi\in\widehat L} m_\rho(\pi)=\infty\ .$$
\end{itemize}
\end{thm}

\begin{proof}
All assertions 
are immediate consequences of the preceding results Prop.~\ref{hering}~(ii), Cor.~\ref{makrele}, and Thm.~\ref{aal}~(ii),(iii).
\end{proof}

Let us mention that our initial proof of  Theorem~\ref{mainbranching}, in particular of Assertion (a),  was completely different.
It rested on the classification of spherical triples in Thm.~\ref{listck} (that $L$ has real rank one in the relevant cases is crucial) and the Casimir operator computations in 
Section~\ref{pbw}.
These results imply that every $H\cap L$-invariant vector in $W_{\rho,K_L}$  
generates a Harish-Chandra module for $L$ and that these modules span
the whole space. 
This together with the easy multiplicity bound $m_\rho(\pi)\le\dim V_\pi^{L\cap H}<\infty$ (see Cor. \ref{inv3}) implies $L$-admissibility.
Moreover, our original proof of part (c) was based on Moore's ergodicity theorem (\cite{Moo66}) valid only for {\em unitary} representations instead of the more elementary Lemma \ref{almosttrivial}. 
We prefer the proof written here resting on Lemma \ref{sigmatheta}
and Delorme's embedding theorem because it is completely independent of the classification. However, the more refined multiplicity assertions (Thm.~\ref{aal}, (i) and (ii), and Thm.~\ref{mainbranching}, (b))  still depend on the classification. A proof 
which is similar in spirit to our first proof has been given in \cite{KK25}.

Recall that an admissible representation $V_\pi$ of finite length of a reductive group with compact center $L$ is said to be of class $L^p$, $p\in [1,\infty)$, if all its 
$K_L\times K_L$-finite matrix coefficients $c_{v,\tilde v}$ belong to $L^p(L)$. If $V_\pi$ is irreducible, it suffices to check this for one non-zero matrix coefficient. 
Of particular importance are the classes 
$L^1$ (integrable discrete series represesentations), $L^2$ (discrete series), and representations that are of class $L^{2+\ve}$ for all $\ve>0$ (tempered representations).
Now let $G/H$ be a semisimple symmetric space, and let $W_\rho$ be an $H$-spherical admissible $G$-representation of finite length on a reflexive Banach space. 
We say that an element $\tilde w\in   ( W_{\tilde\rho,-\infty})^H$ is of relative class $L^p$,  $p\in [1,\infty)$, if $c_{w,\tilde w}\in L^p(G/H)$ for all $w\in W_{\rho,K}$.
Because of Thm.~\ref{aal} (i), we are interested in the case $\dim ( W_{\tilde\rho,-\infty})^H =1$. In this case, being of relative class $L^p$ is really a property of the representation
$W_\rho$. In particular, we can speak of (integrable) relative discrete series and of relatively tempered representations.

\begin{pro}\label{fishy}
Let $(G,H,L)$ be a triple of Type~I, and let $W_\rho$ be an $L$-minimal  $H$-spherical $G$-representation. Then the following assertions are equivalent:
\begin{itemize}
\item[(i)] $W_\rho$ is of relative class $L^p$.
\item[(ii)] All irreducible admissible representations $V_\pi$ of $L$ with $m_\rho(\pi)\ne \{0\}$ are of class $L^p$.
\item[(iii)] There exists an irreducible admissible representation $V_\pi$ of class $L^p$ such that $W_\rho$ is $\pi$-minimal.
\end{itemize}
\end{pro}

\begin{proof}
We fix $0\ne\tilde w\in (W_{\tilde\rho,-\infty})^H$. Let $\Phi\in  \Hom_{(\fl,K_L)}(V_{\pi,K_L},W_{\rho,K})\subset  \Hom_{L}(V_{\pi,\infty},W_{\rho,\infty})$, see Lemma~\ref{sardine}.
Then for $l$ in $L$ we have the following identity of matrix coefficients
\be\label{dadamax} 
 c^G_{\Phi(v),\tilde w}(l)=c^L_{v,\Phi^t(\tilde w)}(l) \ .
\end{equation}
Since $X=G/H \cong L/L\cap H$, restriction to $L$ determines the matrix coefficient $c_{\Phi(v),\tilde w}$ completely. Note that integrability over $L$ and over $X$
are equivalent. We also observe that $\Phi^t(\tilde w)$ which a priori belongs to  $(V_{\tilde\pi,-\infty})^{L\cap H}$ is $K_L$-finite.
This is clear, if $(G,H,L)$ is spherical: by finite-dimensionality we have $(V_{\tilde\pi,-\infty})^{L\cap H}=(V_{\tilde\pi,K_L})^{L\cap H}$. The $K_L$-finiteness
is also easily checked in the group case. So it holds in general. The implication (i)$\Rightarrow$(ii) is now immediate. (ii)$\Rightarrow$(iii) holds by definition of $L$-minimality.
It remains to discuss (iii)$\Rightarrow$(i). Let $W_\rho$ be $\pi$-minimal. We assume that $V_\pi$ is of class $L^p$. Hence $c_{v,\Phi^t(\tilde w)}=c_{\Phi(v),\tilde w}\in L^p(X)$.
Since $L^p(X)$ is $G$-invariant and the vectors $\Phi(v)$ generate $W_\rho$ by $\pi$-minimality we conclude that $c_{w,\tilde w}\in L^p(X)$ for all $w\in W_{\rho,K}$. 
Thus $W_\rho$ is of relative class $L^p$.
\end{proof}

As a first consequence of Prop.~\ref{fishy}
we obtain information on the contribution of integrable relative discrete series of $G$ to eigenspaces of ${\bf D}(G/H)$ on $Y=\Gamma\backslash G/H$.
It could be seen as a certain generalization of Thm.~\ref{something}(c) to general Type~I triples. 

\begin{pro}\label{horn}
Let $(G,H,L)$ be a triple of Type~I, and let $\Gamma\subset L$ be a torsion free uniform lattice. Let $W_\rho$ be an integrable relative discrete series for $G/H$.
Let $\chi$ be the character of ${\bf D}(G/H)$ determined by its action on $(W_{\rho,-\infty})^H$ (note that by unitarity of $W_\rho$ there is a conjugate linear isomorphism $(W_{\rho,-\infty})^H\cong (W_{\tilde\rho,-\infty})^H\ne\{0\}$). Then the eigenspace $E^\infty_\chi(Y)$ is infinite dimensional.
\end{pro}

\begin{proof} Let $\widehat L_\chi$ be as in Cor.~\ref{hussa}. We consider the discrete decomposition
\be\label{h1}W_{\rho}\stackrel{L}{\cong}\widehat{\bigoplus_{\pi\in\widehat L}} M_{\rho}(\pi)\hat\otimes V_\pi\  . \end{equation}
Then $ M_{\rho}(\pi)\ne\{0\}$ implies that $\pi\in\widehat L_\chi$ and, by Prop.~\ref{fishy}, that $\pi$ is an integrable discrete series of $L$.
By the result of Langlands discussed after Prop.~\ref{pavle}, we have $N_\Gamma(\pi)\ne\{0\}$. 
On the other hand we have by Cor.~\ref{hussa} that
$${E}_\chi^{\infty}(Y)\cong \overline{\bigoplus_{\pi\in\widehat L_\chi}} N_\Gamma(\pi)\otimes (V_{\pi,\infty})^{L\cap H}\ .$$
Now we employ Thm.~\ref{mainbranching}, (b),(c) to conclude that already the contribution of the representations of $L$ occuring in (\ref{h1}) to the above direct sum is infinite dimensional.
\end{proof}
The proposition is related to the main result of 
Kassel and Kobayashi in \cite{KK2}, see the beginning of Subsection~\ref{kako}.
The proposition is relevant for all Type~I triples such that there are integrable relative discrete series for $G/H$. Note that this is equivalent to the existence of a relative discrete series
for $G/H$ which in turn is equivalent  to $\rank(G/H)=\rank (K/K_H)$. The latter condition is satisfied for all Type~I triples with $G$ simple except for Case 2 with $n$ odd and Case 7 (see Table~\ref{fix}).

Before we come to further consequences of Prop.~\ref{fishy} - for instance the  non-unitarity of certain $L$-minimal representations which is part of Prop.~\ref{kopp} below -
we return to Cor.~\ref{renoir}. In particular, if $W_\rho$ is induced from $P_{\sigma\gt}$ with generic induction parameter, it gives a way to compute the decomposition of 
Thm.~\ref{mainbranching} explicitly: One just has to do Fourier decomposition along the fibres of $G/P_{\sigma\theta}\rightarrow L/P_L$. These fibres are compact homogeneous
spaces of the form $M_L/Z_{K_L}({\mathfrak a}_{{\mathfrak q}\cap{\mathfrak s}})$.

For a moment, let $G/H$ be a general semisimple symmetric space. Let us describe its standard minimal $\sigma\theta$-stable parabolic $P_{\sigma\theta}$ in more detail. We consider
$Z_G(\fa_{\fq\cap\fs})$. It has a decomposition $Z_G(\fa_{\fq\cap\fs})=M_{\sigma\theta}A_\fq$, where $A_\fq:=\exp(\fa_{\fq\cap\fs})$ and $M_{\sigma\theta}:=\bigcap_\vp \ker\vp^2$. Here the intersection runs over all $\sigma\theta$-invariant characters $\vp: Z_G(\fa_{\fq\cap\fs})\rightarrow \R^*$. Let $N_{\sigma\theta}$ be associated to a system
of positive roots $\Phi^+$ compatible with $\fa_{\fq\cap\fs}$ as in the proof of Cor.~\ref{renoir}. Then we have $P_{\sigma\theta}=M_{\sigma\theta}A_\fq N_{\sigma\theta}$.

$M_{\sigma\theta}$ has the compact subgroups $Z:=Z_K(\fa_{\fq\cap\fs})$, $Z_H:=Z_{K_H}(\fa_{\fq\cap\fs})$, and in the presence of a properly transitive triple $(G,H,L)$ also $Z_L:=Z_{K_L}(\fa_{\fq\cap\fs})$. They were already considered in the proof of Lemma~\ref{sigmatheta}. Since $\fm_{\sigma\theta}\cap\fs\subset\fh$, any irreducible
admissible $M_{\sigma\theta}\cap H$-spherical representation $\delta$ of $M_{\sigma\theta}$ comes from an irreducible representation of $Z$ posessing an $Z_H$-invariant vector.
In particular, $\delta$ is unitary and finite dimensional. For any subgroup $Q\subset K$ we set
$$ \widehat Q_s:=\{\pi\in\widehat Q\mid V_\pi^{Q\cap H}\ne\{0\}\}\ .$$
Here the subscript $s$ stands for spherical. For $\delta\in\widehat Z_s$ and $\lambda\in \fa_{\fq\cap\fs,\C}^*$ we view $\delta$ as an $M_{\sigma\theta}\cap H$-spherical representation  of $M_{\sigma\theta}$ and  set
$$  I^{\delta,\lambda}:=  \Ind_{P_{\sigma\theta}}^{G}\delta\otimes\e^{\lambda}\otimes 1\ .$$
As a $K$-representation it is independent of $\lambda$: $I^\delta=\Ind_{Z}^{K}\delta$. The important features of $I^{\delta,\lambda}$ are: for generic $\lambda$ it is an irreducible
$G$-representation; $(I^{\delta,\lambda}_{-\infty})^H\ne\{0\}$ (see e.g. \cite{Ban88}, Thm.~5.10); Delorme's embedding theorem \cite{Del85} saying that for any irreducible admissible
$H$-spherical $G$-representation $W_\rho$ there exist $(\delta,\lambda)$ as above such that  $W_{\rho,\infty}$ embeds into $I^{\delta,\lambda}_\infty$.

Now let $(G,H,L)$ be a spherical triple. Choose $P_L=M_LA_LN_L$ as in Cor.~\ref{renoir}. Then we obtain immediately the following isomorphisms of $L$-representations:
\begin{eqnarray*}
I^{\delta,\lambda}&\cong& \Ind_{Z_LA_LN_L}^{L}\delta\otimes\e^{\lambda}\otimes 1 \\
&\cong& \bigoplus_{\sigma\in \widehat M_{L,s}} \Hom_{Z_L}(V_\sigma,U_\delta) \otimes \Ind_{M_LA_LN_L}^{L}\sigma\otimes\e^{\lambda}\otimes 1\ ,
\end{eqnarray*}
or shortly
\be\label{schal}
I^{\delta,\lambda}\cong \bigoplus_{\sigma\in \widehat M_{L,s}} \Hom_{Z_L}(V_\sigma,U_\delta) \otimes H^{\sigma,\lambda}\ .
\end{equation}
Here we have identified $ \fa_{\fq\cap\fs,\C}^*$ with  $\fa_{L,\C}^*$ by extending $\lambda\in  \fa_{\fq\cap\fs,\C}^*$ trivially on $ \fa_{H}$ to  $\fa$
and then restricting to $\fa_L$. Note that for generic $\lambda$, (\ref{schal}) already gives the full decomposition of $I^{\delta,\lambda}$ into irreducible $L$-representations.
There is some a priori information on the multiplicity spaces $\Hom_{Z_L}(V_\sigma,U_\delta)$.

\begin{cor}\label{wels}
Let $(G,H,L)$ be a spherical triple.
\begin{itemize}
\item[(a)] For all $\sigma\in\widehat M_{L,s}$ there is a unique $\delta\in \widehat Z_{s}$ such that $\Hom_{Z_L}(V_\sigma,U_\delta)\ne\{0\}$.
\item[(b)] If $\Hom_{Z_L}(V_\sigma,U_\delta)\ne\{0\}$, then 
$\dim\Hom_{Z_L}(V_\sigma,U_\delta)=\dim V_\sigma^{M_L\cap H}$.
\item[(c)] We assume in addition that all group factors of $X=G/H$ are locally isomorphic to $SO_e(1,n)$ or $SU(1,n)$ and that $L$ contains the maximal compact subgroups of the
motion groups of the group factors. Let $V_\pi$ be an irreducible admissible representation of $L$ with $V_\pi^{L\cap H}\ne\{0\}$. Then there is a unique $K_L$-type $\gamma_\pi$ with $V_\pi^{L\cap H}=V_\pi(\gamma_\pi)^{L\cap H}$.
$\gamma_\pi$ occurs in $V_\pi$ with multiplicity one. For $V_\pi=H^{\sigma,\lambda}$ ($\sigma\in\widehat M_{L,s}$, $\lambda$ generic) we denote $\gamma_\pi$ also by $\gamma_\sigma$.
Then the dimension appearing in (b) is also equal to $\dim V_{\gamma_\sigma}^{L\cap H}$.
\end{itemize}
\end{cor} 

\begin{proof} Let $\sigma\in\widehat M_{L,s}$. Since $Z_L\subset M_L$ and therefore also $Z_L\cap H\subset M_L\cap H$ (in fact we have equality, but this is not needed here) there is a $\delta_0\in\widehat Z_{L,s}$ appearing as an irreducible subrepresentation of $V_\sigma$. The spherical representation $U_{\delta_0}$ can be realized in $C^\infty(Z_L/Z_L\cap H)\cong C^\infty(Z/Z_H)$.
For the latter isomorphism we have used
the equality $K=K_LZ_H$ (established at the end of the proof of Lemma~\ref{sigmatheta}) which implies 
\be\label{stand}Z=Z_LZ_H\ .
\end{equation} Therefore there is a $\delta\in \widehat Z_s$ having non-zero $Z_L$-intertwining operator to $U_{\delta_0}$. This proves the existence of $\delta$. For the uniqueness of $\delta$ we employ the uniqueness of minimal $H$-spherical representations (Thm.~\ref{aal}). Indeed, let $\delta_1,\delta_2$ be such that $\Hom_{Z_L}(V_\sigma,U_{\delta_i})\ne\{0\}$, $i=1,2$. By 
the decomposition (\ref{schal}) for generic $\lambda$ both $I^{\delta_1,\lambda}$ and $I^{\delta_2,\lambda}$ are minimal $H$-spherical for the $L$-representation $H^{\sigma,\lambda}$.
Thus $I^{\delta_1,\lambda}\cong I^{\delta_2,\lambda}$. It follows that the representations $I^{\delta_1,\lambda}$ and $I^{\delta_1,\lambda}$ have the same character for all $\lambda\in\fa_{\fq\cap\fs,\C}^*$. From the character formula for parabolically induced representations (see e.g. \cite{Wa92}, Prop.~12.1.1) we deduce that $\delta_1\cong\delta_2$.
The argument is similar to the proof of Thm.~12.1.4 in \cite{Wa92}, but even more elementary (the difference is that we induce a finite dimensional representation while that Thm. treats
the case of induction of a discrete series).

Moreover, the multiplicity formula in Thm.~\ref{aal}, (ii) tells us that $\dim\Hom_{Z_L}(V_\sigma,U_\delta)=\dim (H^{\sigma,\lambda})^{L\cap H}=\dim V_\sigma^{M_L\cap H}$.

Assertion (c) is a direct consequence of Prop.~\ref{mufree}.
\end{proof}

For a finer discussion of (\ref{schal}), in particular in case of reducibility of $I^{\delta,\lambda}$, we need some information on the spherical $K$- and $K_L$-types of $I^\delta$.

\begin{lem}\label{schlips}
Let $(G,H,L)$ be a spherical triple satisfying the assumptions of Cor.~\ref{wels}(c). Let $\delta\in\widehat Z_s$. 
\begin{itemize}
\item[(a)]For any $\tau\in\widehat K_s$ we have $\dim \Hom_K(W_\tau,I^\delta)\le 1$.
\item[(b)]For any $\gamma\in\widehat K_{L,s}$ with $\Hom_{K_L}(V_\gamma,I^\delta)\ne\{0\}$  we have $\dim \Hom_{K_L}(V_\gamma,I^\delta)=\dim V_\gamma^{L\cap H}$.
Moreover, for $L=L_{max}$ the space $\Hom_{K_L}(V_\gamma,I^\delta)$ is one-dimensional while for $L=L_{min}$ (here minimal with respect to our assumptions) it is an irreducible
$C_2$-representation, where $C_2$ is as in the proofs of Lemma~\ref{sardine} and Thm.~\ref{aal}.
\item[(c)] Let $\gamma\in\widehat K_{L,s}$. Then $\Hom_{K_L}(V_\gamma,I^\delta)$  factors through a unique $K$-type $\tau_\gamma$, i.~e.
$$  \Hom_{K_L}(V_\gamma,I^\delta)\cong \Hom_{K_L}(V_\gamma, W_{\tau_\gamma})\otimes \Hom_K(W_{\tau_\gamma},I^\delta)\ .$$
Moreover, $\tau_\gamma\in\widehat K_s$.
\end{itemize}
\end{lem}

\begin{proof}
Let $\gamma\in\widehat K_{L,s}$.  The spherical representation $V_{\gamma}$ can be realized in $C^\infty(K_L/L\cap H)\cong C^\infty(K/K_H)$.
Therefore we can choose
$\tau_\gamma\in\widehat K_s$ with $\Hom_{K_L}(V_\gamma, W_{\tau_\gamma})\ne\{0\}$. (It is in fact unique, `the $\gamma$-minimal $K_H$-spherical $K$-representation', see the arguments proving Claim 3 below.) Any $\tau\in \widehat K_s$  is of the form $\tau_\gamma$ for some $\gamma\in\widehat K_{L,s}$.

Frobenius reciprocity gives  $\Hom_{K_L}(V_\gamma,I^\delta)\cong \Hom_{Z_L}(V_\gamma,U_\delta)$ and $\Hom_K(W_\tau,I^\delta)\cong \Hom_Z(W_\tau,U_\delta)$.
Since $K_L\supset M_L\supset Z_L$ we can decompose $\Hom_{Z_L}(V_\gamma,U_\delta)$ via the decomposition of $V_\gamma$ into $M_L$-isotypic components:
\be\label{kragen}
\Hom_{Z_L}(V_\gamma,U_\delta)\cong \bigoplus_{\sigma\in \widehat M_{L,s}} \Hom_{M_L}(V_\gamma,V_\sigma)\otimes \Hom_{Z_L}(V_\sigma,U_\delta)\ .
\end{equation}
Since $K_L=M_L(L\cap H)$
any non-zero $\Phi\in \Hom_{M_L}(V_\gamma,V_\sigma)$  is injective when restricted to $V_\gamma^{L\cap H}$. This shows that indeed only spherical representations $\sigma$ of $M_L$
can appear in the decomposition.

{\it Claim 1: $\Hom_{Z_L}(V_\gamma,U_\delta)\ne\{0\}\Rightarrow \Hom_Z(W_{\tau_\gamma},U_\delta) \ne\{0\}$.}\\
{\it Proof of Claim 1.}  By assumption and (\ref{kragen}) there exists $\sigma\in \widehat M_{L,s}$ such that $\Hom_{M_L}(V_\gamma,V_\sigma)$ and $\Hom_{Z_L}(V_\sigma,U_\delta)$
are non-zero. We fix non-zero $\Phi_3\in \Hom_{K_L}(W_{\tau_\gamma},V_\gamma)$, $\Phi_2\in \Hom_{M_L}(V_\gamma, V_\sigma)$. As already mentioned above, $\Phi_2$ is  injective when restricted to $V_\gamma^{L\cap H}$. Similarly, the restriction of $\Phi_3$ to $W_{\tau_\gamma}^{K_H}$ is injective since $K=K_LK_H$. Fix non-zero 
$w_0\in W_{\tau_\gamma}^{K_H}$, $u_0\in U_\delta^{Z_H}$. Then $\Phi_2\Phi_3(w_0)\in V_\sigma^{M_L\cap H}$ is also non-zero. Equation (\ref{stand}) implies that elements $\Phi_1\in \Hom_{Z_L}(V_\sigma,U_\delta)$ are uniquely determined by $\Phi_1^*u_0\in V_\sigma^{Z_L\cap H}=V_\sigma^{M_L\cap H}$. By Cor.~\ref{wels}(b) we have $\dim \Hom_{Z_L}(V_\sigma,U_\delta)=\dim V_\sigma^{M_L\cap H}$. This implies that we can choose $\Phi_1$ such that $\langle \Phi_1\Phi_2\Phi_3(w_0),u_0\rangle\ne 0$. Now we consider
$$ \Phi:=\int_{Z/Z_L} \delta(z)\Phi_1\Phi_2\Phi_3\tau_\gamma(z)^{-1} dz = \int_{Z_H/Z_L\cap H} \delta(z)\Phi_1\Phi_2\Phi_3\tau_\gamma(z)^{-1} dz \in  
\Hom_Z(W_{\tau_\gamma},U_\delta) \ .$$
For the second equation we have used (\ref{stand}) again. We see that $\langle \Phi(w_0),u_0\rangle=\langle \Phi_1\Phi_2\Phi_3(w_0),u_0\rangle\ne 0$. Thus $\Phi\ne\{0\}$.

{\it Claim 2: Assume that $Z_L\cap H$ acts transitively on the $Z_H$-orbits on $K/K_H$. Then }
$$\Hom_{Z_L}(W_{\tau_\gamma},U_\delta)=\Hom_Z(W_{\tau_\gamma},U_\delta)\ .$$
{\it Proof of Claim 2.} By (\ref{stand}) any $\Phi\in\Hom_{Z_L}(W_{\tau_\gamma},U_\delta)$ is uniquely determined by the function 
$f_\Phi\in C^\infty(Z_L\cap H\backslash K/K_H)$ given by $f_\Phi(k):=\langle \Phi(\tau_\gamma(k) w_0),u_0\rangle$. Here $w_0,u_0$ are as in the proof of Claim 1.  Our assumption
implies that  $C^\infty(Z_L\cap H\backslash K/K_H)=C^\infty(Z_H\backslash K/K_H)$. Hence $f_\Phi=f_{\overline\Phi}$, where
$$ \overline\Phi:=\int_{Z/Z_L} \delta(z)\Phi\tau_\gamma(z)^{-1} dz = \int_{Z_H/Z_L\cap H} \delta(z)\Phi\tau_\gamma(z)^{-1} dz \in  
\Hom_Z(W_{\tau_\gamma},U_\delta) \ .$$
We conclude that $\Phi=\overline\Phi$.

{\it Claim 3: Assume that for any $\gamma\in\widehat K_{L,s}$, we have $\dim \Hom_{Z_L}(V_\gamma,U_\delta)\le\dim V_\gamma^{L\cap H}$. Then all assertions of the lemma are true.}\\
{\it Proof of Claim 3.}  We have $\{0\}\ne\Hom_{K_L}(V_\gamma, W_{\tau_\gamma})\subset \Hom_{K_L}(V_\gamma,C^\infty(K/K_H))\cong V_{\tilde\gamma}^{L\cap H}$.
It follows from Prop.~\ref{mufree} and the arguments in the proof of Prop.~\ref{staun} that $V_{\tilde\gamma}^{L\cap H}$ and hence also $\Hom_{K_L}(V_\gamma, W_{\tau_\gamma})$
is an irreducible $C_2$-module, whenever $L=L_{min}$. It then follows that in any case $\dim\Hom_{K_L}(V_\gamma, W_{\tau_\gamma})=\dim V_\gamma^{L\cap H}$.
On the other hand, the decomposition of $I^\delta$ into $K$-types yields
$$ \Hom_{K_L}(V_\gamma,I^\delta)\cong\bigoplus_{\tau\in\widehat K} \Hom_{K_L}(V_\gamma, W_\tau)\otimes \Hom_{K}( W_\tau, I^\delta)\cong 
\bigoplus_{\tau\in\widehat K} \Hom_{K_L}(V_\gamma, W_\tau)\otimes \Hom_{Z}( W_\tau,U_\delta)\ .$$
In particular, the natural map
$$  \Hom_{K_L}(V_\gamma, W_{\tau_\gamma})\otimes \Hom_{Z}( W_{\tau_\gamma},U_\delta)\rightarrow \Hom_{Z_L}(V_\gamma,U_\delta)$$
is injective. By Claim 1 it is non-zero whenever $\Hom_{Z_L}(V_\gamma,U_\delta)\ne\{0\}$. Now Claim 3 follows.

{\it Claim 4: Assume that for any $\gamma\in\widehat K_{L,s}$ there is at most one non-trivial summand in the decomposition (\ref{kragen}). Then all assertions of the lemma are true.}\\
{\it Proof of Claim 4.} By Cor.~\ref{wels}(c) we have $\dim\Hom_{M_L}(V_\gamma,V_\sigma)\le 1$ and $\dim\Hom_{Z_L}(V_\sigma,U_\delta)\le \dim V_\gamma^{L\cap H}$. Therefore Claim 3 implies Claim 4.

Using the techniques of the proofs of Prop.~\ref{staun} and Thm.~\ref{aal} we reduce the assertions of the lemma to the case that $(G,H,L)$ is irreducible, $L=L_{min}$.
Furthermore, we may assume without loss of generality that $H$ is connected. 

Let $G/H$ be a group manifold. We use the notation of Table 2. In addition, we denote by $M_1\subset K_1$ the $M$-part of a minimal parabolic of $G_1$.  By our assumptions, $L=L_{max}=G_1\times K_1$. Therefore we have $K=K_L=K_1\times K_1$, $M_L=M_1\times K_1$, and $Z_L=M_1\times M_1$. Now we can apply Claim 4.

It remains to discuss the Cases 1--7 in Table 1, but with $L=L_{min}$. In all cases except 1 and 4, the spaces $X$ and $X_K:=K/K_H$ are simply connected. Thus, it is only in Cases 1 and 4 that we have to discuss what happens if $G$ is replaced by a finite covering of it.

Let $\frak i$ be the largest ideal of $\fk$ contained in $\fk\cap\fh$, and let $I\subset K$ be the corresponding analytic subgroup. For any subgroup $Q\subset K$ we set
$\bar Q:=Q/Q\cap I\subset \bar K=K/I$. All the spherical representations appearing in Claims 1--4 descend to the overlined groups $\bar K, \bar K_L, \bar M_L, \bar Z_L$.
It turns out that some of these overlined groups may coincide: $\bar K=\bar K_L$ except for Cases 5 and 7, $\bar K_L=\bar M_L$ for Cases 1,3,5.

If $\bar K_L=\bar M_L$, i.~e. in the rank one cases 1,3,5, Claim 4 is trivially satisfied. Since the center of $L$ is always contained in $M_L$, the condition $\bar K_L=\bar M_L$
remains true if we replace $G$ by a finite covering. The lemma now follows in these cases. 

For the remaining  rank one case 6 we argue as follows (similar arguments would also work for
Cases 1,3). We have $Z=Z(G)\cdot Z_e$ (where $Z(G)=\{\pm\id\}$), $Z_e=Z_H=SO(7)\times SO(7)$ and $\bar Z_H=\bar K_H=SO(7)$. It follows that $\delta$ is given by a central character of $G$, in particular, $U_\delta$ is one-dimensional. Therefore $\Hom_Z(W_{\tau_\gamma}, U_\delta)\subset \Hom_{\bar Z_H}(W_{\tau_\gamma}, \C)=\Hom_{\bar K_H}(W_{\tau_\gamma}, \C)$.
Here $\C$ denotes the trivial one-dimensional representation. Since $K/K_H$ is symmetric and $W_{\tau_\gamma}$ is  irreducible spherical, the space on the right hand side is one-dimensional. Thus $\Hom_Z(W_{\tau_\gamma}, U_\delta)$ is at most one-dimensional. We have $L\cap H=\{l\in Spin(8)\mid s_+(l)\in SO(7)\}\cong Spin(7)$ and thus
$Z_L\cap H=\{l\in Spin(7)\subset Spin(8)\mid s_-(l)\in SO(7)\}\cong G_2$. Here $s_\pm$ are the half spin representations of $Spin(8)$. We obtain that $\overline{Z_L\cap H}=G_2\subset SO(7)=\bar Z_H$. The non-trivial $\bar Z_H$-orbits on $S^7=K/K_H$ are $6$-dimensional spheres. Since $G_2$ acts transitively on $S^6$ we see that the assumption of Claim 2
is satisfied. It follows that $\dim \Hom_{Z_L}(V_\gamma,U_\delta)=\dim \Hom_{Z_L}(W_{\tau_\gamma}, U_\delta)\le 1$. Eventually, we apply Claim 3.  

In Case 7 we argue similarly as in Case 6. 
We have $Z=SO(2)\times SO(6)\subset SO(8)=K$. Since $SO(6)\subset SO(7)=K_H$ the representation $\delta$ is given by a character of $SO(2)$, in particular it is one-dimensional. We also observe that $K/Z$ is symmetric. Since the algebra
of invariant differential operators acting on sections of line bundles over symmetric spaces is abelian we conclude that $\dim\Hom_Z(W_{\tau_\gamma}, U_\delta)\le 1$.
We have $M_L=Spin(6)\cong SU(4)\subset SO(8)=K$. Thus $Z_L\cap H=M_L\cap H=SU(4)\cap SO(7)=SU(3)$. The non-trivial $Z_H=SO(6)$-orbits on $S^7=K/K_H$ are $5$-dimensional spheres. Since $SU(3)$ acts transitively on $S^5$ we see that the assumption of Claim 2
is satisfied. It follows that $\dim \Hom_{Z_L}(V_\gamma,U_\delta)\le\dim \Hom_{Z_L}(W_{\tau_\gamma}, U_\delta)\le 1$. Eventually, we apply Claim 3. 

The proof for the higher rank cases 2 and 4 will be based on Claim 4. Let us start with Case 2. We have $K_L=SO(2n)$, $M_L=SO(2n-1)$, $Z=S(U(1)\times U(1))\times SO(2n-2)$. Thus $Z_H=S(U(1)\times U(1))\times U(n-1)$, $Z_L=SO(2n-2)$. In addition we have $L\cap H= U(n)$. Thus $\gamma\in\widehat K_{L,s}$ is an irreducible $SO(2n)$-representation possesing
a $U(n)$-invariant vector while $\delta\in \widehat Z_s$ viewed as a representation of $Z_L$ is an irreducible $SO(2n-2)$-representation possesing a $U(n-1)$-invariant vector.
Since $SO(2n)/U(n)$ is symmetric these representations can be described explicitly via highest weights by the Cartan-Helgason theorem (\cite{Hel84}, Ch. V, Thm.~4.1). If $n$ is even,
then $\gamma$ has highest weight
$$ (\gamma_1,\gamma_1,\gamma_2,\gamma_2,\dots,\gamma_{\frac{n}{2}},\gamma_{\frac{n}{2}}),\quad \gamma_i\in\N_0,\ \gamma_1\ge\gamma_2\ge\dots\ge\gamma_{\frac{n}{2}},$$
while $\delta$ (as representation of $SO(2n-2)$) has highest weight
$$  (\delta_1,\delta_1,\delta_2,\delta_2,\dots,\delta_{\frac{n-2}{2}},\delta_{\frac{n-2}{2}},0),\quad \delta_i\in\N_0,\ \delta_1\ge\delta_2\ge\dots\ge\delta_{\frac{n-2}{2}}.$$
For odd $n$ the corresponding highest weights are
$$ (\gamma_1,\gamma_1,\gamma_2,\gamma_2,\dots,\gamma_{\frac{n-1}{2}},\gamma_{\frac{n-1}{2}},0),\quad \gamma_i\in\N_0,\ \gamma_1\ge\gamma_2\ge\dots\ge\gamma_{\frac{n-1}{2}}$$
and
$$  (\delta_1,\delta_1,\delta_2,\delta_2,\dots,\delta_{\frac{n-1}{2}},\delta_{\frac{n-1}{2}}),\quad \delta_i\in\N_0,\ \delta_1\ge\delta_2\ge\dots\ge\delta_{\frac{n-1}{2}}.$$
Here, in order to fix the sign of the last coordinate, we use the following convention for highest weights of representations of $SO(2k)$: The orientation of $\R^{2k}$ induced by the complex structure $e_1+e_2+\dots+e_k\in\fo(2k)$,
where $e_1,\dots,e_k$ is the basis of the Cartan subalgebra with respect to which the highest weights are written, should coincide with the orientation coming from the complex structure
determining the embedding $U(k)\subset SO(2k)$.  

By the classical branching rules for restricting $SO(k)$-representations to $SO(k-1)$ (see e.~g.~\cite{Zhe73} or \cite{GW98}) the only possibility for the highest weight $\mu_\sigma$ of $\sigma\in \widehat M_{L}$ with
$\Hom_{M_L}(V_\gamma, V_{\sigma})\ne\{0\}\ne \Hom_{Z_L}(V_\sigma,U_\delta)$ is
$$ \mu_\sigma=(\gamma_1,\delta_1,\gamma_2,\delta_2,\dots,\delta_{\frac{n-2}{2}},\gamma_{\frac{n}{2}})\ \mbox{ or }\  \mu_\sigma=(\gamma_1,\delta_1,\gamma_2,\delta_2,\dots,\gamma_{\frac{n-1}{2}},\delta_{\frac{n-1}{2}}),$$
respectively. (If the element $\mu_\sigma$ is not a highest weight, i.e. if it is not dominant, then $\Hom_{Z_L}(V_\gamma,U_\delta)=\{0\}$.) Now we apply Claim 4.

For Case 4 the arguments are completely analogous. We have $K=S(U(2)\times U(2n))$, $K_H=Sp(1)\times Sp(n)$, $I=Sp(1)$, $K_L=\{e\}\times S(U(1)\times U(2n))$ ($L=L_{min}$!),
$M_L=\{e\}\times S(\Delta U(1)\times U(2n-1))$, $L\cap H=Sp(n)$.
Moreover, we have $Z=S((\id_2\times \sigma')(\Delta U(2))\times U(2n-2))$, where $\sigma': U(2)\rightarrow U(2)$ is the non-trivial involution leaving the diagonal $U(1)\times U(1)$ pointwise fixed. Hence, $Z_L=\{e\}\times S(\Delta_{2,4}U(1)\times U(2n-2))$, $Z_H=\Delta Sp(1)\times Sp(n-1)$. Now we reduce modulo $I$, i.e.
we take the projections to the $U(2n)$-factor of $K$. We obtain $\bar K=\bar K_L= U(2n)$, $\bar K_H=\overline{L\cap H}=Sp(n)$, $\bar M_L=U(1)\times U(2n-1)$, 
$\bar Z=U(2)\times U(2n-2)$, $\bar Z_H=Sp(1)\times Sp(n-1)$, $\bar Z_L=\{e\}\times U(1)\times U(2n-2)$. We view $\delta\in \widehat Z_{s}$
as representations of $U(2n-2)\subset \bar Z_L\subset\bar Z$ possesing an $Sp(n-1)$-invariant vector. As such, it remains irreducible. Again, we use the Cartan-Helgason theorem
in order to write down the highest weights of the $U(2n)$-representation $\gamma$ and the $U(2n-2)$-representation $\delta$ explicitly:
\begin{eqnarray*} \mu_\gamma&=& (\gamma_1,\gamma_1,\gamma_2,\gamma_2,\dots,\gamma_{n},\gamma_{n}),\quad \gamma_i\in\Z,\ \gamma_1\ge\gamma_2\ge\dots\ge\gamma_{n}\\
\mu_\delta&=&(\delta_1,\delta_1,\delta_2,\delta_2,\dots,\delta_{n-1},\delta_{n-1}),\quad \delta_i\in\Z,\ \delta_1\ge\delta_2\ge\dots\ge\delta_{n-1}.
\end{eqnarray*}
By the classical branching rules for restricting $U(k)$-representations to $U(1)\times U(k-1)$ and $U(k-1)$ (\cite{Zhe73},\cite{GW98}) the only possibility for the highest weight $\mu_\sigma$ of a $U(1)\times U(2n-1)$-representation $\sigma\in \widehat M_{L}$ with
$\Hom_{U(1)\times U(2n-1)}(V_\gamma, V_{\sigma})\ne\{0\}\ne \Hom_{U(2n-2)}(V_\sigma,U_\delta)$ is
$$ \mu_\sigma=(\gamma_n+\sum_{i=1}^{n-1} \gamma_i-\delta_i;\gamma_1,\delta_1,\gamma_2,\delta_2,\dots,\delta_{n-1},\gamma_{n})\ .$$
Now we can apply Claim 4. If we replace $G$  by a finite covering the argument remains unchanged, only the integrality conditions on $\gamma_i$, $\delta_i$ should be relaxed (while
the differences $\gamma_i-\gamma_{i+1}$ and $\delta_i-\delta_{i+1}$ are still required to be integers). 
\end{proof}

Note that spherical triples $(G,H,L)$ satisfying the assumptions of Cor.~\ref{wels}(c) are (at least) very close to be complex spherical. This might be considered as the main reason for the validity of Lemma \ref{schlips}. In fact, Lemma \ref{schlips} could be deduced from Claim 4 and the main result of \cite{KK19} which provides distinguished generators of the
algebra of invariant differential operators for certain complex spherical homogeneous spaces. Since the proof of the latter result also requires an extensive case by case analysis of branching rules we preferred to present a more direct proof here. The proof of Lemma \ref{schlips} also shows the following interesting fact: Let $\delta_\sigma\in\widehat Z_s$ be
the element distinguished by Cor.~\ref{wels}(a). Then the assignment $\sigma \mapsto (\gamma_\sigma,\delta_\sigma)$ defines a bijection between $\widehat M_{L,s}$ and
$\{(\gamma,\delta)\in \widehat K_{L,s}\times \widehat Z_s\mid \Hom_{Z_L}(V_\gamma,U_\delta)\ne\{0\}\}$.

As a first application of Lemma~\ref{schlips} we give a certain generalization of the multiplicity formula $m_\rho(\pi)=\dim V_\pi^{L\cap H}$ of Theorems \ref{aal} and \ref{mainbranching} to irreducible $G$-representations $\rho$
that are only {\em weakly} spherical and even to the multiplicity as a subquotient.

\begin{cor}\label{fliege} Let $(G,H,L)$ be a spherical triple satisfying the assumptions of Cor.~\ref{wels}(c).  Let $V_\pi$ be an irreducible admissible representation of $L$ with $V_\pi^{L\cap H}\ne\{0\}$. Let $W_\rho$ be an irreducible admissible weakly $H$-spherical representation of $G$. By $\widetilde m_\rho(\pi)$ we denote the multiplicity of $V_{\pi,\infty}$ in $W_{\rho,\infty}$
as a subquotient. We assume that $\widetilde m_\rho(\pi)\ne 0$. Then $\widetilde m_\rho(\pi)=\dim V_\pi^{L\cap H}$ and $W_\rho^{K\cap H}\ne\{0\}$. Moreover, if $m_\rho(\pi)\ne 0$, then $\widetilde m_\rho(\pi)=m_\rho(\pi)$.
\end{cor}

\begin{proof}  Delorme's embedding theorem (for not necessarily irreducible admissible spherical representations) \cite{Del85}, Thm.~1, implies that $W_{\rho,\infty}$ appears as an irreducible subquotient of an induced representation
$I^{\delta,\lambda}_\infty$. Hence 
\be\label{malibu}
\widetilde m_\rho(\pi)\le\dim\Hom_{K_L}(V_{\gamma_\pi}, I^\delta)\ . 
\end{equation}
Here $\gamma_\pi$ is  as in Cor.~\ref{wels}(c). More precisely, if $L=L_{min}$ and $\widetilde m_\rho(\pi)\ne 0$, then $\widetilde m_\rho(\pi)$ is the dimension of a non-trivial $C_2$-invariant subspace of $\Hom_{K_L}(V_{\gamma_\pi}, I^\delta)$. By Lemma~\ref{schlips}(b)
we get equality in (\ref{malibu}). Now it is easy to conclude equality for arbirtrary $L$, whenever $\widetilde m_\rho(\pi)\ne 0$. Again by  Lemma~\ref{schlips}(b) and Cor.~\ref{wels}(c)
we obtain $\widetilde m_\rho(\pi)=\dim V_{\gamma_\pi}^{L\cap H}=\dim V_\pi^{L\cap H}$. The same arguments also show that $m_\rho(\pi)\ne 0$ implies 
$m_\rho(\pi)=\dim V_\pi^{L\cap H}$. By Lemma~\ref{schlips}(c) $\Hom_{K_L}(V_{\gamma_\pi}, I^\delta)$ factors through a $K\cap H$-spherical $K$-type $\tau=\tau_{\gamma_\pi}$.
It follows that $W_\rho(\tau)\ne\{0\}$, and hence $W_\rho^{K\cap H}\ne\{0\}$.
\end{proof}

As a second application we will get some information on $L$-minimal $G$-representations and their unique irreducible quotients.
By $W_{\rho^*}$ we denote the conjugate dual of a representation $W_\rho$ on a reflexive Banach space. Note that $W_{\rho,\infty}\cong W_{\rho^*,\infty}$ if and only if $W_{\rho,\infty}$ carries an invariant non-degenerate
(not necessarily definite) continuous Hermitian form. For admissible representations of finite length the latter property only depends on the underlying Harish-Chandra module.

\begin{pro}\label{dark}
Let $(G,H,L)$ be a triple of Type~I. Let $V_\pi$ be an irreducible admissible representation of $L$ with $V_\pi^{L\cap H}\ne\{0\}$, and let $W_\rho$ be the $\pi$-minimal $H$-spherical 
$G$-representation with unique irreducible quotient $W_{\rho_0}$. Then 
\begin{itemize}
\item[(i)]$W_{\rho_0}^{K\cap H}\ne\{0\}$.
\item[(ii)] The natural map $\Hom_L(V_{\pi,\infty},W_{\rho,\infty})\rightarrow \Hom_L(V_{\pi,\infty},W_{\rho_0,\infty})$ is an isomorphism.
\item[(iii)] Assume that $V_{\pi,\infty}\cong V_{\pi^*,\infty}$. Then $W_{\rho_0,\infty}\cong W_{\rho_0^*,\infty}$.
\end{itemize}
\end{pro}

\begin{proof}
If $X=G/H$ is a group manifold, then all three assertions are trivially satisfied. Indeed, $L$-minimal representations are (up to infinitesimal equivalence) of the form $W_\rho=\End_{HS}(V_\pi)\cong V_\pi\hat\otimes V_{\tilde\pi}$, where $HS$ stands for Hilbert-Schmidt and $V_\pi$ is an irreducible admissible Hilbert representation of $L_{min}$. In particular, $W_\rho$ is irreducible.
Therefore we can reduce the assertion to triples $(G,H,L)$ such that $X=G/H$ has no group factors. So we may assume that $(G,H,L)$ is a spherical triple satisfying the assumptions of Cor.~\ref{wels}(c). In this case Assertion (i) is already covered by Cor.~\ref{fliege}. Concerning Assertion (ii) we remark that the injectivity of the map in question
is a direct consequence of Def.~\ref{spass}. Moreover, by Thm.~\ref{aal} and Cor.~\ref{fliege} it is map between vector spaces of the same finite dimension.

We now prove Assertion (iii). By a further reduction we may assume that $G$ is simple. If $V_\pi$ is a discrete series representation, then by Prop.~\ref{fishy} the $\pi$-minimal representation $W_\rho$ is a relative discrete series,
in particular it is irreducible and unitary. Similarly, if $V_\pi$ is tempered, but not square integrable, then by Prop.~\ref{fishy} the representation $W_\rho$ is relatively tempered, but
not a relative discrete series. By \cite{Del85}, Thm.~2, the representation $W_{\rho,\infty}$ allows a non-trivial morphism to some $I^{\delta,\lambda}_\infty$ with $\mathrm {Re}(\lambda)=0$. Here we have used that $X=G/H$ has real rank one. In particular, $W_{\rho_0,\infty}$ is an irreducible subquotient, hence a direct summand, of the unitarizable representation $I^{\delta,\lambda}_\infty$. 

Now we assume that $V_\pi$ is non-tempered and admits an invariant Hermitian form (defined on smooth vectors). Since $L$ has real rank one, by Langlands
classification $V_{\pi,\infty}$ is the unique irreducible submodule of some $H^{\sigma,\lambda}_\infty$ with $\mathrm{Re}(\lambda)<0$. Note that this embedding is given by taking the leading coefficients of the asymptotic expansion of matrix coefficients of $V_\pi$. The existence of an invariant Hermitian form implies that
$\lambda$ is real, and that $w\sigma\cong\sigma$, where $w$ is the nontrivial element of the Weyl group $W(\fl,\fa_L)$, see e.g. \cite{Kn86}, Thm.~16.6 and the remark following it.
In our real rank one situation, that $(G,H,L)$ is spherical just means that $L\cap H$ acts transitively on the unit sphere in $\fl\cap\fs$. In particular, $w$ can be represented by an element 
$l_w\in L\cap H$. Since $p_\fq (\fa_L)=\fa_{\fq\cap\fs}$ we see that $\Ad(l_w)$ acts also on $\fa_{\fq\cap\fs}$ by multiplication by $-1$. Thus $l_w$ represents also
the nontrivial element $w'\in W(\fg,\fa_{\fq\cap\fs})$. Now let $\delta\in\widehat Z_s$ be the element associated to $\sigma\in\widehat M_{L,s}$ by Cor.~\ref{wels}(a).
It follows that then also $w'\delta\cong\delta$. We choose $0\ne\Phi \in  \Hom_{(\fl,K_L)}(V_{\pi,K_L},W_{\rho,K})=  \Hom_{L}(V_{\pi,\infty},W_{\rho,\infty})$.
Using the relation between matrix coeffiicients (\ref{dadamax}) and Cor.~\ref{wels}(a) we find that there is a non-trivial intertwining operator (again given by asymptotics of matrix coefficients) $W_{\rho,\infty}\rightarrow I^{\delta,\lambda}_\infty$. Formally, this could be done by looking at the $Z_LA_L$-map induced by $\Phi$ between the $\fn$-homology
groups $V_{\pi,\infty}/\fn_L V_{\pi,\infty}$ and $W_{\rho,\infty}/\fn_{\sigma\theta}W_{\rho,\infty}$. We find that $W_{\rho_0,\infty}$ can be identified with an irreducible subquotient
of $I^{\delta,\lambda}_\infty$. We have $(I^{\delta,\lambda})^*\cong I^{\delta,-\lambda}$. Since $I^{\delta,-\lambda}\cong I^{w'\delta,w'\lambda}$ the induced representations
$I^{\delta,\lambda}_\infty$ and $I^{\delta,-\lambda}_\infty$ have the same character, and hence the same irreducible subquotients. It follows that also $W_{\rho_0^*,\infty}$ can
be realized as a subquotient of $I^{\delta,\lambda}_\infty$. Note that $W_{\rho,\infty}$ and $W_{\rho^*,\infty}$ are isomorphic as representations of $K$. Now, by Assertion (i) both subquotients
have a common spherical $K$-type. But by Lemma~\ref{schlips}(a) spherical $K$-types have multiplicity one in $I^{\delta,\lambda}_\infty$. We conclude that $W_{\rho,\infty}=W_{\rho^*,\infty}$
(as subquotients of $I^{\delta,\lambda}_\infty$).
\end{proof}

Now we give examples of irreducible unitary representations $V_\pi$ of $L$ such that the corresponding $\pi$-minimal $G$-representation
is not unitary, in some cases even not irreducible. Their existence is the main reason for discussing also non-unitary $G$-representations in this section.

\begin{pro}\label{kopp} Let $(G,H,L)=(SO_e(2,2n), SO_e(1,2n), U(1,n))$. Let $\pi\in\widehat L$ be such that $V_\pi^{L\cap H}\ne\{0\}$. Then, following the notation of  
Section~\ref{detail}, we have $\pi=\pi^k_\lambda$, for some $k\in\Z$, $\lambda\in i[0,\infty)\cup Q_k$. Let $W_\rho$ be the $\pi$-minimal $H$-spherical representation of $G$. Then $W_\rho$ is not unitary if and only if 
\be\label{0kopp}|k|< n, \ \lambda\ne n\end{equation} 
and either 
\be \label{1kopp}k\equiv n (2)\mbox{ and }\lambda\in [0,n-|k|]\end{equation} or 
\be\label{2kopp}k\not\equiv n (2)\mbox{ and }
\lambda\in [1,n-|k|]\ .\end{equation} 
$W_\rho$ is not irreducible if and only if in addition to the above conditions we have $\lambda\in\N_0$, $\lambda\equiv n-|k| \, (2)$.
In the reducible case the irreducible quotient $W_{\rho_0}$ is unitary if and only if $\lambda<n-|k|$, i.e. $\pi^k_\lambda$ is not the end of a complentary series.
\end{pro}

\begin{proof} Let $\pi=\pi^k_\lambda$ be non-trivial and non-tempered (which implies $\lambda\in Q_k$, $\lambda\ne n$)  such that $W_\rho$ is unitary. According to Thm.~\ref{mainbranching} and Prop.~\ref{fishy} there exists an infinite subset 
$$D\subset D_{k,\lambda}:=\{m\in\Z\mid m\equiv k\, (2),\: \lambda\in Q_m,\: \pi_\lambda^m \mbox{ non-tempered}\}$$ 
such that 
$$\rho\cong \widehat{\bigoplus_{m\in D}} \pi_\lambda^m\ .$$
Here we have taken the action of the Casimir $\Omega_G$ and of the element $-\id\in L\subset G$ on $W_\rho$ into account.
On the other hand, the description of $Q_m$ in Prop.~\ref{crux} shows that the sets $D_{k,\lambda}$ are infinite only if $\lambda\in (0,1)$, $k\not\equiv n\,(2)$. This shows
that for $\lambda\ne 0$  the representation $W_\rho$ is not unitary under the Conditions (\ref{0kopp})--(\ref{2kopp}). 

For the case $\lambda=0$ we need a refinement of Prop.~\ref{fishy}.
Note that the representations $\pi_0^k$ are always tempered and that temperedness is equivalent to estimates of $K_L\times K_L$-finite matrix coefficients of the form
$$   |c_{v,\tilde v}(\exp X)|\le C_\ve e^{-(n-\ve)|X|},\quad X\in\fa_L, \ve>0.$$
The exponent $n$ corresponds to $\rho_L\in \fa_L^*$. We call the representation $\pi_0^k$ strongly tempered if the above inequality also holds for $\ve=0$.
Now let  $\pi=\pi^k_0$ be not strongly tempered such that $W_\rho$ is unitary. Theorem~\ref{mainbranching} together with an argument analogous to the proof of Prop.~\ref{fishy} shows that there exists an infinite subset 
$$D\subset D_{k,0}:=\{m\in\Z\mid m\equiv k\, (2),\:  \pi_0^m \mbox{ not strongly tempered}\}$$
such that 
$$\rho\cong \widehat{\bigoplus_{m\in D}} \pi_0^m\ .$$
From the relation between $c$-functions, Knapp-Stein intertwining operators, and the asymptotics of matrix coefficients we see  that  $\pi^m_0$
is not strongly tempered if and only if the family of Knapp-Stein intertwining operators associated to the principal series representations $H^{m,\lambda}$ has a pole  at $\lambda=0$.
The latter condition is equivalent to the irreducibility of $H^{m,0}$
(see e.g. \cite{Kn86}, Thm.~14.16). Thus $\pi^m_0$
is not strongly tempered if and only if $H^{m,0}$ is irreducible. As already mentioned before Prop.~\ref{crux}, this is precisely the case when $m\not\equiv n\, (2)$ or $|m|<n$. 
Thus $D_{k,0}$ is infinite only if $k\not\equiv n\, (2)$. We conclude
that also for $\lambda= 0$  the representation $W_\rho$ is not unitary under the Conditions (\ref{0kopp}) and (\ref{1kopp}). 

Now we assume that $(k,\lambda)$ does not satisfy the conditions (\ref{0kopp})--(\ref{2kopp}). We want to show that $W_\rho$ is unitary. This is clear if $\pi_\lambda^k$ is trivial
or a discrete series representation. In fact, in the latter case $W_\rho$ has to be a relative discrete series representation by Prop.~\ref{fishy} which is unitary.

The unitarity of $W_\rho$ for the remaining $L$-representations as well as the assertions about irreducibility of $W_\rho$ will not follow directly from  Prop.~\ref{fishy}.
We have bring the induced representations $I^{\delta,\lambda}$ of $G$ into play. Note that $Z=\{\pm\id_{2n+2}\}\cdot SO(2n-1)$. $\widehat Z_s$ consists of 2 characters indexed by $[k]$
in $\Z/2\Z$.  We get an embedding of principal series
representations $H^{k,\lambda}\hookrightarrow  I^{[k],\lambda}$. 
Since $ (I^{[k],-\lambda}_{-\infty})^H\ne\{0\}$  at least one irreducible subquotient of  $ I^{[k],\lambda}$ is $H$-spherical.
In addition, if $\lambda\in i(0,\infty)$ or if $k\not\equiv n\, (2)$ and $\lambda\in [0,1)$, then  $I^{[k],\lambda}$ is irreducible and unitary.
Moreover,  $ I^{[k],\lambda}$  is reducible if and only if $\lambda\in\Z$ and $\lambda\equiv n-|k| \, (2)$ (see \cite{HoTa93}). Thus the representation 
$ I^{[k],\lambda}$ is an irreducible unitary $\pi^k_\lambda$-minimal $G$-representation in the first cases and still gives an ireducible $\pi^k_\lambda$-minimal $G$-representation if only the integrality conditions are not satisfied.

In case of reducibility, the irreducible subquotients of $I^{[k],\lambda}$ can be described completely in terms of their $K$-types, see \cite{HoTa93}, pp.~17-18 (one has to 'fold out' the diagrams, see pp.~20--22 of that paper.) In fact, they can be characterized in terms
of their $K_H$-spherical $K$-types. Since $K=SO(2)\times SO(2n)$ these are given by characters $\chi_l$, $l\in\Z$, of $SO(2)\cong U(1)$. Note that $I^{[k]}(\chi_l)\ne\{0\}$ if and only if $l\equiv k   \, (2)$. We are interested in the case $|\lambda|<n$,
only. In this case, $I^{[k],\lambda}$ has precisely $5$ ($3$ for $\lambda=0$ and $k\equiv n\, (2)$) irreducible subquotients: $W_0, W_1^\pm, W_2^\pm$. $W_0$ contains the characters $\chi_l$ with 
$|l|< n-|\lambda|$, $W_1^\pm$ the ones with $n-|\lambda|\le\pm l<n+|\lambda|$, $W_2^\pm$ the ones with $\pm l \ge n+|\lambda|$ ($W_1^\pm=\{0\}$ for $\lambda=0$).
These modules are infinite dimensional. The only subquotients that contain infinitely many characters are $W^\pm_2$. Now it follows from Cor.~\ref{frank} that at least one of these 2 modules, but no other subquotient, is spherical. Since the subquotients $W^\pm_2$ are conjugate via the outer automorphism $\sigma$ of $G$ (fixing $H$), both are spherical.
We also know that $W_2^\pm$ and $W_0$ are unitarizable, $W_1^\pm$ are not (\cite{HoTa93}, p.~30).

Next we discuss the $L$-representations $\pi=\pi^k_0$
with $|k|\ge n$, $k\equiv n\, (2)$, i.e. $\pi$ is a limit of discrete series representations (and is strongly tempered in the sense introduced above). As an irreducible constituent of $H^{k,0}$
the representation $\pi^k_0$ embeds into the unitary $G$-representation $I^{[k],0}$. It contains the $K_L$-type $\gamma_k$ (in the notation of Section~\ref{detail}). It follows that its image in $I^{[k],0}$ contains the character
$\chi_k$, cf. Cor.~\ref{schlips}(c). By the discussion in the previous paragraph it is contained in a spherical irreducible subquotient (namely $W_2^+$ or $W_2^-$) which is in fact an irreducible unitary direct summand giving the $\pi$-minimal $G$-representation. 

Now we show that for $(k,\lambda)$ satisfying the Conditions (\ref{0kopp})--(\ref{2kopp})  with $\lambda\in\N_0$ and 
$\lambda\equiv n-|k| \, (2)$, 
the $\pi^k_\lambda$-minimal $G$-representation $W_\rho$ is reducible. Assume it were irreducible. Then by Delorme's embedding theorem the $G$-representation $W_{\rho,\infty}$ (and therefore also $\pi^k_\lambda$) embeds into $I^{[k],\lambda}$ or  $I^{[k],-\lambda}$ 
(here we have also taken the action of the Casimir and of the center of $G$ into account). As in the previous paragraph we find that the image of $W_\rho$ is a spherical irreducible submodule of $I^{[k],\pm\lambda}$ containing the $K$-character $\chi_k$. Since $|k|<n\le n+\lambda$ this is impossible.

It remains to discuss unitarity of $W_{\rho_0}$ in the reducible situation. We find that $W_{\rho_0}$ is an irreducible subquotient of  $I^{[k],\pm\lambda}$ containing the $K$-character $\chi_k$.
By the above discussion and since $|k|<n$  this subquotient is unitarizable if and only if $|k|<n-\lambda$.
\end{proof}

Analogous lists of unitary $L$-representations $\pi$ having a non-unitary $\pi$-minimal $G$-representation for the other rank one triples are contained in Tables~\ref{super}--\ref{tauto} 
in Section~\ref{eigen}.


\section{Eigendistributions on $\Gamma\backslash G/H$ and $G$-representations}\label{eigen}


The goal of this final section is to bring representation theory of $G$ into the play of the spectral theory on $\Gamma\backslash G/H$. Using the results of the previous
section, this will provide us with interesting reinterpretations and refinements of the results for triples of Type~I obtained in Sections~6--8, whereas for triples of Type~II this opens a new perspective how to understand the expected continuous spectrum.

\subsection{General considerations.}\label{gc} 
Let $X=G/H$ be a semisimple symmetric space, and let $\chi\in\Hom({\bf D}(G/H),\C)$. We consider the eigenspace representations $E_\chi^{\pm\infty}(X)$ of $G$ and their
underlying $(\fg,K)$-module $E_\chi^{\infty}(X)_K=E_\chi^{-\infty}(X)_K$  which is known to be a Harish-Chandra module (\cite{Ban87}, Cor.~3.10).
As a starting point we want to describe its canonical smooth and distribution globalizations, resp., in the sense of Casselman and Wallach. We will conclude that $\Gamma$-invariant
eigendistributions always belong to this distribution globalization, whenever $\Gamma\subset G$ is a discrete subgroup acting properly and cocompactly on $X$.

Our fixed invariant bilinear form together with the choice of a ($\sigma$-stable) maximal compact $K$  provides a norm on $G$ via
$$  \|k\exp(Z)\|:=e^{|Z|}\ , k\in K,\ Z\in\fs.$$
It induces a kind of norm on $X=G/H$ via
$$ \|gH\|_X:=\inf_{h\in H}\|gh\|\ \left (=e^{|Z|}, \mbox{ if } g=k\exp(Z) h \mbox{ for some }  k\in K,\ Z\in\fs\cap \fq,\ h\in H\right ).$$
It is a consequence of non-positive curvature of the Riemannian symmetric space $K\backslash G$ that there is an $s>0$ such that for all $g=k\exp(Z) h\in G$ as above
\be\label{pontius}
\|k\exp(Z) h\|\ge \|gH\|_X^s \|h\|^s\ .
\end{equation}
Similarly, for a properly transitive triple $(G,H,L)$, there is a (different) constant $s>0$ such that for all $l\in L$
\be\label{fixus}   \|l\|^s\le \|lH\|_X\le \|l\|\ .
\end{equation}
The above norms give rise to Schwartz-like spaces of smooth rapidly decreasing functions
$$ {\Cal S}(G):=\{f\in C^\infty(G)\mid \forall Z\in {\Cal U}(\fg), s\in \R,\ \exists C: | l_Zf(g)|\le C \|g\|^{-s}\}$$
and 
$$ {\Cal S}(X):=\{f\in C^\infty(X)\mid \forall Z\in {\Cal U}(\fg), s\in \R,\ \exists C: | l_Zf(x)|\le C \|x\|_X^{-s}\}\ .$$
Inequality (\ref{pontius}) implies that there is a well-defined surjective map ${\Cal S}(G)\ni f\mapsto \bar f\in  {\Cal S}(X)$ given by
$$ \bar f (gH):=\int_H f(gh)\:dh\ .$$
Moreover, it
follows from (\ref{fixus}) that for a properly transitive triple $(G,H,L)$
\be\label{pilatus}  {\Cal S}(X)\cong {\Cal S}(L)^{L\cap H}\ .
\end{equation}
We also consider the spaces of smooth functions of uniform moderate growth
$$ C^\infty_{umg}(X):=\{f\in C^\infty(X)\mid \exists r\in \R\ \forall Z\in {\Cal U}(\fg)\ \exists C: | l_Zf(x)|\le C \|x\|_X^{r}\}$$
and of moderate distributions
$$ {\Cal S}'(X):=({\Cal S}(X))'\subset C^{-\infty}(X)$$
on $X$. Here dualization is with respect to the natural 
Fr\' echet topology on ${\Cal S}(X)$.

\begin{pro}\label{englisch}
Let $W_\rho$ be a continuous $G$-representation on a reflexive Banach space such that $E^{\infty}_\chi(X)_K\subset W_\rho\subset E_\chi^{-\infty}(X)$.
Then $$W_{\rho,\infty}=E^{\infty}_\chi(X)_{umg}:=E^{\infty}_\chi(X)\cap C^\infty_{umg}(X)$$ 
and
$$W_{\rho,-\infty}=E^{-\infty}_\chi(X)_{mod}:=E^{-\infty}_\chi(X)\cap {\Cal S}'(X)\ .$$
Let $\Gamma\subset G$ be a discrete subgroup acting properly and cocompactly on $X=G/H$, and let $Y:=\Gamma\backslash X$. Then
\be\label{klaus1} E^{-\infty}_\chi(Y)=E^{-\infty}_\chi(X)^\Gamma=(E^{-\infty}_\chi(X)_{mod})^\Gamma\ .\nonumber
\end{equation}
\end{pro}

\begin{proof} The first two assertions are direct consequences of Casselman/Wallach theory (\cite{Cas89b}, \cite{Wa92}, Ch.~11, see also \cite{BK14}) and are probably well-known to experts. Therefore we are rather sketchy here. The theory of asymptotic expansions of elements in $E^{\infty}_\chi(X)_K$ shows that $E^{\infty}_\chi(X)_K\subset E^{\infty}_\chi(X)_{umg}$.
On the other hand, it is easily checked that $E^{\infty}_\chi(X)_{umg}$ is a smooth Fr\' echet representation of moderate growth in the sense of Casselman/Wallach. Since there is only
one such representation (up to equivalence) globalizing a given Harish-Chandra module (\cite{Wa92}, Thm.~11.6.7) the first assertion follows.

For the second assertion we recall that the natural actions of $C_c^\infty(G)$ on $W_\rho$ and $W_{\tilde\rho,\pm\infty}$ extend naturally to actions of ${\Cal S}(G)$.
Moreover, any element of ${\Cal S}(G)$ maps $W_{\tilde\rho,-\infty}$ to $W_{\tilde\rho,\infty}$ (this is completely analogous to Lemma~\ref{stock}, (iii)). The delta distribution 
$\delta_{eH}$
supported at the origin of $X$ defines an element of $W_{\tilde\rho,-\infty}$. Hence for any $f\in  {\Cal S}(G)$ the element $\bar f =\tilde\rho(f)\delta_{eH}\in  {\Cal S}(X)$ (see (\ref{pontius}))
defines an element of $W_{\tilde\rho,\infty}$. We obtain a map ${\Cal S}(X)\rightarrow W_{\tilde\rho,\infty}$. A density argument shows that its image contains $W_{\tilde\rho,K}$.
The closed range result \cite{Wa92}, Thm.~11.8.2 now implies that this map is surjective. Dualizing we obtain that $W_{\rho,-\infty}$ is a closed subspace of ${\Cal S}'(X)$, namely 
$E^{-\infty}_\chi(X)\cap {\Cal S}'(X)$.

Concerning the last assertion it suffices to show that $C^{-\infty}(Y)=C^{-\infty}(X)^\Gamma\subset {\Cal S}'(X)$.
Let us first discuss the case that $\Gamma$ comes from a properly transitive triple $(G,H,L)$, i.e. $\Gamma$ is a cocompact lattice of $L$.
It is well-known (\cite{Cas89}) that the push down map
${\Cal S}(L)\rightarrow C^\infty(\Gamma\backslash L)$ given by averaging over $\Gamma$ is well-defined, continuous, and surjective. By (\ref{pilatus}) we obtain an analogous map
${\Cal S}(X)\rightarrow C^\infty(Y)$. Dualizing we obtain the desired inclusion.

For a general discrete $\Gamma\subset G$ acting properly and cocompactly on $X$ we use a recent result of Kassel and Tholozan \cite{KT24} stating that all
these groups are {\em sharp}. We do not  want to state the precise definition of sharpness here, but only the following easy consequence of it (which can be derived using e.g. \cite{KK2}, Lemma~4.4 and Lemma~4.17): For every compact subset $\Omega\subset X$ there are constants $c,s_0>0$ (we can take $s_0$ independent of $\Omega$) such that for all $x\in\Omega$ and
$\gamma\in \Gamma$
$$ \|\gamma x\|_X\ge c\|\gamma\|^{s_0}\ .$$
Using that for every discrete subgroup $\Gamma\subset G$ the Poincar\' e series $\sum_{\gamma\in\Gamma} \|\gamma\|^{-s}$ converges  for sufficiently large $s$ we
easily conclude that the push-down map ${\Cal S}(X)\rightarrow C^\infty(Y)$ given by averaging over $\Gamma$ is again well-defined, continuous, and surjective. 
\end{proof}

In the remainder of this section we will switch the roles of $\rho$ and $\tilde\rho$. We call an admissible $G$-representation $W_\rho$ of finite length on a reflexive Banach space $H$-cospherical, if 
its dual $W_{\tilde\rho}$ is spherical. This is equivalent to: The $G$-representation $W_{\rho,-\infty}$ has an $H$-invariant cyclic vector, see the discussion at the beginning of Section \ref{branching}. If $W_\rho$ is unitary, then $W_{\tilde\rho}$ is equivalent to the conjugate of $W_\rho$, and hence $W_\rho$ is cospherical if and only if it is spherical.
The last assertion also holds for not necessarily unitary {\em irreducible} admissible $W_\rho$ - this fact lies much deeper, see \cite{Bien90}, p.~39.

For simplicity, from now on we will assume that the natural map ${\Cal Z}(\fg)\rightarrow {\bf D}(G/H)$ is surjective. This is justified by the fact that the exceptional symmetric spaces $G/H$
not having this property do not appear as the underlying symmetric space of  a properly transitive triple $(G,H,L)$, see Thm.~\ref{listck} and the second comment following it.
The assumption implies that for every $H$-cospherical representation $W_\rho$ having an infinitesimal character the algebra ${\bf D}(G/H)$ acts on $(W_{\rho,-\infty})^H$ by a 
single character which is uniquely
determined by the infinitesimal character of $W_\rho$. Note that by Delorme's embedding theorem the set of infinitesimal characters of $H$-cospherical representations coincides with the set of infinitesimal characters of the induced representations $I^{\delta,\lambda}$ introduced before Cor.~\ref{wels}. This determines a {\em proper} subset of $\Hom({\bf D}(G/H),\C)$ whenever
$\rank_\R(G/H)<\rank(G/H)$. Let $\widehat G_H\subset \widetilde G_H$ be the sets of (infinitesimal) equivalence classes of irreducible unitary (resp. admissible) $H$-(co)spherical representations of $G$.
Our assumption provides us with a well-defined finite-to-one map $p: \widetilde G_H\rightarrow\Hom({\bf D}(G/H),\C)$. For $\chi\in\Hom({\bf D}(G/H),\C)$ we consider the finite set $\widetilde G_{H,\chi}:=p^{-1}(\{\chi\})$. Below we call a character $\chi\in\Hom({\bf D}(G/H),\C)$ generic if it is generic among the infinitesimal characters of the induced representations $I^{\delta,\lambda}$.

\begin{lem}\label{otto}
Let $\Gamma\subset G$ be a discrete subgroup acting properly and cocompactly on $X=G/H$, and let $Y:=\Gamma\backslash X$. Let $\chi\in\Hom({\bf D}(G/H),\C)$. 
Let $W_\rho$ be an $H$-cospherical representation with infinitesimal character determined by $\chi$. Then the matrix coefficient map (see Section \ref{dist})
$$  W_{\tilde\rho,-\infty}\otimes W_{\rho,-\infty}\longrightarrow C^{-\infty}(G)\ ,\quad \tilde w\otimes w\mapsto c_{\tilde w,w}$$
descends to a map
\be\label{paul}
 (W_{\tilde \rho,-\infty})^\Gamma\otimes (W_{\rho,-\infty})^H\longrightarrow  E^{-\infty}_\chi(Y)\ .
\end{equation}
Moreover, we have
\begin{itemize}
\item[(i)] If $W_\rho$ is the dual of the eigenspace representation considered in Prop.~\ref{englisch}, then the map (\ref{paul}) is surjective. In particular, every eigendistribution on
$Y$ is a distributional matrix coefficient.
\item[(ii)] The map (\ref{paul}) provides an isomorphism of
$$ \bigoplus_{\rho\in \widetilde G_{H,\chi}} (W_{\tilde \rho,-\infty})^\Gamma\otimes (W_{\rho,-\infty})^H$$
onto a closed subspace of  $E^{-\infty}_\chi(Y)$ which coincides with $E^{-\infty}_\chi(Y)$ for generic $\chi$.
\end{itemize}
\end{lem} 
\begin{proof} The first assertion is obvious. Item (i) follows from the fact that any $f\in {\Cal S}'(X)$ can be tautologically written as a matrix coeffcient: $f=c_{f,\delta_{eH}}$.
For (ii) we consider the {\em socle} of the admissible representation of finite length $E^{-\infty}_\chi(X)_{mod}$, i.e. the sum of all its irreducible $G$-subrepresentations (on closed subspaces). It is itself a closed subspace of $E^{-\infty}_\chi(X)_{mod}$ and can be decomposed into isotypic components
$$   \bigoplus_{\rho\in \widetilde G_{H,\chi}} \Hom_G(W_{\tilde \rho,-\infty}, E^{-\infty}_\chi(X)_{mod})\otimes W_{\tilde \rho,-\infty}\ ,$$
and there is a Frobenius reciprocity type isomorphism of the first factor to $(W_{\rho,-\infty})^H$, cf. Prop.~\ref{kunst} and Prop.~\ref{englisch}.
Using the asymptotic expansion of $K$-finite eigenfunctions on $X$ it is not difficult to see that this socle coincides with the whole eigenspace for generic $\chi$, namely for all those
$\chi$ such that all the induced representations $I^{\delta,\lambda}$ having infinitesimal character corresponding to $\chi$ are irreducible. The lemma follows by taking $\Gamma$-invariants.
\end{proof}
A decomposition as in Assertion (ii) of the lemma is only useful if one has some understanding of the spaces of $\Gamma$-invariants $(W_{\tilde \rho,-\infty})^\Gamma$, $\rho\in \widetilde G_{H,\chi}$. This is difficult to achieve for general $\Gamma$. But for uniform lattices $\Gamma\subset L$ , $L\subset G$ connected, the classical decomposition formula (\ref{vangogh})  for $C^\infty(\Gamma\backslash L)$ 
gives some information. 
Note that we have the following description via Frobenius reciprocity of the multiplicity spaces $N_\Gamma(\pi)$, $\pi\in\hat L$, appearing in Lemma~\ref{bassa} (here we use essentially the compactness of $\Gamma\backslash L$):
$$ N_\Gamma(\pi)=\Hom_L(V_\pi, L^2(\Gamma\backslash L)=\Hom_L(V_{\pi,\infty}, C^\infty(\Gamma\backslash L))\cong \Hom_\Gamma(V_{\pi,\infty}, \C)
=  (V_{\tilde \pi,-\infty})^\Gamma\ .
$$
In the following we will identify $N_\Gamma(\pi)$ with $ (V_{\tilde \pi,-\infty})^\Gamma$. We also consider consider the space of intertwining operators $\Hom_L(W_{\rho,\infty}, V_{\pi,\infty})$. From Lemma \ref{stock} (iv) we conclude
that taking adjoints $\Phi\mapsto \Phi^t$ provides an isomorphism
\begin{equation}\label{putzi}   \Hom_L(W_{\rho,\infty}, V_{\pi,\infty})\cong \Hom_L( V_{\tilde\pi,-\infty},W_{\tilde\rho,-\infty})\ .
\end{equation}
Recall from Lemma~\ref{sardine} the role of the space on the right hand side as a multiplicity space for the restriction of $W_{\tilde\rho}$ to $L$.
The following lemma is independent of the theory of symmetric spaces and $H$-(co)spherical representations.

\begin{lem}\label{kirchner}
Let $\Gamma\subset L\subset G$ be as above. Let $W_\rho$ be an admissible $G$-representation of finite length on a reflexive Banach space. Then
$$  (W_{\tilde \rho,-\infty})^\Gamma\cong \overline{\bigoplus_{\pi\in\widehat L}}  \Hom_L(W_{\rho,\infty}, V_{\pi,\infty})\otimes N_\Gamma(\pi)\ ,$$
where the embeddings $i_\pi : \Hom_L(W_{\rho,\infty}, V_{\pi,\infty})\otimes N_\Gamma(\pi)\rightarrow (W_{\tilde \rho,-\infty})^\Gamma$ are given by 
$$i_\pi(\Phi\otimes\tilde v)= \Phi^t\tilde v,\quad \Phi\in \Hom_L(W_{\rho,\infty}, V_{\pi,\infty}), \tilde v\in (V_{\tilde\pi,-\infty})^\Gamma.$$
\end{lem}
\begin{proof} By Frobenius reciprocity we have $(W_{\tilde \rho,-\infty})^\Gamma\cong \Hom_L(W_{\rho,\infty},C^\infty(\Gamma\backslash L))$. We now consider the decomposition  (\ref{vangogh}) $C^\infty(\Gamma\backslash L))\cong \overline{\bigoplus}_{\pi\in\widehat L}
 V_{\pi,\infty}\otimes N_\Gamma(\pi)$ with its defining maps $j_\pi$, $q_\pi$, and $Q_J$ for $J\subset\widehat L$. We can choose $j_\pi$ such that $j_\pi(v\otimes \tilde v)=c_{\tilde v, v}$,  $v\in V_{\pi,\infty}$, $\tilde v\in (V_{\tilde\pi,-\infty})^\Gamma=N_\Gamma(\pi)$.  We define continuous maps 
$$i_\pi: \Hom_L(W_{\rho,\infty}, V_{\pi,\infty}\otimes N_\Gamma(\pi))\rightarrow\Hom_L(W_{\rho,\infty},C^\infty(\Gamma\backslash L))\ ,$$
$p_\pi$ in the other direction, and
$$ P_J: \Hom_L(W_{\rho,\infty},C^\infty(\Gamma\backslash L))\rightarrow\Hom_L(W_{\rho,\infty},C^\infty(\Gamma\backslash L))
$$ 
by composition: $i_\pi(\Phi):=j_\pi\circ\Phi$, $p_\pi(\Psi):=q_\pi\circ\Psi$, $P_J(\Psi):=Q_J\circ\Psi$. Composing with Frobenius reciprocity we consider $i_\pi$, $p_\pi$, $P_J$ also as operators
decomposing  $(W_{\tilde \rho,-\infty})^\Gamma$. These operators satisfy the algebraic properties required for such a decomposition. We have to show that in the situation of  Def.~\ref{kobold}(ii) the
Fourier series $\sum_n P_{J_n}(\tilde w)$ converges to $P_J(\tilde w)\in (W_{\tilde \rho,-\infty})^\Gamma$ (in the strong dual topology). From the convergence of the Fourier
series $\sum_n Q_{J_n}(f)$ in $C^\infty(\Gamma\backslash G)$ we immediately get the convergence in the weak$^*$ topology. Now we use admissibility and finite length
of $W_\rho$. By Casselman/Wallach theory $W_{\rho,\infty}$ can be realized as a closed subspace of the space of smooth sections sections of a vector bundle over $G/P$, $P\subset G$
being a minimal parabolic. It follows that $W_{\rho,\infty}$ is a Montel space. Strong convergence follows, see the remark preceding Lemma \ref{deko}. 
\end{proof}

Note that the validity of the above lemma is independent of the behaviour of the $G$-representation $W_{\tilde\rho}$ when restricted to $L$. It is the discrete decomposablity of
$C^\infty(\Gamma\backslash L)$ that matters here. However, if we have discrete decomposability of $W_{\tilde\rho}$ as an $L$-representation as in Thm.~\ref{mainbranching}, then
an analogous decomposition of  $(W_{\tilde \rho,-\infty})^\Gamma$ holds for all discrete subgroups $\Gamma\subset L$, not only for cocompact ones, cf. Lemma~\ref{deko}.

Let $(G,H,L)$ be a properly transitive triple. 
\begin{pro}\label{paulklee}
Let $W_\rho$ and $V_\pi$ be continuous representations on reflexive Banach spaces of $G$ and $L$, respectively.
Then there is a natural ${\bf D}(G/H)$-equivariant map 
\be\label{thomas}   E: \Hom_L(W_{\rho,\infty},V_{\pi,\infty})\otimes (W_{\rho,-\infty})^H\longrightarrow  (V_{\pi,-\infty})^{L\cap H}\end{equation}
characterized by the following equality of distributional matrix coefficients viewed as distributions on $X$:
$$  c^L_{\tilde v,E(\Phi\otimes w)} =c^G_{\Phi^t\tilde v, w},\quad \tilde v\in V_{\tilde\pi,-\infty},\  \Phi\in \Hom_L(W_{\rho,\infty},V_{\pi,\infty}), \ w\in (W_{\rho,-\infty})^H.$$
Moreover, for irreducible admissible $\pi$, we have:
\begin{itemize}
\item[(i)] Let $\chi\in\Hom({\bf D}(G/H),\C)$, and let $W_\rho$ be the dual of the eigenspace representation considered in Prop.~\ref{englisch}. Then the $E$ is a surjection
onto $(V_{\pi,-\infty})^{L\cap H}_\chi$. 
\item[(ii)] Let $\chi\in\Hom({\bf D}(G/H),\C)$. The map (\ref{thomas}) provides an isomorphism of
$$ \bigoplus_{\rho\in \widetilde G_{H,\chi}} \Hom_L(W_{\rho,\infty},V_{\pi,\infty})\otimes (W_{\rho,-\infty})^H$$
onto a closed subspace of  $(V_{\pi,-\infty})^{L\cap H}_\chi $ which coincides with $(V_{\pi,-\infty})^{L\cap H}_\chi $ for generic $\chi$.
\item[(iii)] Let $(G,H,L)$ be of Type~I. Assume that $\tilde\rho$ is $\tilde\pi$-minimal (see Def.~\ref{spass}). Then $\dim (W_{\rho,-\infty})^H=1$ and $(\ref{thomas})$ is an isomorphism.
Moreover, $E$ induces isomorphisms between certain subspaces, namely
$$\Hom_L(W_{\rho,-\infty},V_{\pi,-\infty})\otimes (W_{\rho,-\infty})^H\longrightarrow  (V_{\pi,\infty})^{L\cap H} $$
and, if in addition $\pi$ and $\rho$ are unitary,
\be\label{gauguin}
\Hom_L(W_{\rho},V_{\pi})\otimes (W_{\rho,-\infty})^H\longrightarrow  V_{\pi}^{L\cap H}\ .
\end{equation}
\end{itemize}
\end{pro}
Evaluating the above characterizing property for $E$ on test functions $\vp\in C_c^\infty(X)$ we get the following alternative characterization of $E$:
\be\label{hodler} \pi(\vp)E(\Phi\otimes w)=\Phi(\rho(\vp) w),\quad \vp\in C_c^\infty(X),\  \Phi\in \Hom_L(W_{\rho,\infty},V_{\pi,\infty}), \ w\in (W_{\rho,-\infty})^H.
\end{equation}
Note that $\rho(\vp):  (W_{\rho,-\infty})^H\rightarrow W_{\rho,\infty}$ is a well-defined operator.
\begin{proof}  In order to construct the map $E$ we employ (\ref{putzi}) and the version of Frobenius reciprocity established in Prop.~ \ref{kunst}. The latter gives
\begin{eqnarray*}(W_{\rho,-\infty})^H&\cong& \Hom_G(W_{\tilde\rho,-\infty},C^{-\infty}(X)),\quad w\mapsto \Phi_w,\  \Phi_w(\tilde w)=c^G_{\tilde w,w},\\
(V_{\pi,-\infty})^{L\cap H}&\cong &\Hom_L(V_{\tilde\pi,-\infty}, C^{-\infty}(X)), \quad v\mapsto \Phi_v,\  \Phi_v(\tilde v)=c^L_{\tilde v,v}.
\end{eqnarray*}
Now we use the natural map
$$ \Hom_L(V_{\tilde\pi,-\infty},W_{\tilde\rho,-\infty})\otimes \Hom_G(W_{\tilde\rho,-\infty}, C^{-\infty}(X))\longrightarrow \Hom_L(V_{\tilde\pi,-\infty},C^{-\infty}(X))\ .$$

The action of ${\Cal S}(L)$ on $V_{\pi,-\infty}$ is given by smoothing operators $\pi(\vp): V_{\pi,-\infty}\rightarrow V_{\pi,\infty}$. Therefore, the distributional matrix coefficients
of $V_\pi$ belong to ${\Cal S}'(L)$. This leads to the following refinement of Prop.~\ref{kunst}: 
$$ (V_{\pi,-\infty})^{L\cap H}\cong \Hom_L(V_{\tilde\pi,-\infty},{\Cal S}'(X))\cong \Hom_L(V_{\tilde\pi,-\infty}, C^{-\infty}(X))\ .$$
In particular, we obtain 
$$ (V_{\pi,-\infty})^{L\cap H}_\chi \cong \Hom_L(V_{\tilde\pi,-\infty}, E^{-\infty}_\chi(X)_{mod})\ .$$
In view of Prop.~\ref{englisch}, Assertion (i) of the present proposition now becomes a tautology. Assertion~(ii) can be obtained along the lines of the proof of
Prop~\ref{otto} (ii). 

Eventually we discuss Assertion (iii). We have already seen in the proof of  Thm.~\ref{aal} that the restriction of $E$ to the subspace $\Hom_L(V_{\tilde \pi,\infty}, W_{\tilde\rho,\infty})\otimes  (W_{\rho,-\infty})^H\cong \Hom_L(W_{\rho,-\infty},V_{\pi,-\infty})\otimes (W_{\rho,-\infty})^H\subset \Hom_L(W_{\rho,\infty},V_{\pi,\infty})\otimes (W_{\rho,-\infty})^H$ - which has a more direct definition - is injective and has dense image in $(V_{\pi,-\infty})^{L\cap H}$, cf. the discussion around (\ref{guru}). Now we assume for a moment that
$(G,H,L)$ is spherical. In this case, all variants of $E$ listed in (iii) coincide and are maps between finite dimensional vector spaces. It follows that they are isomorphisms. Let now $(G,H,L)$
be a general Type~I triple. As in the proof of Thm.~\ref{aal}, we can reduce the assertion to the case $L=L_{min}$. Then, by Lemma \ref{productck}, we have product decompositions $G=G_1\cdot (G_2\times G_2)$, $H= H_1\cdot \Delta G_2$, $L=L_1\cdot (G_2\times \{e\})$ with $(G_1,H_1,L_1)$ spherical. Using corresponding tensor product realisations of the involved representations
we conclude that Assertion~(iii) is true in general if it is true for the 'minimal' group case $(G,H,L)=(L\times L,\Delta L, L\times\{e\})$. Note that here $L\cap H=\{e\}$.
Let $\pi$ be an irreducible admissible representation  of $L$ realized on a Hilbert space $V_\pi$. The $\pi$-minimal  $\Delta L$-spherical $L\times L$-representation is then
given by the space of Hilbert-Schmidt operators 
$W_\rho=\End_{HS}(V_\pi)$. Note that $W_\rho$ is always irreducible, and is unitary if $\pi$ is so. Moreover,
$$ W_{\rho,\infty}=\End_\infty(V_\pi):=\Hom(V_{\pi,-\infty},V_{\pi,\infty}),\ W_{\rho,-\infty}=\End_{-\infty}(V_\pi):=\Hom(V_{\pi,\infty},V_{\pi,-\infty}).$$
Note that $(W_{\rho,-\infty})^H$ is spanned by the identity operator. 
Now let us consider the relevant $\Hom_L$-spaces in its transposed form:
$$ \Hom_L(V_{\tilde\pi,\infty},\End_\infty(V_{\tilde\pi}))\subset \Hom_L(V_{\tilde\pi},\End_{HS}(V_{\tilde\pi}))\subset \Hom_L(V_{\tilde\pi,-\infty},\End_{-\infty}(V_{\tilde\pi}))\ .$$
We claim that this chain of inclusions is canonically isomorphic to 
$V_{\pi,\infty}\subset V_\pi\subset V_{\pi,-\infty}$. Indeed, it is a consequence of the Schur-Dixmier Lemma that every homomorphism in one of three spaces has values
in rank one operators and is of the form
$ \Phi(v_1)(v_2)=\lambda(v_2)v_1$. The canonical isomorphism is now given by $\Phi\mapsto \lambda$. For instance, if $\Phi\in \Hom_L(V_{\tilde\pi,-\infty},\End_{-\infty}(V_{\tilde\pi}))$, then $\lambda\in \Hom(V_{\tilde\pi,\infty},\C)\cong V_{\pi,-\infty}$. It is  easily checked, e.g. using the characterization (\ref{hodler}), that the map $\Phi\mapsto\lambda$ coincides
with $\Phi\mapsto E(\Phi^t\otimes\id_{V_\pi})$. Assertion~(iii) follows.
\end{proof}


\subsection{Results for Type~I triples.}\label{TI} In this subsection we discuss the consequences of Prop.~\ref{paulklee} for triples of Type~I.

Let $\widehat G_H^L$ be the set of (infinitesimal) equivalence classes of $H$-cospherical $G$-representation $W_\rho$ such that $W_{\rho^*}$ is $\pi$-minimal for some $\pi\in \widehat L$ satisfying $V_\pi^{L\cap H}\ne\{0\}$. The set $\widehat G_H^L$ is the disjoint union of its main part $\widehat G_H$ and the two exceptional sets
$$  \widehat G_H':=\{\rho\in\widehat G_H^L\mid \rho\mbox{ irreducible and not unitarizable}\}\ \mbox{ and }\ 
\widehat G_H'':=\{\rho\in\widehat G_H^L\mid \rho\mbox{ not  irreducible}\}. $$ 
Prop.~\ref{kopp} shows that these two sets are non-empty for Triple 1. In contrast, they are empty in the group case 11.

For $\rho\in  \widehat G_H^L$ we set $\widehat L_\rho:=\{\pi\in\widehat L\mid V_\pi^{L\cap H}\ne\{0\},\rho^*\mbox{ is $\pi$-minimal}\}$. Again we have a finite-to-one map $p: \widehat G_H^L\rightarrow \Hom({\bf D}(G/H),\C)$. Its image is contained in the set of $*$-characters $\Xi\subset \Hom({\bf D}(G/H),\C)$. We have $\widehat L_\rho\subset \widehat L_{p(\rho)}$, see Cor.~\ref{hussa} for the definition of  $\widehat L_{p(\rho)}$. Every $\widehat L_\chi$, $\chi\in\Xi$, is a finite (possibly empty) union of sets of the form $\widehat L_\rho$. Our goal is to give a representation
theoretic refinement of the decomposition into joint eigenspaces in Theorem~\ref{selim} indexed by $\rho\in\widehat G_H^L$ instead of $\chi\in\Xi$.

As a first step, if $\rho\in\widehat G_H$, we want to equip the left hand side of (\ref{gauguin}) with a Hilbert space structure  and to understand to what extent (\ref{gauguin}) becomes an isometry of Hilbert spaces.
Since $(W_{\rho,-\infty})^H$ is one-dimensional, the main point is to equip $\Hom_L(W_{\rho},V_{\pi})$ with a suitable Hilbert space structure.
We also want to establish a similar result for general $\rho\in\widehat G_H^L$, even if $\rho$ is not unitary. In the latter case we have to find the correct left hand side of (\ref{gauguin}). 

Let $\rho\in\widehat G_H$. We fix unitary structures on $W_\rho$ and on all $V_\pi$, $\pi\in\widehat L_\rho$.
Now $\Hom_L(W_{\rho},V_{\pi})$
is isomorphic to the multiplicity space $M_{\tilde\rho}(\tilde\pi)$ and carries a natural Hilbert space structure, see (\ref{cezanne}) and Lemma \ref{sardine}. Explicitly, we have
$$ \langle A, B\rangle \id_{V_\pi}:=AB^*,\quad A,B\in \Hom_L(W_{\rho},V_{\pi}).$$

Let now $\rho\in \widehat G_H'\cup \widehat G_H''$. We fix an invariant Hermitian form on the unique irreducible submodule $W_{\rho_0,\infty}\subset W_{\rho,\infty}$ (its existence is provided by Prop.~\ref{dark}(iii)) and unitary structures on all 
$V_\pi$, $\pi\in\widehat L_\rho$. We first discuss the case that the symmetric space $G/H$ has no group factors. In particular, $(G,H,L)$ is spherical. By a slight abuse
of notation, we set 
$$ \Hom_L(W_\rho, V_\pi):=\Hom_L(W_{\rho,\infty}, V_{\pi,\infty})\ .$$
Let $B\in \Hom_L(W_\rho, V_\pi)=\Hom_L(W_{\rho,\infty}, V_{\pi,\infty})$ and let $B_0\in \Hom_L(W_{\rho_0,\infty}, V_{\pi,\infty})$ be its restriction $W_{\rho_0,\infty}$. Let 
$B_0^*\in \Hom_L(V_{\pi,-\infty},W_{\rho_0,-\infty})$ be the adjoint of $B_0$ with respect to the chosen Hermitian forms. 
By Lemma~\ref{sardine} the embedding 
$\Hom_L(V_{\pi,\infty},W_{\rho_0,\infty})\subset\Hom_L(V_{\pi,-\infty},W_{\rho_0,-\infty})$ is an isomorphism. We conclude that $ B_0^*$ maps $V_{\pi,\infty}$ continuously to $W_{\rho_0,\infty}\subset W_{\rho,\infty}$. Thus for all $A,B\in \Hom_L(W_\rho, V_\pi)$ the composition $AB_0^*$ defines an element of $\End_L(V_{\pi,\infty})$. Therefore
we can define a Hermitian form on $\Hom_L(W_\rho, V_\pi)$ by 
$$ \langle A, B\rangle \id_{V_{\pi,\infty}}:=AB_0^*,\quad A,B\in \Hom_L(W_{\rho},V_{\pi}).$$
It follows from the injectivity of the restriction map $\Hom_L(W_{\rho,\infty},V_{\pi,\infty})\rightarrow \Hom_L(W_{\rho_0,\infty},V_{\pi,\infty})$ (see Prop.~\ref{dark}(ii)) and Lemma~\ref{schlips} that  $\langle B, B\rangle=0$ only if $B=0$.

In the general case, we have $G=G_1\cdot(G_2\times G_2)$ with $H=H_1\cdot\Delta G_2$ such that $G_1/H_1$ contains no group factors. Set $L_1:=L\cap G_1$, $K_1:=K\cap G_1$.
We may assume that $W_\rho$ carries a $K_1\cdot(G_2\times G_2)$-invariant Hilbert space structure. Again by a slight abuse of notation we set
$$ \Hom_L(W_\rho, V_\pi):=\Hom_L(W_{\rho,\infty_{G_1}}, V_{\pi,\infty_{L_1}})\ .$$
The statements that follow can be easily checked by decomposing $W_\rho$ and $V_\pi$ into tensor products such that the first factor is a representation of $G_1$ or $L_1$, respectively.
The unitarity of $W_\rho$ as a representation of  $G_2\times G_2$ implies that the invariant Hermitian form on $W_{\rho_0,\infty}$ extends to $W_{\rho_0,\infty_{G_1}}$.
For $B\in\Hom_L(W_\rho, V_\pi)$ let $B_0^*$ be the adjoint of the restriction to $W_{\rho_0,\infty_{G_1}}$. For all $A,B\in \Hom_L(W_\rho, V_\pi)$ the composition $AB_0^*$ defines an element of $\End_L(V_{\pi,\infty_{L_1}})$. We define a non-zero Hermitian form on $\Hom_L(W_\rho, V_\pi)$ by 
$$ \langle A, B\rangle \id_{V_{\pi,\infty_{L_1}}}:=AB_0^*,\quad A,B\in \Hom_L(W_{\rho},V_{\pi}).$$

We remark that at least for some unitary and relatively tempered $\rho$ there are rather canonical normalizations of the scalar product on $(W_{\rho,-\infty})^H$ coming from harmonic analysis on $G/H$ (and depending of course on the normalization of the scalar product on $W_\rho$), see e.g. \cite{vdBFJS}.

\begin{lem}\label{huckepack}
Let $(G,H,L)$ be of Type~I. Let $\rho\in\widehat G^L _H$. We fix a scalar product on the one-dimensional space $(W_{\rho,-\infty})^H$. For $\pi\in\widehat L_\rho$ we consider the space $\Hom_L(W_{\rho},V_{\pi})$ equipped with a non-zero Hermitian form $\langle.,.\rangle$ constructed above.
Then there exists a real number $e_\pi\ne0$ such that
$$  \Hom_L(W_{\rho},V_{\pi})_{e_\pi}:= \left ( \Hom_L(W_{\rho},V_{\pi}),e_\pi\langle.,.\rangle\right )$$
is a Hilbert space and
$$ E: \Hom_L(W_{\rho},V_{\pi})_{e_\pi}\otimes (W_{\rho,-\infty})^H\longrightarrow  V_{\pi}^{L\cap H} $$
is a unitary isomorphism.
\end{lem}

\begin{proof}
We first check that the version of $E$ considered in the lemma is a linear isomorphism. This is clear by Prop.~\ref{paulklee}(iii) if $\rho$ is unitary or if $(G,H,L)$
is spherical (in the latter case $V_\pi^{L\cap H}=(V_{\pi,-\infty})^{L\cap H}$). The general case is easily reduced to these two situations. 
Next we assume that $L=L_{min}$. As in the proof of Thm.~\ref{aal} we see that $E$ is an intertwining operator between irreducible $C_3=Z_G(L)_0$-representations.
The non-zero Hermitian forms on both spaces are $C_3$-invariant. It follows that $E$ respects these forms up to a constant. This proves the lemma for $L=L_{min}$. Again using the techniques of the proof
of Thm.~\ref{aal} one sees that the assertion of the lemma is true for general $L$ if it is true for $L_{min}$.
\end{proof}

The number $e_\pi$ could be viewed as a kind of relative formal dimension of $\pi$ with respect to $\rho$. Indeed, if $\rho$ is a relative discrete series representation and $\pi_1,\pi_2\in\widehat L_\rho$, then $e_{\pi_1}/e_{\pi_2}=d_{\pi_1}/d_{\pi_2}$, where $d_{\pi_i}$ is the formal dimension of the discrete series representation $\pi_i$, $i=1,2$.

Let $\rho\in\widehat G^L _H$ with a fixed scalar product on the one-dimensional space $(W_{\rho,-\infty})^H$. We define subspaces $(W_{\tilde \rho,-\infty})_+^\Gamma\subset (W_{\tilde \rho,-\infty})^\Gamma_0$ of $(W_{\tilde \rho,-\infty})^\Gamma$
as follows. We fix an element $w_0\in (W_{\rho,-\infty})^H$ of norm $1$ and set 
\begin{eqnarray*}
(W_{\tilde \rho,-\infty})_+^\Gamma&:=&\{\tilde w\in (W_{\tilde \rho,-\infty})^\Gamma\mid c_{\tilde w,w_0}\in C^\infty(Y)\}\ ,\\
(W_{\tilde \rho,-\infty})_0^\Gamma&:=&\{\tilde w\in (W_{\tilde \rho,-\infty})^\Gamma\mid c_{\tilde w,w_0}\in L^2(Y)\}\ .\\
\end{eqnarray*}
The second space inherits a Hilbert space structure from $L^2(Y)$. We also get a sesquilinear pairing between $(W_{\tilde \rho,-\infty})^\Gamma$ and $(W_{\tilde \rho,-\infty})_+^\Gamma$. For $\rho\in \widehat G_H''$, let $W_{\rho_1}\subset W_{\tilde\rho}$ be the unique maximal proper $G$-subrepresentation. We  set
$$    (W_{\tilde \rho,-\infty})^\Gamma_{new}:=\left ( (W_{\rho_1,-\infty})^\Gamma\cap  (W_{\rho,-\infty})^\Gamma_+\right )^\perp \subset (W_{\tilde \rho,-\infty})^\Gamma$$
and   $(W_{\tilde \rho,-\infty})^\Gamma_{new,+}:= (W_{\tilde \rho,-\infty})^\Gamma_{new}\cap (W_{\tilde \rho,-\infty})^\Gamma_{+}$,
$(W_{\tilde \rho,-\infty})^\Gamma_{new,0}:= (W_{\tilde \rho,-\infty})^\Gamma_{new}\cap (W_{\tilde \rho,-\infty})^\Gamma_{0}$.
We remark that the quotient map $W_{\tilde\rho}\rightarrow W_{\tilde\rho_0}$ identifies $(W_{\tilde \rho,-\infty})^\Gamma_{new}$ with a closed subspace of
$(W_{\tilde \rho_0,-\infty})^\Gamma$. 

Now we can state the announced refinement of the spectral decomposition given in Theorem~\ref{selim} for triples of Type~I. For Type~I triples, it also extends to all eigenspaces the description given in Lemma~\ref{otto}(ii) of generic eigenspaces.

\begin{thm}\label{gogo}
Let $(G,H,L)$ be a triple of Type~I. Let $\Gamma\subset L$ be a cocompact lattice.
The matrix coefficient map (\ref{paul}) induces a ${\bf D}(G/H)$-equivariant unitary isomorphism of Hilbert spaces
$$ L^2(Y)\cong \widehat{\bigoplus_{\rho\in\widehat G_H\cup \widehat G_H'}} (W_{\tilde \rho,-\infty})_0^\Gamma\otimes  (W_{\rho,-\infty})^H
\oplus \widehat{\bigoplus_{\rho\in\widehat G_H''}} (W_{\tilde \rho,-\infty})_{new,0}^\Gamma\otimes  (W_{\rho,-\infty})^H$$
as well as ${\bf D}(G/H)$-equivariant isomorphisms
\begin{eqnarray*}
C^\infty(Y)&\cong& \overline{\bigoplus_{\rho\in\widehat G_H\cup \widehat G_H'}} (W_{\tilde \rho,-\infty})_+^\Gamma\otimes  (W_{\rho,-\infty})^H
\oplus \overline{\bigoplus_{\rho\in\widehat G_H''}} (W_{\tilde \rho,-\infty})_{new,+}^\Gamma\otimes  (W_{\rho,-\infty})^H\ ,\\
C^{-\infty}(Y)&\cong& \overline{\bigoplus_{\rho\in\widehat G_H\cup \widehat G_H'}} (W_{\tilde \rho,-\infty})^\Gamma\otimes  (W_{\rho,-\infty})^H
\oplus \overline{\bigoplus_{\rho\in\widehat G_H''}} (W_{\tilde \rho,-\infty})_{new}^\Gamma\otimes  (W_{\rho,-\infty})^H\ .
\end{eqnarray*}
The inclusions $(W_{\tilde \rho,-\infty})_+^\Gamma\subset (W_{\tilde \rho,-\infty})_0^\Gamma\subset (W_{\tilde \rho,-\infty})^\Gamma$ have dense images.
These spaces of $\Gamma$-invariants (or rather the corresponding subspaces of `newforms', if $\rho\in\widehat G_H''$) have further decompositions
$$
 \widehat{\bigoplus_{\pi\in\widehat L_\rho}} \Hom_L(W_\rho,V_\pi)_{e_\pi}\otimes N_\Gamma(\pi)\ ,\quad
 \overline{\bigoplus_{\pi\in\widehat L_\rho}} \Hom_L(W_{\rho,\mp\infty},V_{\pi,\mp\infty})\otimes N_\Gamma(\pi)\ .
$$
If $\rho\in\widehat G_H$ is an integrable relative discrete series representation and $\Gamma\subset L$ is torsion free, then $(W_{\tilde \rho,-\infty})_+^\Gamma$ is infinite dimensional.
\end{thm}

\begin{proof} We combine  Lemma~\ref{bassa} with Prop.~\ref{paulklee}(iii), Lemma~\ref{kirchner}, and Lemma~\ref{huckepack}. The last assertion is a reformulation of Prop.~\ref{horn}.
\end{proof}

The decomposition of $L^2(Y)$ with respect to $G$-representations provides more structure and contains more information than just the decomposition into joint eigenspaces.
The difference is similar to the one between the two  decompositions (\ref{GGPS}) (automorphic representations) and (\ref{GGPSbis}) (automorphic forms) in the introduction.
The new point of view is particularly useful since at least the main part of $\widehat G_H^L$ consists of objects that are very well studied in the context of harmonic analysis on semisimple symmetric spaces, namely relative discrete series and the unitary principal series $I^{\delta,\lambda}$ ($\lambda$ imaginary). When $\rank(G/H)>1$, i.e. for Triples 2,4,7,
we also get significantly more information about the spectrum of ${\bf D}(G/H)$ on $L^2(Y)$ than provided by Cor.~\ref{hussa} and Thm.~\ref{selim} in combination with the results of Section~\ref{pbw}.
Indeed, let   
$\pi\in\widehat L$ such that $V_\pi^{L\cap H}\ne\{0\}$. The results of Section~\ref{pbw} are not sufficient to determine completely the corresponding character $\chi_\pi\in\Xi$ provided by Prop.~\ref{staun}, they only
determine the Casimir eigenvalue $\chi_\pi(\Omega_G)$. On the other hand, it is usually not difficult to determine the $\pi$-minimal $G$-representation, or at least its
infinitesimal character, which in turn determines $\chi_\pi$.

The most extensively studied semisimple symmetric spaces (except Riemannian ones and group manifolds) are the real hyperboloids $X=G/H=SO_e(p,q)/SO_e(p-1,q)$. They
appear as the underlying symmetric spaces for Triples 1,5,6 with $(p,q)=(2,2n)$, $(4,4n)$, and $(8,8)$, respectively. For illustration, let us describe the set $\widehat G_H^L$
in these rank one cases. For all unproven statements below we refer to \cite{HoTa93} and \cite{Sch87}. 

The following description of $H$-spherical $G$-representations is valid for all even $q\ge p\ge 2$.  As in the case $(p,q)=(2,2n)$ already discussed in the proof of Prop.~\ref{kopp} the relevant principal series
representations to consider are of the form $I^{[k],\lambda}$, $[k]\in\Z/2\Z$, $\lambda\in \fa_{\fq\cap\fs,\C}^*\cong\C$. Here we identify $\fa_{\fq\cap\fs,\C}^*$ with $\C$
by sending the unique restricted root to $1$. Requiring this identification to be an isometry fixes a normalization of the invariant bilinear form on $\fg$, and therefore also of the Laplacian 
$\Delta$ on $X$. Set $\rho_\fq:=\frac{1}{2}(p+q-2)$. We also consider the $G$-representations on even or odd eigendistributions
$$ E_{[k],\lambda}:=\{f\in C^{-\infty}(X)\mid \Delta f=(\lambda^2-\rho_\fq^2) f,  l_zf=(-1)^k f, z=-\id\in SO_e(p,q)\}\ .$$
The representations $I^{[k],\lambda}$, $I^{[k],-\lambda}$,  $E_{[k],\pm\lambda}$ (or rather their underlying $(\fg,K)$-modules) all have the same irreducible subquotients. They
are irreducible (and hence mutually equivalent) if and only if $\lambda-\rho_\fq\not\in k+2\Z$. The principal series $I^{[\rho_\fq],\lambda}$ and $I^{[\rho_\fq-1],\lambda}$ are irreducible and unitarizable
if and only if $\lambda$ is non-zero imaginary or $\lambda\in i\R\cup (-1,1)$, respectively. For $\lambda\in\N$, we also have the discrete series
$ D_\lambda:= E_{[\rho-\lambda],\lambda}\cap L^2(X)$, and there is the limit of discrete series $D_0\subset  E_{[\rho_\fq],0}$. Alternatively, $D_\lambda$, $\lambda\in\N_0$, can be characterized as the maximal spherical submodule of $I^{[\rho_\fq-\lambda],\lambda}$, cf. the proof of Prop.~\ref{kopp}. For $p>2$, the representations $D_\lambda$, $\lambda\in\N_0$, are irreducible,
while for $p=2$ they split into two irreducible components $D_\lambda=D_\lambda^+\oplus D_\lambda^-$. Then
$$ \widehat G_H=\left\{I^{[\rho_\fq],\lambda},I^{[\rho_\fq-1],\mu}, D_\nu^{(\pm)},\C\mid \lambda\in i(0,\infty),\mu\in i[0,\infty)\cup (0,1), \nu\in\N_0\right\}\ .$$
The exceptional set $\widehat G_H'$ is a subset of the set of irreducible principal series (isomorphic to the corresponding eigenspace representations) with real induction parameter
strictly smaller than $\rho_\fq$: 
$$\left\{I^{[\rho_\fq],\lambda}, I^{[\rho_\fq-1],\mu}\mid \lambda\in (0,\rho_\fq)\setminus 2\Z,\mu\in (1,\rho_\fq)\setminus 1+2\Z\right\}\ .$$
In case of reduciblity of the eigenspaces, $p>2$, and $\lambda\in \{1,2,\dots,\rho_\fq-1\}$ there is exactly one proper submodule $C_\lambda$ of $E_{[\rho_\fq-\lambda],\lambda}$ that is strictly larger 
than $D_\lambda$. It is isomorphic to the quotient of $I^{[\rho_\fq-\lambda],\lambda}$ by its unique non-spherical submodule. For $p=2$, the analogous module splits again into two irreducible modules $C_\lambda^\pm$. Then $\widehat G_H''$ is contained in 
$$\left\{E_{[\rho_\fq-\lambda],\lambda}^*,(C_\mu^{(\pm)})^*\mid \lambda\in \{0,1,\dots\rho_\fq-1\},\mu\in \{1,2,\dots,\rho_\fq-1\}\right\}\ .$$
Here, as usual, the * stands for the conjugate dual representation. In order to determine $\widehat G_H'\cup\widehat G_H''$ precisely one has to determine which of these representations
(ignoring stars) are generated by a unitary $L$-subrepresentation. The complete result  for Triple 1 together with the determination of $\widehat L_\rho$, $\rho\in\widehat G_H^L$,
can be read off directly from the proof of Prop.~\ref{kopp}. We give it in Table \ref{fax} below. Here $\rho_\fq=n$. In particular, we see that almost all candidates for $\widehat G_H^L$ listed above occur, only the representations $I^{[1],\lambda}$, $\lambda\in (n-1,n)$, are missing. Table~\ref{fax} makes transparent how the results of Section~\ref{detail}, in particular Thm.~\ref{something}, fit into our new scheme.  

\begin{table}[h]
\caption{$(G,H,L)=(SO_e(2,2n),SO_e(1,2n), U(1,n))$}\label{fax}
\begin{center}
\renewcommand{\arraystretch}{1.3}
\begin{tabular}{|l||l||l||l|}
  \hline
  &${\bf W_\rho,\ \rho=\rho_\lambda}$ & ${\bf Parameters}$ & ${\bf \pi^k_\lambda\in\widehat L_{\rho_\lambda}}$\\
  \hline
    \hline
  ${\bf \widehat G_H}$& $I^{[n],\lambda}$ & $\lambda\in i(0,\infty)$ & $k\equiv n\, (2)$\\
    \hline
  &$I^{[n-1],\lambda}$& $\lambda\in i[0,\infty)\cup (0,1)$ & $k\equiv n-1 \, (2)$\\
\hline
&$D_\lambda^\pm$ &$\lambda\in\N_0$&$ k\equiv \lambda-n \, (2), \pm k\ge n+\lambda$ \\
\hline
&$\C$ &$\lambda=n$ & $ k=0$\\
   \hline
    \hline
${\bf \widehat G_H'}$&$I^{[\delta],\lambda}$& $\lambda\in (\ve,n-\delta)\setminus\{\ve+2,\ve+4,\dots,n-\delta-2\}$ & $k\equiv\delta\, (2), |k|<n-\lambda$ \\
&&$\delta\in\{0,1\},\ve\in\{0,1\},\ve\equiv n-\delta\, (2)$&$ $\\
 \hline
    \hline
${\bf \widehat G_H''}$&$E_{[n-\lambda],\lambda}^*$&$\lambda\in\{0,1,\dots,n-1\}$&$k\equiv n-\lambda\, (2), |k|<n-\lambda$ \\
\hline
&$(C^\pm_\lambda)^*$&$\lambda\in\{1,2,\dots,n-1\}$&$k=\pm(n-\lambda)$\\
\hline
\end{tabular}
\end{center}
\end{table}

The results of \cite{HoTa93} and \cite{Sch87} combined with the results of Section \ref{branching} and known information about $\widehat L$  imply a comparable precise description for all other rank one cases including Triple 3 not discussed above. We will present it in Tables~\ref{super}--\ref{tauto} after the proof of Prop.~\ref{wanndenn}, where the involved notation is explained. 
Tables~\ref{fax}--\ref{tauto} also provide a complete solution of the branching problem addressed in Section~\ref{branching} for rank one triples. In view of Prop.~\ref{fishy} this includes a classification 
(independent of Harish-Chandra's theory)
of the discrete series representations of $L$ obeying an $L\cap H$-invariant vector. One just has to read the row of  $D_\lambda$ for $\lambda>0$.

One might expect that the contribution to $L^2(Y)$ of non-unitary $G$-representations, i.e. of the exceptional sets $\widehat G_H'$ and $\widehat G_H''$, is small in some sense.
For $\rank(G/H)=1$ we can make this precise.
\begin{pro}\label{wanndenn}
Let $(G,H,L)$ be a triple of Type~I such that $G$ is linear and $\rank(G/H)=1$.  Let $\Gamma\subset L$ be a cocompact lattice.
Then the space
$$   \Big(\widehat{\bigoplus_{\rho\in\widehat G_H'}} (W_{\tilde \rho,-\infty})_0^\Gamma\otimes  (W_{\rho,-\infty})^H\Big)
\oplus\Big( \widehat{\bigoplus_{\rho\in\widehat G_H''}} (W_{\tilde \rho,-\infty})_{new,0}^\Gamma\otimes  (W_{\rho,-\infty})^H\Big)
$$  
is finite dimensional.
\end{pro}

\begin{proof} If $G/H$ is a group manifold, then $\widehat G_H'=\widehat G_H''=\emptyset$. Hence there is nothing to show in this case. Thus we may assume that $(G,H,L)$ is among
the Triples 1,3,5,6. In particular, $(G,H,L)$ is spherical. The decomposition of  $(W_{\tilde \rho,-\infty})_0^\Gamma$ given in Thm.~\ref{gogo} shows that the assertion of the proposition is equivalent to finiteness of the set
$$\left\{\pi\in\bigcup_{\rho\in \widehat G_H'\cup\widehat G_H''} \widehat L_\rho\mid N_\Gamma(\pi)\ne\{0\}\right\}\ .$$
The representations $\pi\in\widehat L_\rho$, $\rho\in \widehat G_H'\cup\widehat G_H''$, are necessarily unitarizable Langlands quotients of
principal series representations $H^{\sigma,\lambda}$ with $\lambda\in (0,\rho_L)$ or irreducible summands of $H^{\sigma,0}$, $\sigma\in\widehat M_{L,s}$ (which, in addition obey
an $L\cap H$-invariant vector).
Later on we will identify also $\fa_L^*$ with $\R$ such that $\rho_L=\rho_\fq$.
By the discreteness of the spectrum 
of the operator $\Omega_L$ when acting on sections of finite dimensional locally homogeneous vector bundles over $\Gamma\backslash L/K_L$, we see that it suffices to show that there is a finite subset $A\subset \widehat M_{L,s}$ such that for the representations $\pi$ of $L$ discussed above one has that $\sigma\in A$.
By Table~\ref{fax} we can take $A:=\{\sigma_k\mid k\in\Z,\,  |k|<n\}$ for Triple~1. Here $\sigma_k$ is as in Section~\ref{detail}.

We recall that in general one can associate to $\sigma\in\widehat M_{L,s}$ the spherical representations $\gamma_\sigma\in \widehat K_{L,s}$ (by Cor.~\ref{wels}(c)) and
$\tau_{\gamma_\sigma}\in\widehat K_s$ (by Lemma~\ref{schlips}(c)). Note that for Triples 5 and 6 with $G=SO_e(p,q)$ and $(p,q)=(4,4n)$ or $(8,8)$, respectively, we have
$\widehat K_s=\{\tau_k\mid k\in\N_0\}$, where $\tau_k$ denotes the spherical harmonics for $SO(p)$ of degree $k$. Then, for $\lambda\in\N_0$, the isotypic
component $I^{[\rho_\fq-\lambda],\lambda}(\tau_k)$ is non-zero and belongs to $D_\lambda\subset I^{[\rho_\fq-\lambda],\lambda}$ if and only if $k\equiv \rho_\fq-\lambda\, (2)$ and
$k\ge \frac{1}{2}(q-p+2)+\lambda$,
i.e. 
\be\label{marcel}
k\ge 2n-1+\lambda\quad\mbox{ or }\quad k\ge \lambda+1\ ,
\end{equation} 
respectively (\cite{HoTa93}).

Next we discuss Triple 6 in detail. We have $L=Spin(1,8)$, $M_L=Spin(7)$, $M_L\cap H\cong G_2$, $\rho_\fq=\rho_L=7$. The set $\widehat M_{L,s}$ is described via highest weights as
$$\left\{\sigma_k\cong\left (\frac{k}{2},\frac{k}{2},\frac{k}{2}\right )\mid k\in\N_0\right\}\ .$$
We claim that we can take $A=\{\sigma_0\}$.

We have embeddings $H^{\sigma_k,\lambda}\hookrightarrow I^{[k],\lambda}$. Since $\rho_\fq$ is odd, the length of the complementary series for $I^{\delta,\lambda}$ is 1 for
$\delta=[0]$ and 0 for $\delta=[1]$. It is classically known (\cite{KnS71}) that the length of the complementary series  for $H^{\sigma_k,\lambda}$ is also 1 for even $k\ne 0$
and $0$ for odd $k$, and that there are no unitarizable Langlands quotients isolated from the complementary series. Thus the length of the complementary series are compatible for 
$k\ne 0$. It remains to check that the ends of the complementary series for $L$ and even $k\ne 0$ or the irreducible constituents of $H^{\sigma_k,0}$, $k$ odd, do not belong to $\widehat L_\rho$ for some exceptional $\rho$.

Let us denote the spherical $K_L$- and $K$-types associated to $\sigma_k$ by $\gamma_k$ and $\tau_k$, respectively. In fact, $\tau_k$ is really the representation on spherical
harmonics of degree $k$ as above. It follows from Cor.~\ref{wels}(c) that the only irreducible subquotient of $H^{\sigma_k,\lambda}$ admitting an $L\cap H$-invariant vector is the one
containing the $K_L$-type $\gamma_k$. As in Section~\ref{detail} we denote the irreducible $L$-representation on this subquotient by $\pi^k_\lambda$.  
We have $H^{\sigma_k,\lambda}(\gamma_k)\hookrightarrow I^{[k],\lambda}(\tau_k)$. Now let $k$ be odd and $\lambda=0$, or $k\ne 0$ even and $\lambda=1$. The
inequalities (\ref{marcel}) tell us that then $I^{[k],\lambda}(\tau_k)\subset D_\lambda$ which implies that the unitary irreducible $G$-representation $D_\lambda$ is $\pi^k_\lambda$-minimal, i.e.
$\pi^k_\lambda$ does not contribute to the exceptional set. This finishes the proof of the claim.

We turn to Triple 5. Here $L_{max}=Sp(1,n)\cdot Sp(1)$ and $L_{min}=Sp(1,n)$. It is a consequence of the last assertion of Prop.~\ref{mufree} that the sets $\widehat L_{max,\rho}$
and $\widehat L_{min,\rho}$ are in one-to-one correspondence. Thus we may restrict our considerations to $L=Sp(1,n)$. (As an aside, let us remark that the arithmeticity of lattices in $Sp(1,n)$, $n\ge 2$,
together with the vanishing of the first Betti number of compact quotients of the quaternionic hyperbolic space implies that any cocompact lattice in $Sp(1,n)\cdot Sp(1)$, $n\ge 2$, has a finite index subgroup
contained in $Sp(1,n)$.) Now we have $M_L=\Delta Sp(1)\times Sp(n-1)$, $M_L\cap H=Sp(n-1)$, $\rho_\fq=\rho_L=2n+1$, and 
$\widehat M_{L,s}=\{\sigma_k\mid k\in\N_0\}$, where $\sigma_k$ is the irreducible representation of $M_L$ of dimension $k+1$ factoring through the $Sp(1)$-factor.
We consider the corresponding elements  $\gamma_k\in \widehat K_{L,s}$ and $\tau_k\in \widehat K_s$. The representation $\tau_k$ is given by spherical harmonics for $SO(4)$
of degree $k$.
We claim that we can take $A=\{\sigma_k\mid k\le 2n-2\}$.

Again we have embeddings $H^{\sigma_k,\lambda}\hookrightarrow I^{[k],\lambda}$ sending $H^{\sigma_k,\lambda}(\gamma_k)$ to $I^{[k],\lambda}(\tau_k)$. The length of the complementary series for $I^{\delta,\lambda}$ is 1 for
$\delta=[0]$ and 0 for $\delta=[1]$. We now apply \cite{BS 81}, Thm.~7.1, which gives the classification of the unitary dual of $Sp(1,n)$, $n\ge 2$, in particular of the unitarizable Langlands quotients. The reader should be aware that this theorem is slightly misprinted. The parameters that seem to describe limits of discrete series actually describe part of the complementary
series. We extract the following information: the length of the complementary series  for $H^{\sigma_k,\lambda}$ is also 1 for even $k\ge 2n$
and $0$ for odd $k\ge 2n-1$, and that there are no unitarizable Langlands quotients isolated from the complementary series for $k\ne 0$ (and the length of the complementary series
for $k\le 2n-2$ is $2n-1-k$). For $n=1$, we have $Sp(1,1)\cong Spin(1,4)$, and thus the complentary series for $k>2n-2= 0$ have the same lenghts (for $k=0$ it has length $3=2n+1$).

We define irreducible $L$-representations $\pi^k_\lambda$ as for Triple 6. Now we consider these representations for $\lambda=0$ and odd $k\ge 2n-1$,
or $\lambda=1$ and even $k\ge 2n$.  Again the
inequalities (\ref{marcel}) tell us that then $I^{[k],\lambda}(\tau_k)\subset D_\lambda$ in these cases. Hence $D_\lambda$ is $\pi^k_\lambda$-minimal. The claim follows.

Finally we discuss Triple 3. Again, $L=Sp(1,n)$, $M_L=\Delta Sp(1)\times Sp(n-1)$, $\rho_\fq=\rho_L=2n+1$, but $M_L\cap H=\Delta U(1)\times Sp(n-1)$. It follows that
$\widehat M_{L,s}=\{\sigma_k\mid k\in\N_0\mbox{ even}\}$, where $\sigma_k$ has the same meaning as for Triple 5. We claim that we can take 
$A=\{\sigma_k\mid k\le 2n-2\mbox{ even}\}$.

We have to  introduce some notation. Since $M_{\sigma\theta}\subset H$ (and, equivalently, $Z\subset K_H$), the relevant principal series representations for $G/H$ are induced from the trivial representation
$\delta$. We denote them by $I^\lambda$, $\lambda\in\C$. We also consider the full eigenspaces
$$ E_{\lambda}:=\{f\in C^{-\infty}(X)\mid \Delta f=(\lambda^2-(2n+1)^2) f\}\ .$$
We consider the embedding $G\subset U(2,2n)\subset SO_e(4,4n)$. Let $T\subset U(2,2n)$ be the group of complex multiples of the identity. It is isomorphic to $U(1)$ and commutes with $G$. We decorate the objects associated with Triple 5 by an $\R$. Then $X_\R$ is a principal $T$-bundle over $X$ and
$$  E_\lambda\cong (E_{[0],\lambda}^\R)^T ,\quad I^\lambda\cong  (I_\R^{[0],\lambda})^T .$$
For odd $\lambda\in \N$ we set $D_\lambda:=( D_\lambda^\R)^T$ and $C_\lambda:= ( C_\lambda^\R)^T$ ($\lambda<2n+1$).

Thm.~4.2 and Paragraph 4.4 in \cite{HoTa93} tell us that the irreducibility and unitarity criteria transfer from Triple~5 to Triple~3. In particular, the complementary series for $I^\lambda$
has length 1, and $D_\lambda$ is an irreducible unitary representation. Now the claim follows in the same way as for Triple 5. 
\end{proof}
\vspace*{1cm}

\setlength{\floatsep}{40pt}

\begin{table}[h]
\caption{$(G,H,L)=(SO_e(8,8),SO_e(7,8), Spin(1,8))$}\label{super}
\begin{center}
\renewcommand{\arraystretch}{1.3}
\begin{tabular}{|l||l||l||l|}
  \hline
  &${\bf W_\rho,\ \rho=\rho_\lambda}$ & ${\bf Parameters}$ & ${\bf \pi^k_\lambda\in\widehat L_{\rho_\lambda}}$\\
  \hline
    \hline
  ${\bf \widehat G_H}$& $I^{[0],\lambda}$ & $\lambda\in i[0,\infty)\cup (0,1)$ & $k\in\N_0$ even  \\
    \hline
  &$I^{[1],\lambda}$& $\lambda\in i(0,\infty)$ & $k\in \N$ odd\\
\hline
&$D_\lambda$ &$\lambda\in\N_0$&$ k\not\equiv \lambda \, (2), k\ge \lambda+1$ \\
\hline
&$\C$ &$\lambda=7$ & $ k=0$\\
   \hline
    \hline
${\bf \widehat G_H'}$&$I^{[0],\lambda}$& $\lambda\in (1,7)\setminus\{3,5\}$ & $k=0$ \\
 \hline
    \hline
${\bf \widehat G_H''}$&$C_\lambda^*$&$\lambda\in\{1,3,5\}$&$k=0$\\
\hline
\end{tabular}
\end{center}
\end{table}

\begin{table}[h]
\caption{$(G,H,L)=(SO_e(4,4n),SO_e(3,4n), Sp(1,n))$}\label{fluous}
\begin{center}
\renewcommand{\arraystretch}{1.3}
\begin{tabular}{|l||l||l||l|}
  \hline
  &${\bf W_\rho,\ \rho=\rho_\lambda}$ & ${\bf Parameters}$ & ${\bf \pi^k_\lambda\in\widehat L_{\rho_\lambda}}$\\
  \hline
    \hline
  ${\bf \widehat G_H}$& $I^{[0],\lambda}$ & $\lambda\in i[0,\infty)\cup (0,1)$ & $k\in\N_0$ even  \\
    \hline
  &$I^{[1],\lambda}$& $\lambda\in i[0,\infty)$ & $k\in \N$ odd\\
\hline
&$D_\lambda$ &$\lambda\in\N_0$&$ k\not\equiv \lambda\, (2), k\ge 2n-1+\lambda$ \\
\hline
&$\C$ &$\lambda=2n+1$ & $ k=0$\\
   \hline
$n\ge 2$:\\
    \hline
${\bf \widehat G_H'}$&$I^{[\delta],\lambda}$& $\lambda\in (1-\delta,2n-1-\delta)\setminus\{3-\delta,5-\delta,\dots,2n-3-\delta\}$ & $k\equiv\delta\, (2)$ \\
&&$\delta\in\{0,1\}$&$ 0\le k<2n-1-\lambda$\\
 \hline
    \hline
${\bf \widehat G_H''}$&$E_{[\lambda+1],\lambda}^*$&$\lambda\in\{0,1,\dots,2n-3\}$&$k\not\equiv \lambda\, (2), 0\le k<2n-1-\lambda$ \\
\hline
&$C_\lambda^*$&$\lambda\in\{1,2,\dots,2n-1\}$&$k=2n-1-\lambda$\\
\hline
$n=1$:\\
  \hline
${\bf \widehat G_H'}$&$I^{[0],\lambda}$& $\lambda\in (1,3)$ & $k=0$ \\
 \hline
    \hline
${\bf \widehat G_H''}$&$C_\lambda^*$&$\lambda=1$&$k=0$\\
\hline
\end{tabular}
\end{center}
\end{table}

\begin{table}[h]
\caption{$(G,H,L)=(SU(2,2n),S(U(1)\times U(1,2n)), Sp(1,n))$}\label{tauto}
\begin{center}
\renewcommand{\arraystretch}{1.3}
\begin{tabular}{|l||l||l||l|}
  \hline
  &${\bf W_\rho,\ \rho=\rho_\lambda}$ & ${\bf Parameters}$ & ${\bf \pi^k_\lambda\in\widehat L_{\rho_\lambda}}$\\
  \hline
    \hline
  ${\bf \widehat G_H}$& $I^{\lambda}$ & $\lambda\in i[0,\infty)\cup (0,1)$ & $k\in\N_0$ even  \\
\hline
&$D_\lambda$ &$\lambda\in\N$ odd&$ k\ge2n-1+\lambda$ even \\
\hline
&$\C$ &$\lambda=2n+1$ & $ k=0$\\
   \hline
$n\ge 2$:\\
    \hline
${\bf \widehat G_H'}$&$I^{\lambda}$& $\lambda\in (1,2n-1)\setminus\{3,5,\dots,2n-3\}$ & $ 0\le k<2n-1-\lambda$ even \\
 \hline
    \hline
${\bf \widehat G_H''}$&$E_{\lambda}^*$&$\lambda\in\{1,3,\dots,2n-3\}$&$ 0\le k<2n-1-\lambda$ even \\
\hline
&$C_\lambda^*$&$\lambda\in\{1,3,\dots,2n-1\}$&$k=2n-1-\lambda$\\
\hline
$n=1$:\\
    \hline
${\bf \widehat G_H'}$&$I^{\lambda}$& $\lambda\in (1,3)$ & $k=0$ \\
 \hline
    \hline
${\bf \widehat G_H''}$&$C_\lambda^*$&$\lambda=1  $&$k=0$\\
\hline
\end{tabular}
\end{center}
\end{table}

\newpage

We finish 
the subsection by establishing that for triples of Type~I (except when $G/H$ is a group manifold) there are cocompact lattices $\Gamma\subset L$ such that the exceptional 
subspace of $L^2(Y)$ considered in Prop.~\ref{wanndenn} is non-zero. We treat the contributions of $\widehat G_H'$ and  $\widehat G_H''$ separately.

\begin{pro}\label{nanu}
Let $(G,H,L)$ be a triple of Type~Ia with $L_{min}$ is locally isomorphic to $SO_e(1,m)$ or $SU(1,m)$, $m\ge 2$. In other words, $(G,H,L)$ is equivalent to one of the Triples 1, 2, 4, 6, 7 (with no additional restrictions on $n$) or to Triple 3 or 5 with $n=1$ in Table \ref{fix}. As in Tables~\ref{fax}--\ref{tauto} we consider the $K_L$-spherical
representations of the complementary series $\pi^0_\lambda$ of $L$, $\lambda\in (0,\rho_L)$. 
Then for every $\ve>0$ there exist a cocompact  lattice $\Gamma\subset L$ and $\lambda\in (\rho_L-\ve,\rho_L)$ such that
$N_\Gamma(\pi_\lambda^0)\ne\{0\}$. For $\ve$ sufficiently small, the corresponding $G$-representation belongs to $\widehat G_H'$. In particular, there exists a compact quotient $Y$ of $X$ such that the  contribution of $\widehat G_H'$ to $L^2(Y)$ is non-zero.
\end{pro}

\begin{proof}
It is well-known that there are compact real and complex hyperbolic manifolds of any dimension that have an arbitrarily small first non-zero Laplace eigenvalue (on functions).
This provides, for every $\ve>0$, the desired lattice $\Gamma\subset L$. Let us give some detail for a possible construction of such manifolds/lattices. We start with a compact real or complex hyperbolic manifold $Z_0=\Gamma_0\backslash L/K_L$ having non-zero first Betti number. Its existence is ensured by Millson \cite{Mill76} in the real and Kazhdan \cite{Kazh77} in the complex case.
In particular, there is a surjective homomorphism $\psi: \Gamma_0\rightarrow \Z$. For $m\in \N$, $m\ge 2$, we consider the non-trivial character $\vp_m:\Gamma_0\rightarrow U(1)$ given by $\vp_m(\gamma):=e^\frac{2\pi i \psi(\gamma)}{m}$, $\gamma\in \Gamma$. Let $\Gamma_m:=\ker \vp_m=\psi^{-1}(m\Z)$. The first Laplace eigenvalues $\mu_1(\vp_m)$ on sections $L^2(Z_0, E_{\vp_m})$ of the corresponding flat vector bundle over $Z_0$ are non-zero. Since for $m\to\infty$ the characters $\vp_m$ converge to the
trivial character,  the eigenvalues $\mu_1(\vp_m)$ converge to $0$, cf. the discussion preceding Prop.~\ref{nottoolate} for all that. We obtain a sequence $\lambda_m\to \rho_L$, $\lambda_m<\rho_L$, with $N_{\Gamma_m}(\pi_{\lambda_m}^0)\ne\{0\}$. Now we set $\Gamma:=\Gamma_m$ and $\lambda:=\lambda_m$ for $m$ sufficiently large.

For $\lambda<\rho_L$ and sufficiently close to $\rho_L$ we consider the principal series representation $I^\lambda$ for $G/H$ induced from the trivial representation of $M_{\sigma\theta}$. Here we consider $\lambda$ as an element of $\fa_{\fq\cap\fs}^*$ as usual. It is an irreducible $H$-spherical representation of $G$ and thus is $\pi_\lambda^0$-minimal.
It cannot be unitary since $G$ is simple of real rank at least two and thus has property $(T)$ (i.e. the trivial representation is isolated in $\widehat G$). This finishes the proof.
\end{proof}

Kassel and Kobayashi conjectured in a much more general context (\cite{KK25}, Conj. 11.2) that any non-zero $L^2$-eigenfunction on $Y$, considered as a $\Gamma$-invariant eigenfunction on $X$, generates a $G$-representation that contains an irreducible {\em unitary} $G$-subrepresentation. This would imply that $\widehat G_H'$ does not contribute to $L^2(Y)$. Thus Proposition \ref{nanu} disproves this conjecture. It also falsifies the weaker Conjecture~11.3 in \cite{KK25}. We remark in addition that the lattices $\Gamma\subset L$ provided by Prop.~\ref{nanu} cannot be of  {\em congruence type}, cf. the discussion below. 

In order to exhibit contributions of $\widehat G_H''$ we need a result of  Bergeron and Clozel \cite{BeCl13},\cite{BeCl17} (in the real hyperbolic case
it goes back to Li \cite{Li93}) on small eigenvalues of  Riemannian locally symmetric spaces of negative curvature that are of congruence type: Let $L$ be among the groups $SO_e(1,m)$, 
$m\ge 4$, $SU(1,m)$, $m\ge 3$, $Sp(1,m)$, $m\ge 1$, and let $B\subset\widehat L$ be the (finite and non-empty) set of the $K_L$-spherical complementary series  representations (without end) and integral infinitesimal character.
Then there exists a cocompact (as well as a non-cocompact) lattice $\Gamma\subset L$ such that
$$  N_\Gamma(\pi) \ne \{0\}\quad\mbox{ for all } \pi\in B\ .$$
In fact, we can take for $\Gamma$ any congruence subgroup of sufficiently high level associated
to a reductive algebraic group $\LL$ defined over $\Q$ such that the Lie group $\LL(\R)$ has only one non-compact simple factor which is, in addition, locally isomorphic to $L$ (for $L=SO_e(1,7)$
some specific groups $\LL$ have to be excluded).

 Let us mention two essential ingredients of the proof of Bergeron and Clozel which are also
of independent interest in our context.
First of all, every representation $\pi\in B$ occurs as a relative discrete series representation for some 
reductive homogeneous
space $L/{H_\pi}$, $H_\pi\subset L$. In fact, by \cite{Sch87},  Thm.~9.1,  we can choose $H_\pi\subset L$ even independent of $\pi$ and symmetric: $H_\pi=SO_e(1,m-1)$, $S(U(1,m-1)\times U(1))$, $Sp(1,m-1)\times Sp(1)$, respectively (for some $\pi$ other choices of $H_\pi$ are possible).

The second ingredient, which is the main theme of \cite{BeCl13} and \cite{BeCl17}, is the vanishing of $N_\Gamma(\pi')$ for all congruence subgroups $\Gamma$ as above and all $K_L$-spherical complementary series representations $\pi'$ with non-integral infinitesimal character and induction parameter
not too close to $0$.

 For more details we refer to \cite{BeCl13}, Prop.~6.8 for the real case, to \cite{BeCl17}, end of Section 2 and \cite{BeCl05}, Thm.~6.5.1,
for the complex case, and to \cite{BeCl17}, end of Section~4.4, for the quaternionic case. We remark that for $L=Sp(1,n)$, $n\ge 2$, the authors claim an analogous result for the end
of the spherical complementary series. The claim is based on a misinterpretation of some results of \cite{Far79}. Indeed, the representation in question does not occur in $L^2(Sp(1,m)/Sp(1,m-1))$.

Now we conclude the following.

\begin{pro}\label{nonu}
Let $(G,H,L)$ be a triple of Type~Ia with $\rank(G/H)=1$ and not equivalent to $(SO_e(2,4),SO_e(1,4),U(1,2))$. In the notation of Tables~\ref{fax}--\ref{tauto}, we consider the $K_L$-spherical representations $\pi^0_\lambda$ of $L$. Depending on the equivalence class of $(G,H,L)$, we set
$$
\Lambda:=\left\{\begin{array}{ll} 
\{1,3,\dots,n-2\}&\mbox{Triple 1 with }n\ge 3\mbox{ odd}\\
\{2,4,\dots,n-2\}&\mbox{Triple 1 with }n\ge 4\mbox{ even}\\
\{1,3,\dots,2n-3\}&\mbox{Triples 3 and 5 with }n\ge 2\\
\{1\}&\mbox{Triples 3 and 5 with }n=1\\
\{1,3,5\}&\mbox{Triple 6}
\end{array}\right. \ . 
$$ 
Then there exists a cocompact (congruence) lattice $\Gamma\subset L$ such that
$N_\Gamma(\pi_\lambda^0)\ne\{0\}$ for all $\lambda\in\Lambda$. In particular, there exists a compact quotient $Y$ of $X$ such that the  contribution of $\widehat G_H''$ to $L^2(Y)$ is non-zero.
\end{pro}

It is quite likely that the same construction also produces non-zero contributions of $\widehat G_H''$ to $L^2(Y)$ for the higher rank triples 2,4,7. To establish this,
one has to show that at least a part of the set $B\subset \widehat L$ considered above is not captured by $\widehat G_H$ (and also not by $\widehat G_H'$).

 
As visible from Tables~\ref{fax}--\ref{tauto}, the $\pi_\lambda^0$-minimal $G$-representation for $\lambda\in\Lambda$ is equal to $C_\lambda$ for Triple 6 as well as
for Triples 3 and 5 with $n=1$. Note that the irreducible quotient of $C_\lambda$ is not unitarizable (this follows from the results of \cite{HoTa93} and arguments as in the proof of 
Prop.~\ref{kopp}). Thus in these cases there is even not a weakly $H$-spherical unitary $G$-representation
that can be associated to the corresponding $\Gamma$-invariant eigenfunctions on $X$. In the remaining cases, the minimal representation is $E_{[0],\lambda}$ or $E_\lambda$, respectively, which has a unitarizable irreducible quotient \cite{HoTa93}.

Combining the previous two propositions with the classification results Theorem \ref{listck} and Prop.~\ref{productck} we easily obtain the following interesting result.

\begin{cor}\label{nanunana}
Let $(G,H,L)$ be a triple of Type~I such that $G/H$ is not a group manifold. Then there exists  a standard quotient $Y$ of $G/H$ 
such that the contribution of $\widehat G_H'\cup \widehat G_H''$  to  $L^2(Y)$ is non-zero.
\end{cor}

\subsection{Results and conjectures for Type~II triples.}\label{TII} Eventually we turn to the discussion of general properly transitive triples $(G,H,L)$ with emphasis on triples of Type~II. We first observe that at least for generic $\chi\in\Hom({\bf D}(G/H),\C)$
we have a description of the distributional eigenspaces $E_\chi^{-\infty}(Y)$ which is completely analogous to the description given in Thm.~\ref{gogo} for Type~I triples.
Indeed, the following statement is an immediate corollary of Lemma~\ref{otto}(ii) and Lemma \ref{kirchner} (alternatively, one can combine Lemma \ref{bassa} with Prop.~\ref{paulklee}(ii)).

\begin{pro}\label{monika}
Let $(G,H,L)$ be a properly transitive triple, and let $\chi\in\Hom({\bf D}(G/H),\C)$ be generic.  Let $\Gamma\subset L$ be a cocompact lattice. Then the map
$$  \Hom_L(W_{\rho,\infty},V_{\pi,\infty})\otimes N_\Gamma(\pi)\otimes (W_{\rho,-\infty})^H\ni \Phi\otimes\tilde v\otimes w \mapsto 
c^G_{\phi^t\tilde v,w}=c^L_{\tilde v, E(\phi\otimes w)}\in C^{-\infty}(Y) $$
induces an isomorphism
$$ \bigoplus_{\rho\in \widetilde G_{H,\chi}}\overline{\bigoplus_{\pi\in\widehat L}}\:  \Hom_L(W_{\rho,\infty},V_{\pi,\infty})\otimes N_\Gamma(\pi)\otimes (W_{\rho,-\infty})^H
\cong E_\chi^{-\infty}(Y)\ .$$
\end{pro}
Note that in the second sum only those $\pi\in\widehat L$ with $(V_{\pi,-\infty})^{L\cap H}_\chi\ne\{0\}$ can occur. Compared to Type~I triples, our understanding of the crucial 
spaces $\Hom_L(W_{\rho,\infty},V_{\pi,\infty})$ is still quite limited in general. Recall that, for unitary $\rho$, they contain the multiplicity spaces $M_{\tilde\rho}(\tilde\pi)$ for
the restriction of $\tilde\rho$ to $L$, see (\ref{putzi}) and Lemma~\ref{sardine}. Thus information on these spaces is usually only available in those cases where the corresponding restriction problem has been already investigated in some detail. This is the case in particular for those strongly reductive $L\subset G$ such that $\Hom_L(W_{\rho,\infty},V_{\pi,\infty})$ is finite dimensional for all admissible $\rho$ and $\pi$ of finite length. Accorging to \cite{KO13}, such pairs $(G,L)$ are characterized by the existence of an open $P_L$-orbit on $G/P$, where $P$ is the minimal parabolic of $G$. A further discussion of these pairs (also called strongly spherical) including their classification (based on \cite{KM14}, \cite{KKPS}, \cite{KKPSII}) can be found in \cite{Mol17}. 
The underlying pairs $(G,L)$ of the first two triples in Prop.~\ref{knarz} below satisfy this criterion. The same is true for the last triple if $G_1=SO_e(1,n)$, $C_1=SO(n)$.
In our context, we would be interested in a weaker criterion ensuring the finite-dimensionality of $\Hom_L(W_{\rho,\infty},V_{\pi,\infty})$ for $H$-cospherical $\rho$, only.
It should be in terms of the $P_L$-action on $G/P_{\sigma\theta}$. Note that for the triples appearing in  Prop.~\ref{knarz} we have $P_{\sigma\theta}=P$.
The latter condition is also satisfied whenever $G/H$ is a group manifold, but it fails for Triples 1--8 in Table~\ref{fix}. In particular, it fails
for Triple~8, which is of Type~II.

The most striking difference to Type~I triples is not that $\Hom_L(W_{\rho,\infty},V_{\pi,\infty})$ might be infinite dimensional, even for $L=L_{max}$, but that a given $\pi$ may
contribute to non-trivial families of representations $\rho$ and characters $\chi$. 
Most importantly, any kind of analogue of Lemma \ref{huckepack} is missing. In particular, we do not have yet a description of the $L^2$-eigenspaces $E_\chi^{(2)}(Y)$ nor
do we know how the distributional eigenspaces for $\chi\in\Xi$ fit together to provide a spectral decomposition. To answer these questions, one would have to establish Conjectures~\ref{C1} and \ref{C2} below and to make their ingredients more explicit. 


\begin{pro}\label{knarz}
Let $(G,H,L)$ be one of the following triples of Type~II:
\begin{itemize}
\item $(SO_e(4,4),SO(3)\times SO_e(1,4), Spin(3,4))$,
\item $(SO(8,\C),SO_e(1,7),Spin(7,\C))$,
\item $(G_1^4,\Delta_{12}G_1\times\Delta_{34}G_1,G_1\times \Delta_{23}G_1\times C_1)$, $rank_\R G_1=1$, $C_1\subset G_1$ compact.
\end{itemize}
We fix a $K_L$-spherical unitary principal series representation of $L$. It defines an element $\pi\in \widehat L$. Let $I^\lambda$, $\lambda\in\fa_{\fq\cap\fs,\C}^*$, be the principal series for $G$ induced from
the minimal $\sigma\theta$-stable parabolic $P_{\sigma\theta}$ and the trivial $M_{\sigma\theta}$-representation $\delta$. Then there exists a complex affine subspace
$\Lambda\subset \fa_{\fq\cap\fs,\C}^*$ such that $\dim_\R(\Lambda\cap i\fa_{\fq\cap\fs}^*)=\dim_\C\Lambda>0$ and  a non-zero meromorphic family
$\Lambda\ni\lambda\mapsto \Phi_\lambda\in \Hom_L(I^\lambda_\infty, V_{\pi,\infty})$ regular on $\Lambda\cap i\fa_{\fq\cap\fs}^*$. More precisely, for the first 2 triples we can take $\Lambda=\fa_{\fq\cap\fs,\C}^*$,
while for the last triple $\Lambda$ is a one-dimensional subspace depending on $\pi$.
\end{pro}

\begin{proof} We start with the first two triples. We fix a $\sigma$-stable maximal abelian subspace $\fa\subset\fs$ (it is four-dimensional) and its three-dimensional subspace
$\fa_{\fq\cap\fs}$. The corresponding parabolics $P$ (the minimal parabolic of $G$) and $P_{\sigma\theta}$ coincide. Therefore we can consider the induced representations $I^\lambda$
not only for $\lambda\in \fa_{\fq\cap\fs,\C}^*$, but also for $\lambda\in \fa_{\C}^*$. We also fix a (three-dimensional) maximal abelian subspace $\fa_{L}\subset\fl\cap\fs$. For $\mu\in \fa_{L,\C}^*$ we consider the spherical principal series representations $H^\mu$ of $L$. We have $V_\pi=H^{\mu_0}$ for some $\mu_0\in i\fa_{L}^*$. By triality the pair $(\fg,\fl)$
is isomorphic to $(\fo(4,4), \fo(3,4))$ or $(\fo(8,\C),\fo(7,\C))$, respectively. It follows that the pair $(G,L)$ is strongly spherical 
(it is even a so-called multiplicity one pair).
Hence by \cite{Mol17}, Thm.~3.3, there exists a non-zero holomorphic family
$$\fa_{\C}^*\times \fa_{L,\C}^*\ni (\lambda,\mu)\mapsto \tilde \Phi_{\lambda,\mu}\in \Hom_L(I^\lambda_\infty, H^\mu_{\infty})\ .
$$
For $\lambda\in \fa_{\fq\cap\fs,\C}^*\subset \fa_{\C}^*$ we would like to set $\Phi_\lambda:=  \tilde \Phi_{\lambda,\mu_0}\in \Hom_L(I^\lambda_\infty, H^{\mu_0}_{\infty})$.
This will work if the restriction of the seven-dimensional family $\tilde \Phi_{\lambda,\mu}$ to the three-dimensional affine subspace $\fa_{\fq\cap\fs,\C}^*\times\{\mu_0\}$
is not identically zero. In general, we choose $(\lambda_1,\mu_1)\in \fa_{\C}^*\times \fa_{L,\C}^*$ such that
$$  \fa_{\fq\cap\fs,\C}^*\times \C\ni (\lambda,z)\mapsto \Phi_{\lambda,z}:= \tilde \Phi_{\lambda+z\lambda_1,\mu_0+z\mu_1} $$
is not identically zero and set $\Phi_\lambda:=\lim_{z\to 0} \Phi_{\lambda,z}/z^k$ for the appropriate $k\in\N_0$.

To deal with the last triple, we first derive a description of $\Hom_L(W_{\rho,\infty},V_{\pi,\infty})$ for certain pairs of representations $(\rho,\pi)$ including all elements of
$\widehat G_H\times\widehat L$ in more accessible terms. More precisely, we consider
$ W_\rho=V_1\hat\otimes V_1'\hat\otimes V_2\hat\otimes V_2'$ and $V_\pi=V_3\hat \otimes V_4\otimes F$, where $V_i$ are admissible Hilbert representations of $G_1$ of finite length and $F$ is a finite dimensional $C_1$-representation.  We drop the assumption  $rank_\R G_1=1$ for a moment.

\begin{lem}\label{tensor} Let $V_1$ and $V_3$ be irreducible. If $V_{1,\infty}\not\cong V_{3,\infty}$, then  $\Hom_L(W_{\rho,\infty},V_{\pi,\infty})=\{0\}$. If $V_1=V_3$, then there is
a canonical isomorphism
$  \Hom_L(W_{\rho,\infty},V_{\pi,\infty})\cong\Hom( [V_{2,\infty}'\otimes F']^{C_1}, Z)$, where
$$ Z=\Hom_{G_1}( (V_2\hat\otimes V_3'\hat\otimes V_4')_\infty,\C)\cong \Hom_{G_1}(V_{2,-\infty}, (V_3\hat \otimes V_4)_{-\infty})\ .$$
Here the actions of $G_1$ and $C_1$ on tensor products are the diagonal actions, while smooth and distribution vectors on tensor products are taken with respect to the actions of $G_1\times G_1(\times G_1)$.
\end{lem}

\begin{proof} We first need a number of identities of quite general nature. The first one is related to Lemma~\ref{sardine}.

{\it Claim 1: Let $G$ be connected semisimple with finite center,  $L\subset G$ be strongly reductive, and let $W_\rho$, $V_\pi$  be  representations on reflexive Banach spaces 
of $G$ and $L$, respectively. We require that $V_\pi$ is admissible of finite length.
Then the natural embeddings induce equalities 
$$\Hom_L(W_{\rho,\infty},V_{\pi,\infty})=\Hom_L(W_{\rho,\infty},V_{\pi,-\infty}),\quad \Hom_L(V_{\pi,-\infty}, W_{\rho,-\infty})=\Hom_L(V_{\pi,\infty}, W_{\rho,-\infty}).$$}
Indeed, let $\Phi\in \Hom_L(W_{\rho,\infty},V_{\pi,-\infty})$. We consider the $L$-representation on $W_\Phi:=W_{\rho,\infty}/\ker\Phi$. Since the $G$-representation $W_{\rho,\infty}$ is a smooth
Fr\' echet-representation of moderate growth we find that the $L$-representation $W_\Phi$ is also a smooth Fr\' echet-representation of moderate growth. Its underlying $(\fl,K_L)$-module $W_{\Phi, K_L}$ is isomorphic to a 
submodule of $V_{\pi,K_L}$. Hence $W_\Phi$ has finite length. Now Casselman-Wallach theory tells us that the map $\Phi_0\in\Hom_{\fl,K_L}(W_{\Phi, K_L}, V_{\pi.K_L})$ induced by $\Phi$ extends to
an element $\Phi_1\in \Hom_L(W_{\Phi},V_{\pi,\infty})$. This shows that $\Phi$ comes from an element in $\Hom_L(W_{\rho,\infty},V_{\pi,\infty})$. In view of Lemma~\ref{stock}(iv) the second equality follows from the first by taking adjoints.

{\it Claim 2: For Hilbert representations $V_i$ of Lie groups $G_i$, $i=1,2$, and any locally convex topologically vector space $U$ we have
$$   \Hom((V_1\hat\otimes V_2)_\infty, U)\cong \Hom(V_{1,\infty},\Hom(V_{2,\infty},U))\ .$$}
A straightforward computation shows that the left hand side injects naturally into the right. In order to obtain the inverse map one uses the fact
that separately continuous bilinear maps on Fr\' echet spaces are continuous. 



{\it Claim 3: Let $G_i$ be connected reductive, $i=1,2$. Let $U_i,V_i$ be Hilbert representations of $G_i$. We require $U_1,V_1$ to be irreducible admissible. Then
$$ \Hom_{G_1}((V_1\hat\otimes V_2)_\infty,(U_1\hat\otimes U_2)_\infty)\cong
\left\{\begin{array}{cc} 
\Hom(V_{2,\infty}, U_{2,\infty})&\mbox{if }\;U_{1,\infty}\cong V_{1,\infty}\\
\{0\}&\mbox{otherwise}
\end{array}\right. \ . 
$$ }
This is a rather direct consequence of Claim 2 and the Schur-Dixmier lemma.

Now we compute $\Hom_L(W_{\rho,\infty},V_{\pi,\infty})=\Hom_{G_1\times\Delta_{23}G_1\times C_1}((V_1\hat\otimes V_1'\hat\otimes V_2\hat\otimes V_2')_\infty ,(V_3\hat \otimes V_4)_\infty\otimes F )$. By Claim~3, this is the zero space unless $V_{1,\infty}\cong V_{3,\infty}$, where we have to consider
$$ \Hom_{\Delta_{12}G_1\times C_1}(( V_3'\hat\otimes V_2\hat\otimes V_2')_\infty ,V_{4,\infty}\otimes F )\ .$$
In view of Claims 1 and 2, the latter can be rewritten as $\Hom( [V_{2,\infty}'\otimes F']^{C_1}, Z_0)$ with
\begin{eqnarray*}  Z_0&=& \Hom_{G_1}(( V_3'\hat\otimes V_2)_\infty, V_{4,-\infty})= \Hom_{G_1}(( V_3'\hat\otimes V_2)_\infty,\Hom(V'_{4,\infty},\C))\\
&\cong& \Hom_{G_1}( (V_2\hat\otimes V_3'\hat\otimes V_4')_\infty,\C)=Z\ .
\end{eqnarray*}
For the last isomorphism we have used Claim 2. Using again Claims 2 and 1, we can rewrite $Z$ as
$$  \Hom_{G_1}( V_{2,\infty}, \Hom(( V_3'\hat\otimes V_4')_\infty,\C))= \Hom_{G_1}( V_{2,\infty}, ( V_3\hat\otimes V_4)_{-\infty})
= \Hom_{G_1}( V_{2,-\infty},( V_3\hat\otimes V_4)_{-\infty})\ .$$
This finishes the proof of the lemma.
\end{proof}
Now we specify to spherical principal series representations. Let $\fa_1\subset\fs_1\subset\fg_1$ be maximal abelian (one-dimensional). We have $\fa_{\fq\cap\fs}\cong \fa_1\times\fa_1$. For $\mu\in\fa_{1,\C}^*$, we consider the corresponding spherical principal series representation $H^\mu$ of $G_1$. We have to consider $W_\rho=I^\lambda$ and
$V_\pi$ as above with $V_i=H^{\mu_i}$, $\mu_i\in \fa_{1,\C}^*$, $i=1,\dots,4$,  and the trivial one-dimensional representation $F$. 
 
According to \cite{BSKZ14} there is an open subset of $(\fa_{1,\C}^*)^3$ containing $(i\fa_{1}^*)^3$ and a holomorphic family on this open subset without zeroes 
$$  \mu=(\mu_1,\mu_2,\mu_3)\mapsto \Psi_{\mu_1,\mu_2,\mu_3}\in \Hom_{G_1}( (H^{\mu_1}\hat\otimes H^{\mu_2}\hat\otimes H^{\mu_3})_\infty,\C) $$
that has a meromorphic continuation to $(\fa_{1,\C}^*)^3$ (a distinguished family of invariant trilinear forms on principal series representations). Now we fix $V_\pi$, i.e. we fix $\mu_3, \mu_4\in i\fa_{1}^*$. We consider the affine subspace 
$$ \Lambda:= \{(\mu_3,\mu)\in \fa_{1,\C}^*\times \fa_{1,\C}^*\cong  \fa_{\fq\cap\fs,\C}^*\mid \mu\in \fa_{1,\C}^*\}\subset \fa_{\fq\cap\fs,\C}^*\ .$$
Let $I_\mu: H^\mu\rightarrow \C$ be the functional given by integration over $K_1\subset G_1$. We then define the desired meromorphic family as follows:
$$ \Lambda\ni \lambda= (\mu_3,\mu)\mapsto \Phi_{\lambda}\in \Hom_L(I^\lambda_\infty, V_{\pi,\infty})\cong \Hom ([H^{-\mu}_\infty]^{C_1}, \Hom_{G_1}((H^\mu\hat\otimes H^{-\mu_3}\hat\otimes H^{-\mu_4})_\infty,\C))$$
by
$$  \Phi_{(\mu_3,\mu)}(f):= I_{-\mu}(f)\cdot \Psi_{\mu,-\mu_3,-\mu_4}\ ,\quad f\in [H^{-\mu}_\infty]^{C_1}\ .$$
This finishes the proof of the proposition.
\end{proof}


\begin{cor}\label{knorz}
Let $(G,H,L)$ be as in Prop.~\ref{knarz}. Let $\Gamma\subset L$ be discrete and cocompact.
Then there exist a spherical principal series representation of $L$ with a corresponding subspace $\Lambda\subset \fa_{\fq\cap\fs,\C}^*$ as in Prop.~\ref{knarz} and 
a non-zero continuous family of eigendistributions $\Lambda\cap i\fa_{\fq\cap\fs}^*\ni\lambda\mapsto F_\lambda\in  E_{\chi_\lambda}^{-\infty}(Y)$.
Here $\chi_\lambda\in \Xi\subset\Hom({\bf D}(G/H),\C)$ is determined by the infinitesimal character of $I^\lambda$. 
\end{cor}

\begin{proof} It is well-known (\cite{DKV79}) that there is a spherical unitary princpial series representation $\pi$ of $L$ with $N_\Gamma(\pi)\ne\{0\}$ (in fact, there are infinitely many).
We consider the corresponding subspace $\Lambda$ as in Prop.~\ref{knarz}. We fix $0\ne \tilde v\in (V_{\tilde\pi,-\infty})^\Gamma\cong N_\Gamma(\pi)$. By the theory of $H$-spherical distribution vectors \cite{Ban88} we can choose a non-zero continuous family $\Lambda\cap i\fa_{\fq\cap\fs}^*\ni\lambda\mapsto w_\lambda\in (I^\lambda_{-\infty})^H$.  Now let $\Phi_\lambda$ as in Prop.~\ref{knarz}. By the identity theorem, $\Phi_\lambda$ does not vanish identically on any non-empty open subset of 
$\Lambda\cap i\fa_{\fq\cap\fs}^*$.
Eventually, we set $F_\lambda=c^G_{\Phi_\lambda^t\tilde v, w_\lambda}$ (cf. Prop.~\ref{monika}).
\end{proof}

We expect that Prop.~\ref{knarz} and therefore also Cor.~\ref{knorz} hold for all triples $(G,H,L)$ of Type~II.
 
The existence of the families $F_\lambda$ indicates that compact quotients corresponding to the above triples (or more generally to triples of Type~II) have not only discrete but also
various kinds of continuous spectrum. In fact, that this is true for the last case in Prop.~\ref{knarz}  with $G_1=PSL(2,\R)$ is the content of Section~\ref{infinitemulti}, cf. also Prop.~\ref{last} below. Note however 
that in general the existence of the families $F_\lambda$ does not a priori exclude the existence of a discrete spectral decomposition of $H=L^2(Y)$ of the form (\ref{disc}). It does not even exclude the existence of a complete orthonormal
system consisting of joint smooth eigenfunctions of ${\bf D}(G/H)$ as in Prop.~\ref{schwoch}. But it does exclude the existence of such a strong discrete spectral decomposition as in 
Thm.~\ref{selim}
and Thm.~\ref{gogo} which also gives decompositions of  $C^{\pm\infty}(Y)$.

In order to show that families of eigendistributions like $F_\lambda$ above are really the building blocks for a non-discrete spectral decomposition of $L^2(Y)$, one has to associate wave packets to them, show that the wave packets define $L^2$-functions (for suitable choices of $\Phi_\lambda$, $w_\lambda$ - notation as in the proof of Cor.~\ref{knorz}) and that $L^2(Y)$ is spanned by $L^2$-eigenfunctions and wave packets (of course associated to more general $\pi$ than just $K_L$-spherical principal series). Thanks to Lemma~\ref{bassa} and Prop.~\ref{paulklee} we can get rid of $\Gamma$ and need only to produce spectral decompositions for $V_\pi^{L\cap H}$, $\pi\in\widehat L$, via wave packets associated to families such as $E(\Phi_\lambda\otimes w_\lambda)\in (V_{\pi,-\infty})^{L\cap H}$.
We believe that the following conjecture is true, at least for generic tempered representations $\pi\in\widehat L$ having an $L\cap H$-invariant vector.

\begin{con}\label{C1}
Let $(G,H,L)$ be a properly transitive triple. There exist
\begin{itemize} 
\item a topological space $\widehat G_H^L\supset \widehat G_H$ whose points are certain (infinitesimal) equivalence classes of $H$-cospherical $G$-representations of finite
length with infinitesimal character which gives rise to a continuous finite-to-one map $p: \widehat G_H^L\rightarrow \Xi$,
\item for each $\pi\in\widehat L$ with $V_\pi^{L\cap H}\ne\{0\}$, a family of subspaces 
$$\widehat G_H^L\ni\rho\mapsto\left [\Hom_L(W_{\rho,\infty},V_{\pi,\infty})\otimes (W_{\rho,-\infty})^H\right ]_0\subset \Hom_L(W_{\rho,\infty},V_{\pi,\infty})\otimes (W_{\rho,-\infty})^H$$ 
carrying the additional structure of a measurable family of Hilbert spaces,
\item for each $\pi\in\widehat L$ with $V_\pi^{L\cap H}\ne\{0\}$, a $\sigma$-finite Borel measure $\mu_\pi$ on $\widehat G_H^L$
\end{itemize}
such that the natural ${\bf D}(G/H)$-equivariant maps 
$$   E: \Hom_L(W_{\rho,\infty},V_{\pi,\infty})\otimes (W_{\rho,-\infty})^H\longrightarrow  (V_{\pi,-\infty})^{L\cap H}$$
considered in Prop.~\ref{paulklee} induce a unitary equivalence
\be\label{D1}   V_\pi^{L\cap H} \cong \int^\oplus_{\widehat G_H^L}\left [\Hom_L(W_{\rho,\infty},V_{\pi,\infty})\otimes (W_{\rho,-\infty})^H\right ]_0\:d\mu_\pi(\rho)
\end{equation}
for all $\pi\in\widehat L$ with $V_\pi^{L\cap H}\ne\{0\}$.
\end{con}

In view of Lemma~\ref{bassa}, tensoring with $N_\Gamma(\pi)$, employing the embeddings 
$$i_\pi : \Hom_L(W_{\rho,\infty}, V_{\pi,\infty})\otimes N_\Gamma(\pi)\rightarrow (W_{\tilde \rho,-\infty})^\Gamma$$
from Lemma~\ref{kirchner}, and summing up over $\pi\in\widehat L$, the validity of Conjecture~\ref{C1} would readily imply the validity of the
following conjecture.

\begin{con}\label{C2}
Let $(G,H,L)$ be a properly transitive triple, and let $\Gamma\subset L$ be  a cocompact lattice. Let $Y=\Gamma\backslash G/H$ as usual.
Then there exist a family of subspaces
$$ \widehat G_H^L\ni\rho\mapsto\left [(W_{\tilde \rho,-\infty})^\Gamma\otimes  (W_{\rho,-\infty})^H\right ]_0\subset (W_{\tilde \rho,-\infty})^\Gamma\otimes  (W_{\rho,-\infty})^H$$
carrying the structure of a measurable family of Hilbert spaces 
and a $\sigma$-finite Borel measure $\mu_\Gamma$ on $\widehat G_H^L$
such that the matrix coefficient map (\ref{paul}) induces a unitary equivalence
\be\label{D2}  L^2(Y) \cong \int^\oplus_{\widehat G_H^L} \left [(W_{\tilde \rho,-\infty})^\Gamma\otimes  (W_{\rho,-\infty})^H\right ]_0\:d\mu_\Gamma(\rho)\ .
\end{equation}
\end{con}

Note that the push-forward of the direct integral decompositions (\ref{D1}) and (\ref{D2}) with respect to the  finite-to-one map $p: G_H^L\rightarrow \Xi$ give coarser decompositions
which are in fact spectral decompositions in the sense of Def.~\ref{kobold1} of the algebra ${\bf D}(G/H)$ acting on the Hilbert spaces  $V_\pi^{L\cap H}$ and $L^2(Y)$, respectively.
The uniqueness of these spectral decompositions still remains to be investigated, see the discussion in Section \ref{G2}.

Let us emphasize that the conjectures are certainly true for triples of Type~I (discrete decompositions by Lemma \ref{huckepack} and Thm.~\ref{gogo}).
In this case, the measure $\mu_\pi$ is just the Dirac measure supported at the conjugate dual of the $\pi$-minimal $G$-representation.
The conjectures are also true for the last triple in Prop.~\ref{knarz} (even without any restriction on the real rank). In this case, they are consequences of the abstract Plancherel theorem for the restriction of tensor products to the diagonal.

\begin{pro}\label{last}
Let $G_1$ be a connected semisimple Lie group with finite center and without compact factors. Let $C_1\subset G_1$ be compact. We consider 
$$ (G,H,L)=(G_1^4,\Delta_{12}G_1\times\Delta_{34}G_1,G_1\times \Delta_{23}G_1\times C_1)\ .$$
Then Conjecture~\ref{C1} (and hence also Conjecture~\ref{C2}) is true for $(G,H,L)$ with $\widehat G_H^L=\widehat G_H$.
\end{pro}
\begin{proof}
Let $\pi\in\widehat L$ such that $V_\pi^{L\cap H}\ne\{0\}$. We have $V_\pi=V_3\hat\otimes V_4\otimes F$ for irreducible unitary representations $V_i$ and $F$ of $G_1$ and $C_1$,
respectively. Now we look at $V_3\hat\otimes V_4$ as a representation of the diagonal subgroup of $G_1\times G_1$. As a unitary representation of $G_1$, it has a direct
integral decomposition into irreducible representations
$$  V_3\hat\otimes V_4\cong \int_{\widehat {G_1}}^\oplus M_\pi(\eta)\hat\otimes V_\eta\: d\nu_\pi(\eta)$$
as in (\ref{lachs}). Tensoring with $F$ and taking $C_1$-invariants (with respect to the diagonal action), we obtain
\be\label{klaus2} V_\pi^{L\cap H}\cong \int_{\widehat{G_1}}^\oplus M_\pi(\eta)\hat\otimes [V_\eta\otimes F]^{C_1}\: d\nu_\pi(\eta)\ .
\end{equation}
We consider the continuous injective map $i_\pi: \widehat{G_1}\rightarrow \widehat G_H$ given by $i_\pi (V_\eta):=V_3\hat\otimes V_3'\hat\otimes V_\eta\hat\otimes V_\eta'$.
Let $\mu_\pi:=i_{\pi,*}\nu_\pi$ be the push-forward measure.

We have $M_\pi(\eta)\hat\otimes [V_\eta\otimes F]^{C_1}=\Hom_{HS}( [V_\eta'\otimes F']^{C_1}, M_\pi(\eta))\subset \Hom( [V_{\eta,\infty}'\otimes F']^{C_1}, M_\pi(\eta))$
and, by Lemma~\ref{sardine}, $M_\pi(\eta)\subset  \Hom_{G_1}(V_{\eta,-\infty}, (V_3\hat \otimes V_4)_{-\infty})$. By Lemma~\ref{tensor}, we obtain a canonical
embedding 
$$    M_\pi(\eta)\hat\otimes [V_\eta\otimes F]^{C_1}\hookrightarrow \Hom_L (i_\pi(V_\eta)_\infty, V_{\pi,\infty})=\Hom_L (i_\pi(V_\eta)_\infty, V_{\pi,\infty})\otimes ((i_\pi(V_\eta)_{-\infty})^H \ .$$
For the last equality, we have used that $(W_{\rho,-\infty})^H$ is canonically isomorphic to $\C$ for any $\rho\in\widehat G_H$ (as always for group manifolds $G/H$).
Let $\left [\Hom_L (i_\pi(V_\eta)_\infty, V_{\pi,\infty})\otimes ((i_\pi(V_\eta)_{-\infty})^H\right ]_0$ be the image of this embedding. If $W_\rho$ is not of the form $i_\pi(V_\eta)$ for some $\eta\in \widehat G_1$, we set 
$\left [\Hom_L(W_{\rho,\infty},V_{\pi,\infty})\otimes (W_{\rho,-\infty})^H\right ]_0=\{0\}$. Now we can rewrite (\ref{klaus2}) in the desired way:
$$   V_\pi^{L\cap H} \cong \int^\oplus_{\widehat G_H}\left [\Hom_L(W_{\rho,\infty},V_{\pi,\infty})\otimes (W_{\rho,-\infty})^H\right ]_0\:d\mu_\pi(\rho)\ .$$
We leave it to the reader to check that this isomorphism is indeed induced by the maps $E$ introduced in Prop.~\ref{paulklee}.
\end{proof}

From the explicit branching rules for tensor products of representations of $G_1=PSL(2,\R)$, as recalled in Section~\ref{infinitemulti}, one can make rather explicit
all the ingredients in (\ref{D2}) for the triple considered in Section~\ref{infinitemulti}, which yields a refinement of the spectral decomposition obtained in Thm.~\ref{surprise}.
For explicit branching rules for tensor product representations of groups not locally isomorphic to $SL(2,\R)$ see e.g. \cite{Martin75}, \cite{Spyros25} and the literature cited therein. 

We note that the strategy of the proof
of Prop.~\ref{last} (combined with Lemma \ref{huckepack}) can be applied to the more general triples of the form 
$$(G,H,L)=(L_1\times L_1\times G_1, \Delta_{12}L_1\times H_1, (L_1\times \Delta_{23} L_1)\cdot C_1), $$
where $(G_1,H_1,L_1)$ is
a triple of Type~I with $L_1$ minimal and $C_1\subset Z_{G_1}(L_1)$ compact. However, for  irreducible triples of Type~II such simple arguments will not work. 

Establishing Conjecture~\ref{C1} 
for more triples $(G,H,L)$ and some $\pi\in\widehat L$, as well as making the ingredients of Conjectures~\ref{C1} and \ref{C2} as explicit as possible, is work in progress. 

We expect that the measure $\mu_\Gamma$ in Conj.~\ref{C2} is the sum of countably many measures $\mu_i$ supported on subsets of $\widehat G_H^L$ that can be naturally parametrized (in a finite-to-one way) by
certain real affine spaces (like $\Lambda\cap i\fa_{\fq\cap\fs}^*$ in Cor.~\ref{knorz}) of dimension $c_i\in\{0,1,\dots,\rank_\R(G/H)\}$ and that $\mu_i$ is absolutely continuous with respect to the push forward of the corresponding Lebesgue measure. The maximum $c$ of these dimensions could be considered as the dimension of the most continuous part of the
spectrum of ${\bf D}(G/H)$ on $L^2(Y)$. Looking at the identification of elements in $\Hom_L(W_{\rho,\infty}, V_{\pi,\infty})$ for induced representations $W_\rho$, $V_\pi$ with
certain $P_L$-invariant distributional sections of bundles over $G/P_{\sigma\theta}$ via Schwartz' kernel theorem (as it has been employed in \cite{Mol17})  our guess for a
geometric description of $c$ in terms of the stabilizers of the $P_L$-action on $G/P_{\sigma\theta}$ is the following: Let $r: P_{\sigma\theta}\rightarrow A_\fq$ be the natural projection.
Then 
$$ c=\rank_\R(G/H)-\min_{g\in G} \dim r\left (P_{\sigma\theta} \cap g^{-1}P_Lg\right )\ .$$
This yields $c=0$ for triples of Type~I (by Lemma~\ref{sigmatheta}, Cor.~\ref{renoir}, and the Bruhat decomposition of $L$) and $c=3,3,1$, respectively, for the triples appearing in Prop.~\ref{knarz}, as it should be.

\appendix
\section{}\label{A}
In the notation of Section \ref{pbw}, we compute $\mathrm{Spec}(G,H,L) =\text{Spec}(F)\cap\text{Spec}(\fl\cap\fs)$ for the Triples 5--10 listed in Table 1 in Section \ref{triples}.
To be more precise, we do this for Triples 6--10, while for Triple 5 we replace $L_{max}$ by $L_{min}$. By the observations made in Section~\ref{pbw}, we  know in advance
that for the spherical triples~5--7 the representation of $L\cap H$ on $\fl\cap \fs$ is irreducible, and thus $\text{Spec}(\fl\cap\fs)$ reduces to a singleton.
\vspace*{0.3cm}


\underline{$\text{\bf Triple 5}$}: $G=SO_e(4,4n)$, $H=SO_e(3,4n)$, $L=Sp(1,n),\;\; n\geq 1$.\\
We identify the pseudo-Euclidean vector space $\R^{4,4n}$ with $\HH^{1+n}$, which carries the structure of a right vector space over the quaternions $\HH$.
Then the Lie algebra $\fl=\fsp(1,n)$ is just the subalgebra of all elements of $\fo(4,4n)$ which are quaternionic linear.
We have identifications
$$  \fl\cap \fs=\Hom_\HH(\HH,\HH^n)\subset \Hom_\R(\HH,\HH^n)=\fs\ .$$
The trace form on $\fo(4,4n)$ corresponds to the bilinear form $\IP{A}{B}=2\Tr_\R(A^t B)$ on $\Hom_\R(\HH,\HH^n)$. Here transposition is with respect to the standard
Euclidean structures on $\HH$ and $\HH^n$. The involution $\sigma$ corresponds to right multiplication with the involution $\sigma_0:p\mapsto -\bar p$ on $\HH$.
The standard embedding $i_1: \HH=\HH^1\hookrightarrow\HH^n$ is quaternionic linear, and thus defines an element of $\fl\cap\fs$.
We have $\IP{i_1}{i_1}=8$ and $\IP{i_1\circ \sigma_0}{i_1}=4$.
We deduce that $\mu=\frac{1}{2}$ is the only ``non-compact eigenvalue''. 

We have $\fl\cap \fk=\fsp(1)\oplus\fsp(n)$ and $\fl\cap \fh=\fsp(n)$. It follows that $F=\fsp(1)$. We observe that $F$ is not only an ideal of $\fl\cap\fk$ but also of
$\fk$: We have the chain of ideals $\fsp(1)\subset\fo(4)\subset\fo(4)\oplus \fo(4n)=\fk$. We also observe that $F\subset \fk$ is not $\sigma$-stable. Since it is a simple
ideal we conclude that $F\cap \sigma(F)=\{0\}$ and thus $F\perp \sigma(F)$. Hence $\lambda=0$ is the only ``compact eigenvalue''. (We have just obtained the classical decomposition $\fo(4)=\fsp(1)\oplus\fsp(1)=F\oplus\sigma(F)$ corresponding to left and right multiplication on $\HH\cong\R^4$.)
\vspace*{0.3cm}

\underline{$\text{\bf Triple 6}$}: $G=SO_{e}(8,8)$, $H=SO_{e}(7,8)$, $L=Spin(1,8)$.\\ 
Let $\{e_0,e_1,\cdots,e_8\}$ be an orthonormal basis of ${\mathbb R}^{1,8}$ equipped with the inner product of signature $(-,+,\cdots,+)$ and consider the associated 
real Clifford algebra ${\bf Cl}({\mathbb R}^{1,8})$ such that $e_0^2=1$, $e_i^2=-1$, $i\ge 1$. This Clifford algebra admits a unique, up to isomorphism, $16$-dimensional real module ${\bf S}\simeq{\mathbb R}^{16}$ (see, for instance, \cite{De99}). Under the action of $e_0$, ${\bf S}$ splits into $\pm 1$-eigenspaces: ${\bf S}={\bf S}^+\oplus {\bf S}^-$. The spin group $L=Spin(1,8)\subset{\bf Cl}({\mathbb R}^{1,8})$ leaves an inner product with signature $(8,8)$ on ${\mathbb R}^{16}$ invariant which makes the splitting orthogonal.
Moreover, the restriction of the inner product to either summand is definite. We will fix isometric (up to sign) identifications ${\bf S}^+\cong\R^8\cong {\bf S}^-$. Among other things, this fixes an embedding $Spin(1,8)\hookrightarrow SO_e(8,8)$. In fact, ${\bf S}$ is the $16$-dimensional real spin representation of $Spin(1,8)$, while ${\bf S}^\pm \simeq{\mathbb R}^8$ are the $8$-dimensional half-spin representations of $L\cap K\simeq Spin(8)\subset Spin(1,8)$. 
The involution $\sigma$ is given as the mapping $(\sigma^+,\text{Id}_{\R^8})$ on $\R^8\oplus\R^8\cong {\bf S}^+\oplus {\bf S}^-$, where $\sigma^+$ is the reflection with respect to $\R^7\subset \R^8$. 

We can view the half spin representations of the Lie algebra of $Spin(8)$ as  homomorphisms $s^\pm:\fo(8)\rightarrow\fo(8)$, which are in fact (outer) automorphisms. This
gives us the following description of $\fl\cap\fk$ as a subalgebra of $\fk=\fo(8)\oplus\fo(8)$:
$$ \fl\cap \fk =\{(s^+(X),s^-(X))\mid X\in\fo(8)\}=\{(Y,\gamma(Y))\mid Y\in\fo(8)\}\subset\fo(8)\oplus\fo(8), $$
where $\gamma:=s^-(s^+)^{-1}$. It follows that
$$ \fl\cap\fh=\{(Y,\gamma(Y))\mid Y\in\fo(7)\}=\{(Y,\gamma(Y))\mid Y\in\fo(8),\sigma^+(Y)=Y\}$$
and that 
$$ F= \{(Y,\gamma(Y))\mid Y\in\fo(8),\sigma^+(Y)=-Y\}\ .$$
For the latter equality we have used that the automorphism $\gamma$ leaves the restriction of $\IP{}{}$ to the second $\fo(8)$-summand invariant (as it does any automorphism with a multiple of the Killing form). Again using this invariance, we obtain for $(Y_i,\gamma(Y_i))\in F$, $i=1,2$:
$$ \langle \sigma(Y_1,\gamma(Y_1)),(Y_2,\gamma(Y_2)\rangle=-\langle Y_1,Y_2\rangle + \langle Y_1,Y_2\rangle= 0.$$
We conclude that there is only one {\it compact} eigenvalue, namely $\lambda=0$.

Finally, observing that elements in ${\mathfrak l}\cap{\mathfrak s}$ are spanned by $X_i=e_0e_i\in \fspin(1,8)\subset{\bf Cl}({\mathbb R}^{1,8})$ for $i=1,\cdots,8$, we have (using the trace form on $\fo({\bf S})$): 
\begin{eqnarray*}
\langle X_i,X_i\rangle&=&\text{Tr}_{{\bf S}}\big(e_0e_i\big)^2\\
&=&\text{Tr}_{{\bf S}}\big(-e_0^2e_i^2\big)\\
&=&16
\end{eqnarray*}
and
\begin{eqnarray*}
\langle \sigma(X_i),X_i\rangle&=&\text{Tr}_{{\bf S}^+}\big(\sigma^+\circ \underbrace{(e_0 e_i e_0 e_i)}_{\text{Id}}\big)+\text{Tr}_{{\bf S}^-}\big((e_0 e_i )\circ\sigma^+\circ\underbrace{(e_0 e_i)}_{(e_0 e_i)^{-1}}\big)\\
&=&2\text{Tr}_{{\bf S}^+}\big(\sigma^+\big)\\
&=&12.
\end{eqnarray*}
We deduce that the only {\it non-compact} eigenvalue is $\mu=\frac{3}{4}$.
\vspace*{0.3cm}

\underline{$\text{\bf Triple 10}$}: $G=SO(8,\mathbb{C})$, $H=SO_e(1,7)$, $L=Spin(7,{\mathbb C})$.\\
We first note that the representation of $L\cap H\cong G_2$ on $\fspin(7)/\fg_2\cong F$ is the seven-dimensional {\it irreducible} representation of $G_2$.
Thus there is only one {\it compact} eigenvalue $\lambda$. To compute it we have to exhibit at least one non-zero element in $F$. We do it using a $\sigma$-stable Cartan algebra $\ft$
of $\fk=\fo(8)$ containing a Cartan subalgebra of $\fspin(7)\subset\fo(8)$. According to the discussion around (\ref{cartanprinciple}) we have $(\fl\cap\fh\cap\ft)^{\perp_{\fl\cap\ft}}\subset F$.
In fact, we consider the standard Cartan subalgebra ${\mathfrak t}\cong\R^4$ of $\fo(8)$ consisting of block-diagonal matrices:
\begin{equation*}
[a_1,a_2,a_3,a_4]:=\left(\begin{array}{cccccc}
\left[\begin{array}{cc}
0 & a_1\\
-a_1 & 0
\end{array}\right] &  &  &  \\
 & \ddots\\
 &  & \ddots \\
 &  &  & \left[\begin{array}{cc}
0 & a_4\\
-a_4 & 0
\end{array}\right]
\end{array}\right)
\end{equation*}
with $a_j\in{\mathbb R}$. It is $\sigma$-stable: $\sigma([a_1,a_2,a_3,a_4])=[-a_1,a_2,a_3,a_4]$. In particular, ${\mathfrak h}\cap {\mathfrak t}=\Big\{[0,a_2,a_3,a_4]\Big\}$.
We may assume that the restriction of $\IP{}{}$ to $\ft$ coincides with the standard Euclidean scalar product on $\R^4$. 

The subalgebra $\fl\cap\fk=\fspin(7)\subset \fo(8)$ can be seen as the image of the spin representation $s_7:\fh\cap\fk=\fo(7)\rightarrow \fo(8)$. We are free to compose it with inner automorphisms
of $\fo(8)$. We use this freedom to adapt it to our choice of $\ft$. We have $s_7={s^+}_{|\fo(7)}$, where $s^+:\fo(8)\rightarrow \fo(8)$ is the half-spin representation
viewed as an (outer) automorphism of $\fo(8)$ as in the discussion of Triple 6. Changing $s^+$ by an inner automorphism, if necessary, we may assume that $s^+$
preserves $\ft$ and the standard positive Weyl chamber $\ft^+\subset\ft$. It follows that $s^+$ sends $\fh\cap\ft$ to $\fl\cap\ft$ and the highest weight vector of
the standard representation $[1,0,0,0]$ to the highest weight vector $[\frac{1}{2},\frac{1}{2},\frac{1}{2},\frac{1}{2}]$ of $s^+$. We conclude that
$$ \fl\cap\ft = \Big\{[a_1,a_2,a_3,a_4]\in{\mathfrak t}\mid\sum_{j=1}^{4}a_j=0\Big\}
\mbox{ and }
\fl\cap\fh\cap \ft = \Big\{[0,a_2,a_3,a_4]\mid\sum_{j=2}^4a_j=0\Big\}.$$
In particular, $\fl\cap\ft$ is three-dimensional and thus a Cartan subalgebra of $\fl\cap\fk$.
Moreover, the orthogonal complement of $\fl\cap\fh\cap \ft$ in $\fl\cap \ft$ is spanned by $[-3,1,1,1]$.
We deduce that $\lambda=-\frac{1}{2}$. 

On the other hand, since ${\mathfrak l}$ is complex, one has: 
$${\mathfrak l}=({\mathfrak l}\cap{\mathfrak k})\oplus i({\mathfrak l}\cap{\mathfrak k})=\big(({\mathfrak l}\cap{\mathfrak h})\oplus F\big)\oplus \big(i({\mathfrak l}\cap{\mathfrak h})\oplus iF\big).$$ 
Now, $\sigma$ being conjugate linear, we get two {\it non-compact} eigenvalues $\mu_1=-(-\frac{1}{2})$ and $\mu_2=-1$. Recall that one has ${\mathfrak l}\cap{\mathfrak h}={\mathfrak l}(1)$ from Lemma \ref{basislem}.
\vspace*{0.3cm}

\underline{$\text{\bf Triple 7}$}: $G=SO(8,\mathbb{C})$, $H=SO(7,\mathbb{C})$, $L=Spin(1,7)$.\\
This triple is locally isomorphic to the triple that arises if we switch the roles of $H$ and $L$ in Triple 10. We already know that there is only one non-compact eigenvalue $\mu$.
Again we have that $F\cong \fspin(7)/\fg_2$ is an irreducible representation of $L\cap H\cong G_2$. Thus the cardinality of the spectrum is at most $2$, in contrast to Triple 10, where it is
3. Now Lemma \ref{basislem}(8) implies:
$$\text{Spec}(SO(8,{\mathbb C}),SO(7,{\mathbb C}),Spin(1,7))=\text{Spec}(SO(8,\mathbb{C}),SO_e(1,7), Spin(7,{\mathbb C}))\setminus\{-1\}\ .$$
We obtain from the results for Triple 10:  $\lambda=-\frac{1}{2}$, $\mu=\frac{1}{2}$.
\vspace*{0.3cm}

\underline{$\text{\bf Triple 9}$}: $G=SO_{e}(4,4)$, $H=SO(3)\times SO_{e}(1,4)$, $L=Spin(3,4)$.\\
The triple is a subtriple of Triple 10: All ingredients arise by intersection of the corresponding object in Triple 10 with the real form $SO_e(4,4)$ of $SO(8,{\mathbb C})$.
In particular, the involution $\sigma$ is inherited from the involution of Triple 10.
We realize $SO(8,{\mathbb C})$ such that it leaves the standard bilinear form on $\C^8$ invariant (as we did implicitly already in the discussion of Triple 10). 
Then we consider two different realizations of  $SO_e(4,4)\subset SO(8,{\mathbb C})$:
\begin{itemize}
\item[(a)] elements in $SO(8,{\mathbb C})$ which leave $i{\mathbb R}^4\oplus{\mathbb R}^4$ invariant,
\item[(b)]  elements in $SO(8,{\mathbb C})$ which leave $\big({i\mathbb R}\oplus{\mathbb R}\big)^4$ invariant.
\end{itemize}
In the first case, the subspace ${\mathfrak t}$ defined for Triple $10$ is still a $\sigma$-stable Cartan subalgebra of ${\mathfrak k}$ containing a Cartan subalgebra of $\fl\cap \fk$. In this way we see that $\lambda=-\frac{1}{2}$ is a {\it compact} eigenvalue for Triple 9 as well. On the other hand, in the realization (b), the subspace $i{\mathfrak t}$ is a $\sigma$-stable maximal  abelian subspace of $\fs$ containing a corresponding space for ${\mathfrak l}\cap{\mathfrak s}$. We exhibit $\mu=-1$, $\frac{1}{2}$ as {\it non-compact} eigenvalues. 
To see that there are not more eigenvalues we again use some irreducibility argument. $\fl\cap\fk$ is the direct sum of three simple ideals of type $A_1$, one of which belonging to $\fl\cap\fh$. The remaining part of $\fl\cap\fh$ is again an ideal of type $A_1$ diagonally embedded in the other two ideals of $\fl\cap\fk$. It follows that $F$ is an irreducible $3$-dimensional representation of $\fl\cap\fh$.  For the non-compact part we observe that the non-compact real Lie algebra $\fg_{2(2)}$ is a subalgebra of $\fl$ with maximal compact subalgebra $\fl\cap\fh$
(in fact $\fg_{2(2)}=\fl\cap \sigma(\fl)$). Therefore $\fg_{2(2)}\cap \fs\subset \fl\cap\fs$ is an irreducible ($8$-dimensional) representation of $\fl\cap\fh$.
One checks that its $4$-dimensional orthogonal complement is also irreducible.
\vspace*{0.3cm}

\underline{$\text{\bf Triple 8}$}: $G=SO_{e}(3,4)$, $H=SO(2)\times SO_{e}(1,4)$, $L=G_{2(2)}$.\\
The triple is a subtriple of Triple 9 in the same sense as Triple 9 is a subtriple of Triple 10. We consider again the realizations (a) and (b) above of $SO_e(4,4)$. Let us embed $SO_e(3,4)$ as the corresponding subgroup leaving the basis vector $e_4$ invariant. In the first realization, $\ft_0:=\ft\cap\fo(3,4)$ serves as a useful Cartan subalgebra of $\fk$, while
$i\ft_0\subset\fs$ is a maximal abelian subspace in the second realization. Note that $\ft_0$ consists of all elements of the form
$[b_1,0,b_2, b_3]=:[b_1,b_2,b_3]$ in $\ft$. Now we see that $F\cap\ft_0$ is spanned by the element $[-2,1,1]$. We exhibit the {\it compact} eigenvalue $\lambda=-\frac{1}{3}$
and the non-compact eigenvalues $\mu=\frac{1}{3}, -1$. It remains to check that $E$ decomposes into only three irreducible $L\cap H$-representations.
$F$ is two-dimensional irreducible. The role of $\fg_{2(2)}$ for Triple 9 is now played by $\fs\fu(1,2)\subset\fl$. This leads to two 4-dimensional irreducible submodules of $\fl\cap\fs$.


\appendix
\setcounter{section}{1}
\section{}\label{B}

We retain the notation of the proof of Proposition \ref{crux}. It turns out to be convenient to parametrize regular integral infinitesimal characters for $L=U(1,n)$ by highest weights
of the corresponding finite dimensional representations of  $L$:
$$ \mu =(m_0,m_1,\dots,m_n),\ m_0\ge m_1\ge\dots\ge m_n,\ m_i\in\Z\ .$$
We fix $\mu$. The minimal $K_L$-types of the corresponding discrete series representations with Harish-Chandra parameters $\Lambda_r\in C_r$, $r=0,\dots,n$, are given by
$$ \nu_{r}=(m_r+n-2r, m_0+1,\dots,m_{r-1}+1,m_{r+1}-1,\dots, m_n-1)\ .$$
In particular,
$$ \nu_0=(m_0+n,m_1-1,\dots,m_n-1)\ .$$ 
We denote the corresponding representation space by $F_{\nu_0}$. Let
$$\fu:=\bigoplus_{\alpha\in\Sigma_0}\fl_{-\alpha}\subset \fl_\C .$$
For any $(\fl,K_L)$-module $V$, we set
$$ C^j_\mu(V):=\Hom_{K_L}(F_{\nu_0}\otimes \Lambda^j\fu, V)$$
and, if $V$ is admissible,
$$ \ind_\mu(V)=\sum_{j=0}^{n}(-1)^j\dim C^j_\mu(V)\ .$$
The number $\ind_\mu(V)$ gives the virtual multiplicity of the $K_L$-type $\nu_0$ in the Euler characteristic of the $\fu$-cohomology of $V$.

\begin{lem}\label{hansi}
\begin{enumerate}
\item[(a)]
If $V$ is a principal series module, then $\ind_\mu(V)=0$.
\item[(b)]
If $V$ has an infinitesimal character that does not correspond to $\mu$, then $\ind_\mu(V)=0$.
\item[(c)]
Let $D_r$ be the discrete series module with 
minimal $K_L$-type $\nu_r$. Then $C^j_\mu(D_r)=\{0\}$ unless $j=r=0$. Moreover,
$\dim C^0_\mu(D_0)=1$. In particular,  $\ind_\mu(D_r)=\left\{\begin{array}{cc} 0&r>0\\1&r=0\end{array} \right.$.
\end{enumerate}
\end{lem}

\begin{proof}
Let $V$ be a principal series module induced from an $M$-representation $E_\sigma$. Then by Frobenius reciprocity $C^j_\mu(V)\cong\Hom_{M}(F_{\nu_0}\otimes \Lambda^j\fu, E_\sigma)$. On the other hand, $\fu$ considered as an $M$-module contains the trivial representation, which implies that the virtual $M$-representation $\sum_{j=0}^{n}(-1)^j \Lambda^j\fu$ is zero. Now, (a) follows.

Assertion (b) is a consequence of the Casselman-Osborne theorem for $\fu$-cohomology \cite{CasO75}.

We now prove (c). The $K_L$-representation $\Lambda^j\fu$ is irreducible with highest weight  $(-j,\underbrace{1,\dots,1}_j,0,\dots,0)$. Thus the highest weights of the irreducible 
components of $F_{\nu_0}\otimes \Lambda^j\fu$ are of the form $(m_0+n-j, l_1,\dots,l_n)$ with
\be\label{kreische}
l_i\le m_i
\end{equation}
with equality for exactly $j$ elements $i\in\{1,\dots,n\}$. If a $K_L$-type $(l_0,l_1,\dots, l_n)$ occurs in $D_r$ (we have to add non-negative linear combinations of elements of $\Sigma_r$
to $\nu_r$), then
\be\label{walter}
l_i\ge m_{i-1}+1>m_i,\ i=1,\dots, r;\quad l_i<m_i,\ i=r+1,\dots,n\ .
\end{equation}
The (in)equalities (\ref{kreische}) and (\ref{walter}) can only hold at the same time if $r=0$ and $j=0$. We have $\dim C^0_\mu(D_0)=1$, since the minimal $K$-type $\nu_0$
has multiplicity one in $D_0$.
\end{proof}

The principal series modules with infinitesimal character given by $\mu$ and $\lambda>0$ are naturally parametrized by pairs of integers $(p,q)$, $p,q\ge 0$, $p+q\le n-1$. The corresponding
principal series parameters are
$$ \sigma_{p,q}=(m_p+m_{n-q}-(p-q); m_0+1,\dots,m_{p-1}+1,m_{p+1},\dots,m_{n-q-1},m_{n-q+1}-1,\dots,m_n-1)$$
and
$$ \lambda_{p,q}=m_p-m_{n-q}+n-(p+q)\ .$$
By $H_{p,q}$ we denote the corresponding principal series module. Let $I_{p.q}$ be its unique irreducible quotient. By Langlands classification the modules $I_{p,q}$ and $D_r$ exhaust the irreducible $(\fl,K_L)$-modules of our fixed infinitesimal character.
\begin{lem}\label{fritzsch}
$$ \ind_\mu(I_{p,q})=\left\{\begin{array}{cc} 0&p>0\\(-1)^{n-q}&p=0\end{array} \right. \ .$$
\end{lem}

\begin{proof}
We make a downward induction on $p+q$ using Lemma \ref{hansi} and the additivity of $\ind_\mu$ with respect to exact sequences. First, let $p+q=n-1$.
There is an exact sequence (see e.g. \cite{Co85})
$$ 0\rightarrow D_p\oplus D_{p+1}\longrightarrow H_{p,q}\longrightarrow I_{p,q}\rightarrow 0\ .$$
We obtain 
$$ \ind_\mu(I_{p,q})=-\ind_\mu(D_p)-\ind_\mu(D_{p+1})=\left\{\begin{array}{cc} 0&p>0\\-1&p=0\end{array} \right.$$
This proves our assertion for $p+q=n-1$. Now assume that the assertion is proved for $p+q>l$, $l\le n-2$. We prove it for $p+q=l$. We consider the module $M_{p,q}$ defined
by the exact sequence
$$ 0\rightarrow M_{p,q}\longrightarrow H_{p,q}\longrightarrow I_{p,q}\rightarrow 0\ .$$
The irreducible subquotients of $M_{p,q}$ are precisely $I_{p,q+1}$, $I_{p+1,q}$, and $I_{p+1,q+1}$, see again \cite{Co85}. If $l=n-2$ the latter module has to be 
replaced by $D_{p+1}$.
By induction hypothesis and Lemma \ref{hansi} all these modules except $I_{p,q+1}$ for $p=0$ have vanishing index $\chi_\mu$ and
$$ \ind_\mu(I_{p,q})=-\ind_\mu (M_{p,q})= \left\{\begin{array}{cc} 0&p>0\\-\ind_\mu(I_{0,q+1})&p=0\end{array} \right\}=\left\{\begin{array}{cc} 0&p>0\\(-1)^{n-q}&p=0\end{array} \right.\ .$$
\end{proof}

It is known (\cite{KnS71}) which of the modules $I_{p,q}$ are unitarizable, i.e. which of them are ends of complementary series. In the above parametrization this has been worked out in Proposition 9.7 in \cite{Ol02}. The result depends strongly
on the infinitesimal character given by $\mu=(m_0,\dots,m_n)$.

\begin{lem}\label{riedel}
The module $I_{p,q}$ is unitarizable if and only if $m_p=m_{n-q}$.
\end{lem}

Now we specify the infinitesimal character to that of the discrete series $\pi_\lambda^k$, $k>n$, $\lambda=k-n-2\ell$, $\ell=0,1,\dots,[\frac{k-n-1}{2}]$.
The corresponding minimal $K_L$-type is $\nu_0=(k-\ell,\ell,0,\dots,0)$ which says that 
$$\mu=(k-\ell-n, \ell+1,1,\dots,1)\ .$$
We study the unitarizability of the modules $I_{0,q}$ having the same infinitesimal character as $\pi_\lambda^k$. The following is now easily deduced from Lemma \ref{riedel}
and the definition of $I_{0,n-1}$.

\begin{cor}\label{heidler}
The module $I_{0,q}$ is unitarizable if and only if $\lambda=1$ and $q=n-1$, or $\lambda=1$ and $k=n+1$.
If $\lambda=1$, then $I_{0,n-1}$ is the underlying $(\fl,K_L)$-module of the representation $\pi'$ in Proposition \ref{nichts}.
\end{cor}

Now we can finish the proof of Prop.~\ref{nichts}. We keep the assumption $k>n>0$. 
We express the ingredients of the Dolbeaut complex associated to the complex vector bundle $E\rightarrow Z$ in terms of the multiplicity spaces $N_\Gamma(\pi)$ as follows:
\begin{eqnarray*}
\Omega^{0,j}(Z,E)&=&\left[ C^\infty(\Gamma\backslash L)\otimes F_{\nu_0}\otimes\Lambda^j\bar\fu^* \right]^{K_L}\\
&\cong& \overline{\bigoplus_{\pi\in\widehat L}}N_\Gamma(\pi)\otimes \left[V_{\pi,\infty}\otimes F_{\nu_0}\otimes\Lambda^j\fu \right]^{K_L}\\
&\cong& \overline{\bigoplus_{\pi\in\widehat L}}N_\Gamma(\bar\pi)\otimes \Hom_{K_L}(V_\pi, F_{\nu_0}\otimes\Lambda^j\fu)\\
&\cong& \overline{\bigoplus_{\pi\in\widehat L}}N_\Gamma(\pi)\otimes C^j_\mu(V_{\pi,K})\ .
\end{eqnarray*}
The last isomorphism is conjugate-linear. Note that only finitely many of the above summands can contribute to cohomology (namely those corresponding to harmonic forms).
Now Lemma \ref{hansi}, (b) implies that
$$ \chi(Z,E)=\sum_\pi \dim N_\Gamma(\pi) \chi_\mu(V_{\pi,K}) \ ,$$
where the sum is over all $\pi\in \widehat L$ having the infinitesimal character corresponding to $\mu$. For $\lambda\ge 2$ we obtain by Lemma \ref{hansi}, (c) and Cor.~\ref{heidler}
$$\chi(Z,E)= \dim N_\Gamma(\pi_\lambda^k)=\dim N_\Gamma(\pi_\lambda^{-k})\ ,$$
while for $\lambda=1$ and $k\ge n+2$ (or $n=1$) we obtain using in addition Lemma~\ref{fritzsch}
$$\chi(Z,E)= \dim N_\Gamma(\pi_\lambda^k)-\dim N_\Gamma(\pi')=\dim N_\Gamma(\pi_\lambda^{-k})-\dim N_\Gamma(\overline{\pi'})\ .$$


\begin{thebibliography}{99}

\baselineskip=12pt

\bibitem{Spyros25} S. Afentoulidis-Almpanis and G. Liu, {\it Tensor product decompositions for rank-one spin groups I: Unitary principal series representations.} Preprint arXiv:2502.19968. To appear in Israel J. of Math.

\bibitem{BS 81} M. W. Baldoni Silva, {\em The unitary dual of $Sp(n,1)$.}
 Trans. Amer. Math. Soc. {\bf 48}  (1981), 549--583.

\bibitem{Ban87} E. van den Ban, {\it Invariant differential operators on a semisimple symmetric space 
and finite multiplicities in a Plancherel formula.} Ark. Mat. {\bf 25} (1987), 175--187. 

\bibitem{Ban88} E. van den Ban, {\it The principal series for a reductive symmetric space I. $H$-fixed distribution vectors.} Ann. Sci.\' Ecole Morm. Sup. {\bf 21} (1988), 359--412. 

\bibitem{vdBFJS} E. van den Ban, M. Flensted-Jensen, and H. Schlichtkrull, {\it Harmonic analysis on semisimple symmetric spaces: a survey of some general results.} In: Representation theory and automorphic forms (Edinburgh 1996), pp. 191--217. Sympos. Pure Math. {\bf 61}, AMS, Providence 1997. 

\bibitem{BMM18} W. Ballmann, H. Matthiessen, and S.Mondal,  {\it Small eigenvalues of surfaces: old and new.} ICCM Not. {\bf 6} (2018), 9--24.

\bibitem{BSKZ14} S. Ben Said, K. Koufany, and G. Zhang, {\it Invariant trilinear forms on spherical principal series of real rank one Lie groups.} Internat. J. Math. {\bf 25} (2014), 35 pp.

\bibitem{BeCl05} N. Bergeron and L. Clozel, {\it Spectre automorphe des vari\' et\' es hyperboliques et applications topologiques.} Ast\' erisque {\bf 303} (2005), 218 pp. 

\bibitem{BeCl13} N. Bergeron and L. Clozel, {\it Quelques cons\' equences des travaux d' Arthur pour le spectre et la topologie des vari\' et\' es hyperboliques.} Invent. Math. {\bf 192} (2013), 505--532. 

\bibitem{BeCl17} N. Bergeron and L. Clozel, {\it Sur le spectre et la topologie des vari\' et\' es hyperboliques de congruence: les cas complexe et quaternionien.} Math. Ann. {\bf 368} (2017), 1333--1358. 

\bibitem{Ber88} J. Bernstein, {\it On the support of the Plancherel measure.} J. Geom. Phys. {\bf 5} (1987), 663--710. 

\bibitem{BK14}  J. Bernstein and B. Kr\"otz, {\it Smooth globalizations of Harish-Chandra modules.} Israel J. Math. {\bf 199} (2014), 45--111. 

\bibitem{Bien90} F. Bien, {\em ${\Cal D}$-modules and spherical representations}. 
Mathematical Notes 39. Princeton University Press, 1990.

\bibitem{BoTr25}M. Boche\' nski and A. Tralle, {\it Standard compact Clifford-Klein forms and Lie algebra decompositions.} Transform. Groups {\bf 30} (2025), 1535--1552. 

\bibitem{Bo63} A. Borel, {\it Compact Clifford-Klein forms of symmetric spaces.} Topology {\bf 2} (1963), 111--122. 

\bibitem{BW80} A. Borel and N. Wallach, {\em Continuous cohomology, discrete subgroups, and representations of reductive groups}. 
Second edition. Mathematical Surveys 67. American Mathematical Society, Providence, R.I., 2000.

\bibitem{BO00} U. Bunke and M. Olbrich, {\it The spectrum of Kleinian manifolds.}
J. Funct. Anal. {\bf 172} (2000), 76--164.

\bibitem{Buser} P. Buser, {\em Geometry and spectra of compact Riemann surfaces.}
Progress in Mathematics 106. Birkh\"auser, 1992.





\bibitem{CM62} E. Calabi and L. Markus, {\em Relativistic space forms.} Ann. Math. (2) {\bf 75} (1962), 63--76.

\bibitem{Car89} Y. Carri\`ere, {\em Autour de la conjecture de L. Markus sur les vari\'et\'es affines.} Invent. Math. {\bf 95} (1989), 615--628.

\bibitem{Cas89} W. Casselman, {\it Introduction to the Schwartz space of $\Gamma\backslash G$.} Canad. J. Math. {\bf 41} (1989), 285--320.

\bibitem{Cas89b} W. Casselman, {\it Canonical extensions of {H}arish-{C}handra modules to representations of {$G$}.} Canad. J. Math. {\bf 41} (1989), 385--438.

\bibitem{CasO75} W. Casselman and M. S. Osborne, {\it The $\fn$-cohomology of representations with infinitesimal character.} Compositio Math. {\bf 31} (1975), 219--227.

\bibitem{CL22} Y. Colin de Verdi\`ere and C. Le Bihan, {\it On essential-selfadjointness of differential operators on closed manifolds.}  Ann. Fac. Sci. Toulouse Math. {\bf 31} (2022), 1287--1302.

\bibitem{Co85} D. H. Collingwood,  {\em Representations of rank one Lie groups.} Pitman, Boston 1985.

\bibitem{DGM25} B. Delarue, C. Guillarmou, and D. Monclair, {\em Spectra of Lorentzian quasi-Fuchsian manifolds.} Preprint arXiv:2504.21762.

\bibitem{De99} P. Deligne, {\it Notes on spinors.} Quantum fields and strings: a course for mathematicians, Vol. {\bf 1, 2}, 99--135, Amer. Math. Soc., Providence, RI, 1999.  

\bibitem{Del85} P. Delorme, {\it Injection de modules sph\'eriques pour les espaces sym\'etriques r\'eductifs dans certaines repr\'esentations induites.} 
With an appendix by E. van den Ban and P. Delorme. Lecture Notes in Math., {\bf 1243}, Noncommutative harmonic analysis and Lie groups (Marseille-Luminy, 1985), 
108--143, Springer, Berlin, 1987.

\bibitem{DKKS} P. Delorme, F. Knop, B. Kr\" otz, and H. Schlichtkrull, {\em Plancherel theory for real spherical spaces: construction of the Bernstein morphisms.}
 J. Amer. Math. Soc. {\bf 34}  (2021), 815--908.

\bibitem{Di} J. Dixmier,  {\em Les $C^*$-alg\`ebres et leurs repr\' esentations.} Gauthier-Villars \& Cie, Paris 1964.

\bibitem{DKV79} J.J. Duistermaat, J.A.C.  Kolk, and V.S. Varadarajan, {\it Spectra of compact locally symmetric manifolds of negative curvature.} Invent. Math. {\bf 52} (1979), 27--93.

\bibitem{DG16} S. Dyatlov and C. Guillarmou, {\it Pollicott-Ruelle resonances for open systems.} Ann. Henri Poincar\' e {\bf 17} (2016), 3089--3146. 

\bibitem{Far79} J. Faraut,  {\em Distributions sph\' eriques sur les espaces hyperboliques.} J. Math. Pures App. {\bf 58} (1979), 369--344.

\bibitem{FSSS72} M. Flato, J. Simon, H. Snellman, and D. Sternheimer, {\it Simple facts about analytic vectors and integrability.} Ann. Sci. \'{ E}cole Norm. Sup. {\bf 5} (1972), 423--434. 

\bibitem{Mol17} J. Frahm, {\it Symmetry breaking operators for strongly spherical reductive pairs.}
Publ. Res. Int. Math. Sci. {\bf 59} (2023), 259--337.


\bibitem{GGPS} I. M.~Gel'fand, M. I.~Graev, and I. I.~Pyatetskii-Shapiro, {\em Representation theory and automorphic functions.} Translated from the Russian by K. A. Hirsch. W. B. Saunders Company, Philadelphia, 1969.

\bibitem{Ghy87}E.~Ghys, {\em Flots d'Anosov dont les feuilletages stables sont diff\' erentiables.} Ann. Sci. \' Ecole Norm. Sup. {\bf 20} (1987), 251--270.

\bibitem{Gol85} W. M.~Goldman, {\em Nonstandard Lorentz space forms.} J. Differential Geometry {\bf 21} (1985), 301--308.

\bibitem{GW98} R.~Goodman and N.~R.~Wallach,  {\em Representations and invariants of the classical groups.} Encyclopedia of Mathematics and its Applications, {\bf 68}, Cambridge University Press, Cambridge 1998.
 
\bibitem{GS77} V. Guillemin and S. Sternberg,  {\em Geometric asymptotics.} Mathematical Surveys {\bf 14}. American Mathematical Society, Providence, R.I., 1977.

\bibitem{He14} H. He,  {\em Generalized matrix coefficients for infinite dimensonal unitary representations.} J. Ramanujan Math. Soc. {\bf 29} (2014), 253--272.

\bibitem{HS94} G. Heckman and H. Schlichtkrull, 
{\it Harmonic Analysis and special functions on symmetric spaces.} Perspectives in Mathematics 14, Academic Press, 1994.


\bibitem{Hel84} S. Helgason, {\it Groups and geometric analysis.} Academic Press, 1984.

\bibitem{HoPa74} R. Hotta and R. Parthasarathy,  {\em Multiplicity formulae for discrete series.} Invent. Math. {\bf 26} (1974), 133--178.

\bibitem{HoTa93} R. Howe and E. Tan,  {\em Homogeneous functions on light cones: the infinitesimal structure of some degenerate principal series representations.} Bull. Amer. Math. Soc. {\bf 28} (1993), 1--74.

\bibitem{HuPa06} J.-S. Huang and P. Pand\v zi\' c, {\em Dirac operators in representation theory.} 
Mathematics: Theory and Apllications. Birkh\"auser, 2006.
 
\bibitem{Ja} K. J\" anich,  {\em Differenzierbare $G$-Mannigfaltigkeiten.} Lecture Notes in Mathematics {\bf 59}. Springer-Verlag, Berlin-New York 1968.

\bibitem{KaK25} K. Kannaka and T. Kobayashi, {\em Deformations of standard locally homogeneous spaces.} Preprint arXiv:2507.14832.

\bibitem{Ka12} F. Kassel, {\em Deformations of proper actions on reductive homogeneous spaces.} Math. Ann. {\bf 353} (2012), 
599--632.

\bibitem{KK11} F. Kassel and T. Kobayashi, {\em Stable spectrum for pseudo-Riemannian locally symmetric spaces.} C. R. Math. Acad. Sci. Paris {\bf 349} (2011), 29--33.

\bibitem{KK2} F. Kassel and T. Kobayashi, {\em Poincar\' e series for non-Riemannian locally symmetric spaces.} Adv. Math. {\bf 287} (2016), 
123--236.

\bibitem{KK19} F. Kassel and T. Kobayashi, {\em Invariant differential operators on spherical homogeneous spaces with overgroups.} J. Lie Theory {\bf 29} (2019), 663--754.

\bibitem{KK20} F. Kassel and T. Kobayashi, {\em Spectral analysis on pseudo-Riemannian locally symmetric spaces.} Proc. Japan Acad. Ser. A Math. Sci. {\bf 96} (2020), 69--74.

\bibitem{KK25} F. Kassel and T. Kobayashi, {\em Spectral Analysis on Standard Locally Homogeneous Spaces.} Lecture Notes in Mathematics {\bf 2367}, Springer 2025. (Also available as arXiv:1912.12601.)

\bibitem{KT24} F. Kassel and N. Tholozan, {\em Sharpness of proper and cocompact actions on reductive homogeneous spaces.} Preprint arXiv:2410.08179.

%
\bibitem{Kazh77} D. Kazhdan, {\it Some applications of the Weil representation.}  J. Analyse Math. {\bf 32} (1977), 235--248.

\bibitem{KSa09} D. Kelmer and P. Sarnak, {\it Strong spectral gaps for compact quotients of products of $PSL(2,\R)$.}  J. Eur. Math. Soc. {\bf 11} (2009), 283--313.

\bibitem{Kli96} B. Klingler, {\em Compl\' etude des vari\' etes lorentziennes \` a courbure constante.} Math. Ann. {\bf 306} (1996), 353--370.

\bibitem{Kn86} A. W. Knapp,  {\em Representation theory of semisimple groups. An overview based on examples.} Princeton University Press, 1986.

\bibitem{KnS71} A. W. Knapp and E. M. Stein, {\it Intertwining operators for semisimple groups.} Ann. of Math. {\bf 93} (1971), 489--578.


\bibitem{KKPS} F. Knop, B. Kr\" otz, T. Pecher, and H. Schlichtkrull, {\em Classification of reductive real spherical pairs I. The simple case.}
 Transform. Groups  {\bf 24}  (2019), 67--114.

\bibitem{KKPSII} F. Knop, B. Kr\" otz, T. Pecher, and H. Schlichtkrull, {\em Classification of reductive real spherical pairs II. The semisimple case.}
 Transform. Groups  {\bf 24}  (2019), 467--510.

\bibitem{KKS} F. Knop, B. Kr\" otz, and H. Schlichtkrull, {\em The tempered spectrum of a real spherical space.}
 Acta math. {\bf 218}  (2017), 319--383.

 \bibitem{Kos89} T. Kobayashi, {\it Proper action on a homogeneous space of reductive type.} Math. Ann. {\bf 285} (1989), 249--263.

 \bibitem{Kos96} T. Kobayashi, {\it Discontinuous groups and Clifford-Klein form of pseudo-Riemannian homgeneous manifolds.} In: Algebraic and Analytic Methods in Representation Theory  (H. Schlichtkrull and B. Orsted, eds.), pp. 99--165. Perspectives in Math 17, Academic Press 1996. 


\bibitem{KM14} T. Kobayashi and T. Matsuki, {\it Classification of finite-multiplicity symmetric pairs.} Transform. Groups {\bf 19} (2014), 457--493. 

\bibitem{KO13} T. Kobayashi and T. Oshima, {\it Finite multiplicity theorems for induction and restriction.} Advances in Math. {\bf 248} (2013), 921--944. 

\bibitem{KY05} T. Kobayashi and T. Yoshino, {\it Compact Clifford-Klein forms of symmetric spaces - Revisited.} Pure Appl. Math. Q. {\bf 1} (2005), 591--663. 

\bibitem{KS16} B. Kr\" otz, and H. Schlichtkrull, {\em Multiplicity bounds and the subrepresentation theorem for real spherical spaces.}
 Trans. Amer. Math. Soc. {\bf 368}  (2016), 2749--2762.

\bibitem{Ku81}R. S. Kulkarni, {\it Proper actions and pseudo-Riemannian space forms.}  Adv. in Math. {\bf 40}  (1981),  10--51. 

\bibitem{La66} R. P. Langlands, {\it Dimensions of spaces of automorphic forms.} 
Proc. Sympos. Pure Math. {\bf 9} (1966), 253--257. 

\bibitem{La76} R. P. Langlands, {\it On the functional equations satisfied by Eisenstein series.} 
Lecture Notes in Mathematics. {\bf 544}, Springer, 1976. 

\bibitem{Li93}J.-S. Li, {\it Theta series and construction of automorphic forms.} In: Representation theory of groups and algebras, Contemp. Math. {\bf 145} (1993), 237--248. 

\bibitem{Martin75}R. P. Martin, {\it On the decomposition of tensor products of principal series representations for real-rank one semisimple Lie groups.}  Trans. Amer. Math. Soc. {\bf 201}  (1975),  177--211. 

\bibitem{MO21} S. Mehdi and M. Olbrich, {\em Spectrum of semisimple locally symmetric spaces and admissibility of spherical representations.} 
Adamovi\'c, Dujella et al. (eds.), Lie groups, number theory, and vertex algebras. Representation theory XVI. Dubrovnik, 2019. 
Contemp. Math. {\bf 768} (2021), 55--63.

\bibitem{MP21} S. Mehdi and P. Pand\v zi\'c, {\it Representation theoretic embeddings of Dirac operators.} Represent. Theory {\bf 25} (2021), 760-779. 

\bibitem{Mill76} J.~J.~Millson, {\it On the first Betti number of a constant negatively curved manifold.}  
Ann. of Math. {\bf 104} (1976), 235--247.

\bibitem{MST23} D. Monclair, J.-M. Schlenker, and N. Tholozan, {\em Gromov-Thurston manifols and anti-de Sitter geometry.} Preprint arXiv:2310.12003.

\bibitem{Moo66} C. C. Moore, {\it Ergodicity of flows on homogeneous spaces.} American J. Math. {\bf 88} (1966), 154--178. 

\bibitem{Mos55} G. D. Mostow, {\it Self-adjoint groups.}  
Ann. of Math. {\bf 62} (1955), 44--55.

\bibitem{Mos57} G. D. Mostow, {\it Equivariant embeddings in Euclidean space.}  
Ann. of Math. {\bf 65} (1957), 432--446.

\bibitem{Nel59} E. Nelson, {\it Analytic vectors.}  
Ann. of Math. {\bf 70} (1959), 572--615.

\bibitem{Ol02} M. Olbrich, {\it Cohomology of convex cocompact groups and invariant distributions on limit sets.} Preprint arXiv:0207301.

 \bibitem{Oni62} A. L. Oni\v s\v cik, {\it Inclusion relations between transitive compact transformation
 groups.} (Russian)   Trudy  Moskov. Mat. Ob\v s\v c.  {\bf 11} (1962), 199--242.

 \bibitem{Oni69} A. L. Oni\v s\v cik, {\it Decompositions of reductive Lie groups.} Math. USSR-Sb. {\bf 9} (1969), 515--554.

 \bibitem{Pou72} N. S. Poulsen, {\it On $C^\infty$-vectors and intertwining bilinear forms for representations of Lie groups.}  
J. Funct. Anal. {\bf 9} (1972), 87--120.

\bibitem{Pow74} R. T. Powers, {\it Selfadjoint algebras of unbounded operators II.}  
Trans. Amer. Math. Soc. {\bf 187} (1974), 261--293.

\bibitem{Rag72} M. S. Raghunathan, {\em Discrete subgroups of Lie groups.} Ergebnisse der Mathematik und ihrer Grenzgebiete 68. Springer, Berlin Heidelberg, 1972.

 \bibitem{Rat06} J. G. Ratcliffe, {\it Foundations of hyperbolic manifolds.} 
 Second edition. Graduate Texts in Mathematics 149. Springer, New York, 2006.

 \bibitem{RSI} M. Reed and B. Simon, {\it Methods of modern mathematical physics.I. Functional analysis.} 
 Second edition. Academic Press, New York - London, 1980.

 \bibitem{RSII} M. Reed and B. Simon, {\it Methods of modern mathematical physics.II. Fourier analysis, self-adjointness.} 
Academic Press, New York - London, 1975.

\bibitem{Rep1} J. Repka, {\it Tensor products of unitary representations of $SL_2({\mathbb R})$.} Bull. Amer. Math. Soc. {\bf 82} (1976), 930--932.

\bibitem{Rep2} J. Repka, {\it Tensor products of unitary representations of $SL_2({\mathbb R})$.} Amer. J.  Math. {\bf 100} (1978), 747--774.


\bibitem{SV} Y. Sakellaridis and A. Venkatesh, {\it Periods and harmonic analysis on spherical varieties.} Asterisque {\bf 396} (2017), 360 pp.

\bibitem{Sch71} H. H. Sch\" afer, {\it Topological vector spaces.} Graduate Texts in Mathematics 3, Springer, Third printing, 1971.

\bibitem{Sch87} H. Schlichtkrull, {\it Eigenspaces of the Laplacian on hyperbolic spaces: composition series and integral transforms.} J. Funct. Anal. {\bf 70} (1987), 194--219.

\bibitem{STV18} H. Schlichtkrull, P. Trapa, and D. A. Vogan, {\it Laplacians on spheres.} Sao Paulo J. Math. Sci. {\bf 12} (2018), 295--358.


\bibitem{Sh94} N. Shimeno, {\it The Plancherel formula for spherical functions with a one-dimensional K-type on a simply connected simple Lie group of Hermitian type.} J. Funct. Anal.  {\bf 121} (1994), 330--388.


\bibitem{To19} K. Tojo, {\it Classification of irreducible symmetric spaces which admit standard compact Clifford-Klein forms.}
 Proc. Japan Acad. Ser. A Math. Sci.  {\bf 95}  (2019),  11--15.

\bibitem{Vi01} E. B. Vinberg, {\it Commutative homogeneous spaces and coisotropic symplectic actions.} Russian Math. Surveys {\bf 56} (2001), 1--60.

\bibitem{Wa76} N. R. Wallach, {\em On the Selberg trace formula in the case of a compact quotient.} Bull. Amer. Math. Soc. {\bf 82}  (1976),  171--195.

\bibitem{Wa88} N. R. Wallach, {\em Real reductive groups I.} Pure and Applied Mathematics 132. Academic Press, 1988.

 \bibitem{Wa92} N. R. Wallach, {\it Real reductive groups II.} Pure and Applied Mathematics 132. Academic Press, 1992.



\bibitem{Ze98} A. Zeghib, {\it On closed anti-de Sitter spacetimes.} Math. Ann.  {\bf 310} (1998), 695--716.

 \bibitem{Zhe73} D.~P.~Zhelobenko, {\it Compact Lie groups and their representations.} Transl. Math. Monogr. 40, Amer. Math. Soc., Providence 1973.

\end{thebibliography}
\end{document}